\documentclass{amsart}
\usepackage[T1]{fontenc}
\usepackage{amsmath,a4wide}
\usepackage{amssymb}
\usepackage{amscd}
\usepackage{epsfig}
\usepackage[T1]{fontenc}
\usepackage{epsfig}
\usepackage{amsmath}
\usepackage{graphicx}
\usepackage{subfigure}
\usepackage{mathrsfs}
\usepackage{epstopdf}
\usepackage{color}
\usepackage{booktabs}
\usepackage{ifpdf}
\usepackage{needspace}
\usepackage{cite}
\usepackage{amsthm}
\usepackage{amsmath}
\usepackage{enumitem}
\usepackage{algorithm,algorithmic}

\usepackage{pgfplots}
\usepgfplotslibrary{groupplots}
\pgfplotsset{compat=1.12}

\usepackage{tikz}
\usetikzlibrary{arrows,intersections}
\usepackage{boldline}
\usepackage{multirow}
\usepackage[margin=1.32in]{geometry}

\definecolor{sblue}{rgb}{0.16, 0.35, 0.59}
\usepackage{varioref}
\usepackage{xr-hyper}
\usepackage{hyperref}
\hypersetup{
    colorlinks,linkcolor=blue,citecolor=blue,filecolor=red,urlcolor=blue}

\def\N{\mathbb{N}}

\def\R{\mathbb{R}}

\def\P{\mathbb{P}}
\definecolor{darkgreen}{rgb}{0,0.35,0}

\allowdisplaybreaks

\makeatletter
\DeclareRobustCommand*\cal{\@fontswitch\relax\mathcal}
\makeatother

\newtheorem{thm}{Theorem}[section]
\newtheorem{pro}[thm]{Proposition}

\newtheorem{lem}[thm]{Lemma}
\numberwithin{equation}{section}

\theoremstyle{definition}
\newtheorem{rem}[thm]{Remark}
\newtheorem{dfn}[thm]{Definition}
\newtheorem{as}[thm]{Assumption}

\usepackage{xparse}
\makeatletter
\RenewDocumentCommand{\title}{om}{%
    \IfNoValueTF{#1}
    {\gdef\shorttitle{}}
    {\gdef\shorttitle{#1}}%
    \gdef\@title{#2}%
}
\makeatother

\makeatletter
\def\@tocline#1#2#3#4#5#6#7{\relax
    \ifnum #1>\c@tocdepth
    \else
    \par \addpenalty\@secpenalty\addvspace{#2}%
    \begingroup \hyphenpenalty\@M
    \@ifempty{#4}{%
        \@tempdima\csname r@tocindent\number#1\endcsname\relax
    }{%
        \@tempdima#4\relax
    }%
    \parindent\z@ \leftskip#3\relax \advance\leftskip\@tempdima\relax
    \rightskip\@pnumwidth plus4em \parfillskip-\@pnumwidth
    #5\leavevmode\hskip-\@tempdima
    \ifcase #1
    \or\or \hskip 2em \or \hskip 2em \else \hskip 3em \fi%
    #6\nobreak\relax
    \hfill\hbox to\@pnumwidth{\@tocpagenum{#7}}\par%
    \nobreak
    \endgroup
    \fi}
\makeatother

\title{Markov-Nash equilibria in mean-field games \\
    under model uncertainty}
\author{Johannes Langner}
\address{MICS, CentraleSup\'{e}lec, Universit\'{e} Paris-Saclay}
\email{johannes.langner@centralesupelec.fr}

\author{Ariel Neufeld}
\address{Division of Mathematical Sciences, Nanyang Technological University}
\email{ariel.neufeld@ntu.edu.sg}

\author{Kyunghyun Park}
\address{Division of Mathematical Sciences, Nanyang Technological University}
\email{kyunghyun.park@ntu.edu.sg}

\thanks{\textit{MSC2020 Subject Classification:} 91A16, 91A10, 90C40, 90C17}
\thanks{\textit{Funding:} J.~Langner is funded by the Deutsche Forschungsgemeinschaft (DFG, German Research Foundation) – 471178162. A.~Neufeld gratefully acknowledges support by the MOE AcRF Tier 1 Grant RG109/25. K.~Park acknowledges financial support by the Presidential Postdoctoral Fellowship of Nanyang Technological University and the National Research Foundation of Korea (grant DOI: RS-2025-02633175).}

\date{\today.}

\begin{document}

\begin{abstract}
    We propose and analyze a framework for mean-field Markov games under model uncertainty. In this framework, a state-measure flow describing the collective behavior of a population affects the given reward function as well as the \textit{unknown} transition kernel of the representative agent. The agent's objective is to choose an optimal Markov policy in order to maximize her worst-case expected reward, where worst-case refers to the most adverse scenario among all transition kernels considered to be feasible to describe the unknown true law of the environment. We prove the existence of a mean-field equilibrium under model uncertainty, where the agent chooses the optimal policy that maximizes the worst-case expected reward, and the state-measure flow aligns with the agent's state distribution under the optimal policy and the worst-case transition kernel. Moreover, we prove that for suitable multi-agent Markov games under model uncertainty the optimal policy from the mean-field equilibrium forms an approximate Markov-Nash equilibrium whenever the number of agents is large enough.
\end{abstract}

\maketitle

\section{Introduction}\label{sec:intro}
Mean-field games introduced by \cite{lasry2007mean,Huang2006nashCertainty} analyze decision-making and interactions of strategic agents within populations. Under the assumption that all agents of a population have the same transition probabilities and reward function and that their interactions only depend on the empirical distribution of all agents, one can simplify the model by approximating the finite agent game by a suitable mean-field game. This framework has led to a wide range of applications, including in finance and economics (e.g., \cite{carmona2015mean,carmona2017mean,lacker2019mean,carmona2020applications,lauriere2022convergence}), crowd motion dynamics (e.g., \cite{lachapelle2011mean,huang2021dynamic}), and epidemiology (e.g., \cite{aurell2022optimal,elie2020contact}).

As a prominent discrete-time mean-field games model, consider a mean-field Markov game denoted by $(S,A,\mu^o,p,r)$: Let $(S,A)$ be state and action spaces and denote by ${\cal P}(S)$\;and\;${\cal P}(A)$ the set of probability measures on $S$\;and\;$A$, respectively. Furthermore, let $\mu^o\in {\cal P}(S)$ be an initial population distribution, $p:S\times A\times {\cal P}(S)\mapsto{\cal P}(S)$ be a transition kernel, and $r:S\times A\times S \times {\cal P}(S)\mapsto \mathbb{R}$ be a one-step reward function.
Assume that a representative agent aims to maximize her total expected reward until the terminal time $T$ by choosing a Markov policy $\pi_{0:T}=(\pi_0,\dots,\pi_{T-1})$ (i.e., a sequence of stochastic kernels $\pi_t:S\mapsto {\cal P}(A)$, $t=0,\dots,T-1$). Given a population measure flow $\mu_{0:T}=(\mu_0,\dots,\mu_{T-1})$ with $\mu_0=\mu^o$ (i.e., a sequence of $\mu_t\in {\cal P}(S)$, $t=0,\dots,T-1$), the central objective the agent faces is to solve the following Markov decision~problem
\begin{align}\label{eq:intro_MDP}
    \sup_{\pi_{0:T}}\mathbb{E}^{\mathbb{P}}\left[\sum_{t=0}^{T-1}r(s_t,a_t,s_{t+1}, \mu_t)\right],
\end{align}
where for given $\pi_{0:T}$, $\mathbb{P}$ is the probability measure (that depends on $\mu_{0:T},\pi_{0:T},$ and $p$) under which the agent's state and action configurations evolve as follows: for every $t= 0, \dots, T-1$
\begin{align}\label{eq:intro_kernel}
    s_0\sim \mu_0(\cdot),\quad a_t\sim \pi_t(\cdot|s_t),\quad s_{t+1}\sim p(\cdot|s_t,a_t,\mu_t).
\end{align}

In this setting, a mean-field equilibrium consists of a Markov policy and a measure flow $(\mu^*_{0:T}, \pi^*_{0:T})$  satisfying that $\pi^*_{0:T}$ is a maximizer of \eqref{eq:intro_MDP} given $\mu_{0:T}^*$, and $\mu_{0:T}^*$ is consistent with the state distribution of the agent acting optimally via $\pi_{0:T}^*$, i.e., $\mu_0^*=\mu^o$ and for $t=0,\dots,T-2$
\begin{align}\label{eq:intro_consist}
    \mu_{t+1}^{*}(ds_{t+1})=\int_{S\times A} p(ds_{t+1}|s_t,a_t,\mu_t^*)\pi_t^*(da_t|s_t)\mu_t^*(ds_t).
\end{align}
In most cases, a mean-field equilibrium attains an approximate Nash equilibrium for an analogous game with a finite number of agents, known as the so-called Nash certainty equivalence principle \cite{Huang2006nashCertainty,bensoussan2013mean,carmona2018probabilistic,cardaliaguet2010notes}. We refer to \cite{gomes2010discrete,elliott2013discrete,gast2012mean,moon2015discrete,saldi2018markov,saldi2019approximate,Elie2020meanfield} for a few articles studying discrete-time mean-field games similar to the setting $(S,A,\mu^o,p,r)$ described above.

Mean-field games commonly involve a significant assumption that the model environment represented by the transition kernel $p$ in the above model $(S,A,\mu^o,p,r)$ is perfectly known to all agents. However, when implemented in practice, the specifics of the model environment are a priori unclear. While some estimation techniques can approximate a ground truth on, e.g., the transition kernel closely, in many cases there exists a margin of misspecification. This might result in an equilibrium that is not consistent with the behavior of large populations in real situations.

As a remedy to {\it model uncertainty}, a number of researchers in various fields have adopted the so-called worst-case (or robust) approach introduced by \cite{gilboa1989maxmin,chen2002ambiguity,epstein1994intertemporal,dow1992uncertainty}.
Here, worst-case refers to considering the most adverse scenario among all probabilities deemed as feasible to describe the unknown law characterizing the environment. The aim of this article is to propose and analyze a framework for mean-field Markov games under model uncertainty, which can be considered as a robust analog of $(S,A,\mu^o,p,r)$ described in \eqref{eq:intro_MDP}-\eqref{eq:intro_consist}.

To that end, let us describe our mean-field Markov game under model uncertainty, which we denote by $(S,A,\mu^o,\mathfrak{P}_{0:T},r)$: Fix $T\in \mathbb{N}$ and let $(S,A,\mu^o,r)$ be the same as the ones given in $(S,A,\mu^o,p,r)$ described above. Furthermore, let $\mathfrak{P}_{0:T}$ be a sequence of {\it set-valued} maps given for every $t=0,\dots,T-1$ by
\begin{align}\label{eq:intro_set_val}
    \mathfrak{P}_{t}:S \times A \times \mathcal{P}(S) \ni (s_t, a_t, \mu_t) \twoheadrightarrow \mathfrak{P}_{t}(s_t, a_t, \mu_t) \subseteq {\cal P}(S).
\end{align}
Then given $({\mu}_{0:T},{\pi}_{0:T})$, denote by $\mathcal{Q}({\mu}_{0:T},{\pi}_{0:T})$ the set of all probability measures $\mathbb{P}$ under which there exists
a sequence of transition kernels $p_{0:T}=(p_0,\dots,p_{T-1})$ satisfying that for every $t=0,\dots,T-1$ and every $(s_t,a_t,\mu_t)\in S \times A\times {\cal P}(S)$
\begin{align}\label{eq:intro_set_val_2}
p_{t}(\cdot | s_t, a_t, \mu_t)\in  \mathfrak{P}_{t}(s_t, a_t, \mu_t),
\end{align}
and the agent's state and action configurations evolve~as follows: for every $t =0, \dots, T-1$
\begin{align}\label{eq:intro_set_val_3}
    s_0\sim \mu_0(\cdot),\quad a_t\sim \pi_t(\cdot|s_t),\quad s_{t+1}\sim p_t(\cdot|s_t,a_t,\mu_t).
\end{align}
In other words, instead of fixing a transition kernel $p:S\times A\times {\cal P}(S)\mapsto{\cal P}(S)$, we consider a set valued map $\mathfrak{P}_t:S \times A \times \mathcal{P}(S) \twoheadrightarrow {\cal P}(S)$ where given $(s_t,a_t,\mu_t)\in S \times A \times \mathcal{P}(S)$, each element of the set $\mathfrak{P}_t(s_t,a_t,\mu_t)$ is considered as a candidate probability measure on $S$ derived from the true but unknown transition kernel. This setting is inspired by \cite{neufeld2023markov,neufeld2024non} which analyzed Markov decision problems under model uncertainty (but without a mean-field measure flow).

Now, given $\mu_{0:T}$ with $\mu_0=\mu^o$, the central objective an agent faces under model uncertainty is to solve the following robust (or worst-case) optimization problem
\begin{align}\label{eq:value_intro}
    V(\mu_{0:T})=\sup_{ \pi_{0:T}}\;\inf_{\mathbb{P}\in {\cal Q}({\mu_{0:T},  \pi_{0:T}})} \mathbb{E}^{\mathbb{P}} \left[\sum_{t=0}^{T-1}r(s_t,a_t,s_{t+1},\mu_t)\right].
\end{align}
We note that the set-valued maps $\mathfrak{P}_{0:T}$ given in \eqref{eq:intro_set_val} induce distributional uncertainty represented by the set $\mathcal{Q}({\mu}_{0:T},{\pi}_{0:T})$, and \eqref{eq:value_intro} and \eqref{eq:intro_MDP} coincide when $\mathfrak{P}_{0:T}$ are singleton-valued.

In this setting, we say
\begin{align}\label{eq:intro_MFE}
    ({\mu}_{0:T}^*,{\pi}_{0:T}^*,{p}_{0:T}^*)
\end{align}
is a mean-field equilibrium of $(S,A,\mu^o,\mathfrak{P}_{0:T},r)$ (see Definition \ref{dfn:robust_mfe}) if the Markov policy  ${\pi}_{0:T}^*$ is optimal to the robust optimization problem $V(\mu_{0:T}^*)$, the transition kernel ${p}_{0:T}^*$ corresponds to the worst-case kernel of $V(\mu_{0:T}^*)$ under $(\mu_{0:T}^*,{\pi}_{0:T}^*)$, and the state-measure flow $\mu_{0:T}^*$ aligns with the agent's state distribution under $({\pi}_{0:T}^*,{p}_{0:T}^*)$, i.e., $\mu_0^*=\mu^o$ and for every~$t=0,\dots,T-2$,
\begin{align}\label{eq:intro_msr_flow}
    \mu_{t+1}^{*}(ds_{t+1})=\int_{S\times A} p_t^*(ds_{t+1}|s_t,a_t,\mu_t^*)\pi_t^*(da_t|s_t)\mu_t^*(ds_t).
\end{align}

\vspace{0.5em}
The main contribution of this paper is twofold:
\begin{itemize}[leftmargin=2.em]
    \item [$\cdot$] In Theorem \ref{thm:MFE}, we prove the existence of a mean-field equilibrium $({\mu}_{0:T}^*,{\pi}_{0:T}^*,{p}_{0:T}^*)$ of
    the mean-field Markov game $(S,A,\mu^o,\mathfrak{P}_{0:T},r)$ described in \eqref{eq:intro_set_val}-\eqref{eq:intro_msr_flow}.

    \item[$\cdot$] We show in Theorem \ref{thm:MNE0} that the optimal Markov policy $\pi_{0:T}^*$ from the mean-field equilibrium of $(S, A, \mu^o, \mathfrak{P}_{0:T}, r)$ forms an approximate Markov-Nash equilibrium (see Definition~\ref{dfn:MNE}) of a multi-agent Markov game under model uncertainty in the sense that the policy $\pi_{0:T}^*$ is (almost) a maximizer for the worst-case objectives of all agents in the multi-agent Markov game (see \eqref{eq:worst_nash}) whenever the number of agents is large enough.
\end{itemize}

In particular, the proof of Theorem~\ref{thm:MFE} combines a dynamic programming principle with a fixed point theorem to establish both optimality and consistency with the population dynamics in \eqref{eq:intro_msr_flow}. The argument relies on the conditions on $\mathfrak{P}$ and $r$ stated in Assumption~\ref{as:msr}. For the approximate Markov-Nash equilibrium result in Theorem~\ref{thm:MNE0}, we additionally require continuity of the mean-field worst-case kernel with respect to the population distribution and uniform approximation of this kernel by the marginals of suitably selected finite-agent worst-case joint kernels when all agents except one follow the mean-field equilibrium policy, while the remaining agent may deviate to an arbitrary Markov policy (see Assumption~\ref{as:weak_conv_ptb}); see Remark~\ref{rem:weak_conv_ptb} for a detailed discussion.

Section~\ref{sec:len} presents an application to a repeated investment problem under model uncertainty and verifies that all these conditions are satisfied. Both main results, Theorem~\ref{thm:MFE} and Theorem~\ref{thm:MNE0}, therefore apply to this problem.

Section~\ref{sec:numeric} illustrates the numerical approximation of mean-field equilibria using the value iterative scheme proposed in Algorithm~\ref{alg:MFE}. We apply this scheme to models of crowd motion (Section~\ref{subsec:crowd}) and illiquid financial markets (Section~\ref{subsec:trade}). In both examples, the uncertainty correspondences $\mathfrak{P}_{t}$ in \eqref{eq:intro_set_val} are specified as Wasserstein balls around a reference transition kernel (see Definition~\ref{dfn:exm:cro} and Definition~\ref{dfn:exm:tra}).

{\it Related literature.} Classic mean-field games (i.e., without uncertainty) are described both in continuous-time (see, e.g., \cite{Huang2006nashCertainty,tembine2013risk,huang2010large,gomes2013continuous,lacker2016general,lacker2022case,lacker2019mean,aurell2022optimal,delarue2020master}) and in discrete-time (see, e.g., \cite{adlakha2015mean,biswas2015mean,gomes2010discrete,gast2012mean,elliott2013discrete,moon2015discrete,moon2016mean,nourian2013linear,saldi2018markov,saldi2019approximate,Elie2020meanfield}); we refer to \cite{carmona2018probabilistic,bensoussan2013mean,gomes2014mean,Lauriere2024meanfield} for survey papers including both settings. We also refer to \cite{gast2011mean,carmona2023model,bauerle2023mean,gu2021mean,gu2023dynamic,motte2022mean} for mean-field control problems in a Markov decision process framework (which corresponds to cooperative models).

In continuous-time settings, several articles have explored mean-field games under distributional or parametric uncertainty (see, e.g., \cite{moon2016linear,FT2026, moon2016robust,huang2013mean,bauso2016robust}). Notably, our notion of a mean-field equilibrium under model uncertainty (described in \eqref{eq:intro_MFE}--\eqref{eq:intro_msr_flow}; see also Definition~\ref{dfn:robust_mfe} and Theorem \ref{thm:MFE}) aligns with those found in continuous-time frameworks (see, e.g., \cite[Proposition~3]{bauso2016robust}, \cite[Theorem~3.2]{huang2013mean}),
where our robust optimization problem \eqref{eq:value_intro} corresponds in their papers to a forward backward system consisting of a Hamilton--Jacobi--Bellman--Isaacs equation, whereas our measure flow \eqref{eq:intro_msr_flow} corresponds in their papers to a Fokker-Planck equation under the associated worst-case measure or parameter. Moreover, \cite[Theorem 6]{moon2016robust} establishes an approximate Nash equilibrium under model uncertainty, which is consistent with ours given in Theorem \ref{thm:MNE0}.

The work \cite{LZZ2026}, which appeared after the first arXiv version of this manuscript, considers a related framework for discrete-time mean-field games under model uncertainty. A key difference lies in the treatment of the population measure flow. While we prescribe this flow in the representative agent's robust optimization problem and require consistency with the optimal policy and a corresponding worst-case transition kernel at equilibrium, \cite{LZZ2026} specify the evolution of the population measure flow for every admissible population policy and transition kernel, incorporating it directly into the robust optimization problem. We refer to Remark~\ref{rem:comp:LZZ} for a detailed comparison. 

While certain proof techniques in our paper bear similarities to \cite{saldi2018markov, saldi2019approximate} which consider mean-field Markov games in a discrete-time setting {\it but without model uncertainty}, the consideration of model uncertainty introduces significant distinctions. Specifically, due to the set-valued maps $\mathfrak{P}_{0:T}$ given in \eqref{eq:intro_set_val}, we cannot directly apply certain existing arguments (including the dynamic programming principle and the fixed point approach). Instead, we establish a robust (i.e., max-min) version of the dynamic programming principle, which constitutes a variant of \cite{neufeld2023markov}. We then propose and study a robust analog of the fixed point approach based on the work of \cite{jovanovic1988anonymous}.
Moreover, we establish the dynamic programming principle for the multi-agent Markov game under model uncertainty and characterize the worst-case measures appearing in both the mean-field and multi-agent Markov games to establish the existence of an approximate Markov-Nash~equilibrium.

\section{Model description}
\subsection{Notation and preliminaries}\label{sec:notat_prelimi}
Throughout this article we work with
Borel spaces. If $X$ is such a space, we denote by ${\cal B}_X$ its Borel $\sigma$-field and ${\cal P}(X)$ the set of all probability measures on $X$  implicitly assumed to be equipped with the topology induced by the weak convergence, i.e., for any $\mathbb{P}\in {\cal P}(X)$ and any $(\mathbb{P}^n)_{n\in\mathbb{N}}\subseteq {\cal P}(X)$, we have
\begin{align}\label{eq:topology_w}
    \mathbb{P}^n \rightharpoonup \mathbb{P}\;\; \mbox{as $n\rightarrow \infty$}\; \Leftrightarrow \;\; \lim_{n\rightarrow \infty} \int_X f(\omega) \mathbb{P}^n(d\omega) = \int_X f(\omega) \mathbb{P}(d\omega)\;\;\mbox{for any $f \in C_b(X;\mathbb{R})$},
\end{align}
where $C_b(X;\mathbb{R})$ is the set of all continuous and bounded functions from $X$ to $\mathbb{R}$.

If $X$ is compact, the weak topology given in \eqref{eq:topology_w} is equivalent to the topology induced by the $1$-Wasserstein distance $d_{W_1}(\cdot,\cdot)$ which we recall to be the following: For any $\mu,\nu\in {\cal P}(X)$, denote by $\operatorname{Cpl}(\mu,\nu)\subseteq {\cal P}(X\times X)$ the subset of all probability measures on $X\times X$ with first marginal $\mu$ and second marginal $\nu$. Then the $1$-Wasserstein distance between $\mu$ and $\nu$ is defined~by
\[
d_{W_1}(\mu,\nu):= \inf_{\gamma \in \operatorname{Cpl}(\mu,\nu)} \int_{X\times X} |x-y| \gamma (dx,d y),
\]
where $|\cdot|$ is the Euclidean norm.

In particular, if we further assume that $X$ is a finite subset in a Euclidean space and denote by $n(X)$ its cardinality,  then ${\cal P}(X)$ can be identified with a simplex in $\mathbb{R}^{n(X)}$, i.e., $\mu\in {\cal P}(X)$ can be treated as an $n(X)$-dimensional vector $(w^\mu_{1},\dots,w^\mu_{n(X)})\in \mathbb{R}^{n(X)}$ with non-negative coordinates $(w_{i}^{\mu})_{i=1,\dots,n(X)}$ which sum up to one.

For each $t\in\mathbb{N}$, we use the abbreviation $X^t:=X\times \cdots \times X$ for the $t$-times Cartesian product of the set $X$, where we endow $X^{t}$ with the corresponding product topology. In analogy, we use $({\cal P}(X))^{t}$ for the corresponding product of ${\cal P}(X)$. Given a sequence of probability measures $(\mathbb{P}_s,\dots,\mathbb{P}_{t-1})\in ({\cal P}(X))^{t-s}$ and $0\leq s < t$, we use the following abbreviation $\mathbb{P}_{s:t}:=(\mathbb{P}_s,\dots,\mathbb{P}_{t-1}).$ The same convention applies to a sequence of other quantities.

\subsection{Mean-field Markov games under model uncertainty}
We specify what we mean by mean-field Markov games under model uncertainty. Let us consider a representative agent who, at each time $t$, observes a state $s_t$ and takes an action $a_t$, whereas a probability measure $\mu_{t}$ describes the population distribution at time $t$.

\begin{dfn}[Mean-field Markov game]\label{dfn:corr} Fix a time horizon $T\in\mathbb{N}$. A mean-field Markov game under model uncertainty, say $(S,A,\mu^{o},\mathfrak{P}_{0:T}, r)$, comprises the following:
    \begin{enumerate}[leftmargin=3.em]
        \item [(i)] $(S,{\cal B}_S)$ and $(A,{\cal B}_A)$ are Borel spaces for the state and action spaces, respectively.
        \item [(ii)] $\mu^{o}\in {\cal P}(S)$ is a given initial distribution for the initial state, which we denote by $s_0$.
        \item [(iii)] For every $t=0,\dots,T-1$, $\mathfrak{P}_{t}:S \times A \times \mathcal{P}(S) \ni (s_t, a_t, \mu_t) \twoheadrightarrow \mathfrak{P}_{t}(s_t, a_t, \mu_t) \subseteq {\cal P}(S)$ is a correspondence (i.e., a set-valued map) at time $t$, inducing {distributional uncertainty} in the next-state configuration.
        \item [(iv)] $r: S\times A \times S \times {\cal P}(S) \mapsto \mathbb{R}$ is a one-step Borel-measurable reward function.
    \end{enumerate}
\end{dfn}

We proceed to describe the set of policies and the set of uncertain probability measures.
\begin{dfn}\label{dfn:corr2}
    Let $(S,A,\mu^o,\mathfrak{P}_{0:T}, r)$ be given in Definition~\ref{dfn:corr}.
    \begin{enumerate}[leftmargin=3.em]
        \item [(i)] Define by $\Pi$ the set of all sequences of Markov policies $\pi_{0:T}$
        such that for $t=0,\dots,T-1$,  $\pi_t:S \ni s_t\mapsto \pi_t(\cdot |s_t)\in {\cal P}(A)$
        is a so-called Markov kernel.
        \item [(ii)]
        Given $({\mu}_{0:T},{\pi}_{0:T})\in ({\cal P}(S))^{T}\times \Pi$ satisfying $\mu_0=\mu^o$, we define by $\mathcal{Q}({\mu}_{0:T},{\pi}_{0:T})\subseteq {\cal P}(S\times (S\times A)^{T})$ the subset of all probability measures $\mathbb{P}:=\mu_0\otimes \mathbb{P}_{({\mu_0,\pi_0,p_0})}\otimes\cdots \otimes \mathbb{P}_{(\mu_{T-1},\pi_{T-1},{p}_{T-1})}$
        such that\footnote{For every $t=0,\dots,T-1$, $\mu_0\otimes \mathbb{P}_{(\mu_0,\pi_0,p_0)}\otimes\cdots \otimes \mathbb{P}_{(\mu_{t},\pi_{t},p_t)}$ denotes an element in $\mathcal{P}(S \times (S \times A)^{t+1})$ satisfying that for every $B \in \mathcal{B}_{S \times (S \times A)^{t+1}}$,
        \begin{align*}
            &\mu_0\otimes \mathbb{P}_{(\mu_0,\pi_0,p_0)}\otimes\cdots \otimes \mathbb{P}_{(\mu_{t},\pi_{t},p_t)}(B)\\ %
            &\;\; := \int_{S} \int_{S\times A} \cdots \int_{S\times A} {\bf 1}_{\{(s_0,s_1,a_0,\dots,s_{t+1},a_{t})\in B\}}\; \mathbb{P}_{(\mu_{t},\pi_{t},p_t)}(ds_{t+1},da_{t} |s_{t}) \cdots \mathbb{P}_{(\mu_0,\pi_0,p_0)}(ds_1,da_0 |s_0) \mu_0(ds_0).
        \end{align*}}
        for every $t=0,\dots,T-1$,
        \[
        \quad \quad \mathbb{P}_{(\mu_t,\pi_t,p_t)}:S\ni s_t \mapsto \mathbb{P}_{(\mu_t,\pi_t,p_t)}(ds_{t+1},da_t |s_t):=p_{t}(ds_{t+1}| s_t, a_t, \mu_t) \pi_{t}(d a_t| s_t)
        \]
        is a stochastic kernel\footnote{Throughout the paper, a stochastic kernel $p$ on $X_{2}$ given $X_{1}$, for some Borel spaces $X_{1}$ and $X_{2}$, is defined as a Borel-measurable mapping from $X_{1}$ to $\mathcal{P}(X_{2})$.} on $S\times A$ given $S$, and $p_t:S\times A\times {\cal P}(S) \mapsto {\cal P}(S)$ is a stochastic kernel satisfying that for every $(s_t,a_t,\mu_t)\in S \times A\times {\cal P}(S)$,
        \[
        p_{t}(ds_{t+1} | s_t, a_t, \mu_t)\in  \mathfrak{P}_{t}(s_t, a_t, \mu_t).
        \]
    \end{enumerate}
\end{dfn}

Denote by $V:({\cal P}(S))^{T}\ni \mu_{0:T}\mapsto V(\mu_{0:T})\in\mathbb{R}$ the robust optimization problem defined by
\begin{align}\label{dfn:value_mfg}
    V( \mu_{0:T}):= \sup_{ \pi_{0:T} \in \Pi} J(\mu_{0:T},\pi_{0:T}),
\end{align}
where the worst-case objective $J:({\cal P}(S))^{T} \times \Pi\ni(\mu_{0:T},\pi_{0:T})\mapsto J(\mu_{0:T},\pi_{0:T})\in \mathbb{R}$ is given by
\begin{align}\label{dfn:worst_value_mfg}
    J(\mu_{0:T},\pi_{0:T}):=\inf_{\mathbb{P}\in {\cal Q}({ \mu_{0:T},  \pi_{0:T}})} \mathbb{E}^{\mathbb{P}} \left[\sum_{t=0}^{T-1}r(s_t,a_t,s_{t+1},\mu_t)\right].
\end{align}

We now introduce what we refer to as a mean\;field equilibrium under model uncertainty.

\begin{dfn}[Mean-field equilibrium]\label{dfn:robust_mfe}
    We call $({\mu}_{0:T}^*,{\pi}_{0:T}^*,{p}_{0:T}^*)$ a~{mean-field equilibrium} of the mean-field Markov game $(S,A,\mu^o,\mathfrak{P}_{0:T}, r)$ (see Definition \ref{dfn:corr}) if the following conditions~hold:
    \begin{enumerate}[leftmargin=2.5em]
        \item[(i)] $({\pi}^*_{0:T},{p}^*_{0:T})$ are optimal for $V({\mu}_{0:T}^*)$, i.e., $\pi^{*}_{0:T}$ is the optimal Markov policy of $V({\mu}_{0:T}^*)$ and $p^{*}_{0:T}$ is the worst-case transition kernel of $V(\mu^{*}_{0:T})$ under $({\mu}_{0:T}^*,{\pi}^*_{0:T})$ in the sense that
        \begin{align*}
            \begin{aligned}
                V( \mu_{0:T}^*) = J(\mu_{0:T}^*,\pi_{0:T}^*) &= \sup_{\pi_{0:T}\in \Pi}\mathbb{E}^{\mathbb{P}({\mu_{0:T}^{*},\pi_{0:T},p_{0:T}^*})} 	\left[\sum_{t=0}^{T-1}r(s_t,a_t,s_{t+1}, \mu^{*}_t)\right] \\
                &= \mathbb{E}^\mathbb{P^*} \left[\sum_{t=0}^{T-1}r(s_t,a_t,s_{t+1},\mu_t^*)\right],
            \end{aligned}
        \end{align*}
        where for every $\pi_{0:T} \in \Pi$,
        \begin{align}\label{eq:worst_arbit_pi}
                \quad \mathbb{P}({\mu_{0:T}^{*},\pi_{0:T},p_{0:T}^*}):= \mu_0^*\otimes \mathbb{P}_{({\mu_0^*,\pi_0,p_0^*})} \otimes \cdots \otimes \mathbb{P}_{(\mu_{T-1}^*,\pi_{T-1},{p}_{T-1}^*)}\in {\cal Q}(\mu_{0:T}^*,\pi_{0:T}),
        \end{align}
        and $\mathbb{P}^{*}:=\mathbb{P}({\mu_{0:T}^{*},\pi_{0:T}^*,p_{0:T}^*})\in {\cal Q}(\mu_{0:T}^*,\pi_{0:T}^*)$ (see Definition~\ref{dfn:corr2}).
        \item[(ii)] ${\mu}_{0:T}^*$ satisfies that $\mu_{0}^*(\cdot) = \mu^{o}(\cdot)$ and for every $t = 0, \dots, T-2$,
        \begin{align} \label{eq:mfe:mu}
            \mu_{t+1}^*(\cdot) = \int_{S\times A}p^*_{t}(\cdot | s_{t}, a_{t}, {\mu}_{t}^*) \pi_{t}^*(da_{t} | s_{t}) \mu_{t}^*(ds_{t}).
        \end{align}
    \end{enumerate}
\end{dfn}

\subsection{Multi-agent Markov games under model uncertainty}\label{subsec:N_agent}
We aim to obtain approximate Markov-Nash equilibria under model uncertainty by using mean-field equilibria under model uncertainty. To that end, in this section, we introduce the framework for multi-agent Markov games under model uncertainty and the notion of their Markov-Nash equilibria.

Let $N\in \mathbb{N}$ be the number of agents and, as before, $S$ and $A$ be the state and action spaces, respectively.
For $i=1,\dots,N$, denote by $s_{t}^i\in S$ and $a_{t}^i\in A$ the state and action configurations of the agent $i$ at time $t$, respectively. Then we set
\[
\mbox{$\overline s^{N}_t:=(s_{t}^1,\dots,s_{t}^N)\in S^N,\qquad\overline a_t^N:=(a_{t}^1,\dots,a_{t}^N)\in A^N$}
\]
to be the state and action configurations of all $N$ agents at time $t$, respectively, and denote by
\begin{align}\label{eq:empiric}
    e^{N}(\overline s_t^N) := \frac{1}{N} \sum_{i=1}^{N} \delta_{s_{t}^i}\in {\cal P}(S)
\end{align}
the empirical distribution of $\overline s^{N}_t$, where $\delta_{s}\in {\cal P}(S)$ denotes the Dirac measure at $s \in S$.

\begin{dfn}[Multi-agent Markov game] \label{dfn:corr_Nplayer}
    Set $N\in \mathbb{N}$. For each $t=0,\dots,T-1$, let $\mathfrak{P}_{t}:S \times A \times \mathcal{P}(S) \ni (s_t, a_t, \mu_t) \twoheadrightarrow \mathfrak{P}_{t}(s_t, a_t, \mu_t) \subseteq {\cal P}(S)$ be the correspondence at time $t$ given in Definition~\ref{dfn:corr}.
    Then an $N$ agent Markov game under model\;uncertainty, say $(S,A,{\mu}^{o},\mathfrak{P}_{0:T}^N, r\;|\;N,\mathfrak{P}_{0:T})$, comprises the following:
    \begin{enumerate}[leftmargin=3.em]
        \item [(i)] $(S,{\cal B}_S)$ and $(A,{\cal B}_A)$ are Borel spaces for the state and action spaces, respectively.

        \item [(ii)] $s_{0}^1,\dots,s_{0}^N$ are independent and identically distributed according to~$\mu^o \in \mathcal{P}(S)$. Furthermore, denote by $\overline \mu^{o,N}(d\overline s_0^N):=\prod_{i=1}^N \mu^o(ds_{0}^i)\in {\cal P}(S^N)$.

        \item [(iii)] For every $t=0,\dots,T-1$, set $\mathfrak{P}_t^N:S^N\times A^N \ni (\overline s_t^N,\overline a_t^N)\twoheadrightarrow \mathfrak{P}_t^N(\overline s_t^N,\overline a_t^N) \subseteq {\cal P}(S^N)$
        to be a correspondence at time $t$ so that for every $(\overline s_t^N,\overline a_t^N)\in S^N\times A^N$,
        \begin{align*}
            \mathfrak{P}_{t}^{N}(\overline{s}_{t}^{N}, \overline{a}_{t}^{N}) := \bigg\{\overline{\mathbb{P}}_t^N(d\overline s_{t+1}^N) := \prod_{i=1}^N p_{t}\big(ds_{t+1}^i \vert s^{i}_{t}, a^{i}_{t}, e^{N}(\overline{s}_{t}^{N})\big)\biggm\vert p_{t} \in \mathcal{K}_{t}\bigg\},
        \end{align*}
        where $\mathcal{K}_{t}$ is the set of all stochastic kernels $p_{t} \colon S \times A \times \mathcal{P}(S) \to \mathcal{P}(S)$ satisfying
        \[
        p_{t}(\,\cdot\, \vert s, a, \mu) \in \mathfrak{P}_{t}(s, a, \mu)\quad \mbox{for every  $(s, a, \mu) \in S \times A \times \mathcal{P}(S)$},
        \]
        and $e^{N}(\cdot)$ is given in \eqref{eq:empiric}. \color{black}

        \item [(iv)] $r: S\times A \times S \times {\cal P}(S) \mapsto \mathbb{R}$ is a one-step Borel-measurable reward function.
    \end{enumerate}
\end{dfn}

\begin{rem}\label{rem:corr_Nplayer}
    For each fixed global state-action configuration $(\overline s_t^N,\overline a_t^N)$, Definition~\ref{dfn:corr_Nplayer}~(iii) requires that, under any admissible joint transition law $\P_t^N \in \mathfrak{P}_t^N(\overline s_t^N,\overline a_t^N)$, agents with identical states and actions have identical marginal transition laws. However, these common marginals may vary across admissible joint transition laws, even for the same global configuration. In Definition~\ref{dfn:corr_Nplayer2} below, admissible joint transition laws may be selected separately across different global configurations, subject to measurability of the resulting joint transition kernel.
\end{rem}

Next, we introduce the set of Markov policies for the multi-agent model given in Definition~\ref{dfn:corr_Nplayer} and the set of probability measures that induce model uncertainty in the underlying Markov game.

\begin{dfn}\label{dfn:corr_Nplayer2}
    Given $N\in \mathbb{N}$, let $(S,A,{\mu}^{o},\mathfrak{P}_{0:T}^N, r\;|\;N,\mathfrak{P}_{0:T})$ be given in Definition~\ref{dfn:corr_Nplayer}.
    \begin{enumerate}[leftmargin=3.em]
        \item [(i)] Denote by $\Pi^N$ the $N$-tuple of sequences of Markov policies $\overline \pi_{0:T}^{N}:=\prod_{i=1}^N \pi^i_{0:T}$
        defined for every $t= 0, \dots, T-1$ by
        \[
        \overline \pi^{N}_t:S^N\ni \overline s_t^N\mapsto \overline \pi^{N}_t(d\overline a_t^N|\overline s_t^N):= \prod_{i=1}^N \pi_{t}^i(da_{t}^i|s_{t}^i) \in {\cal P}(A^N),
        \]
        where for every $i= 1,\dots,N$, $\pi_{t}^i :S\mapsto {\cal P}(A)$ denotes the Markov policy of agent $i$ at time~$t$.

        \item [(ii)] Given $\overline\pi_{0:T}^{N}\in \Pi^N$, define by ${\cal Q}^N(\mu^o,\overline \pi_{0:T}^{N}) \subseteq {\cal P}(S^N\times (S^N\times A^N)^T)$ the subset~of~all~probability measures $\overline {\mathbb{P}}^N:= \overline \mu^{o,N} \otimes \overline {\mathbb{P}}^N_{(\overline \pi^{N}_0,\overline p^{N}_0)}\otimes \cdots \otimes \overline {\mathbb{P}}^N_{(\overline \pi^{N}_{T-1},\overline p^{N}_{T-1})}$ such that for $t=0,\dots,T-1$,
        \begin{align}\label{eq:def_onetime_msr_N}
        \qquad \overline {\mathbb{P}}^N_{(\overline\pi^{N}_t,\overline p^{N}_t)}:S^N\ni \overline s_t^N \mapsto \overline {\mathbb{P}}^N_{(\overline \pi^{N}_t,\overline p^{N}_t)}(d\overline s_{t+1}^N,d \overline a_t^N|\overline s_t^N):=\overline p^{N}_{t}(d\overline s^{N}_{t+1} | \overline s^{N}_{t}, \overline a^{N}_{t}) \overline \pi_t^{N}(d\overline a_t^N|\overline s_t^N)
        \end{align}
        is a stochastic kernel on $S^N \times A^N$ given $S^N$, where $\overline p_t^{N}: S^N \times A^N \mapsto {\cal P}(S^N)$ is a stochastic kernel satisfying
        \begin{align*}
            \overline p^{N}_{t}(d\overline s^{N}_{t+1} | \overline s^{N}_{t},  \overline a^{N}_{t}) \in \mathfrak{P}_t^N(\overline s_t^N,\overline a_t^N)\quad\mbox{for every $(\overline s_t^N,\overline a_t^N)\in S^N \times A^N$}.
        \end{align*}
    \end{enumerate}
\end{dfn}

Having completed the description of the multi-agent Markov game under model uncertainty, we can proceed to describe the worst-case objective function of the individual agent: Given $N\in \mathbb{N}$,
the worst-case objective function $J_i^N:{\cal P}(S)\times \Pi^N \ni (\mu^o, \overline \pi^N_{0:T})\mapsto J_i^N(\mu^o,\overline \pi^N_{0:T})\in \mathbb{R}$ of agent $i$, $i\in\{1, \dots, N\}$, is given by
\begin{align}\label{eq:worst_nash}
    J^{N}_i(\mu^o,\overline\pi_{0:T}^{N}):= \inf_{\overline{\mathbb{P}}^N \in \mathcal{Q}^{N}(\mu^o,\overline \pi_{0:T}^{N})} \mathbb{E}^{\overline{\mathbb{P}}^N}\left[\sum_{t = 0}^{T-1} r\big(s_{t}^i, a_{t}^i, s_{t+1}^i, e^{N}(\overline{s}_{t}^N)\big)\right].
\end{align}

Finally, we introduce the notion of a Markov-Nash equilibrium for the multi-agent Markov game under model uncertainty.

\begin{dfn}[Markov-Nash equilibria]\label{dfn:MNE}
    Given $N\in \mathbb{N}$, we say $(\pi^{*,1}_{0:T},\dots,\pi^{*,N}_{0:T})$ is a Markov-Nash equilibrium of the $N$ agent Markov game $(S,A,{\mu}^{o},\mathfrak{P}_{0:T}^N, r\;|\;N,\mathfrak{P}_{0:T})$ (see Definition \ref{dfn:corr_Nplayer}) if $\overline {\pi}_{0:T}^{N|*}:= \prod_{i=1}^N{\pi}_{0:T}^{*,i}\in \Pi^N$ satisfies that
    for every $i = 1, \dots, N$
    \begin{align*}
        J^{N}_i(\mu^o,\overline \pi_{0:T}^{N|*})  =  \sup_{\pi_{0:T}\in \Pi } J^{N}_i(\mu^o,(\overline \pi_{0:T}^{N|*,-i},\pi_{0:T})),
    \end{align*}
    where we denote by $(\overline \pi_{0:T}^{N|*,-i},\pi_{0:T})\in \Pi^N$ for every $t=0,\dots,T-1$,
    \begin{equation} \label{eq:pi:N-i}
        (\overline \pi^{N|*,-i}_{t},\pi_{t}) := \pi_t(da_t^i|s_t^i)\;\prod_{j=1,j\neq i}^N \pi_{t}^{*,j}(da_{t}^j|s_{t}^j).
    \end{equation}
    Moreover, for a given $\varepsilon > 0$, we say $(\pi^{*,1}_{0:T},\dots,\pi^{*,N}_{0:T})$ is an $\varepsilon$-Markov-Nash equilibrium of the $N$ agent Markov game  $(S,A,{\mu}^{o},$ $\mathfrak{P}_{0:T}^N, r\;|\;N,\mathfrak{P}_{0:T})$ if $\overline {\pi}_{0:T}^{N|*}\in \Pi^N$ satisfies for every $i = 1, \dots, N$~that
    \begin{align*}
        J^{N}_i(\mu^o,\overline \pi_{0:T}^{N|*})  + \varepsilon \geq \sup_{\pi_{0:T}\in \Pi } J^{N}_i(\mu^o,(\overline \pi_{0:T}^{N|*,-i},\pi_{0:T})).
    \end{align*}
\end{dfn}

\section{Main results}
\subsection{Dynamic programming}\label{sec:dpp} We first present some tailored dynamic programming results that will be useful for proving the existence of a mean-field equilibrium under model uncertainty.

\begin{as}\label{as:msr}
    $(S,A,\mu^o,\mathfrak{P}_{0:T},r)$ given in Definition \ref{dfn:corr} satisfies the following conditions:
    \begin{enumerate}%
        \item [(i)] $S$ and $A$ are finite subsets of a (possibly different) Euclidean space.
        \item [(ii)] For every $t=0,\dots,T-1$, $\mathfrak{P}_{t}$ is non-empty, {convex-valued}, compact-valued, and continuous.\footnote{A correspondence between topological spaces is continuous if it is both lower- and upper-hemicontinuous (see, e.g., \cite[Definition 17.2, p.~558]{CharalambosKim2006infinite}).} Furthermore, there exists a constant $L_{\mathfrak{P}_t} > 0$ such that for every $s_{t} \in S$, $a_{t} \in A$, $\mu_{t}, \tilde{\mu}_{t} \in \mathcal{P}(S)$ and for every $\mathbb{P}\in \mathfrak{P}_{t}(s_t, a_t, \mu_t)$, there exists $\tilde{\P} \in \mathfrak{P}_{t}(s_t, a_t, \tilde{\mu}_t)$ satisfying $d_{W_1}(\P, \tilde{\P}) \leq L_{\mathfrak{P}_t} d_{W_1}(\mu_t, \tilde{\mu}_t)$.

        \item [(iii)] $r$ is bounded and Lipschitz continuous in ${\cal P}(S)$, in the sense that there exists some constant $C_r>0$, $L_r>0$ such that
        for every $s_{t}, s_{t+1} \in S$, $a_{t} \in A$, and $\mu_{t}, \tilde \mu_{t} \in \mathcal{P}(S)$,
        $|r(s_t,a_t,s_{t+1},\mu_t)|\leq C_r$ and $| r(s_t, a_t, s_{t+1}, \mu_t) - r(s_t, a_t, s_{t+1}, \tilde{\mu}_t) | \leq L_{r} d_{W_1}(\mu_t ,\tilde{\mu}_t).$
    \end{enumerate}
\end{as}

Let us formulate a sequence of auxiliary mappings $\widehat{V}_{0:T}$ backwards recursively as follows: for $t = T-1, \dots, 0$, define $\widehat V_{t}:S\times ({\cal P}(S))^{T-t}\mapsto \mathbb{R}$ by setting for every $(s_t,\mu_{t:T})\in S\times ({\cal P}(S))^{T-t}$
\begin{align}\label{eq:DPP_maximin2}
    \widehat V_{t}(s_t,\mu_{t:T}):=\sup_{\pi \in \mathcal{P}(A)} \int_{A} \widehat{J}_{t}(s_t, a_t,\mu_{t:T}) \pi(da_t),
\end{align}
where $\widehat {J}_{t}: S\times A \times ({\cal P}(S))^{T-t}\mapsto \mathbb{R}$ is defined as follows: for\;every\;$(s_{T-1},a_{T-1},\mu_{T-1})\in S\times A \times {\cal P}(S)$
\begin{align}\label{eq:DPP_min2}
    \widehat{J}_{T-1}(s_{T-1}, a_{T-1}, \mu_{T-1}):=  \inf_{\P \in \mathfrak{P}_{T-1}(s_{T-1}, a_{T-1}, \mu_{T-1})} \int_{S}  r(s_{T-1}, a_{T-1}, s_T, \mu_{T-1}) \P(ds_T),
\end{align}
whereas if $t\leq T-2$, we set for every $(s_t,a_t,\mu_{t:T})\in S\times A \times ({\cal P}(S))^{T-t}$
\begin{align}\label{eq:DPP_min1}
    \widehat{J}_{t}(s_t,a_t,\mu_{t:T}):=  \inf_{\P \in \mathfrak{P}_{t}(s_t, a_t, \mu_t)} \int_{S} \left( r(s_t, a_t, s_{t+1}, \mu_t) + \widehat{V}_{t+1}(s_{t+1}, \mu_{t+1:T})\right) \P(ds_{t+1}),
\end{align}
with $\mathfrak{P}_{0:T}$ given in Definition \ref{dfn:corr}.

Finally, we define $\widehat{V}:({\cal P}(S))^{T}\mapsto \mathbb{R}$ by setting for every ${\mu}_{0:T}\in ({\cal P}(S))^{T}$
\begin{align} \label{eq:V-1}
    \widehat{V}({\mu}_{0:T}) := \int_{S} \widehat{V}_{0}(s_0, {\mu}_{0:T}) \mu_{0}(ds_0).
\end{align}

\begin{lem}\label{lem:dpp_berge}
    Suppose that Assumption~\ref{as:msr} is satisfied.
    Let $\widehat{V}_{0:T}$ and $\widehat{J}_{0:T}$ be given in \eqref{eq:DPP_maximin2} and \eqref{eq:DPP_min2}--\eqref{eq:DPP_min1}, respectively. Then the following statements hold for every $t=0,\dots,T-1$.
    \begin{itemize}
        \item[(i)] (Minimizer of $\widehat{J}_{t}$)  There exists a measurable selector
        \begin{equation*}
            \widehat p_t \colon S \times A \times  ({\cal P}(S))^{T-t}\ni (s_t, a_t, \mu_{t:T}) \mapsto \widehat {p}_t(\cdot |s_t, a_t, \mu_{t:T}) \in \mathfrak{P}_t(s_t,a_t,\mu_t)
        \end{equation*}
        satisfying that if $t = T-1$, then for every $(s_{T-1}, a_{T-1}, \mu_{T-1})\in S \times A \times  {\cal P}(S)$
        \begin{align}\label{eq:minimizer2}
            \qquad \widehat{J}_{T-1}(s_{T-1}, a_{T-1}, \mu_{T-1})=\int_{S}  r(s_{T-1}, a_{T-1}, s_{T}, \mu_{T-1}) \widehat p_{T-1}(ds_T | s_{T-1}, a_{T-1}, \mu_{T-1}),
        \end{align}
        whereas if $t\leq T-2$, then for every $(s_t, a_t, \mu_{t:T})\in S\times A \times ({\cal P}(S))^{T-t}$
        \begin{align}\label{eq:minimizer}
            \qquad \quad  \widehat{J}_t(s_t,a_t,\mu_{t:T})=\int_{S} \left(r(s_t, a_t, s_{t+1}, \mu_t) + \widehat{V}_{t+1}(s_{t+1}, {\mu}_{t+1:T})\right) \widehat p_t(ds_{t+1} | s_t, a_t, \mu_{t:T}).
        \end{align}

        \item[(ii)] (Maximizer of $\widehat{V}_{t}$) There exists a measurable selector
        \begin{equation*}
            \widehat \pi_t:S\times ({\cal P}(S))^{T-t} \ni (s_t,\mu_{t:T}) \mapsto \widehat \pi_t(\cdot |s_t,\mu_{t:T}) \in \mathcal{P}(A)
        \end{equation*}
        satisfying that for every $(s_t,\mu_{t:T}) \in S \times (\mathcal{P}(S))^{T-t}$
        \begin{align}\label{eq:maximizer}
            \widehat{V}_{t}(s_t,\mu_{t:T})=\int_{A} \widehat{J}_{t}(s_t, a_t,\mu_{t:T}) \widehat \pi_t (da_t|s_t,\mu_{t:T}).
        \end{align}
    \end{itemize}
\end{lem}

\begin{rem} \label{rem:CalPTil}
    Berge's maximum theorem (see, e.g., \cite[Theorem~17.31]{CharalambosKim2006infinite}), as presented in the proof of Lemma \ref{lem:dpp_berge}, ensures the existence of measurable selectors $\widehat{p}_{0:T}$ and $\widehat{\pi}_{0:T}$, as well as the following under the assumption therein: for every $t\leq T-2$, the correspondence $\widehat{\mathfrak{P}}_{t} : S \times A \times (\mathcal{P}(S))^{T-t}\ni (s_t, a_t, \mu_{t:T})  \twoheadrightarrow \widehat{\mathfrak{P}}_{t}  (s_t,a_t,\mu_{t:T})\subseteq {\cal P}(S)$ defined by
    \begin{align*}
        \widehat{\mathfrak{P}}_{t} (s_t, a_t, \mu_{t:T}):= \left\{
        \P \in \mathfrak{P}_{t}(s_t, a_t, \mu_{t}) \left|
        \begin{aligned}
        & \int_{S} \Big(r(s_t, a_t, s_{t+1}, \mu_{t})+\widehat{V}_{t+1}(s_{t+1}, {\mu}_{t+1:T})\Big) \P(ds_{t+1})   \\
        &\;\;=  \widehat{J}_{t}(s_t, a_t, {\mu}_{t:T})
        \end{aligned}
        \right.\right\}
    \end{align*}
    is non-empty, compact-valued, and upper-hemicontinuous (see
    \cite[Theorem 17.31\;(2.),\;(3.)]{CharalambosKim2006infinite}).
    Furthermore, since $\mathfrak{P}_t$ is convex-valued (see Assumption \ref{as:msr}\;(ii)), so is $\widehat{\mathfrak{P}}_{t}$. These observations will be used in Section~\ref{sec:MF_main_thm}.
\end{rem}

As a consequence of Lemma~\ref{lem:dpp_berge}, we obtain the following dynamic programming principle result.

\begin{pro} \label{pro:dpp}
    Suppose that Assumption~\ref{as:msr} is satisfied. Let $\widehat{V}_{0:T}$ and $\widehat{J}_{0:T}$ be given in\;\eqref{eq:DPP_maximin2} and\;\eqref{eq:DPP_min2}--\eqref{eq:DPP_min1},\;respectively.\;Given $\tilde \mu_{0:T}\in ({\cal P}(S))^{T}$,\;the following hold for every $t=0,\dots,T-1$:
    \begin{itemize}
        \item [(i)] There exists a stochastic kernel $p_{t}^*:S\times A\times {\cal P}(S)\mapsto {\cal P}(S)$ so that if $t=T-1$, then for every $(s_{T-1},a_{T-1})\in S\times A$
        \begin{align}\label{eq:minimizer2_given}
            \qquad \widehat{J}_{T-1}(s_{T-1},a_{T-1},\tilde \mu_{T-1})=\int_{S}  r(s_{T-1}, a_{T-1}, s_{T},\tilde \mu_{T-1})p_{T-1}^*(ds_T|s_{T-1},a_{T-1}, \tilde\mu_{T-1}),
        \end{align}
        whereas if $t \leq T-2$, then for every $(s_t,a_t)\in S\times A$
        \begin{align}\label{eq:minimizer_given}
            \quad \quad \widehat{J}_t(s_t,a_t,\tilde \mu_{t:T})=\int_{S} \left(r(s_t, a_t, s_{t+1},\tilde \mu_t) + \widehat{V}_{t+1}(s_{t+1},\tilde {\mu}_{t+1:T})\right) p_t^*(ds_{t+1} | s_t, a_t,\tilde \mu_t).
        \end{align}
        Furthermore, there exists a Markov policy $\pi_{t}^*:S\mapsto {\cal P}(A)$ so that for every $s_t\in S$
        \begin{align}\label{eq:maximizer_given}
            \widehat{V}_{t}(s_t,\tilde {\mu}_{t:T})=\int_{A} \widehat{J}_{t}(s_t, a_t, \tilde\mu_{t:T}) \pi_t^* (da_t|s_t).
        \end{align}
        \item [(ii)] Let $p_{0:T}^*$ and $\pi_{0:T}^*$ be defined as in (i). Define $\mathbb{P}^*(\tilde\mu_{0:T})\in {\cal Q}(\tilde \mu_{0:T},\pi^*_{0:T}) $ by
        \[
        \qquad \qquad\mathbb{P}^*(\tilde\mu_{0:T}) := \mathbb{P}({ \tilde\mu_{0:T},\pi_{0:T}^*,p_{0:T}^*}) := \tilde\mu_0\otimes \mathbb{P}_{({ \tilde\mu_0,\pi_0^*,p_0^*})}\otimes\cdots \otimes \mathbb{P}_{( \tilde\mu_{T-1},\pi_{T-1}^*,{p}_{T-1}^*)}.%
        \]
        Then $V(\tilde\mu_{0:T})$ given in \eqref{dfn:value_mfg} is equal to $\widehat{V}(\tilde\mu_{0:T})$ given in \eqref{eq:V-1}, and $(\pi_{0:T}^*,p_{0:T}^*)$ are optimal for $V(\tilde \mu_{0:T})$, i.e.,
        \begin{align}\label{eq:verification}
            \begin{aligned}
                \qquad V(\tilde \mu_{0:T}) = J(\tilde\mu_{0:T},\pi^*_{0:T}) &= \sup_{\pi_{0:T}\in \Pi}\mathbb{E}^{\mathbb{P}({ \tilde\mu_{0:T},\pi_{0:T},p_{0:T}^*})} 	\left[\sum_{t=0}^{T-1}r(s_t,a_t,s_{t+1}, \tilde\mu_t)\right] \\
                &= \mathbb{E}^{\mathbb{P}^*(\tilde \mu_{0:T})}\left[\sum_{t=0}^{T-1}r(s_t,a_t,s_{t+1}, \tilde\mu_t)\right],
            \end{aligned}
        \end{align}
        with $\mathbb{P}({ \tilde\mu_{0:T},\pi_{0:T},p_{0:T}^*}) := \tilde\mu_0\otimes \mathbb{P}_{({ \tilde\mu_0,\pi_0,p_0^*})}\otimes\cdots \otimes \mathbb{P}_{( \tilde\mu_{T-1},\pi_{T-1},{p}_{T-1}^*)}\in {\cal Q}(\tilde\mu_{0:T},\pi_{0:T})$.
    \end{itemize}
\end{pro}

The proofs of Lemma \ref{lem:dpp_berge} and Proposition \ref{pro:dpp} are presented in Section \ref{sec:proof:dpp:MFG}.

Next, we revisit the multi-agent Markov game given in Definition~\ref{dfn:corr_Nplayer} and Definition~\ref{dfn:corr_Nplayer2} to obtain the corresponding dynamic programming principle result. This will be useful in Section \ref{sec:main_MNE} for determining the worst-case transition kernel for any given Markov policy.

\begin{rem}
    Under Assumption~\ref{as:msr}\;(i) and\;(ii), for every $N \in \N$ and $t \in \{0, \dots, T-1\}$, the correspondence $\mathfrak{P}_{t}^{N}$ given in Definition~\ref{dfn:corr_Nplayer} satisfies the following properties:
    \begin{enumerate}
        \item[(i)] For every $(\overline{s}_{t}^{N}, \overline{a}_{t}^{N}) \in S^{N} \times A^{N}$,
        \begin{align*}
            \hspace{32pt} \mathfrak{P}_{t}^{N}(\overline{s}_{t}^{N}, \overline{a}_{t}^{N}) = \Bigg\{\prod_{i = 1}^{N} \P^{i}(ds_{t+1}^{i}) \Biggm\vert
            \begin{aligned}
                &\P^{i} \in \mathfrak{P}_{t}\big(s_{t}^{i}, a_{t}^{i}, e^{N}(\overline{s}_{t}^{N})\big), \quad \forall i = 1, \dots, N, \\
                &\P^{i} = \P^{j} \text{ whenever } (s_{t}^{i}, a_{t}^{i}) = (s_{t}^{j}, a_{t}^{j}), \quad \forall i, j = 1, \dots, N
            \end{aligned}
            \Bigg\}.
        \end{align*}
        Indeed, inclusion in the right-hand side follows directly from Definition~\ref{dfn:corr_Nplayer}~(iii). Conversely, fix a product measure in the right-hand side. Since $\mathcal{P}(S)$ is Polish and $\mathfrak{P}_{t}$ is lower hemicontinuous with non-empty compact values, the Kuratowski–Ryll-Nardzewski selection theorem (see, e.g., \cite[Theorem~18.13]{CharalambosKim2006infinite}) yields a Borel selector $p_{t}^{0}$ of $\mathfrak{P}_{t}$. Define
        \begin{equation*}
            p_{t}(\,\cdot \mid s, a, \mu) :=
            \begin{cases}
                \P^{i}, &(s, a, \mu) = (s_{t}^{i}, a_{t}^{i}, e^{N}(\overline{s}_{t}^{N})) \text{ for some } i, \\
                p_{t}^{0}(\,\cdot \mid s, a, \mu), &\text{otherwise.}
            \end{cases}
        \end{equation*}
        The equality condition in the right-hand side set makes this definition well-defined. Each modified point is a closed singleton in $S \times A \times \mathcal{P}(S)$ and thus a Borel set, so this finite-point modification preserves Borel measurability. Clearly, $p_{t} \in \mathcal{K}_{t}$ and it generates the prescribed product measure, proving the reverse inclusion.

        \item[(ii)] For every $(\overline{s}_{t}^{N}, \overline{a}_{t}^{N}) \in S^{N} \times A^{N}$, the set $\mathfrak{P}_{t}^{N}(\overline{s}_{t}^{N}, \overline{a}_{t}^{N})$ is non-empty and compact in the weak topology on $\mathcal{P}(S^N)$. Indeed, the tuples $(\P^{1}, \dots, \P^{N})$ satisfying the conditions in~(i) form a closed subset of the compact product $\prod_{i = 1}^{N} \mathfrak{P}_{t}(s_{t}^{i}, a_{t}^{i}, e^{N}(\overline{s}_{t}^{N}))$. This subset is non-empty, as can be seen by evaluating the Borel selector $p_{t}^{0}$ at the corresponding local triples. The product-measure map
        \(
            (\P^{1}, \dots, \P^{N}) \mapsto \prod_{i = 1}^{N} \P^{i}
        \)
        is continuous and, by~(i), maps this subset onto $\mathfrak{P}_{t}^{N}(\overline{s}_{t}^{N}, \overline{a}_{t}^{N})$. Hence, the latter is non-empty and compact.

        \item[(iii)] The correspondence $\mathfrak{P}_{t}^{N}$ is both upper and lower hemicontinuous, and every selector of $\mathfrak{P}_{t}^{N}$ is Borel measurable. Indeed, for every open set $U \subseteq \mathcal{P}(S^{N})$, the sets
        \begin{align*}
            &\{(\overline{s}_{t}^{N}, \overline{a}_{t}^{N}) \mid \mathfrak{P}_{t}^{N}(\overline{s}_{t}^{N}, \overline{a}_{t}^{N}) \cap U \neq \emptyset\},\quad \quad \{(\overline{s}_{t}^{N}, \overline{a}_{t}^{N}) \mid \mathfrak{P}_{t}^{N}(\overline{s}_{t}^{N}, \overline{a}_{t}^{N}) \subseteq U\}
        \end{align*}
        are open, since $S^{N} \times A^{N}$ is finite and discrete. These are the lower and upper inverse images of $U$, respectively, so their openness proves lower and upper hemicontinuity (see \cite[Definition~17.2]{CharalambosKim2006infinite}). Moreover, every selector is Borel-measurable because every subset of its domain is Borel.

        \item[(iv)] For every $\overline{\pi}_{0:T}^{N}\in\Pi^{N}$,
        the set $\mathcal Q^{N}(\mu^{o},\overline{\pi}_{0:T}^{N})$
        is nonempty.  Indeed, fix $\overline{\pi}_{0:T}^{N}\in\Pi^{N}$.
        For each $t\in\{0,\dots,T-1\}$, by~(ii), we can choose one probability measure
        from $\mathfrak P_t^N(\overline s^N,\overline a^N)$
        at each configuration $(\overline s^N,\overline a^N)$.
        Since $S^N\times A^N$ is finite and discrete, these choices
        define a stochastic kernel $\overline p_t^N$ satisfying
        \[
        \overline p_t^N(\,\cdot\,\mid\overline s^N,\overline a^N)
        \in\mathfrak P_t^N(\overline s^N,\overline a^N)
        \quad
        \text{for every }
        (\overline s^N,\overline a^N)\in S^N\times A^N.
        \]
        Combining these transition kernels with the fixed policy
        $\overline{\pi}_{0:T}^{N}$ and the initial distribution
        $\overline{\mu}^{o,N}$ defines a probability measure
        \begin{align*}
            \overline{\mathbb P}^{N}
            (d\overline s_0^N,d\overline a_0^N,\ldots,d\overline s_T^N):=
            \overline{\mu}^{o,N}(d\overline s_0^N)
            \prod_{t=0}^{T-1}
            \overline{\pi}_t^N
            (d\overline a_t^N\mid\overline s_t^N)\,
            \overline p_t^N
            (d\overline s_{t+1}^N
            \mid\overline s_t^N,\overline a_t^N).
        \end{align*}
        By Definition~\ref{dfn:corr_Nplayer2}~(ii), this measure belongs to
        $\mathcal Q^N(\mu^o,\overline{\pi}_{0:T}^{N})$.
        This proves the claim.
    \end{enumerate}
\end{rem}

Set $N\in \mathbb{N}$, and let $\overline \pi_{0:T}^N\in \Pi^N$ and $i\in \{1,\dots,N\}$. Define a sequence of mappings $\widehat{J}_{0:T,i}^{N}(\cdot,\cdot;\overline\pi^N_{0:T})$ backwards recursively as follows: Define for every $(\overline s^N_{T-1},\overline a^N_{T-1})\in S^N\times A^N$
\begin{align}\label{eq:DPP_min_Nash0}
    \widehat{J}_{T-1,i}^{N}(\overline s^N_{T-1},\overline a^N_{T-1};\overline\pi^N_{0:T}):=  \inf_{\P \in \mathfrak{P}_{T-1}^{N}(\overline s_{T-1}^{N}, \overline a_{T-1}^{N})} \int_{S^{N}} r\big(s_{T-1}^{i}, a_{T-1}^{i}, s^{i}_{T}, e^{N}(\overline s_{T-1}^{N})\big)  \P(d\overline s^{N}_{T}),
\end{align}
and for $t\leq T-2$, define for every $(\overline s^N_{t},\overline a^N_{t})\in S^N\times A^N$
\begin{align}\label{eq:DPP_min_Nash}
    \begin{aligned}
        \widehat{J}_{t,i}^{N} (\overline s^N_{t},\overline a^N_{t};\overline\pi^N_{0:T})
        &:=  \inf_{\P \in \mathfrak{P}_{t}^{N}(\overline s^N_{t},\overline a^N_{t})} \int_{S^{N}}\bigg( r\big(s_{t}^{i}, a_{t}^{i}, s^{i}_{t+1}, e^{N}(\overline s_t^{N})\big) \bigg.\\
        &\bigg. \quad\quad + \int_{A^N}\widehat{J}_{t+1,i}^{N} (\overline s^N_{t+1},\overline a^N_{t+1};\overline\pi^N_{0:T})\overline \pi_{t+1}^N(d\overline a_{t+1}^N|\overline s_{t+1}^N)\bigg) \P(d\overline s^{N}_{t+1}),
    \end{aligned}
\end{align}
with $\mathfrak{P}^N_{0:T}$ given in Definition \ref{dfn:corr_Nplayer}.

\begin{lem}\label{lem:dpp_berge_nash}
    Suppose that Assumption~\ref{as:msr} is satisfied. Set $N\in \mathbb{N}$, and let $\overline \pi_{0:T}^N\in \Pi^N$ and $i\in \{1,\dots,N\}$. Furthermore, let $\widehat {J}_{0:T,i}^{N} (\cdot,\cdot;\overline\pi^N_{0:T})$ be given in \eqref{eq:DPP_min_Nash0} and~\eqref{eq:DPP_min_Nash}.
    Then for every $t=0,\dots,T-1$, there exists a measurable selector (i.e., minimizer for $\widehat{J}^N_{t,i}(\cdot,\cdot\,;\overline\pi^N_{0:T})$)
    \begin{equation} \label{eq:N:ker}
        \widehat{p}_{(t,i,\overline \pi^N_{0:T})}: S^N \times A^N \ni (\overline s_t^N, \overline a_t^N)\mapsto \widehat{p}_{(t,i,\overline \pi^N_{0:T})}(\cdot|\overline s_t^N,\overline a_t^N) \in \mathfrak{P}_t^N (\overline s_t^N,\overline a_t^N)
    \end{equation}
    such that, if $t = T-1$, then for every $(\overline s_{T-1}^N,\overline a_{T-1}^N)\in S^N\times A^N$,
    \begin{align*}%
        \begin{aligned}
            \widehat{J}_{T-1,i}^{N}(\overline s^N_{T-1},\overline a^N_{T-1};\overline\pi^N_{0:T})=\int_{S^{N}} r\big(s_{T-1}^{i}, a_{T-1}^{i}, s^{i}_{T}, e^{N}(\overline s_{T-1}^{N})\big) \widehat{p}_{(T-1,i,\overline \pi^N_{0:T})}(d\overline s_{T}^N| \overline s_{T-1}^N,\overline a_{T-1}^N),
        \end{aligned}
    \end{align*}
    whereas if $t\leq T-2$, then for every $(\overline s_t^N,\overline a_t^N)\in S^N\times A^N$
    \begin{align*}
        \begin{aligned}
            \widehat{J}_{t,i}^{N} (\overline s^N_{t},\overline a^N_{t};\overline\pi^N_{0:T})=&\int_{S^{N}}\bigg( r\big(s_{t}^{i}, a_{t}^{i}, s^{i}_{t+1}, e^{N}(\overline s_t^{N})\big)
        \bigg.\\
            &\quad\quad+\left. \int_{A^N}\widehat{J}_{t+1,i}^{N} (\overline s^N_{t+1},\overline a^N_{t+1};\overline\pi^N_{0:T})\overline \pi_{t+1}^N(d\overline a_{t+1}^N|\overline s_{t+1}^N)\right) \widehat{p}_{(t,i,\overline \pi^N_{0:T})}(d\overline s_{t+1}^N| \overline s_t^N,\overline a_t^N).
        \end{aligned}
    \end{align*}
    Moreover, for every $j = 1, \dots, N$, there exists a stochastic kernel (representing the $j$-th marginal of $\widehat{p}_{(t,i,\overline \pi^N_{0:T})}$)
    \begin{equation*}
        \widehat{q}_{(t,i,j,\overline \pi^N_{0:T})}: S^N \times A^N \ni (\overline s_t^N, \overline a_t^N)\mapsto \widehat{q}_{(t,i,j,\overline \pi^N_{0:T})}(\cdot|\overline s_t^N,\overline a_t^N) \in \mathfrak{P}_t\big(s^{j}_{t}, a^{j}_{t}, e^{N}(\overline{s}_{t}^{N})\big),
    \end{equation*}
    such that, for every $(\overline s_t^N, \overline a_t^N) \in S^N \times A^N$,
    \begin{equation*}
        \widehat{p}_{(t,i,\overline \pi^N_{0:T})}(d \overline{s}^{N}_{t+1}|\overline s_t^N,\overline a_t^N) = \prod_{j = 1}^{N} \widehat{q}_{(t,i,j,\overline \pi^N_{0:T})}(d s_{t+1}^{j} \mid \overline s_t^N,\overline a_t^N).
    \end{equation*}
    Furthermore, for every $j, k = 1, \dots, N$ and every $(\overline{s}_{t}^{N}, \overline{a}_{t}^{N}) \in S^{N} \times A^{N}$ such that $(s_{t}^{j}, a_{t}^{j}) = (s_{t}^{k}, a_{t}^{k})$, it holds that
    \begin{equation} \label{eq:qj:qk}
        \widehat{q}_{(t,i,j,\overline \pi^N_{0:T})}(d s_{t+1}^{j} \mid \overline s_t^N,\overline a_t^N) = \widehat{q}_{(t,i,k,\overline \pi^N_{0:T})}(d s_{t+1}^{k} \mid \overline s_t^N,\overline a_t^N).
    \end{equation}
\end{lem}

As a consequence of Lemma~\ref{lem:dpp_berge_nash}, we obtain the following result.
\begin{pro}\label{pro:dpp_Nash}
    Suppose that Assumption~\ref{as:msr} is satisfied. For every $i \in \{1, \dots, N\}$, initial distribution $\mu^{o}$, and $\overline \pi_{0:T}^{N}\in  \Pi^N$, let $\widehat{p}_{(0:T,i,\overline \pi^N_{0:T})}$ be given in Lemma~\ref{lem:dpp_berge_nash}. Then
    \[
        \overline {\mathbb{P}}^N ({\mu^o, \overline \pi^N_{0:T},\widehat{p}_{(0:T,i,\overline \pi^N_{0:T})}}):= \overline \mu^{o,N}\otimes \overline{ \mathbb{P}}^N_{(\overline \pi^{N}_0,\widehat{p}_{(0,i,\overline \pi^N_{0:T})})}\otimes \cdots \otimes \overline{\mathbb{P}}^N_{(\overline \pi^{N}_{T-1},\widehat{p}_{(T-1,i,\overline \pi^N_{0:T})})}\in {\cal Q}^N(\mu^o,\overline \pi^N_{0:T})
    \]
    is the worst-case measure for $J^{N}_i(\mu^o,\overline \pi_{0:T}^{N})$ (given in \eqref{eq:worst_nash}), i.e.,
    \begin{align*}
        \begin{aligned}
            J^{N}_i(\mu^o,\overline \pi_{0:T}^{N}) &= \mathbb{E}^{\overline {\mathbb{P}}^N ({\mu^o,\overline \pi^N_{0:T}},\widehat{p}_{(0:T,i,\overline \pi^N_{0:T})} )}\left[\sum_{t = 0}^{T-1} r\big(s^{i}_{t}, a^{i}_{t}, s^{i}_{t+1}, e^{N}(\overline s_t^N)\big)\right]\\
            &=\int_{S^N}\int_{A^N} \widehat{J}_{0,i}^{N} (\overline s_0^N,\overline a_0^N; \overline\pi^N_{0:T}) \overline \pi_0^N(d\overline a_0^N|\overline s_0^N) \overline \mu^{o,N}(d \overline s_0^N),%
        \end{aligned}
    \end{align*}
    with $\widehat{J}_{0,i}^{N}$ given in \eqref{eq:DPP_min_Nash}.
\end{pro}

The proofs of Lemma \ref{lem:dpp_berge_nash} and Proposition \ref{pro:dpp_Nash} can be found in Section \ref{sec:proof:dpp:MFE}.
\subsection{Existence of mean-field equilibrium}\label{sec:MF_main_thm}

Using the results of the dynamic programming principle derived for the mean-field Markov game in Section~\ref{sec:dpp}, along with Kakutani's fixed point theorem (see, e.g., \cite[Corollary 17.55, p.~583]{CharalambosKim2006infinite}), we will demonstrate the existence of a mean-field equilibrium under model uncertainty in Theorem~\ref{thm:MFE}.

\begin{dfn} \label{dfn:FixPntEq}
    Set $\Xi := (\mathcal{P}(S \times A))^{T}.$ For $\nu_{0:T} \in \Xi$ and $t=0,\dots,T-1$, denote by $\nu_{t,S}$ the marginal of $\nu_t\in \mathcal{P}(S \times A)$ on $S$, i.e., $\nu_{t,S}(\cdot) := \nu_t(\cdot \times A)\in {\cal P}(S)$. Furthermore, denote~by
    \[
    \pi_t^\nu:S\ni s_t\mapsto \pi^\nu_t(\cdot|s_t)\in {\cal P}(A)
    \]
    the disintegrating kernel of $\nu_t$ with respect to $\nu_{t,S}$, i.e., $\nu_t(ds_t,da_t)=\pi_t^{\nu}(da_t|s_t)\nu_{t,S}(ds_t).$
\end{dfn}

\begin{dfn} \label{dfn:FixPntEq2}
    Let $\Xi$ be given in Definition \ref{dfn:FixPntEq}. Let $\widehat{\mathfrak{P}}_{0:T-1}$ be given in Remark \ref{rem:CalPTil}.
    Furthermore, let $\widehat{J}_{0:T}$ be given in \eqref{eq:DPP_min2} and \eqref{eq:DPP_min1}. Define the following correspondences:
    \begin{enumerate}[leftmargin=3.em]
        \item [(i)] ${\cal C}:\Xi \ni\nu_{0:T}\twoheadrightarrow {\cal C}(\nu_{0:T})\subseteq  \Xi $ is defined by
        \begin{align*}
            \qquad
            \begin{aligned}
                \mathcal{C}(\nu_{0:T}) := \bigg\{\tilde{\nu}_{0:T} \in \Xi \;\Big|\;& \tilde{\nu}_{0, S} = \mu^o\;\text{and for every}\;t = 0, \dots,{T-2},\;\;\text{there exists}\\
                &\; p_{t}^{\tilde{\nu}} : S \times A \times (\mathcal{P}(S))^{T-t} \ni (s_t,a_t,\mu_{t:T}) \mapsto p_{t}^{\tilde{\nu}}(\cdot|s_t,a_t,\mu_{t:T}) \in \mathcal{P}(S)\\
                &\; \text{s.t. for every $(s_t,a_t)\in S\times A$, $ p_{t}^{\tilde{\nu}}(\cdot | s_t, a_t, \nu_{t:T,S}) \in \widehat{\mathfrak{P}}_{t}(s_t, a_t, \nu_{t:T,S})$} \\
                &\;\text{and} \;\tilde{\nu}_{t+1, S}(\cdot) = \int_{S \times A}  p_{t}^{\tilde{\nu}}(\cdot | s_t, a_t, \nu_{t:T, S}) \nu_{t}(ds_t, da_t)\bigg\},
            \end{aligned}
        \end{align*}
        and ${\cal B}:\Xi \ni \nu_{0:T} \twoheadrightarrow {\cal B}(\nu_{0:T})\subseteq \Xi$ is defined by
        \begin{align*}
            \mathcal{B}(\nu_{0:T}) := \bigg\{\tilde{\nu}_{0:T} \in \Xi \;\Big|\; \text{for every } t = 0, \dots, T-1, \; \tilde{\nu}_{t}(D_{t}(\nu_{t:T})) = 1\bigg\},
        \end{align*}
        where $D_{t}(\nu_{t:T}) := \big\{(s_t, a_t) \in S \times A\;|\;\max_{a_t' \in A} \widehat{J}_{t}(s_t, a_t', \nu_{t:T, S}) = \widehat{J}_{t}(s_t, a_t, \nu_{t:T, S})\big\}.$
        \item [(ii)] $\Gamma : \Xi \ni \nu_{0:T}  \twoheadrightarrow \Gamma (\nu_{0:T}) \subseteq \Xi $ is defined by
        \begin{align*}
            \Gamma (\nu_{0:T}) := {\cal C}(\nu_{0:T}) \cap {\cal B}(\nu_{0:T}).
        \end{align*}
        We say $\nu_{0:T} \in \Xi$ is a {\it fixed point} of $\Gamma$ if $\nu_{0:T} \in \Gamma(\nu_{0:T})$.
    \end{enumerate}
\end{dfn}

\begin{pro} \label{pro:FixPntEq}
    Suppose that Assumption \ref{as:msr} is satisfied. Then the following hold:
    \begin{itemize}
        \item [(i)] The correspondence $\Gamma$ given in Definition \ref{dfn:FixPntEq2}\;(ii) is non-empty and convex-valued.
        \item [(ii)] The graph of $\Gamma$, i.e. $\mathrm{Gr}(\Gamma) := \{({\nu}_{0:T}, {\xi}_{0:T}) \in \Xi \times \Xi~|~{\xi}_{0:T} \in \Gamma({\nu}_{0:T})\}$, is closed.
        \item [(iii)] There exists a fixed point $\nu_{0:T}^* \in \Xi$ of $\Gamma$, i.e., $\nu_{0:T}^* \in \Gamma(\nu_{0:T}^*)$.
    \end{itemize}
\end{pro}

Using a fixed point $\nu_{0:T}^* \in \Xi$ of $\Gamma$ together with the measurable selectors given in Lemma~\ref{lem:dpp_berge}, we obtain the following main theorem.
\begin{thm} \label{thm:MFE}
    Let $(S,A,\mu^o,\mathfrak{P}_{0:T},r)$ be the mean-field Markov game under model uncertainty given in Definition \ref{dfn:corr}. Suppose that Assumption \ref{as:msr} is satisfied.
    Then there exists a mean-field equilibrium $({\mu}_{0:T}^*,{\pi}_{0:T}^*,{p}_{0:T}^*)$
    of $(S,A,\mu^o,\mathfrak{P}_{0:T},r)$
    (see Definition \ref{dfn:robust_mfe}).
\end{thm}

The proofs of Proposition \ref{pro:FixPntEq} and Theorem \ref{thm:MFE} can be found in Section \ref{sec:proof:MF_main_thm}.

\subsection{Existence of approximate Markov-Nash equilibrium}\label{sec:main_MNE}

Fix a mean-field equilibrium $({\mu}_{0:T}^*,{\pi}_{0:T}^*,{p}_{0:T}^*)$ of the mean-field Markov game $(S,A,\mu^o,\mathfrak{P}_{0:T},r)$ (whose existence is ensured by Theorem~\ref{thm:MFE} under the assumption therein).

In the following, we demonstrate that under certain assumptions, the optimal policy $\pi^{*}_{0:T}$ of the mean-field equilibrium constitutes an approximate Markov-Nash equilibrium of the multi-agent Markov game given in Definition~\ref{dfn:corr_Nplayer} and Definition~\ref{dfn:corr_Nplayer2}. To that end, we first introduce some key notions related to worst-case measures describing the multi-agent Markov game for a given policy.

\begin{dfn}[Worst-case measures]\label{dfn:worst_msr}
    Let $(\pi^{(N)}_{0:T})_{N\in \mathbb{N}}\subseteq \Pi$ be a sequence of arbitrary Markov policies. For every $N\in \mathbb{N}$ and $i\in\{1,\dots,N\}$, we introduce the following.
    \begin{itemize}[leftmargin=3.em]
        \item [(i)] Denote~by
        \[
            {\mathbb{P}}^{*|(N)}:=\mathbb{P}({ \mu_{0:T}^*,\pi_{0:T}^{(N)},p_{0:T}^{*}})\in {\cal Q}( \mu_{0:T}^*,\pi^{(N)}_{0:T}),
        \]
        where $\mathbb{P}({ \mu_{0:T}^*,\pi_{0:T}^{(N)},p_{0:T}^{*}})$ is given in~\eqref{eq:worst_arbit_pi}. Moreover, if $\pi^{(N)}_{0:T}=\pi^{*}_{0:T}$,
        then we denote by
        \begin{align*}%
                \quad{\mathbb{P}}^{*}:=\mathbb{P}^*( \mu_{0:T}^*)={\mathbb{P}}^{*|(N)}%
                \in {\cal Q}( \mu_{0:T}^*,\pi^{*}_{0:T})%
        \end{align*}
        the worst-case measure for $V(\mu_{0:T}^*)$ (see Proposition~\ref{pro:dpp}~(ii)).

        \item [(ii)] For every $t\in \{0,\dots,T-1\}$ and $i, j \in \{1, \dots, N\}$, denote by
        \begin{align}%
                \overline{\pi}_{t,i}^{N|(N)}&:S^N\ni \overline s_t^N\mapsto \overline \pi^{N|(N)}_{t,i}(d\overline a_t^N|\overline s_t^N):= \pi_t^{(N)}(da_t^i|s_t^i)\prod_{j=1,j\neq i}^N \pi_{t}^*(da_{t}^j|s_{t}^j), \label{eq:over:pi:NN} \\
                \overline{p}^{N|(N)}_{t,i}&:S^N \times A^N \ni (\overline s_t^N, \overline a_t^N)\mapsto
                \overline{p}^{N|(N)}_{t,i}(d\overline s_{t+1}^N|\overline s_t^N,\overline a_t^N):=
                \widehat{p}_{(t,i,\overline \pi^{N|(N)}_{0:T,i})} (d\overline s_{t+1}^N|\overline s_t^N,\overline a_t^N), \label{eq:over:p:NN} \\
                \overline{q}^{N|(N)}_{t,i,j}&:S^{N} \times A^{N} \ni (\overline{s}_{t}^{N}, \overline{a}_{t}^{N}) \mapsto
                \overline{q}^{N|(N)}_{t,i,j}(d s_{t+1}|\overline{s}_{t}^{N}, \overline{a}_{t}^{N}) :=
                \widehat{q}_{(t,i,j,\overline \pi^{N \vert (N)}_{0:T, i})}(d s_{t+1} \vert \overline{s}_{t}^{N}, \overline{a}_{t}^{N}) \notag
        \end{align}
        a Markov policy and stochastic kernels, respectively, where $\widehat{p}_{(t,i,\overline \pi^{N|(N)}_{0:T,i})}$ and $\widehat{q}_{(t,i,j,\overline \pi^{N|(N)}_{0:T,i})}$ are defined as in Lemma~\ref{lem:dpp_berge_nash} with respect to $\overline \pi^{N|(N)}_{0:T,i}$. Moreover, let $\overline{\mathbb{P}}_i^{N|(N)}\in{\cal Q}^N(\mu^o,\overline \pi_{0:T,i}^{N|(N)})$~be given by
        \begin{align*}%
            \qquad\quad\overline{\mathbb{P}}_i^{N|(N)}:=\overline {\mathbb{P}}^N ({\mu^o, \overline \pi^{N|(N)}_{0:T,i}},\overline{p}^{N|(N)}_{0:T,i}):= \overline \mu^{o,N}\otimes \overline{ \mathbb{P}}^{N}_{(\overline \pi^{N|(N)}_{0,i},\overline{p}^{N|(N)}_{0,i})}\otimes \cdots \otimes \overline{\mathbb{P}}^{N}_{(\overline \pi^{N|(N)}_{T-1,i},\overline{p}^{N|(N)}_{T-1,i})}%
        \end{align*}
        so that it is the worst-case measure for $J^{N}_i(\mu^o,\overline \pi_{0:T,i}^{N|(N)})$ given in \eqref{eq:worst_nash} (see Proposition~\ref{pro:dpp_Nash}).
        \end{itemize}
\end{dfn}

The notions introduced in the following, which elaborate on certain laws and stochastic kernels for the one-step reward function $r:S\times A\times S \times {\cal P}(S)\mapsto \mathbb{R}$ under the worst-case measures (described above), will be used in Proposition~\ref{pro:ConGNtGam_ptb} and Proposition~\ref{pro:MNE2}.

\begin{dfn}[Laws under worst-case measures]\label{dfn:joint_laws}
    Let $(\pi^{(N)}_{0:T})_{N\in \mathbb{N}}\subseteq \Pi$ be a sequence of arbitrary Markov~policies. For every $N\in \mathbb{N}$ and $i\in\{1,\dots,N\}$, we define the following: Let ${\mathbb{P}}^{*|(N)}\in {\cal Q}( \mu_{0:T}^*,\pi^{(N)}_{0:T})$, ${\mathbb{P}}^{*}\in {\cal Q}( \mu_{0:T}^*,\pi^{*}_{0:T})$, and $\overline {\mathbb{P}}_i^{N|(N)}\in{\cal Q}^N(\mu^o,\overline \pi_{0:T,i}^{N|(N)})$ be given in Definition~\ref{dfn:worst_msr}. Then, for every $t=0,\dots,T-1$, let $\mathbb{Q}_{t}^{*|(N)},\mathbb{Q}_{t, i}^{N|(N)} \in {\cal P}(S\times A\times S \times {\cal P}(S))$ so that
    \begin{itemize}
        \item [$\cdot$] ${\mathbb{Q}}_{t}^{*|(N)}$ is the law of $(s_t,a_t,s_{t+1},\mu_t)$ under ${\mathbb{P}}^{*|(N)}$ at time $t$ with $\mu_{t}=\mu_{t}^*$.
        \item [$\cdot$]   $\mathbb{Q}^{N|(N)}_{t,i}$ is the law of $(s_t^i,a_t^i,s_{t+1}^i,e^{N}(\overline s_t^N))$ under $\overline {\mathbb{P}}_i^{N|(N)}$ at time $t$.
    \end{itemize}
    Moreover, if $\pi_{0:T}^{(N)}=\pi^{*}_{0:T}$, we let $\mathbb{Q}_{t}^*\in {\cal P}(S\times A\times S \times {\cal P}(S))$ be given~by
    \begin{align*}
            \mathbb{Q}_{t}^*:={\mathbb{Q}}_{t}^{*|(N)}
    \end{align*}
    so that it is the law of $(s_t,a_t,s_{t+1},\mu_t)$ under ${\mathbb{P}}^{*}$ at time $t$ with $\mu_{t}=\mu_{t}^*$.
\end{dfn}

We impose the following conditions on the stochastic kernels $\overline{q}^{N|(N)}_{t,i,j}$ and $p_{t}^{*}$.

\begin{as}\label{as:weak_conv_ptb}
    Let $({\mu}_{0:T}^*,{\pi}_{0:T}^*,{p}_{0:T}^*)$ be the (fixed) mean-field equilibrium. The following conditions hold.
    \begin{enumerate}
        \item[(i)] For every $t=0,\dots,T-1$ and $(s_t,a_t)\in S\times A$, the map
        \begin{equation*}
            \mathcal{P}(S) \ni \mu \mapsto p_{t}^{*}(\,\cdot \mid s_{t}, a_{t}, \mu) \in \mathcal{P}(S)
        \end{equation*}
        is continuous in $W_{1}$.

        \item[(ii)] There exists a deterministic sequence $\Delta_{N} \downarrow 0$ such that, for every $N \in \N$, every agent $i$, and any sequence of Markov policies $(\pi^{(N)})_{N \in \N} \subseteq \Pi$, the worst-case joint kernel in Lemma~\ref{lem:dpp_berge_nash} can be selected so that
        \begin{equation*}
            \qquad\max_{0 \leq t \leq T-1} \max_{(\overline{s}^{N}, \overline{a}^{N}) \in S^{N} \times A^{N}} \max_{1 \leq j \leq N} d_{W_{1}}\Big(\overline{q}^{N \vert (N)}_{t, i, j}(\,\cdot \mid \overline{s}^{N}, \overline{a}^{N}), p_{t}^{*}\big(\,\cdot \mid s^{j}, a^{j}, e^{N}(\overline{s}^{N})\big)\Big) \leq \Delta_{N}.
        \end{equation*}
    \end{enumerate}
\end{as}

\begin{rem} \label{rem:weak_conv_ptb}
    Assumption~\ref{as:weak_conv_ptb}~(ii) requires that the finite-agent worst-case state-transition kernels are uniformly approximated by the worst-case transition kernel of the mean-field game with a vanishing approximation error as $N \to \infty$. This means that, for every state-action configuration of the finite-agent game, the state evolution of an individual agent in the finite-agent game is well approximated by that of the representative agent in the mean-field game whenever $N$ is sufficiently large. This condition is crucial for our Nash Certainty Equivalence Principle result under model uncertainty, established in Theorem~\ref{thm:MNE0}. In the classical framework, the transition kernel is prescribed, so this condition and the associated arguments concerning the selection and uniform approximation of worst-case kernels are not needed (see, e.g., \cite{Huang2006nashCertainty,saldi2018markov,saldi2019approximate}).

    The continuity of the limiting worst-case transition kernel in Assumption~\ref{as:weak_conv_ptb}~(i) corresponds to a regularity condition and is the robust counterpart of the continuity assumption imposed on the reference transition kernel in \cite[Assumption~1~(d)]{saldi2018markov}.
\end{rem}

Proposition~\ref{pro:ConGNtGam_ptb} below then shows that, in addition to the convergence of the individual agent's state evolution, the population's empirical state distribution converges to the equilibrium state-measure flow $\mu_{0:T}^*$, thereby allowing us to connect the expected one-step rewards in the finite-agent game $(S,A,\mu^{o},\mathfrak{P}_{0:T}^N,r\;|\;N,\mathfrak{P}_{0:T})$ and the mean-field game $(S,A,\mu^o,\mathfrak{P}_{0:T},r)$.
\color{black}
\begin{pro} \label{pro:ConGNtGam_ptb}
    Suppose Assumption~\ref{as:msr} and Assumption~\ref{as:weak_conv_ptb} are satisfied. Let $(i_{N})_{N \in \N} \subseteq \N$ with $i_{N} \leq N$ be a sequence of deviating agents and $(\pi^{(N)}_{0:T})_{N\in \mathbb{N}}\subseteq \Pi$ be a sequence of arbitrary Markov policies. Moreover, for every $N\in \mathbb{N}$, let $\mathbb{Q}_{0:T}^{*|(N)},\mathbb{Q}^{N|(N)}_{0:T, i_{N}} \subseteq {\cal P}(S\times A\times S \times {\cal P}(S))$ be given in Definition~\ref{dfn:joint_laws}. Then, for every $t=0,\dots,T-1$, the following holds: for every $g\in C_b(S\times A \times S \times {\cal P}(S))$
    \begin{align}\label{eq:conv_4arg_ptb}
        \lim_{N \to \infty} \left| \mathbb{E}^{{\mathbb{Q}}_{t, i_{N}}^{N|(N)}}\big[g\big(s_t^{i_{N}},a_t^{i_{N}},s_{t+1}^{i_{N}},e^{N}(\overline{s}_{t}^{N})\big)  \big]-\mathbb{E}^{{\mathbb{Q}}_{t}^{*|(N)}}\big[ g(s_t,a_t,s_{t+1},\mu_t^{*})  \big] \right| = 0.
    \end{align}
\end{pro}

\begin{rem}
    To prove Proposition~\ref{pro:ConGNtGam_ptb}, we couple the finite-agent system with the auxiliary system in Definition~\ref{dfn:worst_msr_help}, where the agents follow the same policy profile and evolve independently under $p_t^*$ along the fixed equilibrium flow $\mu_{0:T}^*$. At each time $t$, the $N-1$ nondeviating agents in the auxiliary system have i.i.d.\ states with law $\mu_t^*$, while the deviating agent contributes mass $1/N$ to its empirical distribution. The condition~(ii) in Assumption \ref{as:weak_conv_ptb} controls the difference between the selected finite-agent worst-case marginal kernels and $p_t^*$ evaluated at the finite-agent empirical distribution, while the condition~(i) controls the change in $p_t^*$ when this distribution is replaced by $\mu_t^*$. These estimates yield convergence of the finite-agent empirical distribution to $\mu_t^*$ and allow us to approximate the deviating agent's one-step law by its mean-field counterpart under the same policy. The details are provided in Section \ref{sec:pro:ConGNtGam_ptb}.
\end{rem}

The preceding convergence results yield the following approximation of the deviating agent's worst-case objective by its expected cumulative reward under the mean-field equilibrium flow and the transition kernels selected via the mean-field dynamic programming principle established in Proposition~\ref{pro:dpp}. The proof of Proposition \ref{pro:MNE2} is presented in Section \ref{sec:pro:ConGNtGam_ptb}. \color{black}

\begin{pro} \label{pro:MNE2}
    Suppose Assumption~\ref{as:msr} and Assumption~\ref{as:weak_conv_ptb} are satisfied. Let $(i_{N})_{N \in \N} \subseteq \N$ with $i_{N} \leq N$ be a sequence of deviating agents and $(\pi^{(N)}_{0:T})_{N\in \mathbb{N}}\subseteq \Pi$ be a sequence of arbitrary Markov policies. For every $N\in \mathbb{N}$, let $J_{i_{N}}^N(\mu^o,\overline{\pi}_{0:T,i_{N}}^{N|(N)})$ be the worst-case objective function of the agent $i_{N}$ under $(\mu^o,\overline{\pi}_{0:T,i_{N}}^{N|(N)})$ (see Definition \ref{dfn:worst_msr}\;(ii)) and let ${\mathbb{P}}^{*|(N)}\in  {\cal Q}( \mu_{0:T}^*,\pi^{(N)}_{0:T})$ be given in Definition~\ref{dfn:worst_msr}\;(i). Then it holds that
    \begin{align*}
        \lim_{N \to \infty} \Bigg|J_{i_{N}}^N(\mu^o,\overline{\pi}_{0:T,i_{N}}^{N|(N)})- \mathbb{E}^{\mathbb{P}^{*|(N)}} 	\Bigg[\sum_{t=0}^{T-1}r(s_t,a_t,s_{t+1}, \mu_t^*)
        \Bigg] \Bigg|=0.
    \end{align*}
\end{pro}

\begin{rem}\label{rem:MNE2}
    Since it holds that $V(\mu_{0:T}^*)= \mathbb{E}^{\mathbb{P}^{*}(\mu_{0:T}^*)} [\sum_{t=0}^{T-1}r(s_t,a_t,s_{t+1}, \mu_t^*)]$ (see Proposition~\ref{pro:dpp}\;(ii)), Proposition~\ref{pro:MNE2} implies that, for every sequence $(i_{N})_{N \in \N}$ with $i_{N} \in \{1, \dots, N\}$,
    \[
    \lim_{N \to \infty} J_{i_{N}}^N(\mu^o,\overline{\pi}_{0:T}^{N|*})=V(\mu_{0:T}^*),
    \]
    where $\overline{\pi}_{0:T}^{N|*}$ is defined in \eqref{eq:MN_approx}.
\end{rem}

Combining Proposition~\ref{pro:ConGNtGam_ptb} and Proposition~\ref{pro:MNE2} with the optimality of $\pi^{*}_{0:T}$ in the mean-field equilibrium (see Definition \ref{dfn:robust_mfe}\;(i)), we conclude in Theorem \ref{thm:MNE0} that the Markov policy $\pi^{*}_{0:T}$ forms an approximate Markov-Nash equilibrium. The corresponding proof can be found in Section \ref{sec:thm:MNE0}.

\begin{thm} \label{thm:MNE0}
    Suppose Assumption~\ref{as:msr} and Assumption~\ref{as:weak_conv_ptb} are satisfied. Then, for any given $\varepsilon > 0$, there exists $N(\varepsilon)\in \mathbb{N}$ such that for each $N \geq N(\varepsilon)$, $(\pi^*_{0:T},\cdots,\pi^*_{0:T})$ is an $\varepsilon$-Markov-Nash equilibrium of $(S,A,{\mu}^{o},\mathfrak{P}_{0:T}^N, r\;|\;N,\mathfrak{P}_{0:T})$ (see Definition~\ref{dfn:MNE}), i.e.,  $\overline {\pi}_{0:T}^{N|*}\in \Pi^N$ defined for every $t=0,\dots,T-1$ by
    \begin{align}\label{eq:MN_approx}
    \mbox{$\overline{\pi}_{t}^{N|*}:S^N\ni \overline s_t^N\mapsto \overline \pi^{N|*}_{t}(d\overline a_t^N|\overline s_t^N):=\prod_{j=1}^N \pi_{t}^*(da_{t}^j|s_{t}^j)$}
    \end{align}
    satisfies that for every $i = 1, \dots, N$,
    \begin{equation*}
        J^{N}_i(\mu^o,\overline \pi_{0:T}^{N|*})  + \varepsilon \geq \sup_{\pi_{0:T}\in \Pi } J^{N}_i(\mu^o,(\overline \pi_{0:T}^{N|*,-i},\pi_{0:T})),
    \end{equation*}
    where $(\overline \pi_{0:T}^{N|*,-i},\pi_{0:T})$ is defined in \eqref{eq:pi:N-i}.
\end{thm}

\begin{rem} \label{rem:comp:LZZ}
    The work \cite{LZZ2026} also investigates discrete-time mean-field games under model uncertainty, where agents maximize their
    worst-case expected rewards over a set of admissible transition probabilities. While our framework considers a finite time horizon, \cite{LZZ2026} primarily study discounted infinite-horizon games with stationary equilibria. A key distinction lies in how the population measure flow enters the representative agent's robust optimization problem.

    In line with the classical mean-field game formulation of Saldi et al.~\cite{saldi2018markov}, we first prescribe a population
    measure flow $\mu_{0:T}$ and hold it fixed throughout the representative agent's robust optimization problem, including the minimization over admissible transition kernels. At this stage, we do not require the prescribed flow to coincide with the state distributions generated by the policies and transition kernels considered in the optimization. Consistency is instead imposed as part of the definition of a mean-field equilibrium $(\mu_{0:T}^{*}, \pi_{0:T}^{*}, p_{0:T}^{*})$: the prescribed flow must be generated by the equilibrium policy and a corresponding worst-case transition kernel (see Definition~\ref{dfn:robust_mfe}). Under our assumptions, we prove the existence of such an equilibrium in Theorem~\ref{thm:MFE}.

    In contrast, \cite{LZZ2026} specify the population dynamics for every admissible population policy $\pi$ and sequence of transition kernels $p=(p_t)_{t\geq 0}$. For a fixed initial distribution $\mu_0$, denote the resulting population measure flow by $\mu^{\pi,p}$. Written in our notation, its evolution is given by
    \begin{equation*}
        \mu_{t+1}^{\pi,p}(\,\cdot\,)
        =
        \int_{S \times A} p_t(\,\cdot \mid s,a,\mu_t^{\pi,p}) \pi_t(da \mid s,\mu_t^{\pi,p}) \mu_t^{\pi,p}(ds),
        \qquad t \geq 0,
    \end{equation*}
    with $\mu_0^{\pi,p}=\mu_0$ (see \cite[Equation~(2.7)]{LZZ2026}). Thus, the population flow satisfies the population evolution equation for each admissible pair $(\pi,p)$, whether or not $\pi$ is an equilibrium policy. Here, $\pi$ denotes the population policy, which remains fixed when the representative agent considers a unilateral deviation. Different admissible transition kernels may therefore generate different population flows even for the same population policy. These flows enter the representative agent's worst-case evaluation directly (see \cite[Equation~(2.8)]{LZZ2026}).

    This distinction is also reflected in the admissible policies. For each prescribed flow, our Markov policies take the form $\pi_t(da\mid s)$, depending on time and the individual state. The policies in \cite{LZZ2026} may additionally depend on the current population distribution, taking the form $\pi_t(da\mid s,\mu_t)$ and thereby allowing feedback on its evolution.

    Both papers investigate the relationship between mean-field equilibria and Nash equilibria in finite-agent games. In our framework, the policies obtained from a robust mean-field equilibrium yield approximate Markov-Nash equilibria for sufficiently large populations under the assumptions of Theorem~\ref{thm:MNE0}. Under their respective assumptions, \cite[Theorem~2.4]{LZZ2026} establishes the corresponding approximation result for stationary mean-field equilibrium
    strategies that are continuous in the population distribution. Conversely, \cite[Theorem~2.5]{LZZ2026} shows that a uniformly convergent sequence of stationary finite-agent equilibrium strategies has a mean-field equilibrium as its limit, provided that the limiting strategy is continuous in the population distribution. Moreover, \cite[Theorem~2.7]{LZZ2026} establishes the existence of exact stationary Markov-Nash equilibria for each finite $N\geq 2$ under their assumptions. They do not establish a general existence theorem for stationary mean-field equilibria, although \cite[Corollary~2.6]{LZZ2026} provides an existence result under an additional uniform equicontinuity assumption on the finite-agent equilibrium strategies.
\end{rem}

\section{Application: Repeated credit-investment under model uncertainty} \label{sec:len}

In this section, we introduce a repeated credit-investment problem under model uncertainty. In particular, we verify Assumption~\ref{as:msr}, construct an explicit mean-field equilibrium, and identify worst-case transition kernels satisfying Assumption~\ref{as:weak_conv_ptb} for the corresponding finite-agent game. Thus, our main results in Theorem~\ref{thm:MFE} and Theorem~\ref{thm:MNE0} apply.

We consider a financial market populated by banks or other lending institutions. At each date, each institution manages a short-lived loan portfolio that is substantially renewed before the next date. The model is specified in detail below.

\begin{dfn} \label{dfn:exm:len}
    Let $S = A = \{0, 1\}$ denote the state and action spaces, respectively. Fix a time horizon $T \geq 2$ and constants $M, \beta, \delta, \varrho > 0$ and $\lambda, \kappa, c \geq 0$ such that
    \begin{align}
        c + \beta + \delta + \varrho &\leq 1, \label{eq:cbdr}\\
        \kappa &\leq \delta \label{eq:kd}, \\
        M &> \kappa + \delta \big(1 + (T-1) (M + 1 + \lambda + \kappa)\big). \label{eq:M:ineq}
    \end{align}
    The state $s\in\{0,1\}$ is a stylized indicator of the credit quality of the institution's current loan portfolio. A performing portfolio ($s=1$) has satisfactory credit quality, with borrowers generally meeting their repayment obligations, whereas an impaired portfolio ($s=0$) has deteriorated credit quality associated with repayment difficulties or reduced expected recoveries. Since the portfolio is substantially renewed between dates, the next-period state describes the quality of the renewed portfolio. The action $a\in\{0,1\}$ represents the institution's investment in risk management: $a=1$ corresponds to enhanced screening, monitoring, hedging, or capital protection, while $a=0$ corresponds to maintaining baseline risk-management practices without this additional investment.
    \begin{enumerate}
        \item[(i)] For every $t \in \{0, \dots, T-1\}$, define $\mathfrak{P}_{t} \colon S \times A \times \mathcal{P}(S) \ni (s, a, \mu) \twoheadrightarrow \mathfrak{P}_{t}(s, a, \mu) \subseteq \mathcal{P}(S)$ by
        \begin{equation} \label{eq:len:frakP}
            \mathfrak{P}_{t}(s, a, \mu) := \{\mathrm{Bern}(\theta) \mid \theta \in [c + \delta a + \beta \mu(1), c + \delta a + \beta \mu(1) + \varrho]\}.
        \end{equation}

        \item[(ii)] Define $r \colon S \times A \times S \times \mathcal{P}(S) \to \R$ by setting for every $(s, a, \hat{s}, \mu) \in S \times A \times S \times \mathcal{P}(S)$,
        \begin{equation} \label{eq:len:r}
            r(s, a, \hat{s}, \mu) := M s + \hat{s} + \lambda \mu(1) - \kappa a.
        \end{equation}
    \end{enumerate}
\end{dfn}

\begin{rem}
    For a population distribution $\mu \in \mathcal{P}(S)$, $\mu(1)$ denotes the fraction of institutions whose current loan portfolios are in the performing state $s=1$. A larger value of $\mu(1)$ therefore indicates better aggregate credit quality across the population of institutions.

    The probability that the renewed portfolio is performing at the next date is uncertain. The admissible transition laws are
Bernoulli distributions with parameters satisfying
\begin{equation*}
	\theta \in [c+\delta a+\beta\mu(1),\, 	c+\delta a+\beta\mu(1)+\varrho].
\end{equation*}
Thus, enhanced risk management and a larger fraction of institutions with performing portfolios both shift the admissible probability interval upwards. The coefficients have the interpretations given in Table~\ref{tab:fin:int}.

\begin{table}[t]
	\centering
	\caption{Interpretation of the model coefficients.}
	\label{tab:fin:int}
	\renewcommand{\arraystretch}{1.15}
	\begin{tabular}{@{}c p{0.80\linewidth}@{}}
		\toprule
		Coefficient & Interpretation \\
		\midrule
		$c$
		& Baseline lower bound on the probability that the
		renewed portfolio is performing, corresponding to
		$a=0$ and $\mu(1)=0$. \\
		
		$\delta$
		& Effect of enhanced risk management: choosing $a=1$
		instead of $a=0$ increases both probability bounds
		by $\delta$, for a fixed population distribution. \\
		
		$\beta$
		& Effect of aggregate portfolio quality: both
		probability bounds increase by $\beta$ per unit
		increase in the fraction $\mu(1)$ of institutions
		with performing portfolios. \\
		
		$\varrho$
		& Width of the interval of admissible probabilities
		that the renewed portfolio is performing. \\
		
		$M$
		& Additional current reward associated with a
		performing rather than an impaired portfolio,
		represented by the term $Ms$. \\
		
		$\lambda$
		& Coefficient of the direct reward benefit associated
		with aggregate portfolio quality, represented by
		the term $\lambda\mu(1)$. \\
		
		$\kappa$
		& Cost of enhanced risk management, incurred when
		$a=1$. \\
		\bottomrule
	\end{tabular}
\end{table}

The one-period reward
\begin{equation*}
	Ms+\hat{s}+\lambda\mu(1)-\kappa a
\end{equation*}
has four components. The term $Ms$ contributes $M$ when the current portfolio is performing and zero otherwise. The term
$\hat{s}$ contributes one unit when the renewed portfolio is performing at the next date and zero otherwise. The term $\lambda\mu(1)$ captures a direct reward benefit from a larger fraction of institutions with performing portfolios. Finally,
$\kappa a$ is the cost of enhanced risk management.

    We assume that the transition probabilities depend only on the current investment and aggregate market conditions, without direct dependence on the current individual state $s$.

    Regarding the parameters, condition~\eqref{eq:M:ineq} is imposed to obtain the coordinatewise monotonicity used to identify the worst-case kernels in the finite-agent game. It is equivalent to
    \begin{equation*}
        M \big(1 - \delta (T-1)\big) > \kappa + \delta + \delta (T-1) (1 + \lambda + \kappa).
    \end{equation*}
    In particular, \eqref{eq:M:ineq} implies that $\delta (T-1) < 1$. Conversely, once $\delta (T-1) < 1$, the condition~\eqref{eq:M:ineq} is equivalent to
    \begin{equation*}
        M > \frac{\kappa + \delta + \delta (T-1) (1 + \lambda + \kappa)}{1 - \delta (T-1)}.
    \end{equation*}
    Hence, the parameter region is non-empty: first choose a sufficiently small $\delta$, then choose $\kappa \in [0, \delta]$, choose non-negative $c, \lambda$, and positive $\beta, \varrho$ satisfying \eqref{eq:cbdr}, and finally choose $M$ larger than the above threshold.
\end{rem}

We construct a candidate for the mean-field equilibrium in the following way:
\begin{itemize}
    \item[$\cdot$] Define the candidate local worst-case kernel using the lower endpoint of the uncertainty interval:
    \begin{equation} \label{eq:cand:p}
        p_{t}^{*}(\,\cdot \mid s, a, \mu) := \mathrm{Bern}\big(c + \delta a + \beta \mu(1)\big), \quad t = 0, \dots, T-1.
    \end{equation}

    \item[$\cdot$] Let the candidate equilibrium policy be defined by
    \begin{equation} \label{eq:cand:pi}
        \pi_{t}^{*}(\{1\} \mid s) := 1, \quad s \in S, \quad t = 0, \dots, T-1.
    \end{equation}

    \item[$\cdot$] For an arbitrary initial distribution $\mu_{0}^{*} = \mu^{o} \in \mathcal{P}(S)$, define the candidate state-measure flow recursively by
    \begin{equation} \label{eq:cand:mu}
        \mu_{t+1}^{*} := \mathrm{Bern}\big(c + \delta + \beta \mu_{t}^{*}(1)\big), \quad t = 0, \dots, T-1.
    \end{equation}
\end{itemize}
The condition~\eqref{eq:cbdr} guarantees that all numbers in \eqref{eq:cand:p} and \eqref{eq:cand:mu} lie in $[0, 1]$.

Indeed, below we show that the model given in Definition~\ref{dfn:exm:len} satisfies Assumption~\ref{as:msr} and that the above constructed candidate is a mean-field equilibrium in the sense of Definition~\ref{dfn:robust_mfe}.

\begin{lem} \label{lem:len:mfe}
    Under the setup given in Definition~\ref{dfn:exm:len}, the following statements hold.
    \begin{enumerate}
        \item[(i)] The set-valued maps $\mathfrak{P}_{0:T}$ and the one-step reward function $r$ satisfy Assumption~\ref{as:msr}~(ii) and~(iii), respectively.

        \item[(ii)] The triple $(\mu_{0:T}^{*}, \pi_{0:T}^{*}, p_{0:T}^{*})$ defined in \eqref{eq:cand:p}--\eqref{eq:cand:mu} is a mean-field equilibrium. If $\kappa < \delta$, the equilibrium action $1$ is the unique optimal action at every time $t$ and state $s$. If $\kappa \leq \delta$, action $1$ remains optimal, although at $t = T-1$ a tie with the action $0$ is possible if $\kappa = \delta$.
    \end{enumerate}
\end{lem}

\begin{rem}
    In the proof of Lemma~\ref{lem:len:mfe}, we show that the worst-case state-action value (recall~\eqref{eq:DPP_min2} and~\eqref{eq:DPP_min1}) satisfies
    \begin{equation*}
        \widehat{J}_{T-1}(s, 1, \mu_{T-1}^{*}) - \widehat{J}_{T-1}(s, 0, \mu_{T-1}^{*}) = \delta - \kappa \geq 0
    \end{equation*}
    and, for $t = 0, \dots, T-2$,
    \begin{equation*}
        \widehat{J}_{t}(s, 1, \mu_{t:T}^{*}) - \widehat{J}_{t}(s, 0, \mu_{t:T}^{*}) = (1 + M) \delta - \kappa \geq M \delta > 0
    \end{equation*}
    for the difference in value between actions $1$ and $0$ at time $t$. Hence, action $1$ is uniquely optimal for any time $t = 0, \dots, T-2$ and remains uniquely optimal also at time $T-1$, if $\kappa < \delta$. If $\kappa = \delta$, a tie between the two actions is possible at time $t = T-1$.
\end{rem}

Next, we analyze the corresponding finite-agent game, where we again identify a worst-case transition kernel as well as an $\varepsilon$-Markov-Nash equilibrium.

\begin{lem} \label{lem:len:mne}
    Under the setup given in Definition~\ref{dfn:exm:len}, the following statements hold.
    \begin{enumerate}
        \item[(i)] Fix $N \in \N$, a deviating agent $i \in \{1, \dots, N\}$, and a Markov policy $\pi_{0:T}^{(N)} \in \Pi$. Suppose that every agent $j \neq i$ applies the policy $\pi_{0:T}^{*} \in \Pi$ given in~\eqref{eq:cand:pi} and define $\overline{\pi}_{0:T, i}^{N \vert (N)}$ as in~\eqref{eq:over:pi:NN}. At each date, the worst-case kernel in the sense of \eqref{eq:N:ker} can be selected as
        \begin{equation} \label{eq:len:N:ker}
            \widehat{p}_{(t, i, \overline{\pi}_{0:T, i}^{N \vert (N)})} = \overline{p}_{t, i}^{N \vert (N)}(d \overline{s}^{N}_{t+1}|\overline s_t^N,\overline a_t^N) = \prod_{j = 1}^{N} p_{t}^{*}\big(d s_{t+1}^{j} \mid s_t^j, a_t^j, e^{N}(\overline{s}_{t}^{N})\big),
        \end{equation}
        where $p_{t}^{*}$ is given in~\eqref{eq:cand:p} and $\overline{p}_{t, i}^{N \vert (N)}$ is defined as in~\eqref{eq:over:p:NN}.

        \item[(ii)] With the finite-agent joint kernels chosen as in \eqref{eq:len:N:ker}, Assumption~\ref{as:weak_conv_ptb} is satisfied. In particular, for any $\varepsilon > 0$ and sufficiently large $N \in \N$,
        \begin{align*}
            \overline{\pi}_{t}^{N|*}:S^N\ni \overline s_t^N\mapsto \overline \pi^{N|*}_{t}(d\overline a_t^N|\overline s_t^N):=\prod_{j=1}^N \pi_{t}^*(da_{t}^j|s_{t}^j)
        \end{align*}
        forms an $\varepsilon$-Markov-Nash equilibrium.
    \end{enumerate}
\end{lem}

The proofs of Lemma~\ref{lem:len:mfe} and Lemma~\ref{lem:len:mne} are presented in Appendix \ref{sec:pro:len}.

\color{black}
\section{Computing mean-field equilibria via a value iteration algorithm} \label{sec:numeric}
Based on Proposition~\ref{pro:dpp} and Theorem~\ref{thm:MFE}, we derive an iterative scheme that allows to compute approximately a mean-field equilibrium $({\mu}_{0:T}^*,{\pi}_{0:T}^*,{p}_{0:T}^*)$ of $(S,A,\mu^o,\mathfrak{P}_{0:T},r)$. We provide a pseudo-code in Algorithm~\ref{alg:MFE} to show how it can be implemented.\footnote{All the numerical experiments have been performed with the following hardware configurations: a Macbook Air with Apple M1 chip, 8 GBytes of memory, and Mac OS 13.0. All the codes are provided in the following link: \url{https://github.com/JoLa2606/robust_MFE/}}

The algorithm proceeds as follows: Starting with given $\mu_{0:T}^* \in (\mathcal{P}(S))^{T}$, we apply the dynamic programming results as described in \eqref{eq:DPP_maximin2}--\eqref{eq:V-1} to derive the worst-case kernels $p^{*}_{0:T}$ and optimal Markov policies $\pi^{*}_{0:T}$ for $V(\mu^*_{0:T})$ (see Proposition~\ref{pro:dpp}). Next, we update $\mu^*_{0:T}$ by constructing a new sequence of state measures in the sense of Definition~\ref{dfn:robust_mfe}~(ii). This process is iterated until we attain a fixed point $({\mu}_{0:T}^*,{\pi}_{0:T}^*,{p}_{0:T}^*)$ in the sense of Proposition \ref{pro:FixPntEq} and Theorem~\ref{thm:MFE}. Note that as $S$ and $A$ are finite, in line with Assumption\;\ref{as:msr}\;(i), we will construct the corresponding probability measures by interpreting them as elements of a simplex in an Euclidean space.

\begin{algorithm}[ht]
    \caption{An iterative scheme for mean-field equilibrium (MFE) under model uncertainty} \label{alg:MFE}
    \begin{algorithmic}[1]
        {\footnotesize
            \STATE \textbf{Input:} $(S,A)$ with $n(S),n(A)<\infty$ (satisfying Assumption \ref{as:msr}\;(i)), $\mu^o\in {\cal P}(S)$ (i.e., initial distribution),  \\
                            \hskip4.em $(\mathfrak{P}_{0:T},r)$ (satisfying Assumption \ref{as:msr}\;(ii),\;(iii)), and $\mu^{*}_{0:T} \in (\mathcal{P}(S))^{T}$  (a~priori arbitrarily chosen);
            \STATE \textbf{Function} $\mathrm{MFE}\big(\mu_{0:T}^*;\;S,A,\mu^o,\mathfrak{P}_{0:T},r\big)$\textbf{:} \\
            \STATE  \hskip1.5em Set $\mu^{*}_{0}:=\mu^o$; \\
            \STATE  \hskip1.5em \textbf{while} $\mu^{*}_{1:T}$ still changes\\
            \STATE      \hskip3em \textbf{for} $t = T-1$ \textbf{to} $0$ \\
            \STATE          \hskip4.5em \textbf{for} $i = 1$ \textbf{to} $n(S)$ \\
            \STATE              \hskip6em \textbf{for} $j = 1$ \textbf{to} $n(A)$ \\
            \STATE                  \hskip7.5em \textbf{if} $t = T-1$ \\
            \STATE                      \hskip9em Compute $p^{*}_{T-1}(\cdot \mid s_{i}, a_{j}, \mu^{*}_{T-1}) \in \mathfrak{P}_{T-1}(s_{i}, a_{j}, \mu^{*}_{T-1})$ so that \\
            \hskip9em $\widehat{J}_{T-1}(s_{i}, a_{j}, \mu^{*}_{T-1}) = \sum_{s \in S} r(s_{i}, a_{j}, s, \mu^{*}_{T-1}) p^{*}_{T-1}(s \mid s_{i}, a_{j}, \mu^{*}_{T-1}) $;
            \STATE                  \hskip7.5em \textbf{else} \\
            \STATE                      \hskip9em Compute $p^{*}_{t}(\cdot \mid s_{i}, a_{j}, \mu^{*}_{t}) \in \mathfrak{P}_{t}(s_{i}, a_{j}, \mu^{*}_{t})$ so that \\
                                  \hskip9em $\widehat{J}_{t}(s_{i}, a_{j}, \mu^{*}_{t:T}) = \sum_{s \in S} \big(r(s_{i}, a_{j}, s, \mu^{*}_{t}) + \widehat{V}_{t+1}(s, \mu^{*}_{t+1:T})\big)p^{*}_{t}(s \mid s_{i}, a_{j}, \mu^{*}_{t}) $; \\
            \STATE              \hskip7.5em \textbf{end} \\
            \STATE              \hskip6em \textbf{end} \\
            \STATE              \hskip6em Compute $\pi^{*}_{t}(\cdot \mid s_{i}) \in \mathcal{P}(A)$ so that $\widehat V_{t}(s_{i},\mu^{*}_{t:T}) =  \sum_{a \in A} \widehat{J}_{t}(s_{i}, a, \mu^{*}_{t:T})\pi^{*}_{t}(a | s_{i}) $;\\
            \STATE          \hskip4.5em \textbf{end} \\
            \STATE      \hskip3em \textbf{end} \\
            \STATE      \hskip3em \textbf{for} $t = T-2$ \textbf{to} $0$ \\
            \STATE              \hskip4.5em Update $\mu^{*}_{t+1}$ so that $\mu^{*}_{t+1}(s_{i}) := \sum_{s \in S} \sum_{a \in A} p^{*}_{t}(s_{i} | s, a, \mu^{*}_{t})\pi^{*}_{t}(a | s) \mu^{*}_{t}(s) $ $\;\;\forall i=1,\dots,n(S)$; \\
            \STATE      \hskip3em \textbf{end} \\
            \STATE  \hskip1.5em \textbf{end} \\
            \STATE  \hskip1.5em \textbf{Return} $(\mu^{*}_{0:T},\pi^{*}_{0:T},p^{*}_{0:T}$)
        }
    \end{algorithmic}
\end{algorithm}

\subsection{Crowd motion under model uncertainty} \label{subsec:crowd}

We consider the following model, which can be found in \cite[Section 5.7]{Lauriere2024meanfield} and is inspired by the model studied in \cite{Elie2020meanfield}, and extend it by allowing for model uncertainty.

\begin{dfn} \label{dfn:exm:cro}
    Let $S := \{0, 1, \dots, 4\}$ and $A := \{-1, 0, 1\}$ be state and action spaces, respectively. Furthermore, let $T := 2$ be the time horizon, and let $\lambda \geq 0$ and $c > 0$ be given. Agents can decide to move along the one-dimensional (1D) grid world $S$ in both directions or stay where they are; we model these actions by {\it left $= -1$, stay $= 0$}, or {\it right $= 1$}.
    \begin{enumerate}[leftmargin=2.5em]
        \item[(i)] For every $t=0,1$, define $\mathfrak{P}_t^{\lambda}:S \times A \times \mathcal{P}(S)\ni (s_t,a_t,\mu_t) \twoheadrightarrow \mathfrak{P}_t^{\lambda}(s_t,a_t,\mu_t)\subseteq {\cal P}(S)$ by
        \[
            \mathfrak{P}_t^{\lambda}(s_t,a_t,\mu_t):=\Big\{\P \in \mathcal{P}(S) \;\Big|\; d_{{W}_{1}}\big(\P, p^{o}(\cdot | s_t,a_t,\mu_t)\big) \leq \lambda\Big\},
        \]
        where $d_{{W}_{1}}(\cdot,\cdot)$ is the $1$-Wasserstein distance on $S$ and $p^{o} \colon S \times A \times \mathcal{P}(S)\ni (s_t,a_t,\mu_t) \mapsto  p^{o}(\cdot|s_t,a_t,\mu_t)\in\mathcal{P}(S)$ is a reference stochastic kernel on $S$ given $S \times A \times \mathcal{P}(S)$ so that under $p^{o}(\cdot | s_t,a_t,\mu_t)$, $s_{t+1}$ satisfies
        \begin{align*}
            s_{t+1} = \min\big\{4, \max\{0, s_{t} + a_{t} + \varepsilon_{t+1}\}\big\},
        \end{align*}
        where $\varepsilon_{t+1}$ is independently identically distributed according to a uniform distribution with values in $A$.
        \item[(ii)] Define $r: S \times A \times S \times\mathcal{P}(S)\mapsto \mathbb{R}$ by setting for every $(s, a, \hat s, \mu)\in S \times A \times S \times\mathcal{P}(S)$,
        \begin{align*}
            r(s, a, \hat s, \mu) := \Big(1 - \frac{1}{2} \lvert \hat s - 2 \rvert\Big) - \frac{\lvert a \rvert}{4} - \log\big(\mu (\hat s) + c\big).
        \end{align*}
    \end{enumerate}
\end{dfn}

\begin{lem} \label{lem:exm:cro}
    Under the setup given in Definition~\ref{dfn:exm:cro}, let $\lambda \geq 0$ and $c > 0$ be given. Then the set-valued maps $\mathfrak{P}_{0:T}^{\lambda}$ and the one-step reward function $r$ satisfy Assumption~\ref{as:msr}\;(ii) and (iii).
\end{lem}

The proof of the above lemma can be found in Appendix \ref{sec:apdx:numeric}.

\begin{rem}
    The one-step reward $r$ is designed to encourage the agent to move toward the center while avoiding overly crowded areas. Additionally, it discourages unnecessary movement unless it is beneficial. The parameter $c$ allows to model the degree of aversion of crowds. According to the reference kernel $p^{o}$, the agent can either remain in her current position or move to one of the adjacent positions through her own action, while the random disturbance $\varepsilon_{t+1}$ may cause an additional displacement. This reflects scenarios such as a concert where people prefer to be near the center but also wish to avoid excessively crowded spots. Agents try to move around in front of the stage based on their own actions but may also be randomly pushed around by the crowd.
\end{rem}

\begin{figure}[t]
    \centering
    \subfigure[Values for $V(\mu^*_{0:2})$.]{
        \label{fig:V} \includegraphics[scale=0.285]{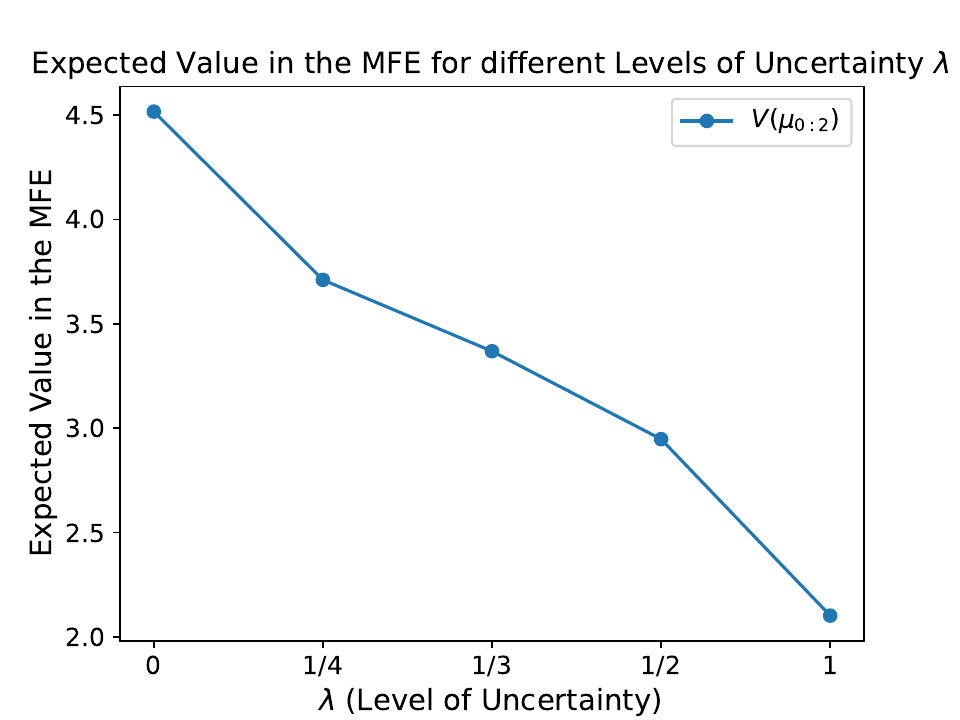}}
    \subfigure[Weights of $\mu^*_1$.]{
        \label{fig:mu1} \includegraphics[scale=0.285]{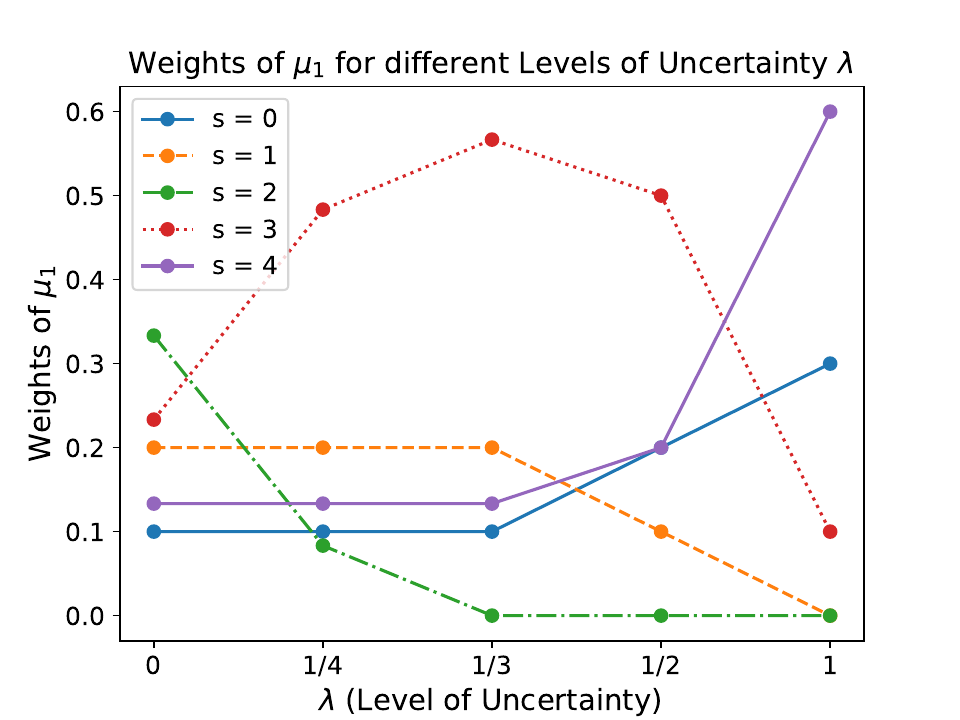}}
    \subfigure[Weights of $\mu^*_2$.]{
        \label{fig:mu2} \includegraphics[scale=0.285]{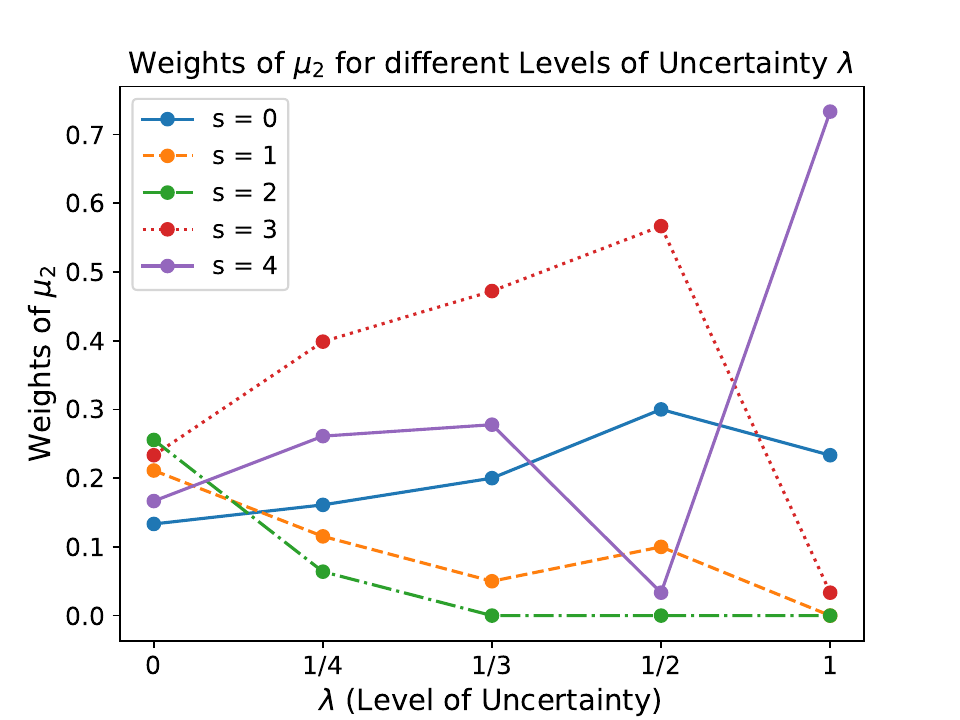}}
    \caption{Sensitivity of values for $V(\mu^*_{0:2})$ and weights of $\mu_{1}^*, \mu_{2}^*$ with respect to uncertainty level $\lambda$ with given $\mu^o=\mu^*_0=(0.2, 0.1, 0.05, 0.25, 0.4)$. }
    \label{fig:mu}
\end{figure}

Explicitly, we fix $c = 10^{-7}$ and consider different levels of uncertainty $\lambda \in \left\{0, \frac{1}{4}, \frac{1}{3}, \frac{1}{2}, 1\right\}$. Let $\mu^o =(w^{\mu^o}_{0},\dots,w^{\mu^o}_{4}) = (0.2, 0.1, 0.05, 0.25, 0.4)$ be the initial state distribution.
\begin{figure}[t]
    \centering
    \subfigure[{\scriptsize~Weights~of~$p^{*}_{0}(\cdot|1, -1, \mu_{0}^*)$.}]{
        \label{fig:wp0_1_-1} \includegraphics[scale=0.21]{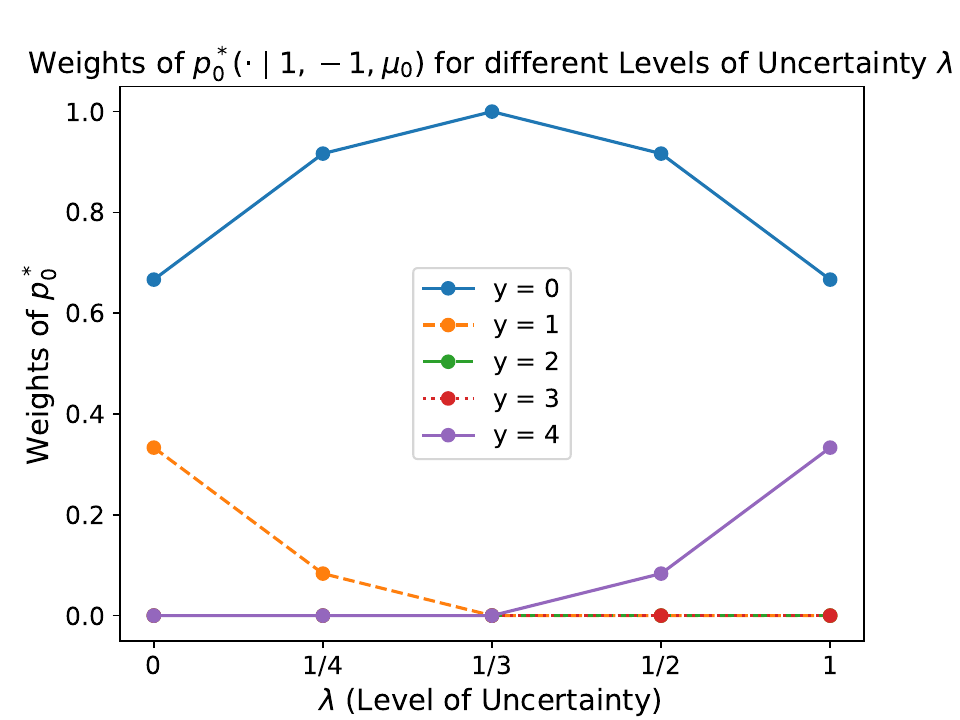}}\;
    \subfigure[{\scriptsize~Weights~of~$p^{*}_{0}(\cdot | 2, 0, \mu_{0}^*)$.}]{
        \label{fig:wp0_2_0} \includegraphics[scale=0.21]{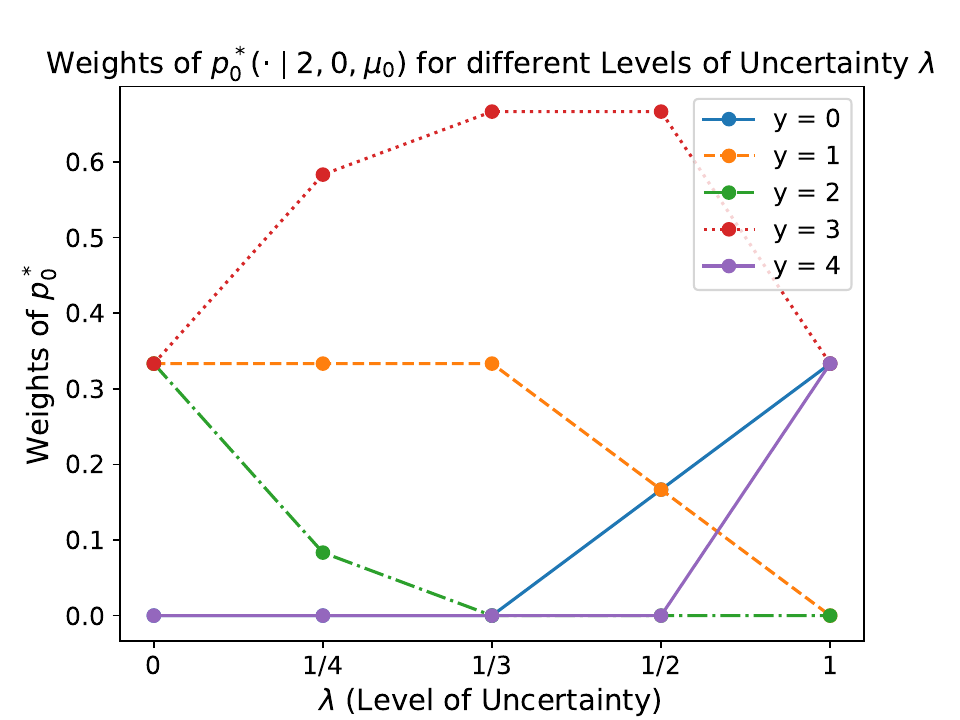}}\;
    \subfigure[{\scriptsize~Weights~of~$p^{*}_{1}(\cdot | 1, 0, \mu_{1}^*)$.}]{
        \label{fig:wp1_1_0} \includegraphics[scale=0.21]{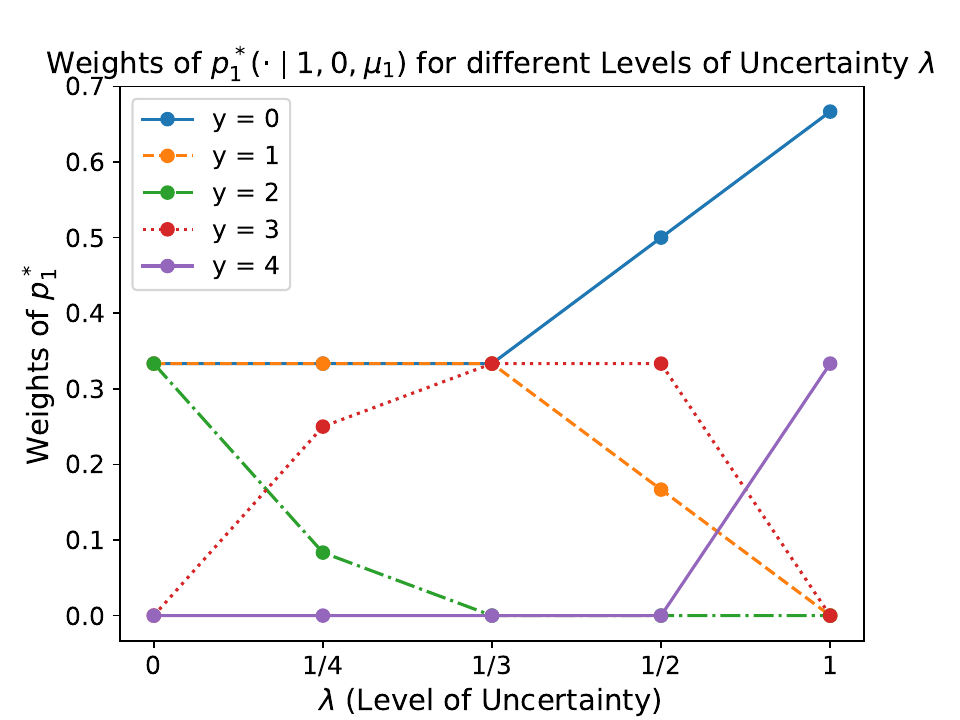}}\;
    \subfigure[{\scriptsize~Weights~of~$p^{*}_{1}(\cdot | 4, 1, \mu_{1}^*)$.}]{
        \label{fig:wp1_4_1} \includegraphics[scale=0.21]{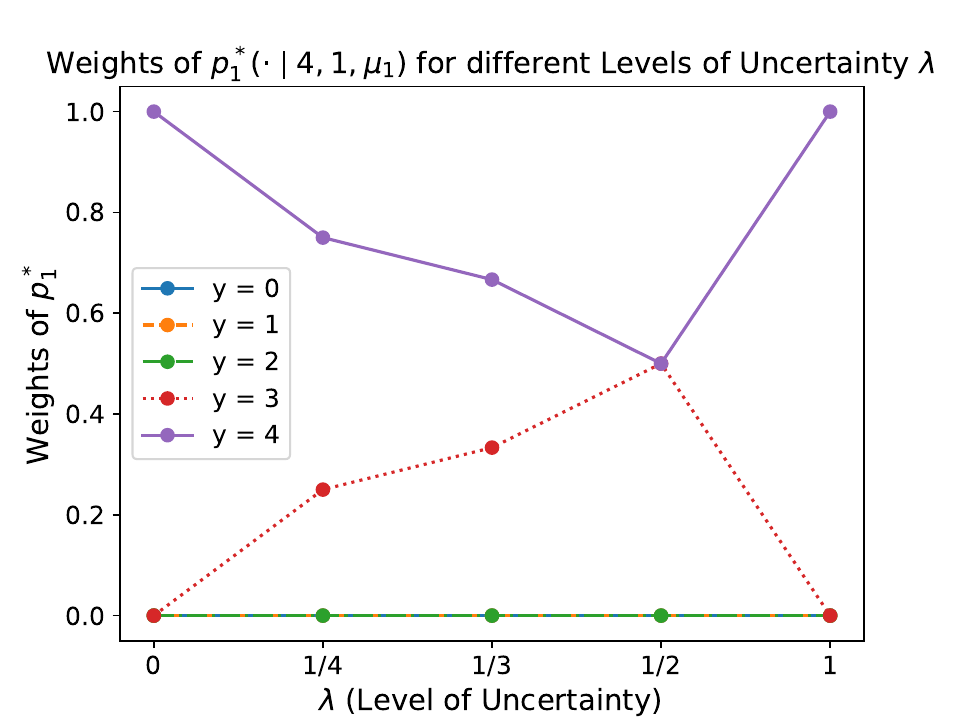}}
    \subfigure[{\scriptsize~Weights~of~$\pi_0^*(\cdot \;|\;0)$.}]{
        \label{fig:wpi0_0} \includegraphics[scale=0.21]{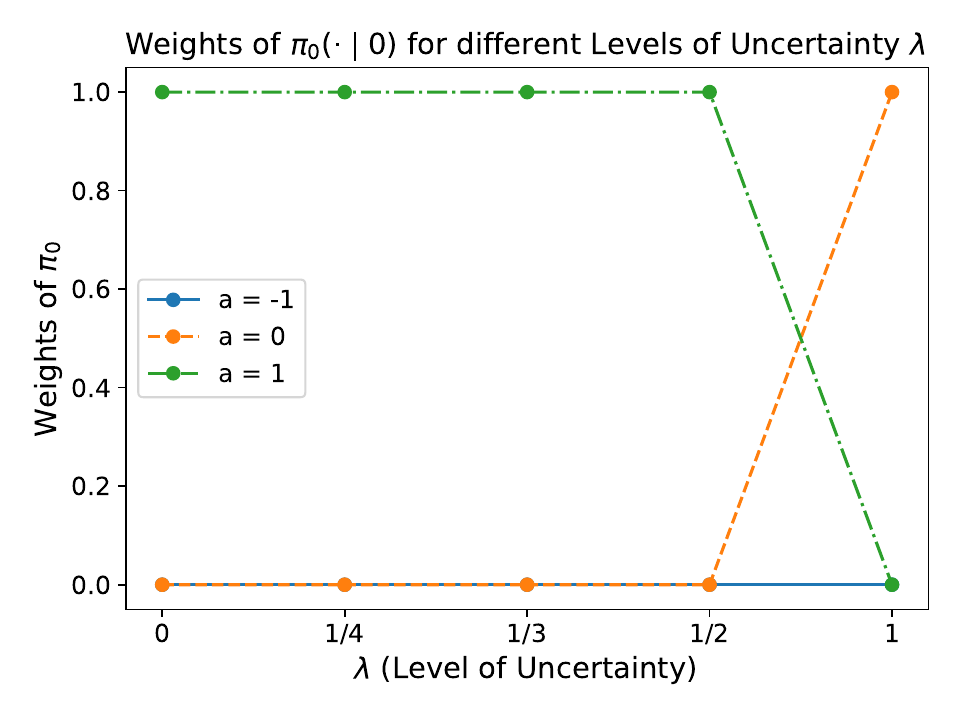}}\;
    \subfigure[Weights of $\pi_0^*(\cdot \;|\;3)$.]{
        \label{fig:wpi0_3} \includegraphics[scale=0.21]{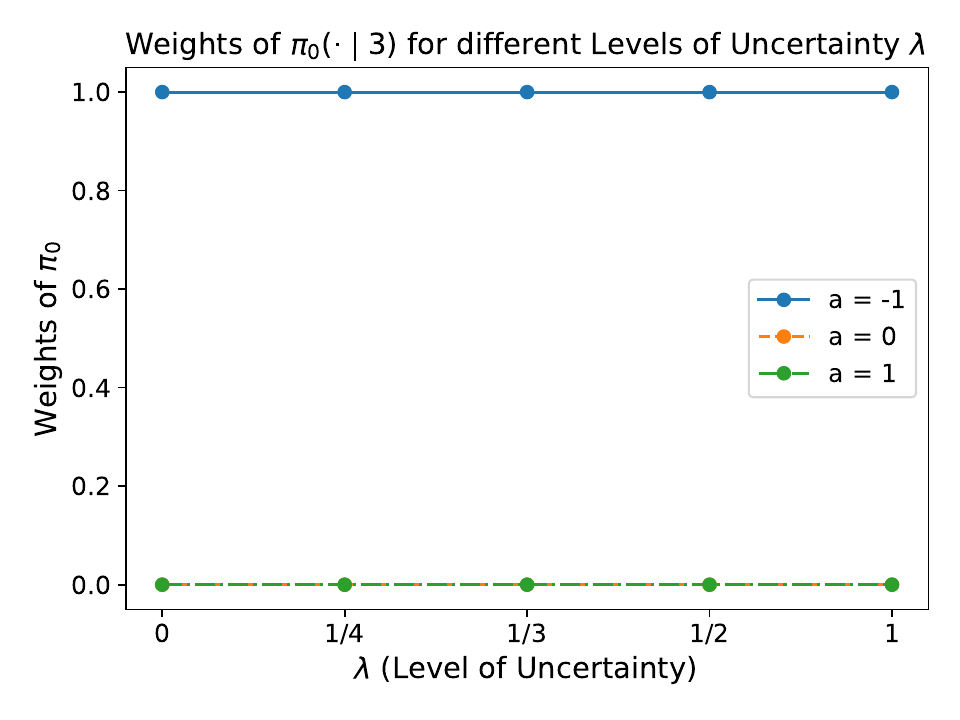}}\;
    \subfigure[Weights of $\pi_1^*(\cdot \;|\;2)$.]{
        \label{fig:wpi1_2} \includegraphics[scale=0.21]{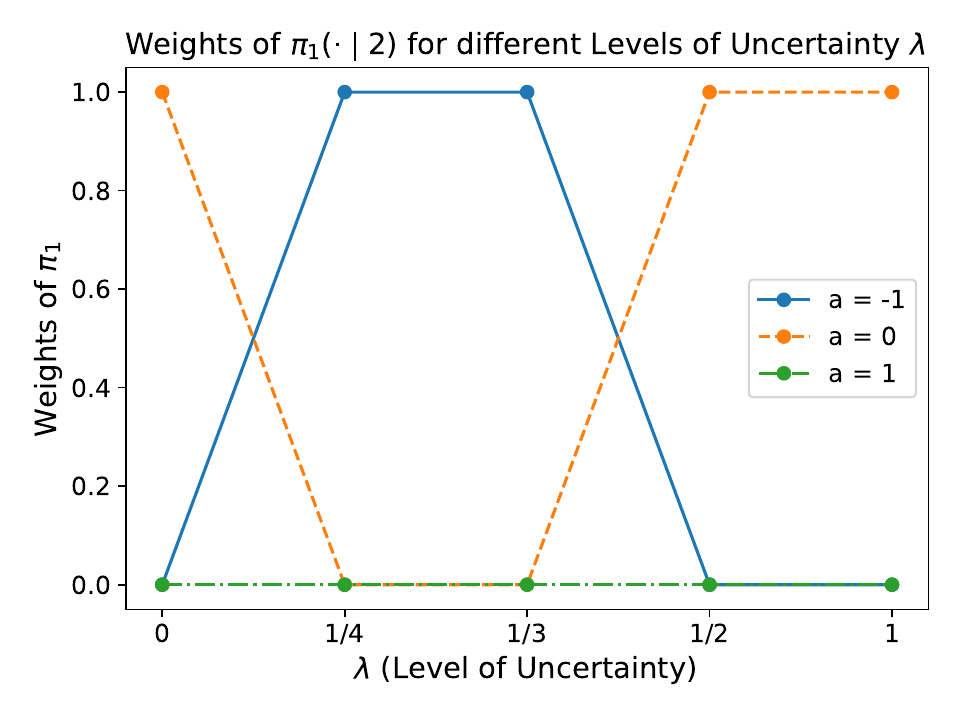}}\;
    \subfigure[Weights of $\pi_1^*(\cdot \;|\;3)$.]{
        \label{fig:wpi1_3} \includegraphics[scale=0.21]{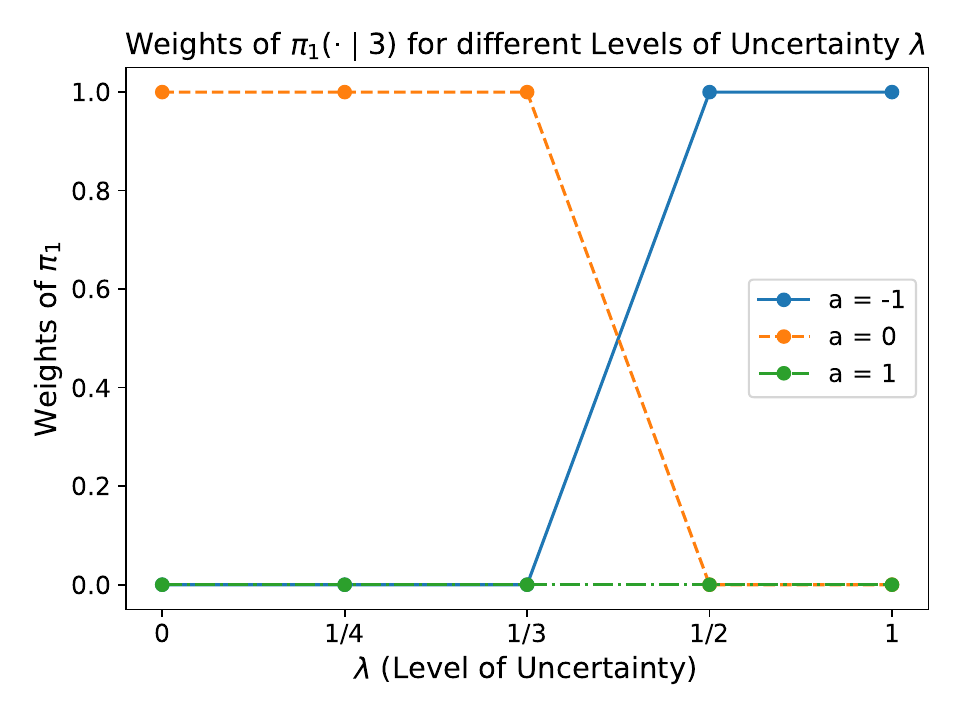}}
    \caption{Sensitivity of $(\pi^*_{0:2},p^*_{0:2})$ with respect to uncertainty level  $\lambda$.}
    \label{fig:wp1}
\end{figure}

Fig.\;\ref{fig:mu}(a) shows that the expected value $V(\mu_{0:2}^*)$ decreases as the uncertainty $\lambda$ increases, which is expected since a higher uncertainty level entails a potentially worse scenario.

Examining the state-measure flow $\mu_1^*$ at $t=1$, in Fig.\;\ref{fig:mu}(b) we observe that in the absence of model uncertainty, the majority of the weight is concentrated at the center position $s = 2$, with some weight distributed to the adjacent positions $s = 1$ and $s = 3$. The least weight is found at the extreme positions $s = 0$ and $s = 4$. This distribution can be interpreted that most individuals move towards the center, while a few choose to remain at the sides to avoid overcrowding. Whereas if the level of uncertainty increases, the distribution shifts, resulting in more weight being moved away from the center and an increase in the mass at $s = 3$. With large uncertainty, the mass is almost entirely shifted to the boundaries, $s = 0$ and $s = 4$. Similar effects are observed in the state-measure flow $\mu_2^*$ at $t=2$, defined analogously to \eqref{eq:mfe:mu}, as shown in Fig.\;\ref{fig:mu}(c).

Fig.\;\ref{fig:wp1} shows the sensitivity of the optimal pair $(\pi_{0:2}^*,p_{0:2}^*)$ in the mean-field equilibrium with respect to uncertainty level $\lambda.$

Although it is hard to interpret the sensitivity of the worst-case stochastic kernels $p^*_{0:2}$ shown in Fig.\;\ref{fig:wp1}\;(a)--(d), we can at least observe that our model uncertainty framework described in Definition\;\ref{dfn:exm:cro}\;(i) is working non-trivially.

Without model uncertainty, i.e. $\lambda=0$, the strategy $\pi_0^*$ at time $t=0$ makes the agent move to the center $s=2$ as the center is not crowded yet, as shown in Fig.\;\ref{fig:wp1}\;(e),\;(f). Indeed, we have seen in Fig.\;\ref{fig:mu}(b) that the weight of $\mu_1^*$ at $s=2$ is dominant. On the other hand, to avoid the crowd at time $t=1$, it becomes beneficial to stay at $s = 3$ rather than trying to move to the center $s=2$ while those already at the center remain there, as shown in Fig.\;\ref{fig:wp1}\;(g),\;(h).

As the uncertainty level increases, we observe some interesting effects. In Fig.~\ref{fig:wp1}\;(a)--(d), similar developments are observed across all presented scenarios for the worst-case kernels $p^{*}_{0}$ and $p^{*}_{1}$. With increasing uncertainty, the probability of getting shifted to overly crowded areas, particularly to $s = 4$, increases. In Fig.~\ref{fig:wp1}(e), the optimal strategy shifts from attempting to move towards the center, $s=2$, to staying at $s = 0$, i.e., avoiding movement to the right. Fig.~\ref{fig:wp1}(g) shows a similar effect: although being in the center is highly beneficial, the optimal strategy $\pi_1^*(\cdot\mid2)$ becomes to resist moving to the crowded areas ($s = 3$ and $s = 4$). In Fig.~\ref{fig:wp1}(h), to avoid staying in the overly crowded area $s = 3$ or moving to $s = 4$, $\pi_1^*(\cdot\mid3)$ changes in order to try to move towards the~center.

\subsection{Trading in illiquid markets under model uncertainty} \label{subsec:trade}

We consider a model of trading in an illiquid market in which each agent either holds an asset or has no position and the population distribution affects both trade execution probabilities and returns through market saturation.

\begin{dfn} \label{dfn:exm:tra}
    Let $S := \{0, 1\}$ and $A := \{-1, 0, 1\}$ denote the state and action spaces, respectively. Furthermore, let $T := 2$ be the time horizon, and let $\lambda \geq 0$, $\beta, \varepsilon_0, \varepsilon_1 \in [0, 1]$, and $R, C_{1}, C_{2} > 0$. The state represents whether the agent holds a position in a (potentially illiquid) asset ($s = 1$) or not ($s = 0$). Actions correspond to the agent’s intention to increase ($a = 1$), decrease ($a = -1$), or maintain ($a = 0$) this position.
    \begin{enumerate}[leftmargin=2.5em]
        \item[(i)] For every $t=0, 1$, define $\mathfrak{P}_t^{\lambda}: S \times A \times \mathcal{P}(S)\ni (s_t,a_t,\mu_t) \twoheadrightarrow \mathfrak{P}_t^{\lambda}(s_t,a_t,\mu_t)\subseteq {\cal P}(S)$ by
        \[
            \mathfrak{P}_t^{\lambda}(s_t,a_t,\mu_t):=\Big\{\P \in \mathcal{P}(S) \;\Big|\; d_{{W}_{1}}\big(\P, p^{o}(\cdot | s_t,a_t,\mu_t)\big) \leq \lambda\Big\},
        \]
        where $d_{{W}_{1}}(\cdot,\cdot)$ is the $1$-Wasserstein distance on $S$ and $p^{o} \colon S \times A \times \mathcal{P}(S) \ni (s_t,a_t,\mu_t) \mapsto  p^{o}(\cdot|s_t,a_t,\mu_t)\in\mathcal{P}(S)$ is a reference stochastic kernel on $S$. For $\mu \in \mathcal{P}(S)$, the reference transition kernel $p^o$ is given by
        \begin{equation*}
            p^o(1 \,|\, s,a,\mu) =
            \begin{cases}
                0, & (s, a) = (0, -1), \\
                \varepsilon_{0}, & (s, a) = (0, 0), \\
                1 - \beta \mu(1), & (s, a) = (0, 1), \\
                \beta \mu(1), & (s, a) = (1, -1), \\
                1 - \varepsilon_{1}, & (s, a) = (1, 0), \\
                1, & (s, a) = (1, 1),
            \end{cases}
        \end{equation*}
        with $p^{o}(0 \,|\, s, a, \mu) = 1 - p^{o}(1 \,|\, s, a, \mu)$.

        \item[(ii)] Define $r: S \times A \times S \times\mathcal{P}(S)\mapsto \mathbb{R}$ by setting for every $(s, a, \hat s, \mu)\in S \times A \times S \times\mathcal{P}(S)$,
        \begin{align*}
            r(s, a, \hat s, \mu) := \big(R - C_{1} \mu(1)\big) \hat s - C_{2} \lvert a \rvert.
        \end{align*}
    \end{enumerate}
\end{dfn}

\begin{lem} \label{lem:exm:tra}
    In the setting of Definition~\ref{dfn:exm:tra}, let $\lambda \geq 0$, $\beta, \varepsilon_{0}, \varepsilon_{1} \in [0, 1]$, and $R, C_{1}, C_{2} > 0$. Then the set-valued maps $\mathfrak{P}_{0:T}^{\lambda}$ and the one-step reward function $r$ satisfy Assumption~\ref{as:msr}\;(ii) and (iii).
\end{lem}

The proof of the above lemma can be found in Appendix \ref{sec:apdx:numeric}.

\begin{rem}
    The one-step reward $r$ can be decomposed into three components reflecting baseline returns, liquidity effects, and trading costs. When the agent holds the asset, she receives a baseline return of size $R$, which is reduced by a liquidity cost $C_{1}$ that increases with market saturation, reflecting the impact of illiquidity. In addition, active adjustments of the position incur operational costs $C_{2}$, thereby discouraging excessive trading.

    According to the reference kernel $p^{o}$, the agent may remain in her current position or adjust her exposure depending on her chosen action. However, due to market illiquidity, these adjustments are not deterministic. Instead, the probability of successfully increasing or decreasing the position depends on the overall level of market saturation. In highly saturated markets, both entering and exiting positions becomes more difficult. If the agent decides to remain in her current position, she can either actively or passively manage it. Active management incurs operational costs but ensures stability of the position. In contrast, passive management exposes the agent to the risk of unintended changes, such as forced liquidation due to margin constraints or accidental exposure arising from other investment activities.
\end{rem}

\begin{figure}[t]
    \centering
    \subfigure[Values for $V(\mu^*_{0:2})$.]{
        \label{fig:V:tra} \includegraphics[scale=0.285]{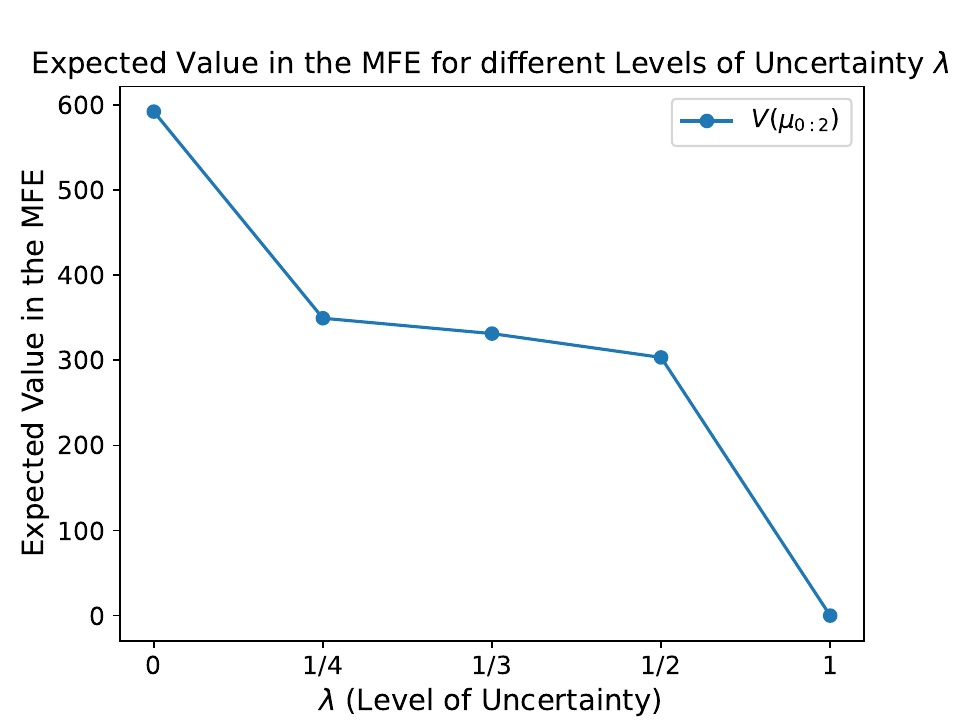}}
    \subfigure[Weights of $\mu^*_1$.]{
        \label{fig:mu1:tra} \includegraphics[scale=0.285]{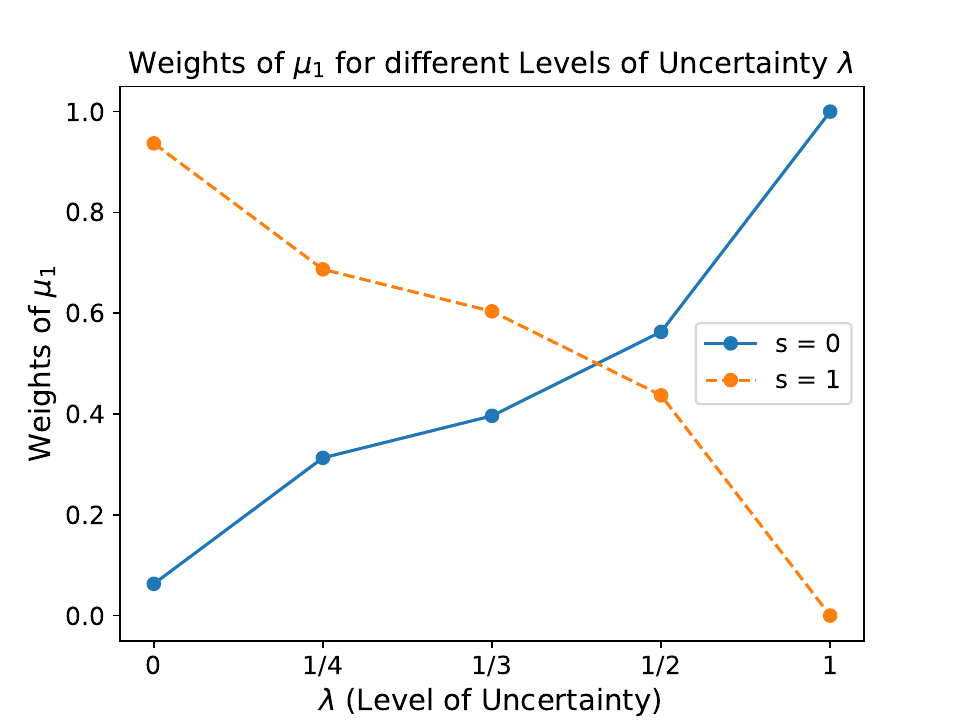}}
    \subfigure[Weights of $\mu^*_2$.]{
        \label{fig:mu2:tra} \includegraphics[scale=0.285]{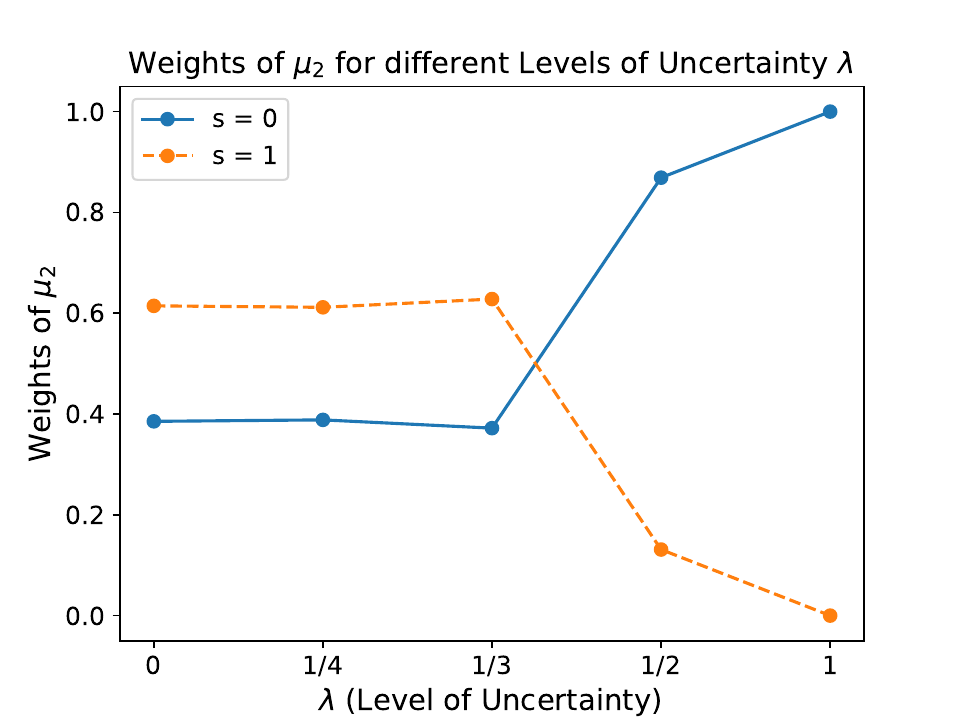}}
    \caption{Sensitivity of values for $V(\mu^*_{0:2})$ and weights of $\mu_{1}^*, \mu_{2}^*$ with respect to uncertainty level $\lambda$ with given $\mu^o=\mu^*_0=(0.9, 0.1)$. }
    \label{fig:mu:tra}
\end{figure}

We fix $R = 1000$, $C_{1} = 2000$, $C_{2} = 50$, $\beta = 0.7$, $\varepsilon_{0} = 0.1$, $\varepsilon_{1} = 0.2$, and consider the uncertainty levels $\lambda \in \left\{0, \frac{1}{4}, \frac{1}{3}, \frac{1}{2}, 1\right\}$. Let $\mu^o =(w^{\mu^o}_{0}, w^{\mu^o}_{1}) = (0.9, 0.1)$ be the initial state distribution.

\begin{figure}[t]
    \centering
    \subfigure[{\scriptsize~Weights~of~$p^{*}_{0}(\cdot|0, 1, \mu_{0}^*)$.}]{
        \label{fig:wp0_0_1:tra} \includegraphics[scale=0.21]{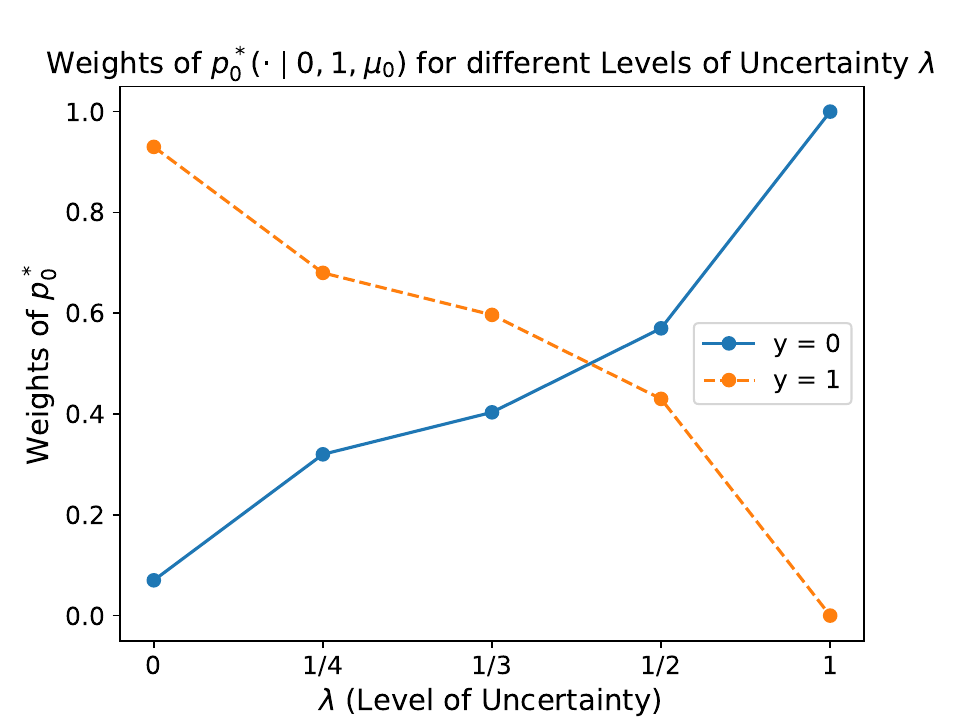}}\;
    \subfigure[{\scriptsize~Weights~of~$p^{*}_{0}(\cdot | 1, -1, \mu_{0}^*)$.}]{
        \label{fig:wp0_1_-1:tra} \includegraphics[scale=0.21]{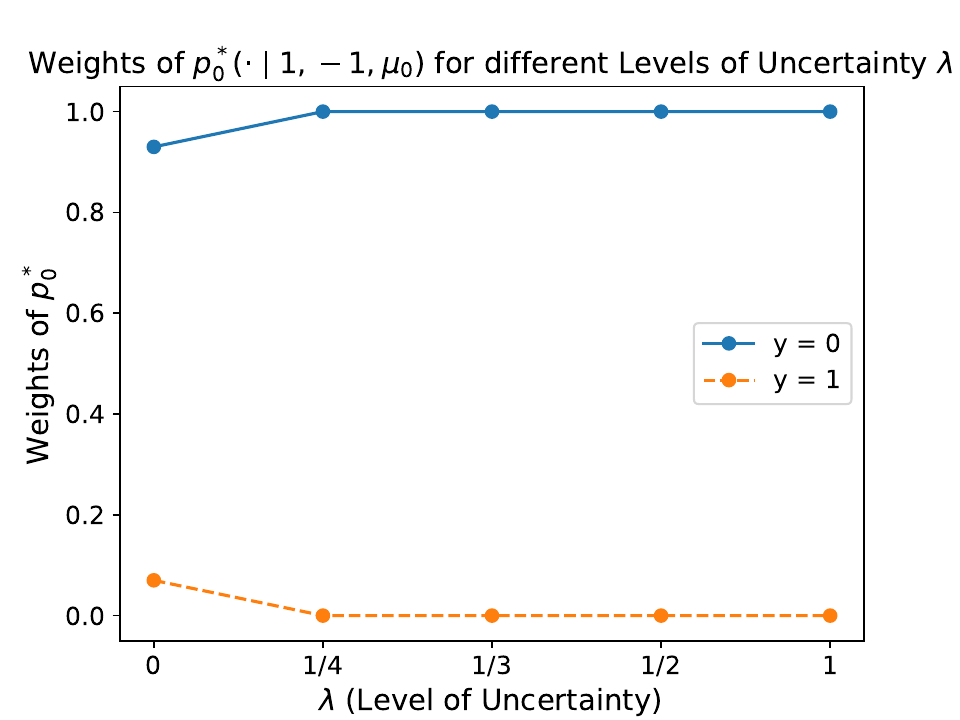}}\;
    \subfigure[{\scriptsize~Weights~of~$p^{*}_{1}(\cdot | 0, 1, \mu_{1}^*)$.}]{
        \label{fig:wp1_0_1:tra} \includegraphics[scale=0.21]{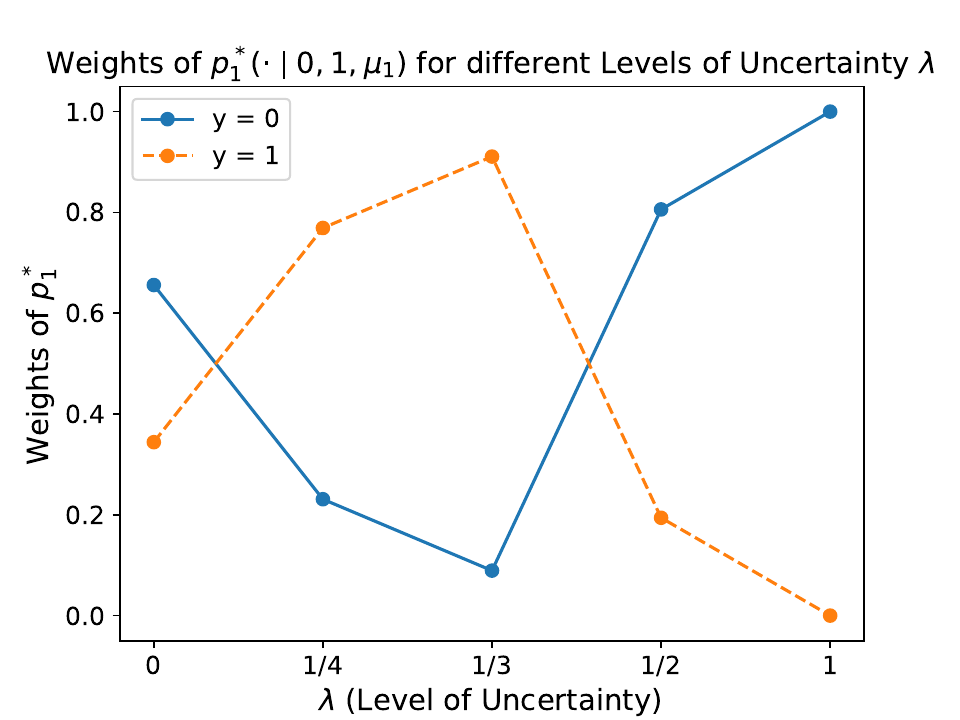}}\;
    \subfigure[{\scriptsize~Weights~of~$p^{*}_{1}(\cdot | 1, -1, \mu_{1}^*)$.}]{
        \label{fig:wp1_1_-1:tra} \includegraphics[scale=0.21]{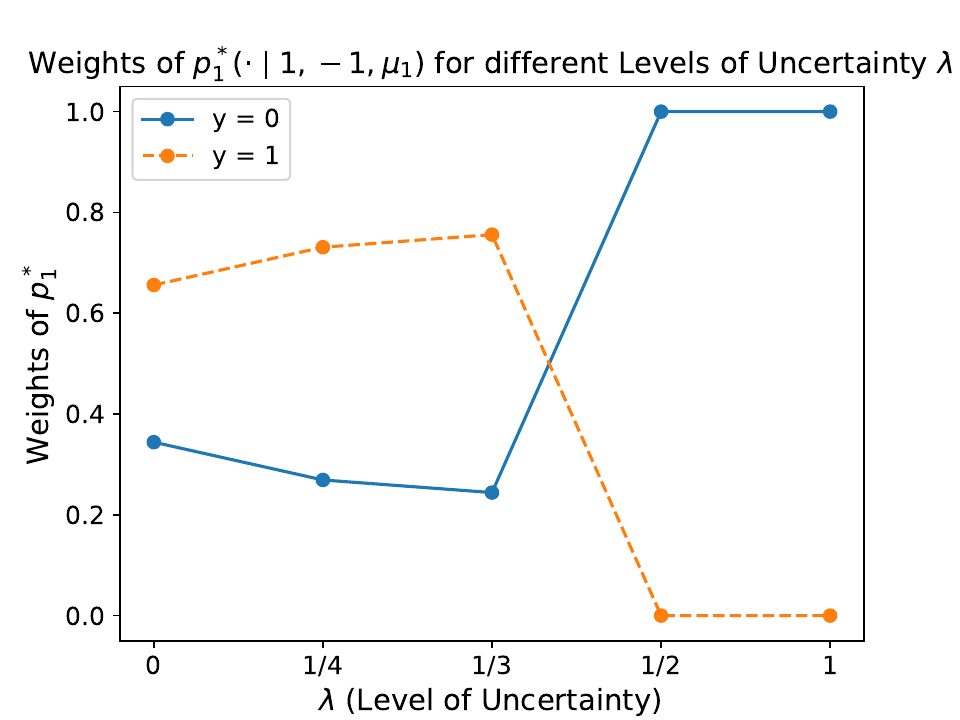}}
    \subfigure[{\scriptsize~Weights~of~$\pi_0^*(\cdot \;|\;0)$.}]{
        \label{fig:wpi0_0:tra} \includegraphics[scale=0.21]{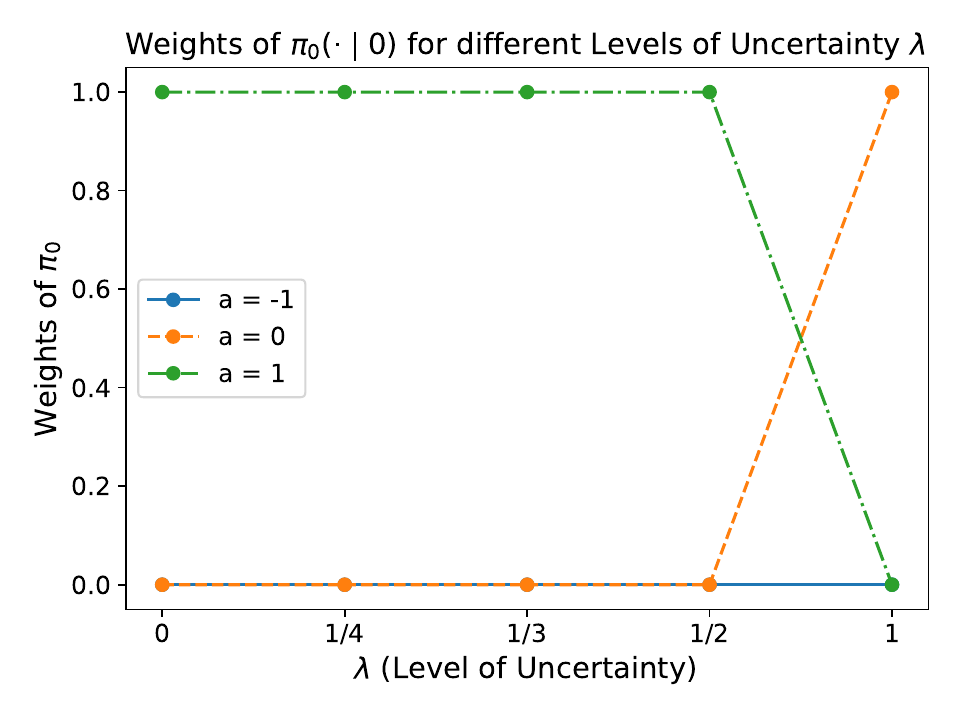}}\;
    \subfigure[Weights of $\pi_0^*(\cdot \;|\;1)$.]{
        \label{fig:wpi0_1:tra} \includegraphics[scale=0.21]{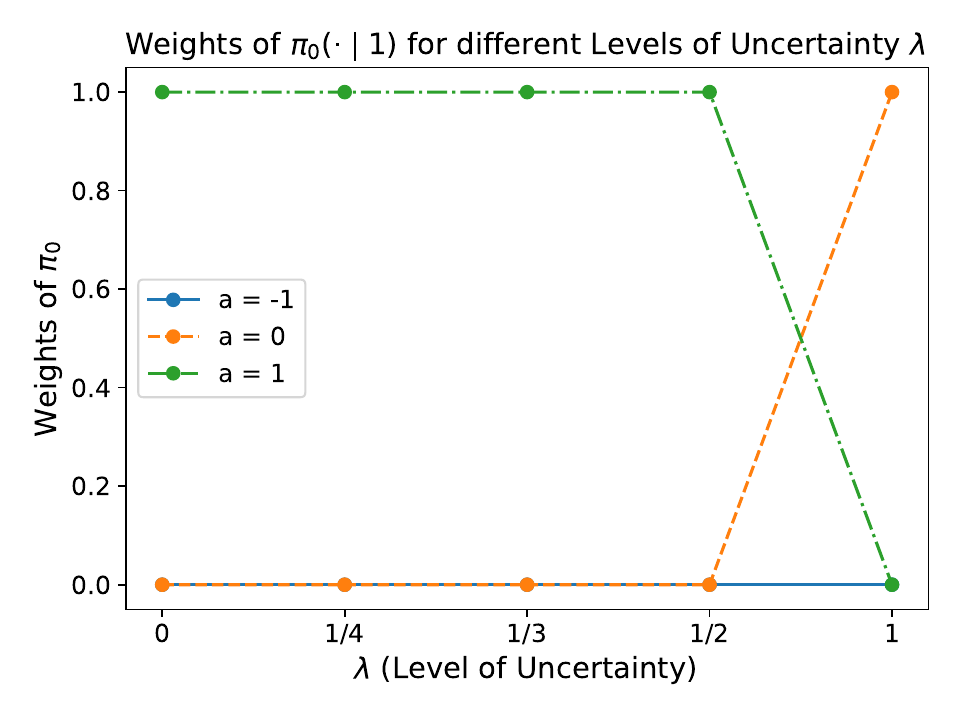}}\;
    \subfigure[Weights of $\pi_1^*(\cdot \;|\;0)$.]{
        \label{fig:wpi1_2:tra} \includegraphics[scale=0.21]{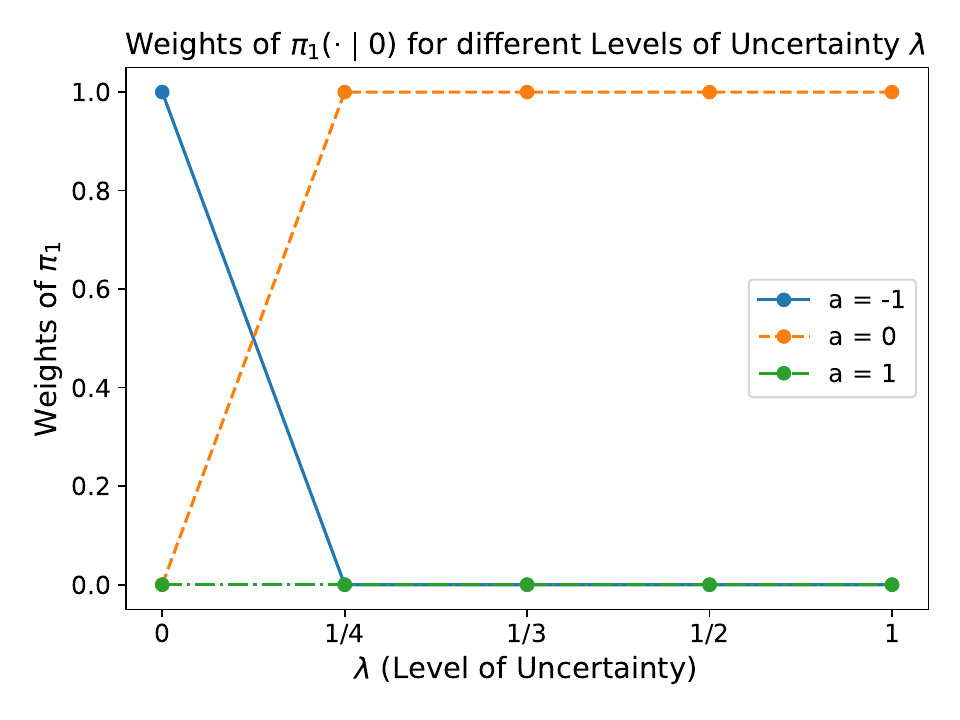}}\;
    \subfigure[Weights of $\pi_1^*(\cdot \;|\;1)$.]{
        \label{fig:wpi1_1:tra} \includegraphics[scale=0.21]{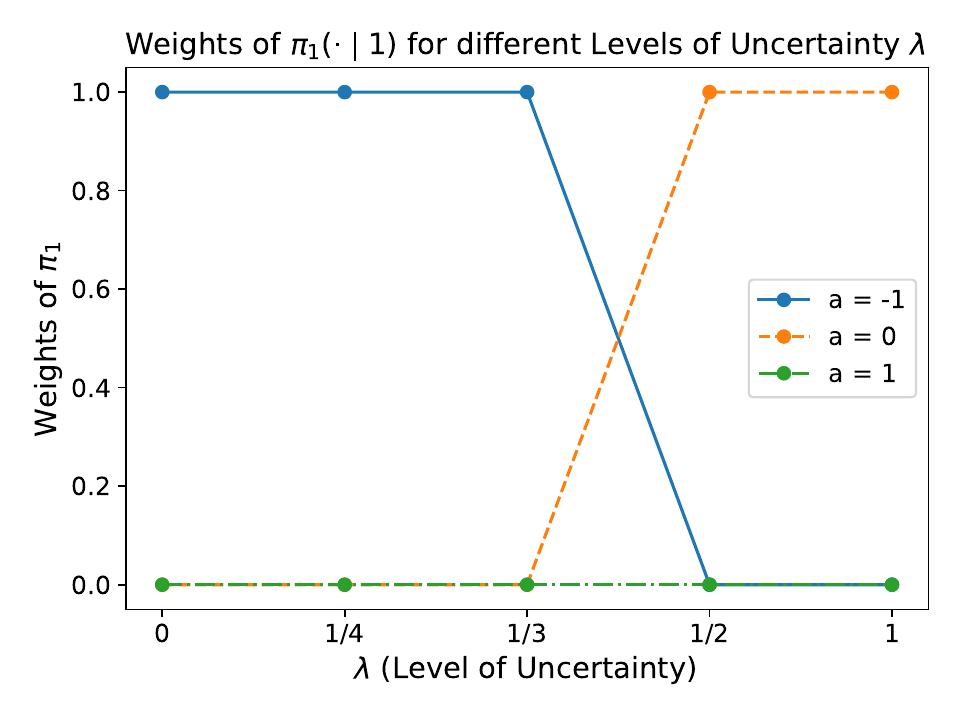}}
    \caption{Sensitivity of $(\pi^*_{0:2},p^*_{0:2})$ with respect to uncertainty level $\lambda$.}
    \label{fig:wp1:tra}
\end{figure}

Fig.~\ref{fig:mu:tra}~(a) shows the expected value $V(\mu^{*}_{0:2})$ for different levels of uncertainty $\lambda$. As anticipated, the expected value decreases as uncertainty increases.

Fig.~\ref{fig:mu:tra}~(b)--(c) show the equilibrium state-measure flow $\mu^{*}_{0:2}$. In the absence of uncertainty, we observe that at time $t = 1$, the state-measure $\mu^{*}_{1}$ is strongly concentrated on $s = 1$, indicating that most of the population is invested in the asset. Since the market is not yet saturated at $t = 0$ and remains highly liquid, agents can enter the position easily and obtain high rewards.

At time $t = 2$, the state-measure $\mu^{*}_{2}$ shifts back toward $s = 0$. As the market becomes highly saturated at $t = 1$, the liquidity cost outweighs the baseline return, making the investment unattractive and inducing agents to reduce their positions. However, due to the illiquidity of the market at this stage, only a fraction of the population is able to successfully exit the asset.

Fig.~\ref{fig:wp1:tra}~(a)--(d) show the worst-case transition kernels $p^{*}_{0:2}$. At time $t = 0$, the asset represents an attractive investment opportunity in a highly liquid market. Consequently, as uncertainty increases, the worst-case kernel reduces the probability of successfully entering the position. This, in turn, explains the composition of the state-measure $\mu^{*}_{1}$ under increasing uncertainty, where we observe that the proportion of agents invested in the asset also decreases.

Interestingly, at time $t = 1$, we observe a change in the dynamics. In the absence of uncertainty, the asset becomes an unattractive investment. Consequently, under the worst-case transition kernel $p^{*}_{1}$, the probability of successfully entering the position increases with uncertainty. This behavior is observed for the tested values $\lambda \leq \frac{1}{3}$. Beyond this threshold, however, we observe a sharp reversal, where the probability of successful entry decreases significantly. A similar pattern arises when the agent attempts to sell the asset. The probability of failing to exit increases up to $\lambda = \frac{1}{3}$ and then drops sharply thereafter. This change can be explained by examining the state-measure $\mu^{*}_{1}$, where the proportion of agents invested in the asset decreases as uncertainty increases. More precisely, for $\lambda \leq \frac{1}{3}$, a majority of agents is still invested, making the asset unattractive. For higher levels of uncertainty, however, the majority of agents is no longer invested, which restores the attractiveness of the asset. This shift in the population composition explains the observed change in the worst-case kernel at time $t = 1$.

Fig.~\ref{fig:wp1:tra}~(e)--(h) show the optimal policy $\pi^{*}_{0:2}$. At time $t = 0$, in the absence of uncertainty, we observe that in both states it is optimal to incur the operational costs $C_{2}$ in order to either attempt to enter the position or ensure that an existing position remains active, i.e., to choose $a = 1$. This behavior persists as uncertainty increases. Among the uncertainty levels considered, only at $\lambda = 1$ does the optimal strategy in both states switch to passive behavior, i.e., to choose $a = 0$. This can be explained by the high probability of transitioning to $s = 0$ under the worst-case transition kernel at this level of uncertainty.

At time $t = 1$, in the absence of uncertainty, it is again optimal in both states to incur the operational costs $C_{2}$, now in order to either prevent entering the position or actively attempt to exit it, i.e., to choose $a = -1$. As uncertainty increases, the optimal strategy in both states switches to passive behavior, i.e., to choose $a = 0$. This change can be explained by recalling the state-measure dynamics at time $t = 1$. As uncertainty increases, the proportion of agents invested in the asset decreases. Hence, the investment becomes less unattractive and eventually turns into a favorable investment for $\lambda \geq \frac{1}{2}$. At the same time, for $\lambda \leq \frac{1}{3}$, the probability of successfully exiting the position decreases under the worst-case transition kernel. In this case, it is no longer optimal to incur the operational cost $C_{2}$, as the benefit of actively attempting to exit the position diminishes while the liquidity costs $C_{1}$ are less severe.

Overall, the example illustrates that in illiquid markets, investment decisions are shaped not only by expected returns, but also by endogenous liquidity effects and the ability to unwind positions, both of which become more pronounced as uncertainty increases.

\section{Proofs of Results in Section \ref{sec:dpp}}\label{sec:proof:dpp}
\subsection{Proof of Lemma \ref{lem:dpp_berge} and Proposition \ref{pro:dpp}}\label{sec:proof:dpp:MFG}
\begin{lem} \label{lem:Vt}
    Suppose that Assumption \ref{as:msr} is satisfied. Let $\widehat{V}_{0:T}$ be given in \eqref{eq:DPP_maximin2}. Fix any $t\in \{0,1,\dots,T-2\}$ and assume that there exist some constants $\widehat{C}_{t+1} \geq 1$ and $\widehat{L}_{t+1} > 0$ such that for every $s_{t+1} \in S$ and every $\mu_{t+1:T}, \tilde \mu_{t+1:T} \in (\mathcal{P}(S))^{T-t-1}$, it holds that
    \begin{align} \label{eq:Ct}
        \begin{aligned}
            \lvert \widehat{V}_{t+1}(s_{t+1},\mu_{t+1:T}) \rvert&\leq \widehat C_{t+1}, \\ %
            \left\lvert \widehat{V}_{t+1}(s_{t+1},\mu_{t+1:T}) - \widehat{V}_{t+1}(s_{t+1}, \tilde{\mu}_{t+1:T}) \right\rvert &\leq \widehat L_{t+1}\sum_{u=t+1}^{T-1} d_{W_1}(\mu_u,\tilde{\mu}_u).
        \end{aligned}
    \end{align}
    Then the following hold:
    \begin{itemize}
        \item [(i)] $\widehat{J}_t$ given in \eqref{eq:DPP_min1} is continuous on $S\times A \times ({\cal P}(S))^{T-t}$. Furthermore, there exists a measurable selector $\widehat p_t: S \times A \times  ({\cal P}(S))^{T-t}\ni (s_t, a_t, \mu_{t:T}) \mapsto \widehat {p}_t(\cdot |s_t, a_t, \mu_{t:T})%
        \in \mathfrak{P}_t(s_t,a_t,\mu_t)$ satisfying~\eqref{eq:minimizer}.

        \item [(ii)] There exists a constant $\widehat{K}_t>0$ such that
        for every $s_{t} \in S$, $a_{t} \in A$, and every $\mu_{t:T}, \tilde \mu_{t:T} \in (\mathcal{P}(S))^{T-t}$, $|\widehat {J}_t(s_t, a_t, \mu_{t:T})-\widehat{J}_t(s_t, a_t, \tilde \mu_{t:T})|\leq \widehat{K}_t \sum_{u=t}^{T-1} {d_{W_{1}}(\mu_u, \tilde{\mu}_u)}.$
        \item[(iii)] $\widehat{V}_{t}$ is continuous on $S\times ({\cal P}(S))^{T-t}$. Furthermore, there exists a measurable selector $\widehat \pi_t:S\times {\cal P}(S)^{T-t} \ni (s_t,\mu_{t:T}) \mapsto \widehat \pi_t(\cdot |s_t,\mu_{t:T}) \in \mathcal{P}(A)$ satisfying \eqref{eq:maximizer}.

        \item[(iv)] There exist some constants $\widehat C_{t} \geq 1$ and $\widehat L_{t} > 0$ such that for every $s_{t} \in S$ and every $\mu_{t:T}, \tilde \mu_{t:T} \in (\mathcal{P}(S))^{T-t}$,
        \begin{align*}
                \qquad \lvert \widehat {V}_{t}(s_t, \mu_{t:T}) \rvert \leq \widehat C_{t},\qquad
                \left\lvert \widehat {V}_{t}(s_t, \mu_{t:T}) - \widehat{V}_{t}(s_t, \tilde \mu_{t:T}) \right\rvert \leq \widehat L_{t}  \sum_{u=t}^{T-1} d_{W_{1}}(\mu_u, \tilde {\mu}_u).
        \end{align*}
    \end{itemize}
\end{lem}
\begin{proof}
    We start by proving (i).
    To that end, set
    \[
    {\cal S}:=\left\{(s_t,a_t,\mu_{t:T},p_t)\Big |(s_t,a_t,\mu_{t:T})\in S \times A \times (\mathcal{P}(S))^{T-t},p_t\in \mathfrak{P}_t(s_t,a_t,\mu_t)\right\}
    \]
    and
    define an auxiliary map $F:{\cal S}\ni (s_t,a_t,\mu_{t:T},p_t) \mapsto F(s_t,a_t,\mu_{t:T},p_t)\in\R$
    by
    \begin{align*}
        F(s_t,a_t,\mu_{t:T},p_t) := \int_{S} \left( r(s_t, a_t, s_{t+1}, \mu_t) + \widehat{V}_{t+1}(s_{t+1}, \mu_{t+1:T})\right) p_t(ds_{t+1}).
    \end{align*}
    Then we consider a sequence $(s_t^{n}, a_t^{n}, {\mu}_{t:T}^{ n}, p_t^{n})_{n\in \mathbb{N}} \subseteq {\cal S}$ such that $(s_t^{n}, a_t^{n})\rightarrow(s_t^\star, a_t^\star),$ ${\mu}_u^{ n}\rightharpoonup\mu_u^\star$ (for every $u=t,\dots,T-1$), {and} $p_t^{n} \rightharpoonup p_t^\star$ as $n \to \infty$,  with some $(s_t^\star, a_t^\star, {\mu}^\star_{t:T}, p_t^\star)\in {\cal S}$.

    By the triangle inequality, for every $n\in\mathbb{N}$,
    \begin{align*}
        &\lvert F(s_t^{n}, a_t^{n}, {\mu}_{t:T}^{ n}, p_t^{n}) - F(s_t^\star, a_t^\star, {\mu}^\star_{t:T}, p_t^\star) \rvert \\
        &\leq \lvert F(s_t^\star, a_t^\star, {\mu}_{t:T}^\star, p_t^{n})  - F(s_t^\star, a_t^\star, {\mu}^\star_{t:T}, p_t^\star) \rvert + \lvert F(s_t^{n}, a_t^{n}, {\mu}_{t:T}^{ n}, p_t^{n})  - F(s_t^\star, a_t^\star, {\mu}^\star_{t:T}, p_t^{n})   \rvert=: \operatorname{I}^n+ \operatorname{II}^n.
    \end{align*}
    We will show that $\operatorname{I}^n$ and $\operatorname{II}^n$ vanish as $n\rightarrow \infty$.

    From Assumption \ref{as:msr}\;(i),\;(iii), and \eqref{eq:Ct}, it follows that $r(s_t^\star, a_t^\star, \cdot, \mu_t^\star) + \widehat{V}_{t+1}(\cdot, \mu_{t+1:T}^\star)$ are continuous and bounded in $S$, i.e., for every $s_{t+1}\in S$,  $|r(s_t^\star, a_t^\star, s_{t+1}, \mu_t^\star) + \widehat{V}_{t+1}(s_{t+1}, \mu^\star_{t+1:T})|\leq (C_r+\widehat C_{t+1}).$
    Furthermore, since $p_t^{n} \rightharpoonup p_t^\star$ as $n\rightarrow \infty$, we obtain that $\lim_{n \to \infty}\operatorname{I}^n = 0$.

    It remains to show the limit of $\operatorname{II}^n$. By Assumption~\ref{as:msr}~(i), $S$ and $A$ are finite. Hence, there exists $N \in \N$ such that for all $n \geq N$, $(s_t^{n}, a_t^{n}) = (s_t^\star, a_t^\star)$. By Assumption \ref{as:msr}\;(iii) and \eqref{eq:Ct}, for every $n \geq N$,
    \begin{align*}
        \operatorname{II}^n
        &\leq \int_{S}\bigg( \Big |r(s_t^{\star}, a_t^{\star}, s_{t+1}, \mu_t^n) - r(s_t^\star, a_t^\star, s_{t+1}, \mu_t^\star)\Big|\Big. \\
        &\quad\qquad +\Big|\widehat{V}_{t+1}(s_{t+1}, \mu_{t+1:T}^n)-\widehat {V}_{t+1}(s_{t+1}, \mu_{t+1:T}^\star)\Big| \bigg)p^{n}_{t}(ds_{t+1}) \\
        &\leq  L_{r} d_{W_{1}}(\mu_t^{n}, \mu_t^\star)+ \widehat L_{t+1}\sum_{u=t+1}^{T-1} d_{W_1}(\mu_u^n,{\mu}_u^\star) .
    \end{align*}
    The limit ${\mu}_u^{ n} \rightharpoonup\mu_u^\star$ (for every $u=t,\dots,T-1$) ensures that
    $\operatorname{II}^n$ vanishes as $n\rightarrow \infty$. Therefore, the map  $F:{\cal S} \to\R$ is continuous.

    Since $\mathfrak{P}_{t}$ is non-empty, compact-valued, and continuous (see Assumption~\ref{as:msr}\;(ii)) and the map $F$ is continuous, an application of Berge's maximum theorem (see, e.g., \cite[Theorem~17.31]{CharalambosKim2006infinite}) ensures the continuity of $\widehat {J}_{t}$ and the existence of the measurable selector $\widehat p_t:S \times A \times  ({\cal P}(S))^{T-t}\ni (s_t, a_t, \mu_{t:T}) \mapsto \widehat p_t(\cdot |s_t, a_t, \mu_{t:T})%
    \in \mathfrak{P}_t(s_t,a_t,\mu_t)$ satisfying~\eqref{eq:minimizer}.

    \vspace{0.5em}
    \noindent Now let us prove (ii). To that end, denote by $\tilde \P := \widehat p_t(\cdot| s_t, a_t, \tilde \mu_{t:T}) \in \mathfrak{P}_{t}(s_t, a_t, \tilde{\mu}_t)$ where $\widehat p_t$ denotes the measurable selector given in Lemma \ref{lem:Vt}\;(i). Furthermore, by Assumption \ref{as:msr}\;(ii), we can choose ${\mathbb{P}}\in \mathfrak{P}_t(s_t,a_t,\mu_t)$ such that the following hold:
    \begin{align}\label{eq:est1}
        d_{W_{1}}({\P}, \tilde{\P}) \leq L_{\mathfrak{P}_t} d_{W_{1}}(\mu_t, \tilde{\mu}_t),
    \end{align}
    and
    \begin{align*}
        \begin{aligned}
            \widehat {J}_{t}(s_t,a_t,\mu_{t:T})-\widehat {J}_{t}(s_t, a_t,\tilde {\mu}_{t:T})
            &\leq\int_{S} \left(r({s}_t,{a_t}, s_{t+1},{\mu}_t) + \widehat{V}_{t+1}(s_{t+1}, {\mu}_{t+1:T})\right) \P(d s_{t+1})\\
            &\quad - \int_{S} \left(r(s_t, a_t, \tilde s_{t+1},\tilde{\mu}_t)+ \widehat{V}_{t+1}(\tilde s_{t+1}, \tilde{\mu}_{t+1:T})\right) \tilde\P(d\tilde s_{t+1})\\
            &=:\operatorname{B}({{\P}, \tilde{\P}}).
        \end{aligned}
    \end{align*}

    Furthermore, since for every\footnote{We refer to Section\;\ref{sec:notat_prelimi} for the definition of $\operatorname{Cpl}({\P}, \tilde{\P})$.} $\gamma \in \operatorname{Cpl}({\P}, \tilde{\P})$, by Assumption \ref{as:msr}\;(i),\;(iii), and \eqref{eq:Ct}, we have
    \begin{align*}
        \operatorname{B}({{\P}, \tilde{\P}}) &= \int_{S\times S} \Big(r({s}_t,{a_t}, s_{t+1},{\mu}_t)-r(s_t, a_t, \tilde s_{t+1},\mu_t) + r({s}_t,{a_t}, \tilde s_{t+1},{\mu}_t)-r(s_t, a_t, \tilde s_{t+1},\tilde{\mu}_t) \Big.\\
        &\quad \quad \quad\quad\Big. + \widehat{V}_{t+1}(s_{t+1}, {\mu}_{t+1:T})-\widehat{V}_{t+1}(\tilde s_{t+1}, \mu_{t+1:T})\Big.\\
        &\quad \quad \quad\quad \left. + \widehat{V}_{t+1}(\tilde s_{t+1}, {\mu}_{t+1:T})-\widehat{V}_{t+1}(\tilde s_{t+1}, \tilde{\mu}_{t+1:T})\right)\gamma(ds_{t+1},d\tilde{s}_{t+1}) \\
        &\leq \int_{S\times S} \Big(L'_{r} \lvert s_{t+1} - \tilde s_{t+1} \rvert + L_{r} d_{W_{1}}(\mu_{t}, \tilde \mu_{t}) \Big.\\
        &\quad \quad \quad\quad \Big.+ \widehat{L}'_{t+1} \lvert s_{t+1} - \tilde s_{t+1} \rvert + \widehat{L}_{t+1} \sum_{u = t+1}^{T-1} d_{W_{1}}(\mu_{u}, \tilde \mu_{u})\Big)\gamma(ds_{t+1},d\tilde{s}_{t+1}),
    \end{align*}
    where $L'_{r}, \widehat{L}'_{t+1} > 0$ can be chosen appropriately thanks to Assumption\;\ref{as:msr}\;(i).

    It thus holds that
    \begin{align*}
        \begin{aligned}
            \operatorname{B}({{\P}, \tilde{\P}})
            &\leq \widehat{K}_t\bigg(\sum_{u=t}^{T-1}d_{W_1}(\mu_u,\tilde \mu_u)+\inf_{\gamma \in\operatorname{Cpl}({\P}, \tilde{\P})} \int_{S\times S}|s_{t+1}-\tilde{s}_{t+1}|\gamma(ds_{t+1},d\tilde{s}_{t+1})\bigg)\\
            &= \widehat{K}_t \bigg(\sum_{u=t}^{T-1}d_{W_1}(\mu_u,\tilde \mu_u)+d_{W_1}({\P}, \tilde{\P})\bigg).
        \end{aligned}
    \end{align*}
    where $\widehat{K}_t:=(L_r + L'_r + \widehat L_{t+1} + \widehat L'_{t+1})>0$.

    Combined with \eqref{eq:est1}, this ensure that
    \begin{align*}
        \widehat {J}_{t}(s_t,a_t,\mu_{t:T})-\widehat {J}_{t}(s_t, a_t,\tilde {\mu}_{t:T}) \leq \widehat{K}_t(1+L_{{\mathfrak{P}}_t}) \sum_{u=t}^{T-1}d_{W_1}(\mu_u,\tilde \mu_u).
    \end{align*}

    Using the same arguments as those used in the above upper bound, we can obtain the lower bound $\widehat {J}_{t}(s_t,a_t,\mu_{t:T})-\widehat {J}_{t}(s_t, a_t,\tilde {\mu}_{t:T}) \geq -\widehat{K}_t(1+L_{{\mathfrak{P}}_t}) \sum_{u=t}^{T-1}d_{W_1}(\mu_u,\tilde \mu_u),$ by using the same constant $\widehat{K}_t>0$. This completes the proof.

    The proof of part (iii) follows from similar arguments as those used in the proof of (i). We define a map ${G}:S\times ({\cal P}(S))^{T-t}\times {\cal P}(A)\ni (s_t,\mu_{t:T},\pi_t)\rightarrow G(s_t,\mu_{t:T},\pi_t)\in \mathbb{R}$ by
    \[
    G(s_t,\mu_{t:T},\pi_t):= \int_{A} \widehat {J}_{t}(s_t, a_t,\mu_{t:T}) \pi(da_t).
    \]
    Then we consider a sequence $(s_t^n,\mu^n_{t:T},\pi_t^n)_{n\in \mathbb{N}} \subseteq S\times ({\cal P}(S))^{T-t}\times {\cal P}(A)$ such that $s_t^n\rightarrow s_t^\star$ ${\mu}_u^{ n}\rightharpoonup\mu_u^\star$ (for every $u=t,\dots,T-1$),  and $\pi_t^{n} \rightharpoonup \pi_t^\star$, as $n \to \infty$ with some $(s_t^\star, {\mu}_{t:T}^\star, \pi_t^\star)\in S\times {\cal P}(S)^{(T-t)}\times {\cal P}(A)$.

    By the triangle inequality, for every $n\in\mathbb{N}$,
    \begin{align*}
        &\lvert G (s_t^n,\mu^n_{t:T},\pi_t^n) - G (s_t^\star,\mu^\star_{t:T},\pi_t^\star) \rvert \\
        &\;\; \leq \lvert G (s_t^\star,\mu^\star_{t:T}, \pi_t^{n}) - G (s_t^\star,\mu^\star_{t:T},\pi_t^\star) \rvert+ \lvert G (s_t^{n},\mu^{n}_{t:T},\pi_t^{n}) - G (s_t^\star, \mu^\star_{t:T}, \pi_t^n) \rvert=: \operatorname{III}^n+\operatorname{IV}^n.
    \end{align*}
    We will show that $\operatorname{III}^n$ and $\operatorname{IV}^n$ vanish as $n\rightarrow \infty$.

    Since $\widehat {J}_{t}(s_t^\star,\cdot,\mu^\star_{t:T})$ is continuous on $A$ (see Lemma \ref{lem:Vt}\;(i)) and the action space $A$ is finite (see Assumption \ref{as:msr}\;(i)), the limit $\pi_t^{n} \rightharpoonup \pi_t^\star$ ensures that $\operatorname{III}^n$ vanishes as $n\rightarrow \infty$.

    Furthermore, as $S$ is also finite (see Assumption~\ref{as:msr}~(i)), there exists $N \in \N$ such that for every $n \geq N$ we have $s^{n}_{t} = s^{\star}_{t}$. By Lemma \ref{lem:Vt}\;(ii), we then have for every $n \geq N$,
    \begin{align*}
        \operatorname{IV}^n \leq \int_{A} \left| \widehat{J}_{t}(s_t^{\star}, a_t,\mu^{n}_{t:T}) - \widehat{J}_{t}(s_t^\star, a_t,\mu^{\star}_{t:T}) \right| \pi_t^n(da_t)\leq  \widehat{K}_t \sum_{u=t}^{T-1} {d_{W_{1}}(\mu_u^n, {\mu}^\star_u)}.
    \end{align*}
    Combined with the limit ${\mu}_u^{ n} \rightharpoonup \mu_u^\star$ (for every $u=t,\dots,T-1$), this ensures that $\operatorname{IV}^n$ vanish as $n\rightarrow \infty$. Therefore, the map $G$ is continuous.

    Since $\mathcal{P}(A)$ is compact (noting that $A$ is finite) and $G$ is continuous, an application of Berge's maximum theorem ensures the continuity of $\widehat{V}_{t}$ and the existence of the measurable selector $\widehat \pi_t:S\times ({\cal P}(S))^{T-t} \ni (s_t,\mu_{t:T}) \mapsto \widehat \pi_t(\cdot | s_t,\mu_{t:T}) \in \mathcal{P}(A)$ satisfying \eqref{eq:maximizer}.

    Lastly we prove part (iv). By Assumption \ref{as:msr}\;(iii) and \eqref{eq:Ct}, for every $(s_t, a_t, \mu_{t:T}) \in S \times A \times (\mathcal{P}(S))^{T-t}$ and every $\mathbb{P} \in \mathfrak{P}_t(s_t,a_t,\mu_t)$, we have
    \begin{equation*}
        -C_r-\widehat{C}_{t+1} \leq \int_S \bigl(r(s_t,a_t,s_{t+1},\mu_t) + \widehat{V}_{t+1}(s_{t+1},\mu_{t+1:T})\bigr)\,\mathbb{P}(ds_{t+1})\leq C_r+\widehat{C}_{t+1}.
    \end{equation*}
    Hence, by taking the infimum over $\mathbb{P} \in \mathfrak{P}_t(s_t,a_t,\mu_t)$, then integrating with respect to $\pi \in \mathcal{P}(A)$, and finally taking the supremum over $\pi \in \mathcal{P}(A)$, we obtain
    \begin{align*}
        \lvert \widehat{V}_{t}(s_t, \mu_{t:T}) \rvert  \leq \widehat C_{t},\qquad \mbox{where}\quad   \widehat C_{t}:=C_{r}+\widehat C_{t+1}. 
    \end{align*}

    To have the other estimates, denote by ${\pi}:=\widehat \pi_t(\cdot |s_t,\mu_{t:T})\in {\cal P}(A)$ where $\widehat \pi_t$ is the measurable selector given in Lemma \ref{lem:Vt}\;(iii). Then since $\pi$ is not necessarily a maximizer for $\widehat{V}_t(s_t,\tilde{\mu}_{t:T})$ but for $\widehat{V}_{t}(s_t, \mu_{t:T})$, it holds
    \begin{align}
            \widehat{V}_{t}(s_t, \mu_{t:T})- \widehat{V}_t(s_t,\tilde{\mu}_{t:T}) \leq \int_A \left(\widehat {J}_{t}(s_t, a_t,\mu_{t:T}) -\widehat {J}_{t}(s_t,  a_t,\tilde \mu_{t:T}) \right) \pi(da_t).
    \end{align}
    Further, by Lemma \ref{lem:Vt}\;(ii), $\int_A \widehat {J}_{t}(s_t, a_t,\mu_{t:T}) -\widehat {J}_{t}(s_t,  a_t,\tilde \mu_{t:T}) \pi(da_t)\leq \widehat {K}_t \sum_{u=t}^{T-1} {d_{W_{1}}(\mu_u, \tilde{\mu}_u)}$,
    which leads to the upper bound estimates with letting $\widehat L_t:=\widehat{K}_t$.

    Using the same arguments as those used in the above estimates, we obtain that $\widehat{V}_{t}(s_t, \mu_{t:T})- \widehat{V}_t(s_t,\tilde{\mu}_{t:T})\geq- \widehat L_t\sum_{s=t}^{T-1} {d_{W_{1}}(\mu_u, \tilde{\mu}_u)}$, with the same constant $\widehat {L}_t>0$. This completes the~proof.
\end{proof}

\begin{proof}[Proof of Lemma \ref{lem:dpp_berge}]
    We will prove the parts (i)\;and\;(ii) together. First we claim that when  $t=T-1$, there exists a measurable selector $\widehat p_{T-1}:S \times A \times  {\cal P}(S)\ni (s_{T-1}, a_{T-1}, \mu_{T-1}) \mapsto \widehat p_{T-1}(\,\cdot\, |s_{T-1}, a_{T-1}, \mu_{T-1})%
    \in \mathfrak{P}_{T-1}(s_{T-1},a_{T-1},\mu_{T-1})$ satisfying \eqref{eq:minimizer2}. Indeed, since $\widehat{J}_{T-1}$ has a simple integrand $r(\cdot,\cdot,\cdot,\cdot)$ (see \eqref{eq:DPP_min2}), the same arguments as for the proof of Lemma~\ref{lem:Vt}\;(i) (applying Berge's maximum theorem), but with respect to the map
    $F:{\cal S} \to \mathbb{R}$ given by
    \[
    F(s_{T-1},a_{T-1},\mu_{T-1},{p}_{T-1}) := \int_{S}  r(s_{T-1}, a_{T-1}, s_{T}, \mu_{T-1}) p_{T-1}(ds_{T}),
    \]
    with ${\cal S}=\{(s_{T-1},a_{T-1},\mu_{T-1}, p_{T-1})\in S \times A \times \mathcal{P}(S) \times \mathcal{P}(S),\;p_{T-1}\in\mathfrak{P}_{T-1}(s_{T-1},a_{T-1},\mu_{T-1})\}$, ensure the existence of the selector $\widehat{p}_{T-1}$.

    Analogously, if $t=T-1$, there exists a measurable selector $\widehat {\pi}_{T-1}:S\times  {\cal P}(S)\ni (s_{T-1}, \mu_{T-1}) \mapsto \widehat \pi_{T-1}(\cdot |s_{T-1}, \mu_{T-1})%
    \in {\cal P}(A)$ satisfying \eqref{eq:maximizer}.  Indeed, we first claim that there is $\widehat{K}_{T-1}>0$ such that for every $s_{T-1} \in S$, $a_{T-1} \in A$, $\mu_{T-1}, \tilde \mu_{T-1} \in \mathcal{P}(S)$, it holds that
    \begin{align}\label{eq:est_J}
        |\widehat {J}_{T-1}(s_{T-1},a_{T-1},\mu_{T-1})-\widehat {J}_{T-1}(s_{T-1}, a_{T-1},\tilde \mu_{T-1})|\leq \widehat{K}_{T-1}  {d_{W_{1}}(\mu_{T-1}, \tilde{\mu}_{T-1})}.
    \end{align}
    By the existence of $\widehat p_{T-1}$ satisfying \eqref{eq:minimizer2}, the arguments devoted for the proof of Lemma~\ref{lem:Vt}\;(ii) using $\widehat p_{T-1}$ and Assumptions~\ref{as:msr}\;(i),\;(iii)\ ensure that we have $\widehat{K}_{T-1}>0$ satisfying \eqref{eq:est_J}.

    By \eqref{eq:est_J}, we can use the same arguments presented for the proof of Lemma~\ref{lem:Vt}\;(iii) using Berge's maximum theorem to have the existence of the measurable selector $\widehat {\pi}_{T-1}$ satisfying~\eqref{eq:maximizer}.

    So far we have proven (i) and (ii) for the case $t=T-1$. The other cases (i.e., $t\leq T-2$) can be proven by applying Lemma \ref{lem:Vt} under the condition of the existence of constants $\widehat{C}_{T-1} \geq 1$, $\widehat{L}_{T-1} > 0$ such that
    for every $s_{T-1} \in S$ and $\mu_{T-1}, \tilde \mu_{T-1} \in \mathcal{P}(S)$, it holds
    \begin{align*}
        \begin{aligned}
            \lvert \widehat {V}_{T-1}(s_{T-1},\mu_{T-1}) \rvert&\leq \widehat C_{T-1},\\
            \left\lvert \widehat {V}_{T-1}(s_{T-1},\mu_{T-1}) - \widehat{V}_{T-1}(s_{T-1}, \tilde{\mu}_{T-1}) \right\rvert &\leq \widehat L_{T-1}d_{W_1}(\mu_{T-1},\tilde{\mu}_{T-1}).
        \end{aligned}
    \end{align*}

    By the existence of $\widehat {p}_{T-1}$ and $\widehat \pi_{T-1}$ and the estimates given in \eqref{eq:est_J}, we can use the same arguments presented for the proof of Lemma~\ref{lem:Vt}\;(iv) to obtain those constants satisfying~the above estimates.
\end{proof}

\begin{proof}[Proof of Proposition \ref{pro:dpp}]
    By the existence of $\widehat{p}_{0:T}$ and $\widehat{\pi}_{0:T}$ given in Lemma \ref{lem:dpp_berge}, it is straightforward to prove the part (i). Indeed for every $t=0,\dots,T-1$, we can define sequences of stochastic kernels by for every $(s_t,a_t,\mu_{t})\in S\times A \times {\cal P}(S)$,
    \begin{align*}
        p_t^*(\cdot | s_t,a_t,\mu_t):=\left\{
        \begin{aligned}
            &\widehat {p}_t(\cdot | s_t,a_t,\mu_t,\tilde \mu_{t+1:T})\quad &&\mbox{if}\;\;t\leq T-2; \\
            & \widehat p_{t}(\cdot|s_t,a_t,\mu_{t})\quad &&\mbox{if}\;\;t=T-1,%
        \end{aligned}
        \right.
    \end{align*}
    and for every $s_t\in S$,
    \[
    \pi_t^*(\cdot|s_t):= \widehat{\pi}_t(\cdot|s_t,\tilde \mu_{t:T}).
    \]
    By the optimality of $\widehat{p}_{0:T}$ and $\widehat{\pi}_{0:T}$ (see \eqref{eq:minimizer2}-\eqref{eq:maximizer}), ${p}^*_{0:T}$ and ${\pi}^*_{0:T}$  constructed above satisfy \eqref{eq:minimizer2_given}-\eqref{eq:maximizer_given}.

    \vspace{0.5em}
    \noindent Now let us prove (ii). Let $\overline{\mathbb{P}}:=\tilde \mu_0\otimes  \mathbb{P}_{(\tilde\mu_0,\pi_0^*,p_0)}\otimes \cdots \otimes { \mathbb{P}}_{(\tilde\mu_{T-1},\pi_{T-1}^*,p_{T-1})}\in {\cal Q}(\tilde \mu_{0:T},\pi^*_{0:T})$ and denote by for every $t=1,\dots,T-1$, $\overline{\mathbb{P}}_{0:t} := \tilde \mu_0\otimes  \mathbb{P}_{(\tilde\mu_0,\pi_0^*,p_0)}\otimes \cdots \otimes { \mathbb{P}}_{(\tilde\mu_{t},\pi_{t}^*,p_{t})}$ and $\overline{ \mathbb{P}}_0=\tilde\mu_0$.

    Note that by the definitions of $\widehat {V}_{0:T}$ and $\widehat{J}_{0:T}$ given in \eqref{eq:DPP_maximin2}-\eqref{eq:DPP_min1} and the optimality of $\pi^*_{0:T}$ given in \eqref{eq:maximizer_given},
    \begin{align}\label{eq:cmpris1}
        \begin{aligned}
            &\mathbb{E}^{\overline{ \mathbb{P}}}\Big[r(s_{T-1}, a_{T-1}, s_T, \tilde \mu_{T-1}) \Big]\\
            &\;\;=\mathbb{E}^{\overline{ \mathbb{P}}_{0:T-1}} \left[\int_{S\times A} r(s_{T-1}, a_{T-1}, s_T, \tilde \mu_{T-1}){ \mathbb{P}}_{(\tilde\mu_{T-1},\pi_{T-1}^*,p_{T-1})}(ds_T,da_{T-1}|s_{T-1})  \right]\\
            & \;\; \geq \mathbb{E}^{\overline{ \mathbb{P}}_{0:T-1}}\left[ \int_A \widehat{J}_{T-1}(s_{T-1},a_{T-1},\tilde\mu_{T-1}) \pi^*_{T-1}(da_{T-1}|s_{T-1})\right]
            = \mathbb{E}^{\overline{ \mathbb{P}}}\Big[ \widehat{V}_{T-1}(s_{T-1},\tilde\mu_{T-1}) \Big],
        \end{aligned}
    \end{align}
    and that for every $t\leq T-2$,
    \begin{align}\label{eq:cmpris2}
        \begin{aligned}
            &\mathbb{E}^{\overline{ \mathbb{P}}}\Big[r(s_{t}, a_{t}, s_{t+1}, \tilde \mu_{t})+ \widehat{V}_{t+1}(s_{t+1}, \tilde\mu_{t+1:T})\Big]\\
            &\;\;= \mathbb{E}^{\overline{ \mathbb{P}}_{0:t}}\left[\int_{S\times A} \left(r(s_{t}, a_{t}, s_{t+1}, \tilde \mu_{t})+ \widehat{V}_{t+1}(s_{t+1}, \tilde\mu_{t+1:T}) \right){ \mathbb{P}}_{(\tilde\mu_{t},\pi_{t}^*,p_{t})}(ds_{t+1},da_{t}|s_{t})  \right] \\
            &\;\;\geq \mathbb{E}^{\overline{ \mathbb{P}}_{0:t}}\left[ \int_A \widehat{J}_{t}(s_{t},a_{t},\tilde\mu_{t:T}) \pi^*_{t}(da_{t}|s_{t})\right] = \mathbb{E}^{\overline{ \mathbb{P}}}\Big[\widehat{V}_{t}(s_{t},\tilde\mu_{t:T})\Big].
        \end{aligned}
    \end{align}

    By \eqref{eq:cmpris1} and \eqref{eq:cmpris2}, we hence have
    \begin{align*}
        \begin{aligned}
            \mathbb{E}^{\overline{ \mathbb{P}}} \left[\sum_{t=0}^{T-1}r(s_t,a_t,s_{t+1},\tilde\mu_t)\right] &= \mathbb{E}^{\overline{ \mathbb{P}}} \left[\sum_{t=0}^{T-2}r(s_t,a_t,s_{t+1},\tilde\mu_t)+r(s_{T-1}, a_{T-1}, s_T, \tilde\mu_{T-1}) \right]\\
            &\geq  \mathbb{E}^{\overline{ \mathbb{P}}} \left[\sum_{t=0}^{T-2}r(s_t,a_t,s_{t+1},\tilde\mu_t)+\widehat{V}_{T-1}(s_{T-1},\tilde\mu_{T-1}) \right]\\
            &\geq \cdots \geq \mathbb{E}^{\overline{ \mathbb{P}}} \Big[ \widehat{V}_0(s_{0},\tilde\mu_{0:T}) \Big]= \widehat{ V}(\tilde\mu_{0:T}).
        \end{aligned}
    \end{align*}
    Since $\overline{ \mathbb{P}}$ is arbitrary in ${\cal Q}(\tilde\mu_{0:T},\pi^*_{0:T})$, we have
    \begin{align*}
        \begin{aligned}
            \inf_{\overline\P\in {\cal Q}(\tilde\mu_{0:T},\pi^*_{0:T})} \mathbb{E}^{\overline \P} \left[\sum_{t=0}^{T-1}r(s_t,a_t,s_{t+1},\tilde\mu_t)\right] \geq \widehat{V}(\tilde\mu_{0:T})= \mathbb{E}^{\mathbb{P}^*(\tilde\mu_{0:T})} \left[\sum_{t=0}^{T-1}r(s_t,a_t,s_{t+1},\tilde\mu_t)\right],
        \end{aligned}
    \end{align*}
    with $\mathbb{P}^*(\tilde\mu_{0:T}):=\mathbb{P}({ \tilde\mu_{0:T},\pi_{0:T}^*,p_{0:T}^*})=  \tilde\mu_0\otimes \mathbb{P}_{({ \tilde\mu_0,\pi_0^*,p_0^*})}\otimes\cdots \otimes \mathbb{P}_{( \tilde\mu_{T-1},\pi_{T-1}^*,{p}_{T-1}^*)} \in {\cal Q}(\tilde\mu_{0:T},\pi^*_{0:T})$.  Furthermore, since $\pi_{0:T}^* \in \Pi$, we hence have $V( \tilde\mu_{0:T})\geq  \widehat {V}(\tilde\mu_{0:T})$.

    \vspace{0.5em}
    Let $\pi_{0:T}\in \Pi$ and $\underline{\mathbb{P}}:=\tilde \mu_0\otimes  \mathbb{P}_{(\tilde\mu_0,\pi_0,p_0^*)}\otimes \cdots \otimes { \mathbb{P}}_{(\tilde\mu_{T-1},\pi_{T-1},p_{T-1}^*)}\in {\cal Q}(\tilde \mu_{0:T},\pi_{0:T})$ and denote by for every $t=1,\dots,T-1$, $\underline{\mathbb{P}}_{0:t}:=\tilde \mu_0\otimes  \mathbb{P}_{(\tilde\mu_0,\pi_0,p_0^*)}\otimes \cdots \otimes { \mathbb{P}}_{(\tilde\mu_{t},\pi_{t},p_{t}^*)}$ and $\underline{ \mathbb{P}}_0=\tilde\mu_0$.

    From the definitions of $\widehat {V}_{0:T}$ and $\widehat{J}_{0:T}$ given in \eqref{eq:DPP_maximin2}-\eqref{eq:DPP_min1}  and the optimality of $p^*_{0:T}$ given in~\eqref{eq:minimizer2_given} and \eqref{eq:minimizer_given}, it follows that
    \begin{align}\label{eq:cmpris3}
        \begin{aligned}
            &\mathbb{E}^{\underline{ \mathbb{P}}}\Big[r(s_{T-1}, a_{T-1}, s_T, \tilde \mu_{T-1}) \Big]\\
            &\;\;=\mathbb{E}^{\underline{ \mathbb{P}}_{0:T-1}}  \left[\int_{S\times A} r(s_{T-1}, a_{T-1}, s_T, \tilde \mu_{T-1}){ \mathbb{P}}_{(\tilde\mu_{T-1},\pi_{T-1},p^*_{T-1})}(ds_T,da_{T-1}|s_{T-1})  \right]\\
            & \;\; = \mathbb{E}^{\underline{ \mathbb{P}}_{0:T-1}}\left[ \int_A \widehat{J}_{T-1}(s_{T-1},a_{T-1},\tilde \mu_{T-1}) \pi_{T-1}(da_{T-1}|s_{T-1})\right]
            \leq  \mathbb{E}^{\underline{ \mathbb{P}}}\Big[ \widehat{V}_{T-1}(s_{T-1},\tilde \mu_{T-1}) \Big],
        \end{aligned}
    \end{align}
    and that for every $t\leq T-2$,
    \begin{align}\label{eq:cmpris4}
        \begin{aligned}
            &\mathbb{E}^{\underline{ \mathbb{P}}}\Big[r(s_{t}, a_{t}, s_{t+1}, \tilde \mu_{t})+ \widehat{V}_{t+1}(s_{t+1}, \tilde\mu_{t+1:T})\Big]\\
            &\;\;= \mathbb{E}^{\underline{ \mathbb{P}}_{0:t}}\left[\int_{S\times A} \left(r(s_{t}, a_{t}, s_{t+1}, \tilde \mu_{t})+ \widehat{V}_{t+1}(s_{t+1}, \tilde\mu_{t+1:T}) \right){ \mathbb{P}}_{(\tilde\mu_{t},\pi_{t},p^*_{t})}(ds_{t+1},da_{t}|s_{t})  \right] \\
            &\;\;= \mathbb{E}^{\underline{ \mathbb{P}}_{0:t}}\left[ \int_A \widehat{J}_{t}(s_{t},a_{t},\tilde\mu_{t:T}) \pi_{t}(da_{t}|s_{t})\right] \leq  \mathbb{E}^{\underline{ \mathbb{P}}}\Big[\widehat {V}_{t}(s_{t},\tilde \mu_{t:T})\Big].
        \end{aligned}
    \end{align}

    By \eqref{eq:cmpris3} and \eqref{eq:cmpris4}, we hence have
    \begin{align*}
        \begin{aligned}
            \mathbb{E}^{\underline{ \mathbb{P}}} \left[\sum_{t=0}^{T-1}r(s_t,a_t,s_{t+1},\tilde\mu_t)\right] &= \mathbb{E}^{\underline{ \mathbb{P}}} \left[\sum_{t=0}^{T-2}r(s_t,a_t,s_{t+1},\tilde\mu_t)+r(s_{T-1}, a_{T-1}, s_T, \tilde\mu_{T-1}) \right]\\
            &\leq  \mathbb{E}^{\underline{ \mathbb{P}}} \left[\sum_{t=0}^{T-2}r(s_t,a_t,s_{t+1},\tilde\mu_t)+\widehat{V}_{T-1}(s_{T-1},\tilde\mu_{T-1}) \right]\\
            &\leq \cdots \leq  \mathbb{E}^{\underline{ \mathbb{P}}} \Big[ \widehat{V}_0(s_{0},\tilde\mu_{0:T}) \Big]= \widehat{V}(\tilde\mu_{0:T}),
        \end{aligned}
    \end{align*}
    which ensures that
    \begin{align}\label{eq:opw_1}
        \inf_{\mathbb{P}\in{\cal Q}(\tilde\mu_{0:T},\pi_{0:T})}\mathbb{E}^{{\mathbb{P}}} \left[\sum_{t=0}^{T-1}r(s_t,a_t,s_{t+1},\tilde\mu_t)\right]\leq \mathbb{E}^{\underline{ \mathbb{P}}} \left[\sum_{t=0}^{T-1}r(s_t,a_t,s_{t+1},\tilde\mu_t)\right]  \leq  \widehat{V}(\tilde\mu_{0:T}).
    \end{align}
    Since $\pi_{0:T}$ is arbitrary in $\Pi$, we have $V( \tilde\mu_{0:T})\leq  \widehat{V}(\tilde\mu_{0:T})$.

    \vspace{0.5em}
    It remains to show the equality of $V(\Tilde{\mu}_{0:T})$ to the supremum in \eqref{eq:verification}. Since the last inequality given in \eqref{eq:opw_1} holds for any $\pi_{0:T}\in \Pi$ (with recalling $\underline{\mathbb{P}}=\mathbb{P}(\tilde{\mu}_{0:T},\pi_{0:T},p^*_{0:T})=\tilde \mu_0\otimes  \mathbb{P}_{(\tilde\mu_0,\pi_0,p_0^*)}\otimes \cdots \otimes { \mathbb{P}}_{(\tilde\mu_{T-1},\pi_{T-1},p_{T-1}^*)}\in {\cal Q}(\tilde \mu_{0:T},\pi_{0:T})$), it follows that
    \[
    \sup_{\pi\in \Pi}\mathbb{E}^{\mathbb{P}(\tilde{\mu}_{0:T},\pi_{0:T},p^*_{0:T})} \left[\sum_{t=0}^{T-1}r(s_t,a_t,s_{t+1},\tilde\mu_t)\right]  \leq \mathbb{E}^{\mathbb{P}^*(\tilde \mu_{0:T})} 	\left[\sum_{t=0}^{T-1}r(s_t,a_t,s_{t+1}, \tilde\mu_t)\right]= V(\tilde{\mu}_{0:T})
    \]
    where the last equality follows from above (i.e., $V(\tilde{\mu}_{0:T})= \widehat{V}(\tilde\mu_{0:T})$).

    On the other hand, since $\pi^*_{0:T}\in \Pi$ and $\mathbb{P}^*(\tilde \mu_{0:T})=\mathbb{P}(\tilde{\mu}_{0:T},\pi_{0:T}^*,p^*_{0:T})=\tilde \mu_0\otimes  \mathbb{P}_{(\tilde\mu_0,\pi_0^*,p_0^*)}\otimes \cdots \otimes { \mathbb{P}}_{(\tilde\mu_{T-1},\pi_{T-1}^*,p_{T-1}^*)}\in {\cal Q}(\tilde \mu_{0:T},\pi_{0:T}^*)$, the above inequality establishes equality. This completes the~proof.
\end{proof}

\subsection{Proof of Lemma \ref{lem:dpp_berge_nash} and Proposition \ref{pro:dpp_Nash}}\label{sec:proof:dpp:MFE}

\begin{proof}[Proof of Lemma \ref{lem:dpp_berge_nash}]
    Fix $\overline\pi^N_{0:T} \in \Pi^{N}$ and $i \in \{1, \dots, N\}$, let $t \leq T-1$ and set
    \[
    {\cal S}:=\left\{(\overline s^N_{t}, \overline a^N_{t}, \overline p^{N}_t) \, \Big | \,(\overline s^N_{t}, \overline a^N_{t})\in S^{N} \times A^{N}, \overline{p}^{N}_t \in \mathfrak{P}^{N}_t(\overline s^N_{t}, \overline a^N_{t})\right\}.
    \]
    Define an auxiliary map $F:{\cal S}\ni (\overline s^N_{t}, \overline a^N_{t}, \overline p^{N}_t) \mapsto F(\overline s^N_{t}, \overline a^N_{t}, \overline p^{N}_t)\in\R$ by
    \begin{align*}
        F(\overline s^N_{t}, \overline a^N_{t}, \overline p^N_{t}) := \int_{S^{N}} f(\overline s^N_{t}, \overline a^N_{t}, \overline{s}^{N}_{t+1}) \overline{p}^{N}_{t}(d\overline s_{t+1}^N),
    \end{align*}
    where if $t = T-1$, then we set
    \begin{align*}
        f(\overline s^N_{t}, \overline a^N_{t}, \overline{s}^{N}_{t+1}) := r\big(s_{t}^{i}, a_{t}^{i}, s^{i}_{t+1}, e^{N}(\overline s_t^{N})\big),
    \end{align*}
    whereas if $t\leq T-2$, then we set
    \begin{align*}
        f(\overline s^N_{t}, \overline a^N_{t}, \overline{s}^{N}_{t+1}) := r\big(s_{t}^{i}, a_{t}^{i}, s^{i}_{t+1}, e^{N}(\overline s_t^{N})\big) + \int_{A^N}\widehat{J}_{t+1,i}^{N} (\overline s^N_{t+1},\overline a^N_{t+1};\overline\pi^N_{0:T})\overline \pi_{t+1}^N(d\overline a_{t+1}^N|\overline s_{t+1}^N).
    \end{align*}
    Since both $S^{N}$ and $A^{N}$ are finite, $F$ is continuous in $(\overline{s}^{N}_{t},\overline{a}^{N}_{t})$. Again, by the finiteness of $S^{N}$ and $A^{N}$, we get that $f$ is continuous. Hence, $F$ is continuous in $\overline p^{N}_{t}$. From here, we can follow the same ideas as presented in the proofs of Lemma~\ref{lem:dpp_berge} and Lemma~\ref{lem:Vt} to obtain a minimizer
    \begin{equation*}
        \widehat{p}_{(t,i,\overline \pi^N_{0:T})}: S^N \times A^N \ni (\overline s_t^N, \overline a_t^N)\mapsto \widehat{p}_{(t,i,\overline \pi^N_{0:T})}(\cdot|\overline s_t^N,\overline a_t^N) \in \mathfrak{P}_t^N (\overline s_t^N,\overline a_t^N).
    \end{equation*}
    By Definition~\ref{dfn:corr_Nplayer}~(iii), for each $(\overline{s}_{t}^{N}, \overline{a}_{t}^{N}) \in S^{N} \times A^{N}$, there exists a kernel $p_{t}^{(\overline{s}_{t}^{N}, \overline{a}_{t}^{N})} \in \mathcal{K}_{t}$ such that
    \begin{equation*}
        \widehat{p}_{(t,i,\overline \pi^N_{0:T})}(d\overline{s}_{t+1}^{N} \mid \overline s_t^N,\overline a_t^N) = \prod_{j = 1}^{N} p_{t}^{(\overline{s}_{t}^{N}, \overline{a}_{t}^{N})}\big(d s_{t+1}^{j} \mid s_{t}^{j}, a_{t}^{j}, e^{N}(\overline{s}_{t}^{N})\big).
    \end{equation*}
    Now, for each $j \in \{1, \dots, N\}$, define $\widehat{q}_{(t,i,j,\overline \pi^N_{0:T})}: S^N \times A^N \to \mathcal{P}(S)$ by
    \begin{equation*}
        \widehat{q}_{(t,i,j,\overline \pi^N_{0:T})}(d s_{t+1} \mid \overline s_t^N,\overline a_t^N) := p_{t}^{(\overline{s}_{t}^{N}, \overline{a}_{t}^{N})}\big(d s_{t+1} \mid s_{t}^{j}, a_{t}^{j}, e^{N}(\overline{s}_{t}^{N})\big).
    \end{equation*}
    As $S^{N} \times A^{N}$ is finite, it holds that $\widehat{q}_{(t,i,j,\overline \pi^N_{0:T})}$ is measurable for any $j \in \{1, \dots, N\}$. Moreover, for every $j, k \in \{1, \dots, N\}$ and every $(\overline{s}_{t}^{N}, \overline{a}_{t}^{N}) \in S^{N} \times A^{N}$ such that $(s_{t}^{j}, a_{t}^{j}) = (s_{t}^{k}, a_{t}^{k})$, it holds that
    \begin{align*}
        \widehat{q}_{(t,i,j,\overline \pi^N_{0:T})}(d s_{t+1} \mid \overline s_t^N,\overline a_t^N) &= p_{t}^{(\overline{s}_{t}^{N}, \overline{a}_{t}^{N})}\big(d s_{t+1} \mid s_{t}^{j}, a_{t}^{j}, e^{N}(\overline{s}_{t}^{N})\big) = p_{t}^{(\overline{s}_{t}^{N}, \overline{a}_{t}^{N})}\big(d s_{t+1} \mid s_{t}^{k}, a_{t}^{k}, e^{N}(\overline{s}_{t}^{N})\big) \\
        &= \widehat{q}_{(t,i,k,\overline \pi^N_{0:T})}(d s_{t+1} \mid \overline s_t^N,\overline a_t^N). \qedhere
    \end{align*}
\end{proof}

\begin{proof}[Proof of Proposition \ref{pro:dpp_Nash}]
    We can use the same approach as presented in the proof of Proposition~\ref{pro:dpp}~(ii) to show that
    \begin{align*}
        \begin{aligned}
            J^{N}_i(\mu^o,\overline \pi_{0:T}^{N})&= \inf_{\overline{\mathbb{P}}^N \in \mathcal{Q}^{N}(\mu^o,\overline \pi_{0:T}^{N})} \mathbb{E}^{\overline{\mathbb{P}}^N}\left[\sum_{t = 0}^{T-1} r\big(s_{t}^i, a_{t}^i, s_{t+1}^i, e^{N}(\overline{s}_{t}^N)\big)\right]\\
            &=\mathbb{E}^{\overline {\mathbb{P}}_i^N ({\mu^o,\overline \pi^N_{0:T}},\widehat{p}_{(0:T,i,\overline \pi^N_{0:T})} )}\left[\sum_{t = 0}^{T-1} r\big(s^{i}_{t}, a^{i}_{t}, s^{i}_{t+1}, e^{N}(\overline s_t^N)\big)\right],
        \end{aligned}
    \end{align*}
    where the second equality follows by definition of $\widehat{J}_{0,i}^{N}$.
\end{proof}

\section{Proofs of Results in Section \ref{sec:MF_main_thm}}\label{sec:proof:MF_main_thm}
\subsection{Preliminary lemmas}
Let us provide some simple observations that play an instrumental role in the proof of Proposition \ref{pro:FixPntEq} and Theorem \ref{thm:MFE}.

Let us begin with a measurable extension of mappings into stochastic kernels defined on probability spaces.
\begin{lem} \label{lem:ExtKer}
    Suppose that Assumption~\ref{as:msr} is satisfied. Let $t \in \{0, \dots, T-1\}$ and $\tilde \mu_{t:T} \in (\mathcal{P}(S))^{T-t}$. Furthermore, let $p_{t} \colon S \times A\ni (s_t,a_t) \mapsto  p_t(\cdot|s_t,a_t)\in \mathcal{P}(S)$ be a mapping.
    Then there exists a Borel-measurable mapping (i.e., stochastic kernel) $\overline {p}_{t} : S \times A \times (\mathcal{P}(S))^{T-t}\ni(s_t,a_t,\mu_{t:T}) \mapsto \overline{p}_t(\cdot|s_t,a_t,\mu_{t:T})\in \mathcal{P}(S)$ such that for every $(s_t,a_t) \in S\times A$
    \begin{align*}
        \overline{p}_{t}(\cdot | s_t, a_t, \tilde{\mu}_{t:T}) = p_{t}(\cdot | s_t, a_t).
    \end{align*}
\end{lem}
\begin{proof}%
    By the existence of measurable selectors given in Lemma \ref{lem:dpp_berge}\;(i), we can and do choose a stochastic kernel $p'_{t} \colon S \times A \times (\mathcal{P}(S))^{T-t} \ni (s_t,a_t,\mu_{t:T}) \mapsto {p}'_{t}(\cdot|s_t,a_t,\mu_{t:T})\in  \mathcal{P}(S)$. Then define $\overline {p}_{t}: S \times A \times (\mathcal{P}(S))^{T-t}\ni (s_t,a_t,\mu_{t:T}) \mapsto \overline {p}_{t}(\cdot |s_t,a_t,\mu_{t:T}) \in \mathcal{P}(S)$ by
    \begin{align} \label{eq:ext_mbl}
        \overline{p}_{t}(\cdot | s_t, a_t, {\mu}_{t:T}) = \left\{
        \begin{aligned}
            &p_{t}(\cdot | s_t, a_t)\quad &&\text{if }\;\;{\mu}_{t:T} = \tilde{\mu}_{t:T}, \\
            &p'_{t}(\cdot | s_t, a_t, {\mu}_{t:T})\quad &&\text{else}.
        \end{aligned}
        \right.
    \end{align}

    It is sufficient to show that $\overline{p}_{t}$ is Borel-measurable. To that end, recall that ${\cal B}_{{\cal P}(S)}$ and ${\cal B}_{S \times A \times (\mathcal{P}(S))^{T-t}}$ denote the Borel $\sigma$-field of ${\cal P}(S)$ and $S \times A \times (\mathcal{P}(S))^{T-t}$, respectively.

    Let $E\in{\cal B}_{{\cal P}(S)}$. Then since
    \begin{align*}
        \overline{p}_{t}^{-1}(E) &= \left\{(s_t, a_t,  {\mu}_{t:T}) \in S \times A \times (\mathcal{P}(S))^{T-t}~\Big|~\overline{p}_{t}(\cdot | s_t, a_t,  {\mu}_{t:T}) \in E\right\} \\
        &= \left\{(s_t, a_t,  {\mu}_{t:T}) \in S \times A \times  \{\tilde{\mu}_{t:T}\} ~\Big|~\overline{p}_{t}(\cdot | s_t, a_t, {\mu}_{t:T}) \in E\right\} \\
        &\quad\cup \left\{(s_t, a_t, \mu_{t:T}) \in S \times A \times \big((\mathcal{P}(S))^{T-t} \setminus \{\tilde{\mu}_{t:T}\}\big)~\Big|~\overline{p}_{t}(\cdot | s_t, a_t, {\mu}_{t:T}) \in E\right\} =: E_1\cup E_2,
    \end{align*}
    we will show that $E_1$, $E_2\in{\cal B}_{S \times A \times (\mathcal{P}(S))^{T-t}}$.

    Note that by \eqref{eq:ext_mbl},
    \begin{align*}
        \begin{aligned}
            E_1&= \left\{(s_t, a_t) \in S \times A ~\Big|~{p}_{t}(\cdot | s_t, a_t) \in E\right\}\times  \{\tilde{\mu}_{t:T}\},  \\
            E_2&= \left\{(s_t, a_t, \mu_{t:T}) \in S \times A \times (\mathcal{P}(S))^{T-t} ~\Big|~{p}'_{t}(\cdot | s_t, a_t, {\mu}_{t:T}) \in E\right\} \\
            &\quad \setminus \left(\left\{(s_t, a_t) \in S \times A ~\Big|~ {p}'_{t}(\cdot | s_t, a_t,\tilde{\mu}_{t:T}) \in E\right\}\times  \{\tilde{\mu}_{t:T}\}\right)=: E_{2,1}\setminus E_{2,2}.
        \end{aligned}
    \end{align*}
    Since $S$ and $A$ are finite (see Assumption \ref{as:msr}\;(i)), $p_t$ is Borel-measurable. Hence this implies that $E_1 \in {\cal B}_{S \times A \times (\mathcal{P}(S))^{T-t}}$. For the same reason, it follows that $E_{2,2}\in{\cal B}_{S \times A \times (\mathcal{P}(S))^{T-t}}$. Furthermore, since  $p'_{t}$ is Borel-measurable, $E_{2,1}\in {\cal B}_{S \times A \times (\mathcal{P}(S))^{T-t}}$.
\end{proof}

The following two lemmas link the correspondences ${\cal C}$,\;${\cal B}$ (given in Definition \ref{dfn:FixPntEq2}\;(i)) into the dynamic programming results given in Lemma \ref{lem:dpp_berge} (and Proposition \ref{pro:dpp}).
\begin{lem}\label{lem:CnuOpt}
    Suppose that Assumption \ref{as:msr} is satisfied. Let $\tilde{\nu}_{0:T}\in {\cal C}({\nu}_{0:T})$ and denote by $ p^{\tilde{\nu}}_{0:T-1}$ the corresponding kernels enabling $\tilde{\nu}_{0:T}\in {\cal C}(\nu_{0:T})$ (see Definition \ref{dfn:FixPntEq2}\;(i)). For every $t=0,\dots,T-1$, define $\overline{p}_t^{\tilde{\nu}}:S\times A \times {\cal P}(S)\ni (s_t,a_t,\mu_t)\mapsto \overline{p}_t^{\tilde{\nu}}(\cdot|s_t,a_t,\mu_t)\in {\cal P}(S)$~by
    \begin{align*}
        \overline{p}_t^{\tilde{\nu}}(\cdot|s_t,a_t,\mu_t):= \left\{
        \begin{aligned}
            &{p}_t^{\tilde{\nu}}(\cdot|s_t,a_t,\mu_t,\nu_{t+1:T,S})\quad  &&\mbox{if}\;\;t\leq T-2; \\
            &\widehat{p}_t(\cdot|s_t,a_t,\mu_t) \quad &&\mbox{if}\;\;t=T-1,
        \end{aligned}
        \right.
    \end{align*}
    where $\widehat{p}_t$ is the stochastic kernel given in Lemma \ref{lem:dpp_berge}\;(i). Then for every $(s_t,a_t)\in S\times A$, $\overline{p}_t^{\tilde{\nu}}(\cdot |s_t,a_t,\nu_{t,S})$ is optimal for $\widehat{J}_t(s_t,a_t,\nu_{t:T,S})$, i.e., if $t=T-1$,
    \begin{align*}
            \int_{S}  r(s_{T-1}, a_{T-1}, s_T, \nu_{T-1,S})\overline{p}_{T-1}^{\tilde{\nu}}(ds_{T}|s_{T-1},a_{T-1},\nu_{T-1,S})=\widehat{J}_{T-1}(s_{T-1},a_{T-1},\nu_{T-1,S}),
    \end{align*}
    whereas if $t\leq T-2$,
    \begin{align*}\int_{S} \left(r(s_t, a_t, s_{t+1}, \nu_{t,S}) + \widehat {V}_{t+1}(s_{t+1},\nu_{t+1:T,S})\right) \overline p_t^{\tilde{\nu}} (ds_{t+1} |s_t, a_t,  \nu_{t,S})=\widehat {J}_t(s_t,a_t,\nu_{t:T,S}).
    \end{align*}
\end{lem}
\begin{proof}
    It is straightforward to show the case where $t=T-1$ by the optimality of $\widehat{p}_{T-1}(={p}_{T-1}^{\tilde{\nu}})$ presented in Lemma \ref{lem:dpp_berge}\;(i). For the case where $t\leq T-2$, since for every $(s_t,a_t)\in S\times A$
    \[
    \overline{p}_t^{\tilde{\nu}}(\cdot|s_t,a_t,\nu_{t,S})=p_t^{\tilde{\nu}}(\cdot|s_t,a_t,\nu_{t:T,S})\in \widehat{\mathfrak{P}}_t(s_t,a_t,\nu_{t:T,S})
    \]
    (see Definition \ref{dfn:FixPntEq2}\;(i) and Remark \ref{rem:CalPTil}), $\overline{p}_t^{\tilde{\nu}}(\cdot|s_t,a_t,\nu_{t,S})$ is optimal for $\widehat{J}_t(s_t,a_t,\nu_{t:T,S})$. This completes the proof.
\end{proof}

\begin{lem}\label{lem:BnuOpt}
    Suppose that Assumption \ref{as:msr} is satisfied. Let $\nu_{0:T}$,~$\tilde{\nu}_{0:T} \in \Xi$ and denote by $\pi^{\tilde{\nu}}_{0:T}$ the disintegrating kernels of $\tilde{\nu}_{0:T}$ (see Definition \ref{dfn:FixPntEq}). Furthermore, denote for every $t=0,\dots,T-1$ by $\tilde w_{t}(s)$ the weight of the measure $\tilde \nu_{t, S}(\cdot)$ at each point $s \in S$ (i.e., $\sum_{s\in S}\tilde w_{t}(s)=1$ with $\tilde w_{t}(s)\geq 0$ for $s\in S$). Then the following hold:
    \begin{itemize}
        \item [(i)] $\tilde{\nu}_{0:T}\in {\cal B}(\nu_{0:T})$ (see Definition \ref{dfn:FixPntEq2}\;(i)) if and only if for every $t=0,\dots,T-1$ and $s_t\in S$ such that $\tilde w_{t}(s_t)>0$,
        $ \pi_t^{\tilde \nu}(\cdot|s_t)\in {\cal P}(A)$ is optimal for $\widehat{V}_{t}(s_t,{\nu}_{t:T,S})$ (see \eqref{eq:DPP_maximin2}).
        \item [(ii)] Let $\tilde{\nu}_{0:T}\in {\cal B}(\nu_{0:T})$. For every $t=0,\dots,T-1$, define $\overline \pi^{\tilde \nu}_t:S\ni s_t \mapsto \overline \pi^{\tilde \nu}_t(\cdot|s_t)\in{\cal P}(A)$ by
        \begin{align}\label{eq:policy_const}
            \overline \pi^{\tilde \nu}_t(\cdot|s_t):=\left\{
            \begin{aligned}
                & \pi_t^{\tilde{\nu}}(\cdot|s_t)\quad &&\mbox{if}\;\;\tilde w_{t}(s_t)>0; \\
                &\widehat \pi_t(\cdot|s_t,\nu_{t:T,S}) &&\mbox{else},
            \end{aligned}
            \right.
        \end{align}
        where $\widehat \pi_t$ is the measurable selector given  in Lemma \ref{lem:dpp_berge}. Then it holds
        \begin{align}\label{eq:disinteg_reinteg}
            \tilde{\nu}_t(ds_t,da_t)= \overline \pi^{\tilde \nu}_t(da_t|s_t) \tilde \nu_{t,S}(ds_t).
        \end{align}
        Furthermore, $\overline \pi_t^{\tilde \nu}(\cdot|s_t)$ is optimal for $\widehat{V}_{t}(s_t,{\nu}_{t:T,S})$ for every $s_t\in S$.
    \end{itemize}

\end{lem}
\begin{proof}%
    We start by proving the statement (i). Suppose $\tilde{\nu}_{0:T}\in {\cal B}(\nu_{0:T})$. Fix any $t=0,\dots,T-1$. Then since $\tilde{\nu}_{t}(D_{t}(\nu_{t:T})) = 1$,
    \begin{align*}
        1 &= \int_{S}\int_A  \mathbf{1}_{\{(s_t,a_t)\in D_{t}(\nu_{t:T})\}} \pi^{\tilde{\nu}}_{t}(da_t| s_t) \tilde \nu_{t, S}(ds_t) \\
        &= \sum_{s_t\in S}\tilde w_{t}(s_t) \pi^{\tilde{\nu}}_{t}\big(\{a_t\in A|(s_t,a_t)\in D_{t}(\nu_{t:T})\} \big\vert s_t\big).
    \end{align*}
    This implies that for every $s_t\in S$ such that $\tilde w_{t}(s_t)>0$, $\pi^{\tilde{\nu}}_{t}\big(\{a_t\in A|(s_t,a_t)\in D_{t}(\nu_{t:T})\} \big\vert s_t\big)=1$.

    We hence have that for every $s_t \in S$ such that $\tilde w_{t}(s_t)>0$, it holds
    \begin{align}\label{eq:equal_distinteg}
        \begin{aligned}
            \int_{A} \widehat{J}_{t}(s_t, a_t, \nu_{t:T, S}) \pi^{\tilde{\nu}}_{t}(da_t| s_t)&=\int_{A} \widehat{J}_{t}(s_t, a_t, \nu_{t:T, S}) {\bf 1}_{\{a_t\in A|(s_t,a_t)\in D_{t}(\nu_{t:T})\}}\pi^{\tilde{\nu}}_{t}(da_t| s_t) \\
            &= \int_{A} \max_{a_t' \in A} \widehat{J}_{t}(s_t, a_t', \nu_{t:T, S}){\bf 1}_{\{a_t\in A|(s_t,a_t)\in D_{t}(\nu_{t:T})\}}\pi^{\tilde{\nu}}_{t}(da_t|s_t) \\
            &= \max_{a_t' \in A} \widehat{J}_{t}(s_t, a_t', \nu_{t:T, S}).
        \end{aligned}
    \end{align}
    Furthermore, since
    \[
    \max_{a_t' \in A}\widehat{J}_{t}(s_t, a_t', \nu_{t:T, S})\geq\sup_{\pi \in \mathcal{P}(A)} \int_{A} \widehat{J}_{t}(s_t, a_t, \nu_{t:T, S}) \pi(da_{t}) = \widehat{V}_{t}(s_t, \nu_{t:T, S}),
    \]
    it follows from $\pi^{\tilde{\nu}}_{t}(da_t| s_t) \in {\cal P}(A)$ and \eqref{eq:equal_distinteg} that $\pi_t^{\tilde{\nu}}(\cdot|s_t)$ is optimal for $\widehat{V}_t(s_t,\nu_{t:T,S})$.

    Now suppose that for every $t=0,\dots,T-1$ and $s_t\in S$ such that $\tilde w_{t}(s_t)>0$, $\pi_t^{\tilde{\nu}}(\cdot|s_t)$ is optimal for $\widehat{V}_t(s_t,\nu_{t:T,S})$.  Assume that there exists some $t\leq T-1$ such that $\tilde{\nu}_{t}(D_{t}(\nu_{t:T})) < 1$.

    Set $S' := \big\{s_t \in S \big| \pi^{\tilde{\nu}}_{t}(\{a_t\in A|(s_t,a_t)\in D_{t}(\nu_{t:T})\}| s_t) < 1\;\mbox{and}\;\tilde w_{t}(s_t)>0\big\}$, which is non-empty (due to $\tilde{\nu}_{t}(D_{t}(\nu_{t:T})) < 1$). Let $s_t \in S'$. Since $A$ is finite, there exists $a_{t}^{*} \in A$ such that
    \begin{equation*}
        \widehat{J}_t(s_t,a_t^*,\nu_{t:T,S}) = \max_{a_t' \in A} \widehat{J}_t(s_t,a_t',\nu_{t:T,S}).
    \end{equation*}
    Define $\pi_t' := \delta_{a_t^*} \in \mathcal{P}(A)$. Since $s_t \in S'$, we have
    \begin{equation*}
        \pi_t^{\widetilde{\nu}}\left(\left\{a_t \in A \,\middle|\,\widehat{J}_t(s_t,a_t,\nu_{t:T,S}) < \max_{a_t' \in A} \widehat{J}_t(s_t,a_t',\nu_{t:T,S})\right\}\,\middle|\, s_t\right) > 0
    \end{equation*}
    and thus,
    \begin{align*}
        \int_{A} \widehat{J}_{t}(s_t, a_t, \nu_{t:T, S}) \pi^{\tilde{\nu}}_{t}(da_t|s_t) &< \int_{A} \max_{a_t' \in A} \widehat{J}_{t}(s_t, a_t', \nu_{t:T,S}) \pi^{\tilde{\nu}}_{t}(da_t|s_t) = \max_{a_t' \in A}\widehat{J}_{t}(s_t, a_t', \nu_{t:T, S}) \\
        &= \widehat{J}_{t}(s_t, a_t^*, \nu_{t:T, S}) = \int_{A} \widehat{J}_{t}(s_t, a_t, \nu_{t:T, S}) \pi_{t}'(da_t) \leq \widehat{V}_{t}(s_t, \nu_{t:T, S}).
    \end{align*}
    This contradicts the optimality of $\pi_t^{\tilde \nu}(\cdot|s_t)$ for $\widehat{V}_t(s_t,\nu_{t:T,S})$.

    Thus, $\tilde{\nu}_{t}(D_{t}(\nu_{t:T})) = 1$ for every $t=0,\dots, T-1$, i.e., $\tilde{\nu}_{0:T}\in {\cal B}(\nu_{0:T})$.

    \vspace{0.5em}
    \noindent Now let us prove (ii). By the construction given in \eqref{eq:policy_const}, it is straightforward to see that \eqref{eq:disinteg_reinteg} holds. Hence it remains to show the optimality of $\overline \pi^{\tilde \nu}_t(\cdot|s_t)$ for $\widehat{V}_t(s_t,\nu_{t:T,S})$ for every $s_t\in S$.

    Let $s_t\in S$ be such that $\tilde w_{t}(s_t)>0$. Then Lemma \ref{lem:BnuOpt}\;(i) ensures that $\overline \pi^{\tilde \nu}_t(\cdot|s_t)=\pi_t^{\tilde \nu}(\cdot|s_t)$ is optimal for $\widehat{V}_t(s_t,\nu_{t:T,S})$. For the other case where $s_t\in S$ with $\tilde w_{t}(s_t)=0$, since $\overline \pi^{\tilde \nu}_t(\cdot|s_t)=\widehat \pi_t(\cdot|s_t,\nu_{t:T,S})$, the optimality of $\widehat \pi_t$ given in Lemma \ref{lem:dpp_berge}\;(ii) ensures that $\overline \pi^{\tilde \nu}_t(\cdot|s_t)$ is optimal for $\widehat{V}_t(s_t,\nu_{t:T,S})$. This completes the proof.
\end{proof}

\subsection{Proof of Proposition \ref{pro:FixPntEq}}
\begin{proof}[Proof of Proposition \ref{pro:FixPntEq}\;(i)]
    We first note that by the existence of $\widehat p_{0:T}$ given in Lemma~\ref{lem:dpp_berge}\;(i), ${\cal C}$ (given in Definition \ref{dfn:FixPntEq2}\;(i)) is non-empty.

    We claim that $\Gamma$ is non-empty. To that end, let $\nu_{0:T} \in \Xi$ and choose an arbitrary $\tilde \nu_{0:T} \in \mathcal{C}(\nu_{0:T})$.
    Now for every $t=0,\dots,T-1$, set
    \[
    \tilde \nu'_{t}(ds_t, da_t) := \widehat \pi_{t}(da_t |s_t, \nu_{t:T,S}) \tilde \nu_{t, S}(ds_t),
    \]
    where $\widehat \pi_t$ is the measurable selector given in Lemma~\ref{lem:dpp_berge}\;(ii).

    Then since $\tilde \nu'_{t,S}(\cdot)=\tilde \nu_{t, S}(\cdot)$ and $\tilde \nu_{0:T} \in \mathcal{C}(\nu_{0:T})$, it is clear that $\tilde \nu'_{0:T}\in {\cal C}(\nu_{0:T})$. Hence it remains to show that $\tilde \nu'_{0:T} \in {\cal B}(\nu_{0:T})$. Indeed, since the disintegrating kernel $\pi_t^{\tilde{\nu}'}(\cdot|s_t)$ equals $\widehat \pi_{t}(\cdot | s_t, \nu_{t:T,S})$ for every $s_t\in S$, $\pi_t^{\tilde{\nu}'}(\cdot|s_t)$ is optimal for $\widehat{V}_t(s_t,\nu_{t:T,S})$ for every $s_t\in S$.  From this, Lemma~\ref{lem:BnuOpt}\;(i) ensures the claim to hold.

    \vspace{0.5em}
    Next we claim that $\Gamma$ is convex-valued. Let $\nu_{0:T}\in \Xi$, $\nu'_{0:T}, \nu''_{0:T} \in \Gamma(\nu_{0:T})$, and $\lambda \in (0, 1)$. For every $t=0,\dots,T-1$, define $\tilde \nu_t \in {\cal P}(S\times A)$ by
    \[
    \tilde \nu_{t}(ds_t,da_t) := \lambda \nu'_{t}(ds_t,da_t) + (1 - \lambda) \nu''_{t}(ds_t,da_t).
    \]

    We claim that $\tilde \nu_{0:T} \in \Gamma (\nu_{0:T})$. Since it is straightforward to see that $\tilde \nu_{0:T}\in {\cal B}(\nu_{0:T})$, we will show that $\tilde \nu_{0:T} \in {\cal C}(\nu_{0:T})$.

    It is clear that $\tilde \nu_{0,S}= \lambda \nu'_{0,S}+(1-\lambda)\nu''_{0,S}= \mu^o$ (since $\nu'_{0,S}=\nu''_{0,S}=\mu^o$; see Definition~\ref{dfn:FixPntEq2}\;(i)).  Denote by $p^{\nu'}_{0:T-1}$ and $p^{\nu''}_{0:T-1}$ the sequences of kernels enabling $\nu'_{0:T}\in {\cal C}(\nu_{0:T})$ and $\nu''_{0:T}\in {\cal C}(\nu_{0:T})$ respectively.

    Then for every $t=0,\dots,T-2$, we define $p^{\tilde {\nu}}_t : S\times A\times ({\cal P}(S))^{T-t}\rightarrow {\cal P}(S)$ by for every $(s_t,a_t,\mu_{t:T})\in S\times A\times ({\cal P}(S))^{T-t}$,
    \[
    p^{\tilde {\nu}}_t (\cdot|s_t,a_t,\mu_{t:T}):=  \lambda p^{\nu'}_{t}(\cdot | s_t, a_t, \mu_{t:T})+(1-\lambda) p^{\nu''}_{t}(\cdot | s_t, a_t, \mu_{t:T}).
    \]

    Note that for every $t=0,\dots,T-2$, $p^{\nu'}_{t}(\cdot | s_t, a_t, \nu_{t:T,S})$, $p^{\nu''}_{t}(\cdot | s_t, a_t, \nu_{t:T,S}) \in \widehat {\mathfrak{P}}_t(s_t,a_t,\nu_{t:T,S})$ for every $(s_t,a_t)\in S \times A$ and $ \widehat {\mathfrak P}_t$ is convex-valued (see Remark \ref{rem:CalPTil}). Therefore, for every $t=0,\dots,T-2$, it holds that $p^{\tilde {\nu}}_t (\cdot|s_t,a_t,\nu_{t:T,S})\in \widehat{\mathfrak P}_t(s_t,a_t,\nu_{t:T,S})$ for every $(s_t,a_t)\in S\times A$.

    Furthermore, it also holds for every $t=0,\dots,T-2$ that
    \begin{align*}
        \tilde \nu_{t+1,S}(\cdot)=  \lambda \nu'_{t+1,S}(\cdot) + (1 - \lambda) \nu''_{t+1,S}(\cdot)= \int_{S \times A} p^{\tilde {\nu}}_t (\cdot|s_t,a_t,\nu_{t:T,S}) \nu_{t}(ds_t, da_t).
    \end{align*}
    We hence have that  $\tilde \nu_{0:T} \in \mathcal{C}(\nu_{0:T})$. This completes the proof.
\end{proof}

\begin{proof}[Proof of Proposition \ref{pro:FixPntEq}\;(ii)]	Let $({\nu}_{0:T}^{n}, {\xi}_{0:T}^{n})_{n \in \N} \subseteq \Xi \times \Xi$ be a sequence such that for every $n\in \mathbb{N}$, ${\xi}_{0:T}^{n} \in \Gamma({\nu}_{0:T}^{n})$ and that for every $t=0,\dots,T-1$ as $n \to \infty$,
    \begin{align}\label{eq:conv_msr}
        {\nu}_{t}^{n} \rightharpoonup
        {\nu}_{t}^{\star }, \quad {\xi}_{t}^{n} \rightharpoonup {\xi}_{t}^{\star},
    \end{align}
    with some $ ({\nu}_{0:T}^{\star }, {\xi}_{0:T}^{\star}) \in \Xi \times \Xi$.

    To prove $\mathrm{Gr}(\Gamma)$ is closed, it is sufficient to prove that ${\xi}^\star_{0:T} \in \Gamma({\nu}_{0:T}^{\star })$.

    \vspace{0.5em}
    \noindent {\it Step 1.}~We show that ${\xi}^\star_{0:T} \in {\cal C}({\nu}_{0:T}^{\star })$.
    Since $\xi_{0,S}^n=\mu^o$ for every $n\in\mathbb{N}$ (due to ${\xi}_{0:T}^{n} \in {\cal C}({\nu}_{0:T}^{n})$), by \eqref{eq:conv_msr} it holds that $\xi_{0,S}^\star =\mu^o$.

    For every $n\in \mathbb{N}$, let $p_{0:T-1}^{\xi^n}$ be a sequence of kernels enabling ${\xi}_{0:T}^{n} \in {\cal C}({\nu}_{0:T}^{n})$ (see Definition~\ref{dfn:FixPntEq2}\;(i)). For notational simplicity, set $p_{0:T-1}^{n}:=p_{0:T-1}^{\xi^n}$.

    Then for every $n\in \mathbb{N}$ and $t=0,\dots,T-2$, it holds that
    \begin{align} \label{eq:Xi(n)C}
        \xi^{n}_{t+1, S}(\cdot) = \int_{S \times A} p_{t}^{n}(\cdot | s_t, a_t, \nu^{n}_{t:T, S}) \nu^{n}_{t}(ds_t, da_t),
    \end{align}
    (due to ${\xi}_{0:T}^{n} \in {\cal C}({\nu}_{0:T}^{n})$) and that for every $(s_t, a_t) \in S \times A$,
    \begin{align}\label{eq:conv_msr1_1}
        \mathbb{P}_{t,s_t,a_t}^n:=p^{n}_{t}(\cdot | s_t, a_t, \nu^{n}_{t:T, S}) \in \widehat {\mathfrak{P}}_{t}(s_t, a_t, \nu^{n}_{t:T,S}).
    \end{align}

    \vspace{0.5em}
    Fix any $t\in \{0,\dots,T-2\}$. Let $(s_t,a_t)\in S\times A$.
    Since $\mathbb{P}^n_{t,s_t,a_t}\in  \widehat {\mathfrak{P}}_{t}(s_t, a_t, \nu^{n}_{t:T,S})$ for every $n\in \mathbb{N}$ and for every $u=t,\dots,T-1$, $\nu^{n}_{u,S} \rightharpoonup \nu_{u,S}^\star$ as $n\rightarrow\infty$ (see \eqref{eq:conv_msr}), the compact-valueness and upper-hemicontinuity of the correspondence $ \widehat {\mathfrak{P}}_{t}$  (see Remark \ref{rem:CalPTil}) ensure that there exist a subsequence $(\P^{n_{k}}_{t, s_t, a_t})_{k \in \N}$ and some $\P_{t, s_t, a_t} \in \widehat {\mathfrak{P}}_{t}(s_t, a_t, \nu_{t:T,S}^\star)$ such that
    \begin{align}\label{eq:conv_msr2}
        \P^{n_{k}}_{t, s_t, a_t} \rightharpoonup \P_{t, s_t, a_t}\quad \mbox{as $k\rightarrow\infty$}
    \end{align}
    (see \cite[Theorem 17.20]{CharalambosKim2006infinite}). Since both $S$ and $A$ are finite (see Assumption \ref{as:msr}\;(i)), by using the same arguments presented for \eqref{eq:conv_msr2} a finite number of times, we can and do choose a subsequence $(\P^{n_{k}}_{t, s_t, a_t})_{k \in \N}$  of the one in \eqref{eq:conv_msr1_1} and have $(\mathbb{P}_{t,s_t,a_t})_{(s_t,a_t)\in S\times A}$ (for notational simplicity, we do not relabel that sequence) for which \eqref{eq:conv_msr2} holds with $\mathbb{P}_{t,s_{t},a_{t}}\in \widehat {\mathfrak{P}}_{t}(s_t, a_t, \nu_{t:T, S}^\star)$ for every $(s_t, a_t) \in S \times A$.

    From this, we can define a mapping
    \begin{align}\label{eq:conv_msr3}
        p^\star_{t} : S \times A\ni (s_t, a_t)   \mapsto p^\star_{t} (\cdot |s_t,a_t):= \mathbb{P}_{t, s_t, a_t} \in \widehat{\mathfrak{P}}_{t}(s_t, a_t, \nu_{t:T, S}^\star).
    \end{align}
    Lemma~\ref{lem:ExtKer} enables to extend $p^\star_{t}$ as a stochastic kernel $\overline {p}^\star_{t} : S \times A \times (\mathcal{P}(S))^{T-t}\ni(s_t,a_t,\mu_{t:T}) \mapsto \overline{p}^\star_t(\cdot|s_t,a_t,\mu_{t:T})\in \widehat{\mathfrak{P}}_{t}(s_t, a_t, \mu_{t:T})$ such that for every $(s_t,a_t) \in S\times A$, it holds
    \begin{align}\label{eq:conv_msr4}
        \overline{p}_{t}^\star(\cdot | s_t, a_t, \nu_{t:T,S}^\star) = p_{t}^\star(\cdot | s_t, a_t).
    \end{align}

    By the consecutive constructions given in \eqref{eq:conv_msr3} and \eqref{eq:conv_msr4}, the limit \eqref{eq:conv_msr2} together with~\eqref{eq:conv_msr1_1} ensures that for every $(s_t, a_t) \in S \times A$, as $k\rightarrow \infty$,
    \begin{align}\label{eq:conv_msr5}
        p^{n_k}_{t}(\cdot | s_t, a_t, \nu^{n_k}_{t:T, S}) \rightharpoonup \overline{p}_{t}^\star(\cdot | s_t, a_t, \nu_{t:T,S}^\star).
    \end{align}

    Now we claim that as $k\rightarrow \infty$,
    \begin{align}\label{eq:conv_claim}
        \int_{S \times A} p_{t}^{n_k}(\cdot | s_t, a_t, \nu^{n_k}_{t:T, S}) \nu^{n_k}_{t}(ds_t, da_t) \rightharpoonup \int_{S \times A} \overline p_{t}^{\star}(\cdot | s_t, a_t, \nu^{\star}_{t:T, S}) \nu^{\star}_{t}(ds_t, da_t).
    \end{align}

    To that end, for every $k\in \mathbb{N}$ denote by $w_t^{n_k}(s_t, a_t)$ and $w_t^\star(s_t, a_t)$ the weights of $\nu^{n_k}_{t}$ and $\nu_{t}^\star$ at $(s_t, a_t) \in S \times A$ and by $w^{n_k}_{t,s_t,a_t}(s_{t+1})$ and $\overline w^\star_{t,s_t,a_t}(s_{t+1})$ the weights of $p_{t}^{n_k}(ds_{t+1} | s_t, a_t, \nu^{n_k}_{t:T, S})$ and $\overline p^\star_{t}(ds_{t+1} |s_t, a_t, \nu_{t:T, S}^{\star})$ at $s_{t+1} \in S$.  Then by \eqref{eq:conv_msr} and \eqref{eq:conv_msr5} (since $S$ and $A$ are finite; see Assumption \ref{as:msr}\;(i)), it holds that for every $s_t, s_{t+1} \in S$ and $a_t \in A$, as $k\rightarrow \infty$,
    \begin{align}\label{eq:conv_msr6}
        w_t^{n_k}(s_t, a_t)\rightarrow w_t^\star(s_t, a_t), \quad w^{n_k}_{t,s_t,a_t}(s_{t+1}) \rightarrow \overline w^\star_{t,s_t,a_t}(s_{t+1}).
    \end{align}

    Let $g:S\rightarrow \mathbb{R}$ be any mapping (which is obviously in $C_b(S;\mathbb{R})$ as $S$ is finite). Then since for every $k\in \mathbb{N}$
    \begin{align*}
    &\int_{S \times A} \int_{S} g(s_{t+1}) p_{t}^{n_k}(ds_{t+1} | s_t, a_t, \nu^{n_k}_{t:T, S}) \nu^{n_k}_{t}(ds_t, da_t) \\
    &\quad =\sum_{(s_t, a_t) \in S \times A} 	w_t^{n_k}(s_t, a_t) \sum_{s_{t+1} \in S} w^{n_k}_{t,s_t,a_t}(s_{t+1}) g(s_{t+1}),
    \end{align*}
    from \eqref{eq:conv_msr6} (together with the finiteness of $S$ and $A$), it follows that
    \begin{align*}
        &\lim_{k \to \infty} \int_{S \times A} \int_{S} g(s_{t+1}) p_{t}^{n_k}(ds_{t+1} | s_t, a_t, \nu^{n_k}_{t:T, S}) \nu^{n_k}_{t}(ds_t, da_t) \\
        &\quad=  \sum_{(s_t, a_t) \in S \times A}  w_t^\star(s_t, a_t) \sum_{s_{t+1} \in S} \overline w^\star_{t,s_t,a_t}(s_{t+1}) g(s_{t+1}) \\
        &\quad= \int_{S \times A} \int_{S} g(s_{t+1}) \overline p_{t}^{\star}(ds_{t+1} | s_t, a_t, \nu^{\star}_{t:T, S}) \nu^{\star}_{t}(ds_t, da_t),
    \end{align*}
    which ensures the claim given in \eqref{eq:conv_claim} to hold.

    Using \eqref{eq:conv_claim} together with \eqref{eq:Xi(n)C} and \eqref{eq:conv_msr}, we hence have that
    \begin{align*}
        \xi_{t+1, S}^\star(\cdot) = \int_{S \times A} \overline p_{t}^\star(\cdot | s_t, a_t, \nu^\star_{t:T, S}) \nu^{\star}_{t}(ds_t, da_t),
    \end{align*}
    where we recall that $\overline p_{t}^\star$ satisfies \eqref{eq:conv_msr4} for every $(s_t,a_t)\in S\times A$. Since this holds for any $t=0,\dots,T-2$,  we hence have that $\xi_{0:T}^\star \in \mathcal{C}(\nu_{0:T}^\star)$.

    \vspace{0.5em}
    \noindent{\it Step 2.}~It remains to show that $\xi_{0:T}^\star\in \mathcal{B}(\nu_{0:T}^\star)$. We follow the arguments in the proofs of \cite[Proposition 3.9]{saldi2018markov} and \cite[Proposition 3.9]{A2026}.

    Fix $t=0,\dots,T-1$. Let $(s_{t}^n)_{n \in \N} \subseteq S$ and $s_t\in S$ be such that $s_{t}^n \rightarrow s_t$ as $n\rightarrow\infty$. Since $\widehat{J}_{t}(\cdot,\cdot, \nu^{n}_{t:T, S}) \colon S \times A \to \R$ converges continuously\footnote{Suppose $g$ and $(g_{n})_{n \in \N}$ are measurable functions on a metric space $E$. The sequence $(g_{n})_{n \in \N}$ is said to converge to $g$ continuously if $\lim_{n \to \infty} g_{n}(e_{n}) = g(e)$ for any sequence $(e_{n})_{n \in \N}$ with $e_{n} \to e \in E$.} to $\widehat{J}_{t}(\cdot,\cdot, \nu_{t:T,S}^\star ) \colon S \times A \to \R$ (by Lemma~\ref{lem:Vt}\;(i) and \eqref{eq:conv_msr}) and $A$ is finite, it holds that
    \begin{align}\label{eq:conv_conti_max_J}
        \lim_{n \to \infty} \max_{a_t \in A} \widehat{J}_{t}(s_{t}^n, a_t, \nu^{n}_{t:T, S}) = \max_{a_t \in A} \widehat{J}_{t}(s_t, a_t, \nu^\star_{t:T, S}),
    \end{align}
    which implies that $\max_{a_t \in A} \widehat{J}_{t}(\cdot, a_t, \nu^{n}_{t:T, S})$ converges continuously to $\max_{a_t \in A} \widehat{J}_{t}(\cdot, a_t, \nu^\star_{t:T, S})$.

    Set $D_{t}^{n} := D_{t}(\nu_{t:T}^{n})$, $n \in \N$, and $D_{t}^{\star} := D_{t}(\nu_{t:T}^{\star})$, where $D_{t}(\,\cdot\,)$ is defined in Definition~\ref{dfn:FixPntEq2}\;(i). Since $S \times A$ is finite, there exists $\varepsilon > 0$ such that
    \begin{align}\label{eq:e_M_seq}
        E_{t} := \left\{(s_t, a_t)\in S\times A  \Big| \max_{a_t' \in A} \widehat{J}_{t}(s_t, a_t', \nu^\star_{t:T, S}) \geq \widehat{J}_{t}(s_t, a_t, \nu^\star_{t:T, S}) + \varepsilon\right\} = (D_{t}^{\star})^{c}
    \end{align}
    Then, for any $n \in \N$,
    \begin{align*}
        \begin{aligned}
            1 - \xi^{n}_{t}(D_{t}^\star \cap D^{n}_{t}) &= \xi^{n}_{t}(D^{n}_{t}) - \xi^{n}_{t}(D^\star_{t} \cap D^{n}_{t})\\
            & = \xi^{n}_{t}((D^\star_{t})^{c}  \cap D^{n}_{t}) = \xi^{n}_{t}(E_{t} \cap D^{n}_{t}).
        \end{aligned}
    \end{align*}
    This ensures that
    \begin{align}\label{eq:conv_msr_flow}
        \begin{aligned}
            1 &= \limsup_{n \to \infty} \bigg\{\xi^{n}_{t}(D_{t}^\star \cap D^{n}_{t}) + \xi^{n}_{t}(E_{t} \cap D^{n}_{t})\bigg\}.
        \end{aligned}
    \end{align}

    We claim that
    \begin{align}\label{eq:forza_claim}
        \limsup_{n \to \infty} \xi^{n}_{t}(E_{t} \cap D^{n}_{t})= \limsup_{n \to \infty}\int_{S\times A}\mathbf{1}_{\{(s_t,a_t)\in E_{t} \cap D^{n}_{t} \}}  \xi_t^n(ds_t,da_t) = 0.
    \end{align}
    We first show that $\mathbf{1}_{\{(s_t,a_t)\in E_{t} \cap D^{n}_{t} \}}:S\times A \mapsto \mathbb{R}$ converges continuously to $0$ as $n\rightarrow \infty$. Let $(s_{t}^n, a_{t}^n)_{n\in \mathbb{N}}$ be a sequence such that $(s_{t}^n, a_{t}^n) \to (s_t^\star, a_t^\star) \in E_{t}$ as $n\rightarrow \infty$. Then by \eqref{eq:conv_conti_max_J} and \eqref{eq:e_M_seq},
    \begin{align*}
        \begin{aligned}
            \lim_{n \to \infty} \max_{a_t \in A} \widehat{J}_{t}(s_{t}^n, a_t, \nu^{n}_{t:T, S}) &= \max_{a_t \in A} \widehat{J}_{t}(s_t^\star, a_t, \nu^\star_{t:T, S}) \\
            &\geq \widehat{J}_{t}(s_t^\star, a_t^\star, \nu^\star_{t:T, S}) + \varepsilon \\
            &= \lim_{n \to \infty} \widehat{J}_{t}(s_{t}^n, a_{t}^n, \nu^{n}_{t:T,S}) + \varepsilon.
        \end{aligned}
    \end{align*}
    Hence, for sufficiently large $n$, we have $\max_{a_t \in A} \widehat{J}_{t}(s_{t}^n, a_t, \nu^{n}_{t:T, S}) > \widehat{J}_{t}(s_{t}^n, a_t^{n}, \nu^{n}_{t:T,S})$ which implies that $(s_{t}^n, a_{t}^n) \not\in D^{n}_{t}$. Hence we have that $\mathbf{1}_{\{(s_t,a_t)\in E_{t} \cap D^{n}_{t} \}}$ converges continuously to $0$ as $n \to \infty$.

    From this and the limit ${\xi}_{t}^{n} \rightharpoonup {\xi}_{t}^{\star}$ as $n\rightarrow \infty$ (see \eqref{eq:conv_msr}), an application of \cite[Theorem 3.5]{L1981} ensures that the claim given in \eqref{eq:forza_claim} holds. Combining this with \eqref{eq:conv_msr_flow}, we have
    \begin{align*}
        1 = \limsup_{n \to \infty} \xi^{n}_{t}(D_{t}^\star \cap D^{n}_{t}) \leq \limsup_{n \to \infty} \xi^{n}_{t}(D_{t}^\star).
    \end{align*}
    Furthermore, since $D_{t}^\star$ is closed, the Portmanteau theorem (see e.g., \cite[Theorem 2.1]{billingsley2013convergence}, \cite[Theorem 8.2.3]{bogachev2007constructions}) implies that $\limsup_{n \to \infty} \xi^{n}_{t}(D_{t}^\star) \leq \xi_{t}^\star(D_{t}^\star).$ Hence, we have shown that $\xi_{t}^\star(D_{t}^\star) = 1$.

    Since this holds for any $t=0,\dots,T-1$,  we conclude that $\xi_{0:T}^\star \in \mathcal{B}(\nu_{0:T}^\star)$.
\end{proof}

\begin{proof}[Proof of Proposition \ref{pro:FixPntEq}\;(iii)]
    Note that $\Xi$ is a compact convex topological space. Furthermore, $\Gamma$ is non-empty, convex-valued and its graph is closed (see Proposition~\ref{pro:FixPntEq}~(i), (ii)). Therefore, by Kakutani's fixed point theorem (see, e.g., \cite[Corollary 17.55, p.~583]{CharalambosKim2006infinite}), $\Gamma$ has a fixed point $\nu^*_{0:T}$, i.e., $\nu^*_{0:T}\in \Gamma (\nu^*_{0:T})$.
\end{proof}

\subsection{Proof of Theorem~\ref{thm:MFE}}
\begin{proof}[{Proof of Theorem~\ref{thm:MFE}}]
    By Proposition~\ref{pro:FixPntEq}~(iii), $\Gamma$ has a fixed point $\nu^*_{0:T}$, i.e., $\nu^*_{0:T}\in \Gamma (\nu^*_{0:T})$.

    Then, since ${\nu}_{0:T}^* \in {\cal C}(\nu^*_{0:T})$, it holds that $\nu_{0,S}^* = \mu^o$. Furthermore, Lemma \ref{lem:CnuOpt} ensures that for every $t=0,\dots,T-1$, there exists $\overline{p}_t^{{\nu}^*}:S\times A \times {\cal P}(S)\ni (s_t,a_t,\mu_t)\mapsto \overline{p}_t^{{\nu}^*}(\cdot|s_t,a_t,\mu_t)\in {\cal P}(S)$ defined~by
    \begin{align}\label{eq:kernel_nash}
        \overline{p}_t^{{\nu}^*}(\cdot|s_t,a_t,\mu_t):= \left\{
        \begin{aligned}
            &{p}_t^{{\nu}^*}(\cdot|s_t,a_t,\mu_t,\nu^*_{t+1:T, S})\quad  &&\mbox{if}\;\;t\leq T-2; \\
            &\widehat{p}_t(\cdot|s_t,a_t,\mu_t) \quad &&\mbox{if}\;\;t=T-1,
        \end{aligned}
        \right.
    \end{align}
    where $\widehat{p}_{0:T}$ is the sequence of the measurable selectors given in Lemma \ref{lem:dpp_berge}\;(i) and ${p}_{0:T-1}^{{\nu}^*}$ is the sequence of the corresponding kernels enabling ${\nu}_{0:T}^* \in {\cal C}(\nu^*_{0:T})$, i.e., for $t=0,\dots,T-2$,
    \begin{align}\label{eq:kernel_nash_sub}
        \begin{aligned}
            &\text{$ p_{t}^{{\nu}^*}(\cdot | s_t, a_t, {\nu}^*_{t:T,S}) \in \widehat{\mathfrak{P}}_{t}(s_t, a_t, {\nu}^*_{t:T,S})~$ for every $(s_t,a_t)\in S\times A$,}\\
            &\text{and}\;\;{\nu}_{t+1, S}^*(\cdot) = \int_{S \times A}  p_{t}^{\nu^*}(\cdot | s_t, a_t, \nu_{t:T, S}^*) \nu_{t}^*(ds_t, da_t),
        \end{aligned}
    \end{align}
    (see Definition~\ref{dfn:FixPntEq2}), and that for all $t=0,\dots,T-1$,  $\overline{p}_t^{{\nu}^*}(\cdot|s_t,a_t,\nu_{t,S}^*)$ is optimal for $\widehat{J}_t(s_t,a_t,\nu^*_{t:T,S})$ for every $(s_t,a_t)\in S\times A$.

    Furthermore, since ${\nu}_{0:T}^* \in {\cal B}(\nu^*_{0:T})$, Lemma \ref{lem:BnuOpt}\;(ii) ensures that for every $t=0,\dots,T-1$, there exists $\overline \pi^{\nu^*}_t:S\ni s_t \mapsto \overline \pi^{ \nu^*}_t(\cdot|s_t)\in{\cal P}(A)$ defined by
    \begin{align}\label{eq:policy_nash}
        \overline \pi^{\nu^*}_t(\cdot|s_t):=\left\{
        \begin{aligned}
            & \pi_t^{{\nu}^*}(\cdot|s_t)\quad &&\mbox{if}\;\; w^*_{t}(s_t)>0; \\
            &\widehat \pi_t(\cdot|s_t,\nu_{t:T,S}^*) &&\mbox{else},
        \end{aligned}
        \right.
    \end{align}
    where $w^*_{t}(s_t)$ is the weight of $\nu^*_{t,S}$ at $s_t\in S$ and $\widehat \pi_{0:T}$ is the sequence of measurable selectors given in Lemma \ref{lem:dpp_berge}\;(ii), and that for every $t=0,\dots,T-1$,  it holds
    \begin{align}\label{eq:disinteg_reinteg_nash}
        {\nu}^*_t(ds_t,da_t)= \overline \pi^{ \nu^*}_t(da_t|s_t)  \nu^*_{t,S}(ds_t),
    \end{align}
    and that $\overline \pi_t^{\nu^*}(\cdot|s_t)$ is optimal for $\widehat{V}_{t}(s_t,{\nu}^*_{t:T,S})$ for every $s_t\in S$.

    The optimality of $\overline{p}_{0:T}^{{\nu}^*}$ and $\overline{\pi}^{\nu^*}_{0:T}$ ensures that $(\overline{\pi}^{\nu^*}_{0:T},\overline{p}_{0:T}^{{\nu}^*})$ is optimal for $V(\nu_{0:T,S}^*)$, i.e., the condition (i) given in Definition \ref{dfn:robust_mfe} holds. Furthermore, combining \eqref{eq:disinteg_reinteg_nash} with \eqref{eq:kernel_nash} and \eqref{eq:kernel_nash_sub} ensures that $(\nu^*_{0:T,S},\overline{\pi}^{\nu^*}_{0:T},\overline{p}_{0:T}^{{\nu}^*})$ satisfies condition (ii) given in Definition \ref{dfn:robust_mfe}. Hence, $(\nu^*_{0:T,S},\overline{\pi}^{\nu^*}_{0:T},\overline{p}_{0:T}^{{\nu}^*})$ is a mean-field equilibrium of $(S,A,\mu^o,\mathfrak{P}_{0:T},r)$.
\end{proof}

\section{Proofs of Results in Section \ref{sec:main_MNE}}\label{sec:proof:main_MNE}
\subsection{Preliminary lemmas}
The proof of Proposition~\ref{pro:ConGNtGam_ptb} relies on the coupling constructed in Lemma~\ref{lem:P:cpl} to show that the Wasserstein distance between the finite-agent and mean-field one-step laws tends to zero as $N\to\infty$. To prove this coupling result, we introduce two preliminary lemmas.

The first provides a coupling under which the two next states differ with probability equal to the total variation distance between their laws.
\begin{lem} \label{lem:max:cpl}
    Let $\mu,\nu \in \mathcal{P}(S)$. Then there exists
    a coupling measure $\gamma \in \mathrm{Cpl}(\mu, \nu)$ such that
    \begin{equation*}
        \gamma(\{(x, y) \in S \times S \mid x \neq y\}) = \lVert \mu - \nu \rVert_{\mathrm{TV}},
    \end{equation*}
    where $\lVert \,\cdot\, \rVert_{\mathrm{TV}}$ is the total variation norm. In this case, $\gamma$ is called a maximal coupling.
\end{lem}
\begin{proof}%
    Fix $\mu,\nu\in{\cal P}(S)$. For each $s\in S$, set
    $r(s):=\min\{\mu(s),\nu(s)\}$ and let
    $\alpha:=\sum_{s\in S}r(s)$. Then, since
    \(
    r(s)=\frac12\big(\mu(s)+\nu(s)-|\mu(s)-\nu(s)|\big)
    \)
    for any $s\in S$, we have
    \begin{align*}
        \alpha
        &=\sum_{s\in S}\frac12
        \big(\mu(s)+\nu(s)-|\mu(s)-\nu(s)|\big)
        =1-\|\mu-\nu\|_{\mathrm{TV}}.
    \end{align*}
    If $\alpha=1$, take
    $\gamma(x,y)=\mathbf{1}_{\{x=y\}}\mu(x)$.
    Otherwise, define $\gamma$ by
    \begin{equation*}
        \gamma(x,y)
        =\mathbf{1}_{\{x=y\}}r(x)
        +\frac{\big(\mu(x)-r(x)\big)\big(\nu(y)-r(y)\big)}{1-\alpha},
        \qquad x,y\in S.
    \end{equation*}
    Then
    \begin{align*}
        \gamma(S\times S)
        &=\sum_{x\in S}\sum_{y\in S}
        \bigg(
        \mathbf{1}_{\{x=y\}}r(x)
        +\frac{\big(\mu(x)-r(x)\big)\big(\nu(y)-r(y)\big)}{1-\alpha}
        \bigg)\\
        &=\alpha+\frac{1}{1-\alpha}
        \sum_{x\in S}\big(\mu(x)-r(x)\big)
        \sum_{y\in S}\big(\nu(y)-r(y)\big)
        =\alpha+\frac{(1-\alpha)^2}{1-\alpha}=1.
    \end{align*}
    We claim that $\gamma\in\mathrm{Cpl}(\mu,\nu)$.
    Indeed, for any $s\in S$,
    \begin{align*}
        \gamma(\{s\}\times S)
        &=\sum_{y\in S}
        \bigg(
        \mathbf{1}_{\{y=s\}}r(s)
        +\frac{\big(\mu(s)-r(s)\big)\big(\nu(y)-r(y)\big)}{1-\alpha}
        \bigg)\\
        &=r(s)+\frac{\mu(s)-r(s)}{1-\alpha}
        \sum_{y\in S}\big(\nu(y)-r(y)\big)
        =r(s)+\frac{\mu(s)-r(s)}{1-\alpha}(1-\alpha)
        =\mu(s).
    \end{align*}
    Analogously, $\gamma(S\times\{s\})=\nu(s)$ for any $s\in S$.

    Finally, since
    $(\mu(s)-r(s))(\nu(s)-r(s))=0$ for any $s\in S$,
    we obtain that
    \begin{align*}
        \gamma(\{(x,y)\in S\times S\mid x\neq y\})
        &=\sum_{x\neq y}
        \bigg(
        \mathbf{1}_{\{x=y\}}r(x)
        +\frac{\big(\mu(x)-r(x)\big)\big(\nu(y)-r(y)\big)}{1-\alpha}
        \bigg)\\
        &=\sum_{x\neq y}
        \frac{\big(\mu(x)-r(x)\big)\big(\nu(y)-r(y)\big)}{1-\alpha}\\
        &=\frac{1}{1-\alpha}
        \sum_{x\in S}\big(\mu(x)-r(x)\big)
        \sum_{y\in S}\big(\nu(y)-r(y)\big)
        =\|\mu-\nu\|_{\mathrm{TV}}. \qedhere
    \end{align*}
\end{proof}

Equip $S$ and $A$ with its Euclidean metric. Define
\begin{equation}\label{eq:diam:S:A}
    D_{S} := \max_{s, s' \in S} \lvert s - s' \rvert, \quad d_{S} := \min_{s, s' \in S, s \neq s'} \lvert s - s' \rvert > 0, \quad \text{and} \quad D_{A} := \max_{a, a' \in A} \lvert a - a' \rvert.
\end{equation}
For each $t = 0, \dots, T-1$, define the uniform continuity modulus $\omega_{t} \colon [0, \infty) \to \R$ by
\begin{equation} \label{eq:uni:cont:mod}
    \omega_{t}(r) := \max_{\substack{s \in S, a \in A, \\ \mu, \nu \in \mathcal{P}(S), d_{W_{1}}(\mu, \nu) \leq r}} d_{W_{1}}\big(p_{t}^{*}(\cdot \mid s, a, \mu), p_{t}^{*}(\cdot \mid s, a, \nu)\big),\quad r\geq 0,
\end{equation}
where $p_t^*$ is the worst-case state-transition kernel of the mean-field equilibrium; see Theorem~\ref{thm:MFE}.

\begin{rem}\label{rem:UCM}
    Let $t = 0, \dots, T-1$.  By Assumption~\ref{as:weak_conv_ptb}~(i) and compactness of
    $\mathcal P(S)$, for each $(s,a)\in S\times A$, the map
    $\mu\mapsto p_t^*(\,\cdot\,\mid s,a,\mu)$ is uniformly continuous.
    Moreover, since $S\times A$ is finite, it holds that $\omega_t(r)\to0$ as $r\downarrow0$.
    Last, $\omega_t$ is non-decreasing by definition and is constant
    for $r\geq D_S$, since $d_{W_1}(\mu,\nu)\leq D_S$ for all
    $\mu,\nu\in\mathcal P(S)$.
\end{rem}

The second lemma provides a function that bounds the uniform continuity modulus \eqref{eq:uni:cont:mod} from above and has the properties used in the proof of Lemma~\ref{lem:P:cpl}.
\begin{lem} \label{lem:ome:maj}
    Suppose that Assumption~\ref{as:msr} and Assumption~\ref{as:weak_conv_ptb} are satisfied. Then, for any $t=0,\dots,T-1$ there exists a continuous, non-decreasing, concave function $\overline{\omega}_{t} \colon [0, \infty) \to \R$ such that $\overline{\omega}_{t}(r)\geq \omega_{t}(r)$ for all $r\geq 0$ and $\overline{\omega}_{t}(0) = 0$.
\end{lem}
\begin{proof}
    Fix $t\in\{0,\dots,T-1\}$. To construct $\overline{\omega}_t$ with the desired properties, we consider a map $g\colon[0,D_S]\to\R$ defined by
    \begin{equation}\label{eq:const_UCM}
        g(r) = \sup \Bigg\{\sum_{k = 1}^{m} \lambda_{k} \omega_{t}(r_{k}) \Biggm\vert
        \begin{array}{l}
            m \in \N, \quad \lambda_{k} \geq 0, \quad \sum_{k = 1}^{m} \lambda_{k} = 1, \\
            r_{k} \in [0, D_{S}], \quad \sum_{k = 1}^{m} \lambda_{k} r_{k} = r
        \end{array}
        \Bigg\},\quad r\in [0,D_S].
    \end{equation}

    \noindent {\it Step 1.} For each $r\in[0,D_S]$, taking $m=1$, $\lambda_1=1$, and $r_1=r$
    in the definition of $g$ (see \eqref{eq:const_UCM}) gives $g(r)\geq\omega_t(r)$. Moreover, we claim that $g$ is concave on $[0,D_S]$.

    Indeed, let $r^{1}, r^{2} \in [0, D_{S}]$ and $\varepsilon > 0$. Then, for $i = 1, 2$, there exist $m^{i} \in \N$, $\lambda^{i}_{k} \geq 0$, $r^{i}_{k} \in [0, D_{S}]$ such that
    \begin{equation} \label{eq:g:sum}
        \sum_{k = 1}^{m^{i}} \lambda^{i}_{k} = 1, \quad \sum_{k = 1}^{m^{i}} \lambda^{i}_{k} r^{i}_{k} = r^{i}, \quad \text{and} \quad \sum_{k = 1}^{m^{i}} \lambda^{i}_{k} \omega_{t}(r^{i}_{k}) \geq g(r^{i}) - \varepsilon.
    \end{equation}
    If we now combine these two convex combinations via $\theta \in [0, 1]$, we obtain that
    \begin{equation} \label{eq:conv:comb}
        \sum_{k = 1}^{m^{1}} \theta \lambda^{1}_{k} + \sum_{k = 1}^{m^{2}} (1 - \theta) \lambda^{2}_{k} = 1 \quad \text{and} \quad \sum_{k = 1}^{m^{1}} \theta \lambda^{1}_{k} r^{1}_{k} + \sum_{k = 1}^{m^{2}} (1 - \theta) \lambda^{2}_{k} r^{2}_{k} = \theta r^{1} + (1 - \theta) r^{2}.
    \end{equation}
    By definition of $g$ in \eqref{eq:const_UCM},~\eqref{eq:g:sum}, and~\eqref{eq:conv:comb}, it follows that
    \begin{align*}
        g(\theta r^{1} + (1 - \theta) r^{2}) \geq
        \theta \sum_{k = 1}^{m^{1}} \lambda^{1}_{k} \omega_{t}(r^{1}_{k}) + (1 - \theta) \sum_{k = 1}^{m^{2}} \lambda^{2}_{k} \omega_{t}(r^{2}_{k})
        \geq
        \theta g(r^{1}) + (1 - \theta) g(r^{2}) - \varepsilon.
    \end{align*}
    As $\varepsilon > 0$ was arbitrary, \(
        g(\theta r^{1} + (1 - \theta) r^{2}) \geq \theta g(r^{1}) + (1 - \theta) g(r^{2}).
    \)
    Thus, $g$ is concave on $[0,D_S]$.

    \noindent {\it Step 2.} For $r=0$, if $\sum_{k=1}^{m}\lambda_k r_k=0$ with
    $\lambda_k\geq0$ and $r_k\geq0$, then $r_k=0$ whenever
    $\lambda_k>0$. Thus, every admissible sum equals $\omega_t(0)$,
    and hence \(
        g(0)=\omega_t(0)=0.\)

    Set $M:=\omega_t(D_S)$. By monotonicity of $\omega_t$ (see Remark \ref{rem:UCM}) and
    definition of $g$ in \eqref{eq:const_UCM},
    \[
    0\leq\omega_t(r)\leq g(r)\leq M,
    \qquad r\in[0,D_S].
    \]
    In particular, $g(D_S)=M$. Since $g$ is finite and concave,
    it is continuous on $(0,D_S)$.

    \noindent {\it Step 3.} We claim that $g$ is continuous at $0$. Indeed, let $\varepsilon>0$.
    Since $\omega_t(r)\to0$ as $r\downarrow0$ (see Remark~\ref{rem:UCM}),
    we can choose $\delta>0$ such that $\omega_t(r)\leq\varepsilon$
    for all $r\in[0,\delta]$.

    Let $r\in[0,D_S]$ with $0<r<\delta$. Recalling the definition of $g$ in \eqref{eq:const_UCM}, let $m\in\mathbb N$, $\lambda_k\geq0$, and $r_k\in[0,D_S]$ satisfy $\sum_{k=1}^{m}\lambda_k=1$ and $\sum_{k=1}^{m}\lambda_k r_k=r$. Set $I_1:=\{k\leq m\mid r_k\leq\delta\}$ and $I_2:=\{k\leq m\mid r_k>\delta\}$. Since $r=\sum_{k=1}^{m}\lambda_k r_k\geq\delta\sum_{k\in I_2}\lambda_k$, we have $\sum_{k\in I_2}\lambda_k\leq r/\delta$. Therefore,
    \begin{equation*}
        \sum_{k=1}^{m}\lambda_k\omega_t(r_k)\leq\varepsilon\sum_{k\in I_1}\lambda_k+M\sum_{k\in I_2}\lambda_k\leq\varepsilon+\frac{Mr}{\delta}.
    \end{equation*}
    Since the choice of $m$, $\lambda_k$, and $r_k$ was arbitrary, taking the supremum gives $0\leq g(r)\leq\varepsilon+Mr/\delta$.

    Hence $\limsup_{r\downarrow0}g(r)\leq\varepsilon$. Since $\varepsilon>0$ was arbitrary and $g(0)=0$, we obtain $g(r)\to g(0)$ as $r\downarrow0$. This proves the claim.

    \noindent {\it Step 4.} Next, we show that $g$ is non-decreasing on $[0, D_{S}]$. Let $0 \leq r_{1} < r_{2} \leq D_{S}$. Then
    \begin{equation*}
        r_{2} = \lambda r_{1} + (1 - \lambda) D_{S}, \quad \text{where } \lambda := \frac{D_{S} - r_{2}}{D_{S} - r_{1}}.
    \end{equation*}
    Since $g$ is concave and $g \leq M = g(D_{S})$ on $[0, D_S]$ (by Steps 1 and 2), we have
    \begin{equation*}
        g(r_{2}) - g(r_{1}) \geq \lambda g(r_{1}) + (1 - \lambda) g(D_{S}) - g(r_{1}) = (1 - \lambda) \big(g(D_{S}) - g(r_{1})\big) \geq 0.
    \end{equation*}

    \noindent {\it Step 5.} Finally, define $\overline{\omega}_{t} \colon [0, \infty) \to \R$ by
    \begin{equation*}
        \overline{\omega}_{t}(r) :=
        \begin{cases}
            g(r), &\text{if } 0 \leq r \leq D_{S}, \\
            M, &\text{if } D_{S} < r.
        \end{cases}
    \end{equation*}
    Since $g$ is concave and non-decreasing on $[0,D_S]$ and $g(D_S)=M$ (by Steps 1, 2, 4), its constant extension $\overline{\omega}_t$ is concave and non-decreasing on $[0,\infty)$. As $\overline{\omega}_t$ is finite and concave, it is continuous on $(0,\infty)$. Continuity at $0$ follows from that of $g$ (by Step 3), since $\overline{\omega}_t=g$ on $[0,D_S]$.
\end{proof}
\color{black}

\subsection{Proof of Proposition~\ref{pro:ConGNtGam_ptb} and Proposition~\ref{pro:MNE2}}\label{sec:pro:ConGNtGam_ptb}
To connect the one-step laws $\mathbb Q_{t,i}^{N|(N)}$ and $\mathbb Q_t^{*|(N)}$ introduced in Definition~\ref{dfn:joint_laws}, we introduce the following auxiliary $N$-agent system. It retains the initial distribution and policy of the finite-agent game, but uses the mean-field worst-case transition kernels $p_{0:T-1}^*$ evaluated along the equilibrium flow $\mu_{0:T}^*$.

\begin{dfn} \label{dfn:worst_msr_help}
    Let $(\pi^{(N)}_{0:T})_{N\in \mathbb{N}}\subseteq \Pi$ be a sequence of arbitrary Markov policies. For every $N\in \mathbb{N}$ and $i\in\{1,\dots,N\}$ and $t\in \{0,\dots,T-1\}$, define the stochastic kernel by
    \begin{align*}%
        \begin{aligned}
            \underline{p}^{N|(N)}_{t,i}&:S^N \times A^N \ni (\overline{\underline{s}}_t^N, \overline{\underline{a}}_t^N)\mapsto
            \underline{p}^{N|(N)}_{t,i}(d\overline{\underline{s}}_{t+1}^N|\overline{\underline{s}}_t^N,\overline{\underline{a}}_t^N):=
            \prod_{j=1}^N p^*_t(d\underline{s}_{t+1}^{j}|\underline{s}_t^{j},\underline{a}_t^{j}, \mu_{t}^{*}).
        \end{aligned}
    \end{align*}
    Moreover, let $\underline{\mathbb{P}}_i^{N|(N)} \in \mathcal{P}(S^{N} \times (S^{N} \times A^{N})^{T})$ be given by
    \begin{align*}%
        \qquad\quad\underline{\mathbb{P}}_i^{N|(N)}:=\overline {\mathbb{P}}^N ({\mu^o, \overline \pi^{N|(N)}_{0:T,i}},\underline{p}^{N|(N)}_{0:T,i}):= \overline \mu^{o,N}\otimes \overline{ \mathbb{P}}^{N}_{(\overline \pi^{N|(N)}_{0,i},\underline{p}^{N|(N)}_{0,i})}\otimes \cdots \otimes \overline{\mathbb{P}}^{N}_{(\overline \pi^{N|(N)}_{T-1,i},\underline{p}^{N|(N)}_{T-1,i})},
    \end{align*}
    where, for each $t\in\{0,\dots,T-1\}$, the one-step kernel $\overline{\mathbb P}^{N}_{(\overline\pi^{N|(N)}_{t,i},\underline p^{N|(N)}_{t,i})}$ is defined as in \eqref{eq:def_onetime_msr_N}, under policy $\overline\pi^{N|(N)}_{t,i}$ (see \eqref{eq:over:p:NN}) and transition kernel $\underline p^{N|(N)}_{t,i}$.
\end{dfn}

\begin{rem}\label{rem:worst_msr_help}
    Under the auxiliary measure $\underline{\mathbb P}_i^{N|(N)}$ in Definition \ref{dfn:worst_msr_help}, the agents evolve independently along the fixed flow $\mu_{0:T}^*$, since the initial states are independent, each agent's policy depends only on its own state, and the policy and transition kernels have product form. For each $t\in\{0,\dots,T-1\}$, $\mathbb Q_t^{*|(N)}$ is the law of $(\underline s_t^i,\underline a_t^i,\underline s_{t+1}^i,\mu_t^*)$ for the deviating agent $i$, while $\mathbb Q_t^*$ is the law of $(\underline s_t^j,\underline a_t^j,\underline s_{t+1}^j,\mu_t^*)$ for every $j\neq i$, as defined in Definition~\ref{dfn:joint_laws}. Moreover, by definition of the mean-field equilibrium, $\underline s_t^j$ has law $\mu_t^*$ for every $j\neq i$ and $t\in\{0,\dots,T\}$. Hence, for each $t$, the random variables $(\underline s_t^j)_{j\neq i}$ are i.i.d.\ with common law $\mu_t^*$.
\end{rem}

\begin{lem} \label{lem:P:cpl}
    Suppose that Assumption~\ref{as:msr} and Assumption~\ref{as:weak_conv_ptb} are satisfied. Let $(i_N)_{N\in\N}\subseteq\N$ be a sequence of deviating agents with $i_N\leq N$ for every $N\in\N$, and let $(\pi^{(N)}_{0:T})_{N\in\N}\subseteq\Pi$ be a sequence of arbitrary Markov policies. For each $N\in\N$, let $\overline{\mathbb P}_{i_N}^{N|(N)}$ and $\underline{\mathbb P}_{i_N}^{N|(N)}$ be defined in Definition~\ref{dfn:worst_msr}~(ii) and Definition~\ref{dfn:worst_msr_help}, respectively. Then there exists a sequence of couplings $\mathfrak C_{i_N}^{N}\in\mathrm{Cpl}(\overline{\mathbb P}_{i_N}^{N|(N)},\underline{\mathbb P}_{i_N}^{N|(N)})$, $N\in\N$, such that, as $N\to\infty$,
    \begin{align}\label{eq:PiN:eN}
        &\mathbb E_{\mathfrak C_{i_N}^N}\big[d_{W_1}(e^N(\overline s_t^N),\mu_t^*)\big]\to0\quad
        \text{for every }t\in\{0,\dots,T-1\},\\
        \label{eq:PiN:siN:aiN}
        &\begin{aligned}
            \mathfrak C_{i_N}^N(s_t^{i_N}\neq\underline s_t^{i_N})&\to0
            &&\text{for every }t\in\{0,\dots,T\},\\
            \mathfrak C_{i_N}^N(a_t^{i_N}\neq\underline a_t^{i_N})&\to0
            &&\text{for every }t\in\{0,\dots,T-1\}.
        \end{aligned}
    \end{align}
\end{lem}
\begin{proof}
    {\it Step 1.} Fix $N\in\N$. We construct a coupling $\mathfrak C_{i_N}^N\in\mathrm{Cpl}(\overline{\mathbb P}_{i_N}^{N|(N)},\underline{\mathbb P}_{i_N}^{N|(N)})$. Since the initial state vectors under both measures have distribution $\overline{\mu}^{o,N}$, we couple them diagonally.

    For each $t\in\{0,\dots,T-1\}$ and each pair of current state configurations $(\overline s_t^N,\overline{\underline s}_t^N)$, we couple the actions as follows. For each agent $j\in\{1,\dots,N\}$, the two action laws are $\pi_t^*(\,\cdot\,\mid s_t^j)$ and $\pi_t^*(\,\cdot\,\mid\underline s_t^j)$ if $j\neq i_N$, and $\pi_t^{(N)}(\,\cdot\,\mid s_t^{i_N})$ and $\pi_t^{(N)}(\,\cdot\,\mid\underline s_t^{i_N})$ if $j=i_N$. We use their diagonal coupling when $s_t^j=\underline s_t^j$ and their product coupling otherwise. Taking the product of these couplings over $j=1,\dots,N$ defines a stochastic kernel $C_t^N\colon S^N\times S^N\to\mathcal P(A^N\times A^N)$.

    Next, fix $t\in\{0,\dots,T-1\}$, a paired state-action configuration $(\overline s_t^N,\overline{\underline s}_t^N,\overline a_t^N,\overline{\underline a}_t^N)$, and an agent $j\in\{1,\dots,N\}$. The next-state laws of agent $j$ in the original and auxiliary systems are $\overline q^{N|(N)}_{t,i_N,j}(\,\cdot\,\mid\overline s_t^N,\overline a_t^N)$ and $p_t^*(\,\cdot\,\mid\underline s_t^j,\underline a_t^j,\mu_t^*)$, respectively, where $\overline q^{N|(N)}_{t,i_N,j}$ is defined in \eqref{eq:over:p:NN}. Lemma~\ref{lem:max:cpl} provides a coupling
    \begin{align*}
        D_{t,j}^N(\,\cdot\,\mid\overline s_t^N,\underline s_t^j,\overline a_t^N,\underline a_t^j)
        \in\mathrm{Cpl}\big(
        \overline q^{N|(N)}_{t,i_N,j}(\,\cdot\,\mid\overline s_t^N,\overline a_t^N),
        p_t^*(\,\cdot\,\mid\underline s_t^j,\underline a_t^j,\mu_t^*)
        \big)
    \end{align*}
    such that
    \begin{align}\label{eq:kernel_D_max}
        \begin{aligned}
        &D_{t,j}^N\big(\{(s_{t+1}^{j}, \underline{s}_{t+1}^{j} )\in S\times S\mid s_{t+1}^{j}\neq \underline{s}_{t+1}^{j}\}\mid\overline s_t^N,\underline s_t^j,\overline a_t^N,\underline a_t^j\big)\\
        &\quad =\big\|\overline q^{N|(N)}_{t,i_N,j}(\,\cdot\,\mid\overline s_t^N,\overline a_t^N)
        -p_t^*(\,\cdot\,\mid\underline s_t^j,\underline a_t^j,\mu_t^*)\big\|_{\mathrm{TV}}.
        \end{aligned}
    \end{align}

    Define the stochastic kernel $D_t^N\colon S^N\times S^N\times A^N\times A^N\to\mathcal P(S^N\times S^N)$ for the pair of next-state vectors by
    \begin{align*}
        D_t^N(d\overline s_{t+1}^N,d\overline{\underline s}_{t+1}^N\mid\overline s_t^N,\overline{\underline s}_t^N,\overline a_t^N,\overline{\underline a}_t^N)
        :=\prod_{j=1}^{N}D_{t,j}^N(ds_{t+1}^j,d\underline s_{t+1}^j\mid\overline s_t^N,\underline s_t^j,\overline a_t^N,\underline a_t^j).
    \end{align*}

    Set
    \(
        \Lambda_0^N(d\overline s_0^N,d\overline{\underline s}_0^N):=\overline\mu^{o,N}(d\overline s_0^N)\delta_{\overline s_0^N}(d\overline{\underline s}_0^N).
    \)
    Starting from $\Lambda_0^N$, we apply $C_t^N$ and then $D_t^N$ at each time $t=0,\dots,T-1$ to obtain the joint law of the two trajectories,
    \begin{equation*}
        \mathfrak C_{i_N}^N:=\Lambda_0^N\otimes C_0^N\otimes D_0^N\otimes\cdots\otimes C_{T-1}^N\otimes D_{T-1}^N.
    \end{equation*}
    By construction, the first marginals of $C_t^N$ and $D_t^N$ are the action and transition kernels of the original system in Definition~\ref{dfn:worst_msr}, respectively, and their second marginals are those of the auxiliary system in Definition~\ref{dfn:worst_msr_help}. Moreover, since both marginals of $\Lambda_0^N$ equal $\overline\mu^{o,N}$, it follows that $\mathfrak C_{i_N}^N\in\mathrm{Cpl}(\overline{\mathbb P}_{i_N}^{N|(N)},\underline{\mathbb P}_{i_N}^{N|(N)})$.

    \noindent {\it Step 2.} Let $j\in\{1,\dots,N\}$. Then, by definition of $\Lambda_0^N$, it holds that
    \[
        \mathfrak C_{i_N}^N(s_0^j\neq\underline s_0^j)=0.
    \]
    For every $t\in\{0,\dots,T-1\}$, the kernel $C_t^N$ couples the actions of agent $j$ diagonally whenever $s_t^j=\underline s_t^j$  (see Step 1), so it holds that $\mathfrak C_{i_N}^N(a_t^j\neq\underline a_t^j,\,s_t^j=\underline s_t^j)=0$.
    Consequently, for every $t\in\{0,\dots,T-1\}$,
    \begin{equation}\label{eq:action:cpl}
        \mathfrak C_{i_N}^N(a_t^j\neq\underline a_t^j)=\mathfrak C_{i_N}^N(a_t^j\neq\underline a_t^j,\,s_t^j\neq\underline s_t^j)\leq\mathfrak C_{i_N}^N(s_t^j\neq\underline s_t^j).
    \end{equation}

    Let $t\in\{0,\dots,T-1\}$. When $\underline{s}_{t}^{j} = s_{t}^{j}$ and $\underline{a}_{t}^{j} = a_{t}^{j}$, by the triangle inequality, Assumption~\ref{as:weak_conv_ptb}, and the definition of  the uniform continuity modulus $\omega_{t}$ in \eqref{eq:uni:cont:mod}, we have that
    \begin{align} \label{eq:q:p:ineq}
        \begin{aligned}
            &d_{W_{1}}\big(\overline{q}^{N \vert (N)}_{t, i_{N}, j}(\,\cdot \mid \overline{s}_{t}^{N}, \overline{a}_{t}^{N}), p_{t}^{*}(\,\cdot \mid \underline{s}_{t}^{j}, \underline{a}_{t}^{j}, \mu_{t}^{*})\big) \\
            &\leq d_{W_{1}}\Big(\overline{q}^{N \vert (N)}_{t, i_{N}, j}(\,\cdot \mid \overline{s}_{t}^{N}, \overline{a}_{t}^{N}), p_{t}^{*}\big(\,\cdot \mid s_{t}^{j}, a_{t}^{j}, e^{N}(\overline{s}_{t}^{N})\big)\Big) \\
            &\hspace{90pt}+ d_{W_{1}}\Big(p_{t}^{*}\big(\,\cdot \mid s_{t}^{j}, a_{t}^{j}, e^{N}(\overline{s}_{t}^{N})\big), p_{t}^{*}(\,\cdot \mid s_{t}^{j}, a_{t}^{j}, \mu_{t}^{*})\Big) \\
            &\leq \Delta_{N} + \omega_{t}\big(d_{W_{1}}(e^{N}(\overline{s}_{t}^{N}), \mu_{t}^{*})\big),
        \end{aligned}
    \end{align}

    For each $t = 0, \dots, T-1$, set \(
        \mathcal{F}_{t} := \sigma(\overline{s}_{0}^{N}, \overline{\underline{s}}_{0}^{N}, \overline{a}_{0}^{N}, \overline{\underline{a}}_{0}^{N}, \dots, \overline{s}_{t}^{N}, \overline{\underline{s}}_{t}^{N}, \overline{a}_{t}^{N}, \overline{\underline{a}}_{t}^{N}).
    \)

    Then, by the definition of the coupling ${\mathfrak{C}}_{i_{N}}^{N}$ as a maximal coupling for the next-state distributions (see \eqref{eq:kernel_D_max}), on the event that $\underline{s}_{t}^{j} = s_{t}^{j}$ and thus $\underline{a}_{t}^{j} = a_{t}^{j}$ (as the actions are coupled diagonally on this event), it follows that
    \begin{align*}
        \mathbf{1}_{\{\underline{s}_{t}^{j} = s_{t}^{j}\}} {\mathfrak{C}}_{i_{N}}^{N}(s_{t+1}^{j} \neq \underline{s}_{t+1}^{j} \mid \mathcal{F}_{t}) &= \lVert \overline{q}^{N \vert (N)}_{t, i, j}(\,\cdot \mid \overline{s}_{t}^{N}, \overline{a}_{t}^{N}) - p_{t}^{*}(\,\cdot \mid \underline{s}_{t}^{j}, \underline{a}_{t}^{j}, \mu_{t}^{*}) \rVert_{\mathrm{TV}} \mathbf{1}_{\{\underline{s}_{t}^{j} = s_{t}^{j}\}} \\
        &\leq \frac{1}{d_{S}} d_{W_{1}}\big(\overline{q}^{N \vert (N)}_{t, i_{N}, j}(\,\cdot \mid \overline{s}_{t}^{N}, \overline{a}_{t}^{N}), p_{t}^{*}(\,\cdot \mid \underline{s}_{t}^{j}, \underline{a}_{t}^{j}, \mu_{t}^{*})\big) \mathbf{1}_{\{\underline{s}_{t}^{j} = s_{t}^{j}\}} \\
        &\leq \frac{1}{d_{S}} \Big(\Delta_{N} + \omega_{t}\big(d_{W_{1}}(e^{N}(\overline{s}_{t}^{N}), \mu_{t}^{*})\big)\Big) \mathbf{1}_{\{\underline{s}_{t}^{j} = s_{t}^{j}\}},
    \end{align*}
    where the first inequality follows from the bound $\lVert\mu-\nu\rVert_{\mathrm{TV}}\leq d_{W_1}(\mu,\nu)/d_S$ for all $\mu,\nu\in\mathcal P(S)$ (recalling that $d_S$ is defined in \eqref{eq:diam:S:A}), and the second inequality follows from \eqref{eq:q:p:ineq}.

    Thus,
    \begin{align*}
        &{\mathfrak{C}}_{i_{N}}^{N}(s_{t+1}^{j} \neq \underline{s}_{t+1}^{j}, s_{t}^{j} = \underline{s}_{t}^{j}) = \mathbb{E}_{{\mathfrak{C}}_{i_{N}}^{N}}\bigg[\mathbf{1}_{\{s_{t}^{j} = \underline{s}_{t}^{j}\}} {\mathfrak{C}}_{i_{N}}^{N}(s_{t+1}^{j} \neq \underline{s}_{t+1}^{j} \mid \mathcal{F}_{t})\bigg] \\
        &\leq \mathbb{E}_{{\mathfrak{C}}_{i_{N}}^{N}}\bigg[\mathbf{1}_{\{s_{t}^{j} = \underline{s}_{t}^{j}\}} {\mathfrak{C}}_{i_{N}}^{N}(s_{t+1}^{j} \neq \underline{s}_{t+1}^{j}\mid \mathcal{F}_{t})\bigg]\\
        &\leq \mathbb{E}_{{\mathfrak{C}}_{i_{N}}^{N}}\bigg[\mathbf{1}_{\{s_{t}^{j} = \underline{s}_{t}^{j}\}} \frac{1}{d_{S}} \Big(\Delta_{N} + \omega_{t}\big(d_{W_{1}}(e^{N}(\overline{s}_{t}^{N}), \mu_{t}^{*})\big)\Big)\bigg] \leq \frac{1}{d_{S}} \bigg(\Delta_{N} + \mathbb{E}_{{\mathfrak{C}}_{i_{N}}^{N}}\Big[\omega_{t}\big(d_{W_{1}}(e^{N}(\overline{s}_{t}^{N}), \mu_{t}^{*})\big)\Big]\bigg).
    \end{align*}
    This implies that
    \begin{align} \label{eq:s:neq:s_}
        \begin{aligned}
            {\mathfrak{C}}_{i_{N}}^{N}(s_{t+1}^{j} \neq \underline{s}_{t+1}^{j}) &\leq {\mathfrak{C}}_{i_{N}}^{N}(s_{t+1}^{j} \neq \underline{s}_{t+1}^{j}, s_{t}^{j} \neq \underline{s}_{t}^{j}) + {\mathfrak{C}}_{i_{N}}^{N}(s_{t+1}^{j} \neq \underline{s}_{t+1}^{j}, s_{t}^{j} = \underline{s}_{t}^{j}) \\
            &\leq {\mathfrak{C}}_{i_{N}}^{N}(s_{t}^{j} \neq \underline{s}_{t}^{j}) + {\mathfrak{C}}_{i_{N}}^{N}(s_{t+1}^{j} \neq \underline{s}_{t+1}^{j}, s_{t}^{j} = \underline{s}_{t}^{j}) \\
            &\leq {\mathfrak{C}}_{i_{N}}^{N}(s_{t}^{j} \neq \underline{s}_{t}^{j}) + \frac{1}{d_{S}} \bigg(\Delta_{N} + \mathbb{E}_{{\mathfrak{C}}_{i_{N}}^{N}}\Big[\omega_{t}\big(d_{W_{1}}(e^{N}(\overline{s}_{t}^{N}), \mu_{t}^{*})\big)\Big]\bigg).
        \end{aligned}
    \end{align}

    \noindent {\it Step 3.} For each $t \in\{ 0, \dots, T\}$, set
    \begin{equation}\label{eq:def:alpha:rho:mu}
        \alpha_{t}^{N} := \frac{1}{N} \sum_{j = 1}^{N} {\mathfrak{C}}_{i_{N}}^{N}(s_{t}^{j} \neq \underline{s}_{t}^{j}), \quad \rho_{t}^{N} := \mathbb{E}_{{\mathfrak{C}}_{i_{N}}^{N}}\big[d_{W_{1}}(e^{N}(\overline{s}_{t}^{N}), \mu_{t}^{*})\big].
    \end{equation}
    We claim that for every $t\in\{0,\dots,T-1\},$
    \begin{equation}\label{eq:alpha:ineq}
        \alpha_{t+1}^N\leq\alpha_t^N+\frac{1}{d_S}\bigg[\Delta_N+\overline{\omega}_t\bigg(D_S\alpha_t^N+\frac{D_S}{N}+\frac{C_S}{\sqrt{N-1}}\bigg)\bigg].
    \end{equation}

    To this end, set $C_S:=D_S\sqrt{|S|}/2$, where $D_S$ is defined in \eqref{eq:diam:S:A}. For each $t\in\{0,\dots,T\}$, the auxiliary states $(\underline s_t^j)_{j\neq i_N}$ are i.i.d.\ with common law $\mu_t^*$ (see Remark~\ref{rem:worst_msr_help}). Hence, for $N\geq2$, the standard finite-state empirical estimate applied to these $N-1$ auxiliary states yields that for every $t\in\{0,\dots,T-1\},$
    \begin{equation}\label{eq:dw:fsN-1}
        \mathbb E_{\mathfrak C_{i_N}^N}\big[d_{W_1}(\underline{\mu}_{t}^{N, -i_{N}},\mu_t^*)\big]\leq\frac{C_S}{\sqrt{N-1}}, \qquad \mbox{where } \underline{\mu}_{t}^{N, -i_{N}}:=\frac{1}{N-1} \sum_{j \neq i_{N}} \delta_{\underline{s}_{t}^{j}}
    \end{equation}
    where $\underline\mu_t^{N,-i_N}$ denotes their empirical distribution. Adding agent $i_N$ changes the empirical distribution by at most $D_S/N$ in $d_{W_1}$. Hence, by the triangle inequality, for every $t\in\{0,\dots,T-1\},$
    \begin{align}\label{eq:dw:total:N}
        \begin{aligned}
        \mathbb E_{\mathfrak C_{i_N}^N}\Big[d_{W_1}(e^N(\overline{\underline s}_t^N),\mu_t^*)\Big]
        &\leq\mathbb E_{\mathfrak C_{i_N}^N}\big[d_{W_1}(e^N(\overline{\underline s}_t^N),\underline{\mu}_{t}^{N, -i_{N}})\big]
        +\mathbb E_{\mathfrak C_{i_N}^N}\big[d_{W_1}(\underline{\mu}_{t}^{N, -i_{N}},\mu_t^*)\big]\\
        &\leq\frac{D_S}{N}+\frac{C_S}{\sqrt{N-1}}.
        \end{aligned}
    \end{align}

    Pairing the states of each agent across the two systems gives a coupling $\gamma_t^N:=\frac{1}{N}\sum_{j=1}^{N}\delta_{(s_t^j,\underline s_t^j)}$ of the empirical distributions $e^N(\overline s_t^N)$ and $e^N(\overline{\underline s}_t^N)$. Hence, for every $t\in\{0,\dots,T-1\},$
    \begin{align}\label{eq:dw:sep:state}
        d_{W_1}\big(e^N(\overline s_t^N),e^N(\overline{\underline s}_t^N)\big)
        \leq\int_{S\times S}|s-\underline s|\,\gamma_t^N(ds,d\underline s)=\frac{1}{N}\sum_{j=1}^{N}|s_t^j-\underline s_t^j|
        \leq\frac{D_S}{N}\sum_{j=1}^{N}\mathbf{1}_{\{s_t^j\neq\underline s_t^j\}}.
    \end{align}
    For $N\geq2$, the triangle inequality and the preceding estimates in \eqref{eq:dw:fsN-1},\eqref{eq:dw:total:N}, and \eqref{eq:dw:sep:state} yields that for every $t\in\{0,\dots,T-1\},$
    \begin{equation}\label{eq:rho:alpha:ineq}
        \begin{aligned}
            \rho_t^N
            &\leq\mathbb E_{\mathfrak C_{i_N}^N}\Big[d_{W_1}(e^N(\overline s_t^N),e^N(\overline{\underline s}_t^N))\Big]
            +\mathbb E_{\mathfrak C_{i_N}^N}\Big[d_{W_1}(e^N(\overline{\underline s}_t^N),\mu_t^*)\Big]\\
            &\leq\frac{D_S}{N}\sum_{j=1}^{N}\mathfrak C_{i_N}^N(s_t^j\neq\underline s_t^j)
            +\frac{D_S}{N}+\frac{C_S}{\sqrt{N-1}}=D_S\alpha_t^N+\frac{D_S}{N}+\frac{C_S}{\sqrt{N-1}},
        \end{aligned}
    \end{equation}
    recalling that $\rho_t^N$ and $\alpha_t^N$ are defined in \eqref{eq:def:alpha:rho:mu}.

    By \eqref{eq:s:neq:s_}, it follows that for every $t\in\{0,\dots,T-1\},$
    \begin{equation}\label{eq:ine:dW:alpha}
        \alpha_{t+1}^{N} \leq \alpha_{t}^{N} + \frac{1}{d_{S}} \bigg(\Delta_{N} + \mathbb E_{\mathfrak C_{i_N}^N}\Big[\omega_t\big(d_{W_1}(e^N(\overline s_t^N),\mu_t^*)\big)\Big]\bigg).
    \end{equation}

    By Lemma~\ref{lem:ome:maj}, choose a continuous, non-decreasing, concave function $\overline{\omega}_t$ such that $\overline{\omega}_t(0)=0$ and $\omega_t(r)\leq\overline{\omega}_t(r)$ for all $r\geq0$. Applying this bound and then Jensen's inequality, we obtain that for every $t\in\{0,\dots,T-1\},$
    \begin{equation}\label{eq:omega:rho}
        \mathbb E_{\mathfrak C_{i_N}^N}\Big[\omega_t\big(d_{W_1}(e^N(\overline s_t^N),\mu_t^*)\big)\Big]
            \leq\mathbb E_{\mathfrak C_{i_N}^N}\Big[\overline{\omega}_t\big(d_{W_1}(e^N(\overline s_t^N),\mu_t^*)\big)\Big]\leq\overline{\omega}_t(\rho_t^N).
    \end{equation}
    Combining \eqref{eq:ine:dW:alpha} and \eqref{eq:omega:rho} with the monotonicity of $\overline{\omega}_t$ and  \eqref{eq:rho:alpha:ineq}, we obtain that \eqref{eq:alpha:ineq} holds, as claimed.

    \noindent {\it Step 4.} We first show by induction on $t$ that $\alpha_t^N\to0$ as $N\to\infty$. The diagonal initial coupling in Step~2 gives $\alpha_0^N=0$ for every $N\in\N$. Suppose that $\alpha_t^N\to0$ for some $t\in\{0,\dots,T-1\}$. Since $\Delta_N\to0$ and $\overline{\omega}_t$ is continuous at $0$ with $\overline{\omega}_t(0)=0$ (see Lemma~\ref{lem:ome:maj}), the right-hand side of \eqref{eq:alpha:ineq} tends to zero as $N\to\infty$. Hence $\alpha_{t+1}^N\to0$. By induction, $\alpha_t^N\to0$ for every $t\in\{0,\dots,T\}$. It then follows from \eqref{eq:rho:alpha:ineq} that $\rho_t^N\to0$ for every $t\in\{0,\dots,T\}$.

    We next consider the deviating agent $i_N$. For each $t\in\{0,\dots,T-1\}$, \eqref{eq:s:neq:s_} and \eqref{eq:omega:rho} give
    \begin{equation*}
        \mathfrak C_{i_N}^N(s_{t+1}^{i_N}\neq\underline s_{t+1}^{i_N})
        \leq\mathfrak C_{i_N}^N(s_t^{i_N}\neq\underline s_t^{i_N})
        +\frac{1}{d_S}\big(\Delta_N+\overline{\omega}_t(\rho_t^N)\big).
    \end{equation*}
    The last term tends to zero as $N\to\infty$, since $\Delta_N\to0$, $\rho_t^N\to0$, and $\overline{\omega}_t$ is continuous at $0$ with $\overline{\omega}_t(0)=0$. Since the diagonal initial coupling gives $\mathfrak C_{i_N}^N(s_0^{i_N}\neq\underline s_0^{i_N})=0$, induction yields
    \begin{equation*}
        \mathfrak C_{i_N}^N(s_t^{i_N}\neq\underline s_t^{i_N})\to 0
        \qquad\text{as }N\to\infty
    \end{equation*}
    for every $t\in\{0,\dots,T\}$. Finally, by \eqref{eq:action:cpl}, for every $t\in\{0,\dots,T-1\}$,
    \begin{equation*}
        \mathfrak C_{i_N}^N(a_t^{i_N}\neq\underline a_t^{i_N})
        \leq\mathfrak C_{i_N}^N(s_t^{i_N}\neq\underline s_t^{i_N})
        \to 0
        \qquad\text{as }N\to\infty.
    \end{equation*}
    This completes the proof.
\end{proof}

\begin{proof}[Proof of Proposition~\ref{pro:ConGNtGam_ptb}]
    Fix $t \in \{0, \dots, T-1\}$. As $S \times A \times S \times \mathcal{P}(S)$ is compact, the $1$-Wasserstein distance $d_{W_{1}}$ metricizes the weak topology defined in \eqref{eq:topology_w}. Hence, it is sufficient to show that $d_{W_{1}}(\mathbb{Q}_{t, i_{N}}^{N|(N)}, \mathbb{Q}^{*|(N)}_{t, i_{N}})$ converges to $0$.

    Let $({\mathfrak{C}}_{i_{N}}^{N})_{N \in \N}$ be the sequence of couplings ${\mathfrak{C}}_{i_{N}}^{N} \in \mathrm{Cpl}(\overline{\mathbb{P}}_{i_{N}}^{N|(N)}, \underline{\mathbb{P}}_{i_{N}}^{N|(N)})$ constructed in Lemma~\ref{lem:P:cpl}.
    We can formally define both $\mathbb{Q}^{N|(N)}_{t, i_{N}}$ and $\mathbb{Q}_{t}^{*|(N)}$ as a pushforward measure of $\overline{\mathbb{P}}_{i_{N}}^{N|(N)}$ and $\underline{\mathbb{P}}_{i_{N}}^{N|(N)}$, respectively. Indeed, define
    \begin{equation*}
        \Phi_{t, 1}^{N}, \Phi_{t, 2}^{N} \colon S^{N} \times (S^{N} \times A^{N})^{T} \to S \times A \times S \times \mathcal{P}(S)
    \end{equation*}
    by, for every $(\overline{s}_{0}^{N}, \overline{a}_{0}^{N}, \dots, \overline{s}_{T-1}^{N}, \overline{a}_{T-1}^{N}, \overline{s}_{T}^{N}), (\overline{\underline{s}}_{0}^{N}, \overline{\underline{a}}_{0}^{N}, \dots, \overline{\underline{s}}_{T-1}^{N}, \overline{\underline{a}}_{T-1}^{N}, \overline{\underline{s}}_{T}^{N}) \in S^{N} \times (S^{N} \times A^{N})^{T}$,
    \begin{eqnarray*}
        &&\Phi_{t, 1}^{N}(\overline{s}_{0}^{N}, \overline{a}_{0}^{N}, \dots, \overline{s}_{T-1}^{N}, \overline{a}_{T-1}^{N}, \overline{s}_{T}^{N}) := (s_{t}^{i_{N}}, a_{t}^{i_{N}}, s_{t+1}^{i_{N}}, e^{N}(\overline{s}_{t}^{N})),\\
        &&\Phi_{t, 2}^{N}(\overline{\underline{s}}_{0}^{N}, \overline{\underline{a}}_{0}^{N}, \dots, \overline{\underline{s}}_{T-1}^{N}, \overline{\underline{a}}_{T-1}^{N}, \overline{\underline{s}}_{T}^{N}) := (\underline{s}_{t}^{i_{N}}, \underline{a}_{t}^{i_{N}}, \underline{s}_{t+1}^{i_{N}}, \mu_{t}^{*}).
    \end{eqnarray*}
    Then it holds that  \(
        \mathbb{Q}^{N|(N)}_{t, i_{N}} = \overline{\mathbb{P}}_{i_{N}}^{N|(N)} \circ (\Phi_{t, 1}^{N})^{-1}\)  and \(\mathbb{Q}_{t}^{*|(N)} = \underline{\mathbb{P}}_{i_{N}}^{N|(N)} \circ (\Phi_{t, 2}^{N})^{-1}.\)

    Now define
    \begin{equation*}
        \Phi_{t}^{N} \colon (S^{N} \times (S^{N} \times A^{N})^{T})^2 \to \big(S \times A \times S \times \mathcal{P}(S)\big)^2
    \end{equation*}
    by for every $(\overline{s}_{0}^{N}, \overline{a}_{0}^{N}, \dots, \overline{s}_{T-1}^{N}, \overline{a}_{T-1}^{N}, \overline{s}_{T}^{N}, \overline{\underline{s}}_{0}^{N}, \overline{\underline{a}}_{0}^{N}, \dots, \overline{\underline{s}}_{T-1}^{N}, \overline{\underline{a}}_{T-1}^{N}, \overline{\underline{s}}_{T}^{N})\in (S^{N} \times (S^{N} \times A^{N})^{T})^2$,
    \begin{align*}
        &\Phi_{t}^{N}(\overline{s}_{0}^{N}, \overline{a}_{0}^{N}, \dots, \overline{s}_{T-1}^{N}, \overline{a}_{T-1}^{N}, \overline{s}_{T}^{N}, \overline{\underline{s}}_{0}^{N}, \overline{\underline{a}}_{0}^{N}, \dots, \overline{\underline{s}}_{T-1}^{N}, \overline{\underline{a}}_{T-1}^{N}, \overline{\underline{s}}_{T}^{N}) \\
        &:= (\Phi_{t,1}^{N}(\overline{s}_{0}^{N}, \overline{a}_{0}^{N}, \dots, \overline{s}_{T-1}^{N}, \overline{a}_{T-1}^{N}, \overline{s}_{T}^{N}), \Phi_{t, 2}^{N}(\overline{\underline{s}}_{0}^{N}, \overline{\underline{a}}_{0}^{N}, \dots, \overline{\underline{s}}_{T-1}^{N}, \overline{\underline{a}}_{T-1}^{N}, \overline{\underline{s}}_{T}^{N})),
    \end{align*}
    and then we define the pushforward measure
    \begin{equation*}
        \mathbb{Q}^{N}_{t, i_{N}} := {\mathfrak{C}}_{i_{N}}^{N} \circ (\Phi_{t}^{N})^{-1}.
    \end{equation*}
    Then, it holds that $\mathbb{Q}^{N}_{t, i_{N}} \in \mathrm{Cpl}(\mathbb{Q}^{N|(N)}_{t, i_{N}}, \mathbb{Q}_{t}^{*|(N)})$. Indeed, for any measurable $B \subseteq S \times A \times S \times \mathcal{P}(S)$,
    \begin{align*}
        \mathbb{Q}^{N}_{t, i_{N}}\Big(B \times \big(S \times A \times S \times \mathcal{P}(S)\big)\Big) &= {\mathfrak{C}}_{i_{N}}^{N}\bigg((\Phi_{t}^{N})^{-1}\Big(B \times \big(S \times A \times S \times \mathcal{P}(S)\big)\Big)\bigg) \\
        &= {\mathfrak{C}}_{i_{N}}^{N}\bigg((\Phi_{t, 1}^{N})^{-1}(B) \times (\Phi_{t, 2}^{N})^{-1}\big(S \times A \times S \times \mathcal{P}(S)\big)\bigg) \\
        &= {\mathfrak{C}}_{i_{N}}^{N}\bigg((\Phi_{t, 1}^{N})^{-1}(B) \times \big(S^{N} \times (S^{N} \times A^{N})^{T}\big)\bigg) \\
        &= \overline{\P}_{i_{N}}^{N|(N)}\big((\Phi_{t, 1}^{N})^{-1}(B)\big) = \mathbb{Q}^{N|(N)}_{t, i_{N}}(B).
    \end{align*}
    Analogously, we have \(
        \mathbb{Q}^{N}_{t, i_{N}}((S \times A \times S \times \mathcal{P}(S)) \times B) = \underline{\P}_{i_{N}}^{N|(N)}((\Phi_{t, 2}^{N})^{-1}(B)) = \mathbb{Q}^{*|(N)}_{t}(B).\)

    Hence, it holds that $\mathbb{Q}^{N}_{t, i_{N}} \in \mathrm{Cpl}(\mathbb{Q}^{N|(N)}_{t, i_{N}}, \mathbb{Q}_{t}^{*|(N)})$.

    Using the coupling $\mathbb Q_{t,i_N}^N=\mathfrak C_{i_N}^N\circ(\Phi_t^N)^{-1}$, we obtain
    \begin{align*}
        &d_{W_1}\big(\mathbb Q_{t,i_N}^{N|(N)},\mathbb Q_t^{*|(N)}\big)\\
        &\leq\mathbb E_{\mathfrak C_{i_N}^N}\Big[
        |s_t^{i_N}-\underline s_t^{i_N}|+|a_t^{i_N}-\underline a_t^{i_N}|
        +|s_{t+1}^{i_N}-\underline s_{t+1}^{i_N}|
        +d_{W_1}(e^N(\overline s_t^N),\mu_t^*)\Big]\\
        &\leq D_S\mathfrak C_{i_N}^N(s_t^{i_N}\neq\underline s_t^{i_N})
        +D_A\mathfrak C_{i_N}^N(a_t^{i_N}\neq\underline a_t^{i_N})
        +D_S\mathfrak C_{i_N}^N(s_{t+1}^{i_N}\neq\underline s_{t+1}^{i_N})+\mathbb E_{\mathfrak C_{i_N}^N}\big[d_{W_1}(e^N(\overline s_t^N),\mu_t^*)\big],
    \end{align*}
    where $D_S$ and $D_A$ are defined in \eqref{eq:diam:S:A}.

    By the convergence \eqref{eq:PiN:eN} and \eqref{eq:PiN:siN:aiN}  in Lemma~\ref{lem:P:cpl}, all four terms tend to zero as $N\to\infty$. This completes the proof.
\end{proof}
\color{black}

\begin{proof}[Proof of Proposition~\ref{pro:MNE2}]
    For $N\in \mathbb{N}$, let $\mathbb{Q}_{0:T}^{*|(N)},\mathbb{Q}^{N|(N)}_{0:T, i_{N}}$ be given in Definition~\ref{dfn:joint_laws}. It holds that, for every $t\in\{0,\dots,T-1\}$,
    \begin{align*}
        \mathbb{E}^{\overline{\mathbb{P}}_{i_{N}}^{N|(N)}}\big[r({s}^{i_{N}}_{t}, {a}^{i_{N}}_{t}, {s}^{i_{N}}_{t+1}, {e}^{N}_{t}(\overline{s}_t^N))\big] &= \int_{S \times A \times S \times {\cal P}(S)} r(s_t, a_t, s_{t+1}, \mu_t) \mathbb{Q}^{N|(N)}_{t, i_{N}}(ds_t, da_t, ds_{t+1}, d\mu_t),\\
        \mathbb{E}^{{\mathbb{P}}^{*|(N)}}\big[r({s}_{t}, {a}_{t}, {s}_{t+1}, \mu_t^*)\big] &= \int_{S \times A \times S \times {\cal P}(S)} r(s_t, a_t, s_{t+1}, \mu_t) \mathbb{Q}^{*|(N)}_{t}(ds_t, da_t, ds_{t+1}, d\mu_t).
    \end{align*}
    Moreover, $r \in C_{b}(S \times A \times S \times \mathcal{P}(S))$ by Assumption~\ref{as:msr}\;(i) and\;(iii). By Proposition~\ref{pro:ConGNtGam_ptb}, it then follows that
    \begin{equation*}
        \lim_{N \to \infty} \left|\mathbb{E}^{\overline{\mathbb{P}}_{i_{N}}^{N|(N)}}\big[r({s}^{i_{N}}_{t}, {a}^{i_{N}}_{t}, {s}^{i_{N}}_{t+1}, {e}^{N}_{t}(\overline{s}_t^N))\big] -\mathbb{E}^{{\mathbb{P}}^{*|(N)}}\big[r({s}_{t}, {a}_{t}, {s}_{t+1}, \mu_t^*)\big]  \right|= 0.
    \end{equation*}
    Hence,
    \begin{align*}
        &\lim_{N \to \infty} \left|J_{i_{N}}^N(\mu^o,\overline{\pi}_{0:T,i_{N}}^{N|(N)})- \mathbb{E}^{\mathbb{P}^{*|(N)}} \left[\sum_{t=0}^{T-1}r(s_t,a_t,s_{t+1}, \mu_t^*)\right] \right|\\
        &\quad = \lim_{N \to \infty} \bigg\lvert \sum_{t = 0}^{T-1} \mathbb{E}^{\overline{\mathbb{P}}_{i_{N}}^{N|(N)}}\big[r({s}^{i_{N}}_{t}, {a}^{i_{N}}_{t}, {s}^{i_{N}}_{t+1}, {e}^{N}_{t}(\overline{s}_t^N))\big]-\sum_{t=0}^{T-1} \mathbb{E}^{{\mathbb{P}}^{*|(N)}}\big[r({s}_{t}, {a}_{t}, {s}_{t+1}, \mu_t^*)\big] \bigg\rvert \\
        &\quad \leq \sum_{t = 0}^{T-1} \lim_{N \to \infty} \left|\mathbb{E}^{\overline{\mathbb{P}}_{i_{N}}^{N|(N)}}\big[r({s}^{i_{N}}_{t}, {a}^{i_{N}}_{t}, {s}^{i_{N}}_{t+1}, {e}^{N}_{t}(\overline{s}_t^N))\big] -\mathbb{E}^{{\mathbb{P}}^{*|(N)}}\big[r({s}_{t}, {a}_{t}, {s}_{t+1}, \mu_t^*)\big]  \right| = 0.
    \end{align*}
    This completes the proof.
\end{proof}

\subsection{Proof of Theorem \ref{thm:MNE0}}\label{sec:thm:MNE0}
\begin{proof}[Proof of Theorem \ref{thm:MNE0}]
    Let $\varepsilon > 0$. For a permutation $\sigma$ of the agent labels, let $(\sigma \overline{s}^{N})^{j} = s ^{\sigma^{-1}(j)}$, and use the analogous convetion for actions and policy profiles. Definition~\ref{dfn:corr_Nplayer}~(iii) implies that
    \begin{equation*}
        \mathfrak{P}_{t}^{N}(\sigma \overline{s}_{t}^{N}, \sigma \overline{a}_{t}^{N}) = \{\sigma_{\#} \P \mid \P \in \mathfrak{P}_{t}^{N}(\overline{s}_{t}^{N}, \overline{a}_{t}^{N})\},
    \end{equation*}
    where $\sigma_{\#} \P := \P \circ \sigma^{-1}$ denotes the push-forward of $\P$ under the permutation $\sigma$. For each admissible joint kernel $p_{t}^{N}$, the relabeled kernel defined by
    \begin{equation*}
        p_{t}^{N, \sigma}(B \mid \overline{s}^{N}, \overline{a}^{N}) = p_{t}^{N}(\sigma^{-1} B \mid \sigma^{-1} \overline{s}^{N}, \sigma^{-1} \overline{a}^{N})
    \end{equation*}
    remains admissible. Since the initial law is i.i.d., relabeling every state and action coordinate of a path gives
    \begin{equation*}
        \mathcal{Q}^{N}(\mu^{o}, \sigma \overline{\pi}_{0:T}^{N}) = \{\sigma_{\#} \P \mid \P \in \mathcal{Q}^{N}(\mu^{o}, \overline{\pi}_{0:T}^{N})\}.
    \end{equation*}
    The profile $\overline{\pi}_{0:T}^{N \vert *}$ is unchanged by relabeling, and a profile with one deviating agent is relabeled by moving that deviator to the corresponding new label.

    Therefore, it is enough to establish the claimed inequality for agent~$1$, i.e., to show that there exists $N(\varepsilon)\in \mathbb{N}$ such that for each $N\geq N(\varepsilon)$,
    \begin{align*}
        J^{N}_1(\mu^o,\overline \pi_{0:T}^{N|*})  + \varepsilon \geq \sup_{\pi_{0:T}\in \Pi } J^{N}_1(\mu^o,(\overline \pi_{0:T}^{N|*,-1},\pi_{0:T})),
    \end{align*}
    where $J^{N}_1$ denotes the worst-case reward for agent $1$.

    For each $N \in \mathbb{N}$, let $\pi^{(N)}_{0:T}\in \Pi$ be a sequence of policies satisfying that
    \begin{align}\label{eq:MNE_step1}
        J^{N}_1(\mu^o,(\overline \pi_{0:T}^{N|*,-1},\pi_{0:T}^{(N)}))> \sup_{\pi_{0:T}\in \Pi } J^{N}_1(\mu^o,(\overline \pi_{0:T}^{N|*,-1},\pi_{0:T}))-\frac{\varepsilon}{3}.
    \end{align}

    By Proposition~\ref{pro:dpp}\;(ii) (by replacing $\tilde \mu_{0:T}$ as $\mu^*_{0:T}$; see \eqref{eq:verification}) it holds that
    \begin{align*}%
        \begin{aligned}
            \sup_{\pi_{0:T}\in \Pi}\mathbb{E}^{\mathbb{P}({ \mu_{0:T}^*,\pi_{0:T},p_{0:T}^*})} 	\left[\sum_{t=0}^{T-1}r(s_t,a_t,s_{t+1}, \mu_t^*)\right] &= \mathbb{E}^{\mathbb{P}^*} 	\left[\sum_{t=0}^{T-1}r(s_t,a_t,s_{t+1}, \mu_t^*)\right]=V(\mu_{0:T}^*),
        \end{aligned}
    \end{align*}
    where $\mathbb{P}^*=\mathbb{P}({ \mu_{0:T}^*,\pi_{0:T}^{*},p_{0:T}^{*}})\in {\cal Q}( \mu_{0:T}^*,\pi^{*}_{0:T})$ (see Definition\;\ref{dfn:worst_msr}\;(i)).

    Moreover since  $J^{N}_1(\mu^o,(\overline \pi_{0:T}^{N|*,-1},\pi_{0:T}^{(N)}))=J_1^N(\mu^o,\overline{\pi}_{0:T,1}^{N|(N)})$ and  $\mathbb{P}({ \mu_{0:T}^*,\pi_{0:T}^{(N)},p_{0:T}^{*}})=\mathbb{P}^{*|(N)}$ (see Definition \ref{dfn:worst_msr}), we apply Proposition \ref{pro:MNE2} to have
    \begin{align}\label{eq:MNE_step2}
        \begin{aligned}
            \limsup_{N \to \infty} J^{N}_1(\mu^o,(\overline \pi_{0:T}^{N|*,-1},\pi_{0:T}^{(N)}))&%
            =\limsup_{N \to \infty}\mathbb{E}^{\mathbb{P}({ \mu_{0:T}^*,\pi_{0:T}^{(N)},p_{0:T}^{*}})} 	\left[\sum_{t=0}^{T-1}r(s_t,a_t,s_{t+1}, \mu_t^*)\right]\\
            &\leq \sup_{\pi_{0:T}\in \Pi}\mathbb{E}^{\mathbb{P}({ \mu_{0:T}^*,\pi_{0:T},p_{0:T}^*})} 	\left[\sum_{t=0}^{T-1}r(s_t,a_t,s_{t+1}, \mu_t^*)\right] \\
            &= V(\mu_{0:T}^*).
        \end{aligned}
    \end{align}

    Combining \eqref{eq:MNE_step1}, \eqref{eq:MNE_step2} and Remark \ref{rem:MNE2}, we can choose $N(\varepsilon)\in \mathbb{N}$ such that for all $N\geq N(\varepsilon)$%
    \begin{align*}
        \sup_{\pi_{0:T}\in \Pi } J^{N}_1(\mu^o,(\overline \pi_{0:T}^{N|*,-1},\pi_{0:T}))- \varepsilon &<
        J^{N}_1(\mu^o,(\overline \pi_{0:T}^{N|*,-1},\pi_{0:T}^{(N)}))-\frac{2\varepsilon}{3}\\
        &\leq V(\mu_{0:T}^*) - \frac{\varepsilon}{3}\\
        &\leq J^{N}_1(\mu^o,\overline{\pi}_{0:T}^{N|*}).
    \end{align*}
    This completes the proof.
\end{proof}

\appendix
\section{Supplementary proofs}\label{sec:apdx}

\subsection{Proofs of Results in Section~\ref{sec:len}} \label{sec:pro:len}
\begin{proof}[Proof of Lemma~\ref{lem:len:mfe}]
    Fix $t\in\{0,\dots,T-1\}$ and $(s,a,\mu)\in S\times A\times\mathcal P(S)$. Since $a,\mu(1)\in[0,1]$ and $c,\delta,\beta,\varrho\geq0$, \eqref{eq:cbdr} gives
        \[
        0\leq c+\delta a+\beta\mu(1)\leq c+\delta a+\beta\mu(1)+\varrho\leq c+\delta+\beta+\varrho\leq1.
        \]
        Thus, both endpoints in \eqref{eq:len:frakP} lie in $[0,1]$, and $\mathfrak P_t(s,a,\mu)$ is a nonempty subset of $\mathcal P(S)$.

    As $S = \{0, 1\}$ only contains two elements, we have \(
        d_{W_{1}}(\mu, \nu) = \lvert \mu(1) - \nu(1) \rvert\)
    for any $\mu, \nu \in \mathcal{P}(S)$. In particular, for $p, q \in [0, 1]$,
    \begin{equation} \label{eq:Bern:W1}
        d_{W_{1}}\big(\mathrm{Bern}(p), \mathrm{Bern}(q)\big) = \lvert p - q \rvert.
    \end{equation}
    The parameter interval in \eqref{eq:len:frakP} is non-empty and compact. The map
    \begin{equation} \label{eq:Bern:cont}
        [0, 1] \ni \theta \mapsto \mathrm{Bern}(\theta) \in \mathcal{P}(S)
    \end{equation}
    is continuous by \eqref{eq:Bern:W1}. Its image is therefore compact.

    To show convexity, let $\mathrm{Bern}(\theta_{1}), \mathrm{Bern}(\theta_{2}) \in \mathfrak{P}_{t}(s, a, \mu)$ and $\alpha \in [0, 1]$. Then
    \begin{equation*}
        \alpha \mathrm{Bern}(\theta_{1}) + (1 - \alpha) \mathrm{Bern}(\theta_{2}) = \mathrm{Bern}(\alpha \theta_{1} + (1 - \alpha) \theta_{2}).
    \end{equation*}
    The convex combination $\alpha \theta_{1} + (1 - \alpha) \theta_{2}$ remains in the parameter interval. Thus, $\mathfrak{P}_{t}(s, a, \mu)$ is convex. Because $S$ and $A$ are finite, it is sufficient to show continuity in $\mu$. Let
    \begin{equation*}
        \ell_{a}(\mu) := c + \delta a + \beta \mu(1).
    \end{equation*}
    The two ambiguity intervals for $\mu$ and another measure $\nu \in \mathcal{P}(S)$ have the same length $\varrho$, and their endpoints differ by
    \begin{equation*}
        \lvert \ell_{a}(\mu) - \ell_{a}(\nu) \rvert = \beta \lvert \mu(1) - \nu(1) \rvert = \beta d_{W_{1}}(\mu, \nu).
    \end{equation*}
    Using \eqref{eq:Bern:W1}, the Hausdorff distance in $(\mathcal{P}(S), d_{W_{1}})$ (see, e.g., \cite[Definition~3.70]{CharalambosKim2006infinite}) is therefore
    \begin{align*}
        &d_{H}^{W_{1}}\big(\mathfrak{P}_{t}(s, a, \mu), \mathfrak{P}_{t}(s, a, \nu)\big) \\
        &= \max\bigg\{\sup_{\theta \in [\ell_{a}(\mu), \ell_{a}(\mu) + \varrho]} \inf_{\vartheta \in [\ell_{a}(\nu), \ell_{a}(\nu) + \varrho]} d_{W_{1}}\big(\mathrm{Bern}(\theta), \mathrm{Bern}(\vartheta)\big), \\
        &\qquad\qquad \sup_{\vartheta \in [\ell_{a}(\nu), \ell_{a}(\nu) + \varrho]} \inf_{\theta \in [\ell_{a}(\mu), \ell_{a}(\mu) + \varrho]} d_{W_{1}}\big(\mathrm{Bern}(\theta), \mathrm{Bern}(\vartheta)\big)\bigg\} \\
        &= \max\bigg\{\sup_{\theta \in [\ell_{a}(\mu), \ell_{a}(\mu) + \varrho]} \inf_{\vartheta \in [\ell_{a}(\nu), \ell_{a}(\nu) + \varrho]} \lvert \theta - \vartheta \rvert, \sup_{\vartheta \in [\ell_{a}(\nu), \ell_{a}(\nu) + \varrho]} \inf_{\theta \in [\ell_{a}(\mu), \ell_{a}(\mu) + \varrho]} \lvert \theta - \vartheta \rvert\bigg\}.
    \end{align*}
    For any $\theta \in [\ell_{a}(\mu), \ell_{a}(\mu) + \varrho]$, we have
    \begin{equation*}
        \inf_{\vartheta \in [\ell_{a}(\nu), \ell_{a}(\nu) + \varrho]} \lvert \theta - \vartheta \rvert =
        \begin{cases}
            \ell_{a}(\nu) - \theta, &\theta \leq \ell_{a}(\nu), \\
            0, &\ell_{a}(\nu) < \theta \leq \ell_{a}(\nu) + \varrho, \\
            \theta - (\ell_{a}(\nu) + \varrho), &\ell_{a}(\nu) + \varrho < \theta,
        \end{cases}
    \end{equation*}
    as well as
    \begin{align*}
        \ell_{a}(\nu) - \theta &\leq \ell_{a}(\nu) - \ell_{a}(\mu) \leq \lvert \ell_{a}(\mu) - \ell_{a}(\nu) \rvert \leq \beta d_{W_{1}}(\mu, \nu), \\
        \theta - (\ell_{a}(\nu) + \varrho) &\leq \ell_{a}(\mu) + \varrho - \ell_{a}(\nu) - \varrho = \ell_{a}(\mu) - \ell_{a}(\nu) \leq \lvert \ell_{a}(\mu) - \ell_{a}(\nu) \rvert \leq \beta d_{W_{1}}(\mu, \nu).
    \end{align*}
    Hence,
    \begin{equation*}
        \sup_{\theta \in [\ell_{a}(\mu), \ell_{a}(\mu) + \varrho]} \inf_{\vartheta \in [\ell_{a}(\nu), \ell_{a}(\nu) + \varrho]} \lvert \theta - \vartheta \rvert \leq \beta d_{W_{1}}(\mu, \nu).
    \end{equation*}
    The same bound holds for the second supremum. Thus,
    \begin{equation*}
        d_{H}^{W_{1}}\big(\mathfrak{P}_{t}(s, a, \mu), \mathfrak{P}_{t}(s, a, \nu)\big) \leq \beta d_{W_{1}}(\mu, \nu).
    \end{equation*}
    The Hausdorff continuity of a compact-valued correspondence implies both upper- and lower-hemicontinuity (see, e.g., \cite[Theorem~17.15]{CharalambosKim2006infinite}). Hence, $\mathfrak{P}_{t}$ is continuous.

    Next, take an arbitrary
    \begin{equation*}
        \P = \mathrm{Bern}(\theta) \in \mathfrak{P}_{t}(s, a, \mu)
    \end{equation*}
    and set
    \begin{equation*}
        \widetilde{\theta} := \theta + \beta \big(\nu(1) - \mu(1)\big), \quad \widetilde{\P} := \mathrm{Bern}(\widetilde{\theta}).
    \end{equation*}
    Because both endpoints of the interval in \eqref{eq:len:frakP} undergo precisely this shift, when $\mu$ is replaced by $\nu$, we have $\widetilde{\P} \in \mathfrak{P}_{t}(s, a, \nu)$. Then
    \begin{equation*}
        d_{W_{1}}(\P, \widetilde{\P}) = \lvert \theta - \widetilde{\theta} \rvert = \beta \lvert \mu(1) - \nu(1) \rvert = \beta d_{W_{1}}(\mu, \nu).
    \end{equation*}
    Thus, we can choose $C_{\mathfrak{P}_{t}} := \beta > 0$ for every $t$. Consequently, $\mathfrak{P}_{t}$ satisfies Assumption~\ref{as:msr}~(ii).

    For the reward function, the bounds $s,\hat s,a\in\{0,1\}$ and $\mu(1)\in[0,1]$, together with the sign conditions on the coefficients, give $-\kappa\leq r(s,a,\hat s,\mu)\leq M+1+\lambda$. Hence, $|r(s,a,\hat s,\mu)|\leq C_r:=M+1+\lambda+\kappa$. Moreover,
    \[
    |r(s,a,\hat s,\mu)-r(s,a,\hat s,\nu)|\leq\lambda|\mu(1)-\nu(1)|=\lambda d_{W_1}(\mu,\nu).
    \]
    Thus, Assumption~\ref{as:msr}~(iii) holds with $L_r:=1+\lambda$. This proves part~(i).

    We next verify that the candidate in \eqref{eq:cand:p}--\eqref{eq:cand:mu} is a mean-field equilibrium in the sense of Definition~\ref{dfn:robust_mfe}. Fix the state-measure flow $\mu_{0:T}^*$ given in \eqref{eq:cand:mu}. At time $T-1$, the dynamic programming principle in \eqref{eq:DPP_maximin2}--\eqref{eq:V-1} implies the minimization problem $\int_S r(s,a,\hat s,\mu_{T-1}^*)\,\P(d\hat s)$ over $\P\in\mathfrak P_{T-1}(s,a,\mu_{T-1}^*)$. Since $\hat s$ enters the reward with coefficient $1$, this integral is increasing in the Bernoulli success probability. The minimum is therefore attained by $p_{T-1}^*(\,\cdot\,\mid s,a,\mu_{T-1}^*)$, whose success probability is the lower endpoint $c+\delta a+\beta\mu_{T-1}^*(1)$. Thus,
    \begin{equation}\label{eq:len:JT-1}
        \widehat J_{T-1}(s,a,\mu_{T-1}^*)=Ms+c+(\lambda+\beta)\mu_{T-1}^*(1)+(\delta-\kappa)a.
    \end{equation}
    By \eqref{eq:kd}, the gain from choosing action $1$ over action $0$ is $\delta-\kappa\geq0$. Hence, action $1$ is optimal at time $T-1$ and is uniquely optimal if $\kappa<\delta$. Substituting $a=1$ in \eqref{eq:len:JT-1} gives
    \begin{equation}\label{eq:len:VT-1}
        \widehat V_{T-1}(s,\mu_{T-1}^*)=Ms+C_{T-1},
    \end{equation}
    where $C_{T-1}:=c+(\lambda+\beta)\mu_{T-1}^*(1)+\delta-\kappa$ is independent of $s$.

    Suppose now that, for some $t\in\{0,\dots,T-2\}$, $\widehat V_{t+1}(s,\mu_{t+1:T}^*)=Ms+C_{t+1}$ for every $s\in S$, where $C_{t+1}$ is independent of $s$. For each $(s,a)\in S\times A$, the integrand in the minimization is
    \[
    r(s,a,\hat s,\mu_t^*)+\widehat V_{t+1}(\hat s,\mu_{t+1:T}^*)
    =Ms+\lambda\mu_t^*(1)-\kappa a+(1+M)\hat s+C_{t+1}.
    \]
    Since $1+M>0$, its expectation is strictly increasing in the Bernoulli success probability. The minimum is therefore attained by $p_t^*(\,\cdot\,\mid s,a,\mu_t^*)$, which has the smallest admissible success probability. Thus,
    \begin{equation}\label{eq:len:Jt}
        \widehat J_t(s,a,\mu_{t:T}^*)=Ms+\lambda\mu_t^*(1)-\kappa a+C_{t+1}+(1+M)\big(c+\delta a+\beta\mu_t^*(1)\big).
    \end{equation}
    Since $\kappa\leq\delta$ and $M,\delta>0$, the gain from choosing action $1$ over action $0$ is $(1+M)\delta-\kappa\geq M\delta>0$. Hence, action $1$ is uniquely optimal at time $t$.

    Substituting $a=1$ in \eqref{eq:len:Jt} gives $\widehat V_t(s,\mu_{t:T}^*)=Ms+C_t$, where $C_t:=\lambda\mu_t^*(1)-\kappa+C_{t+1}+(1+M)(c+\delta+\beta\mu_t^*(1))$ is independent of $s$.

    Together with the result at time $T-1$, this completes the backward induction. For the fixed flow $\mu_{0:T}^*$, the policy $\pi_{0:T}^*$ is optimal and $p_{0:T}^*$ is a corresponding worst-case transition kernel.

    Lastly, we need to show that $\mu_{0:T}^{*}$ evolves according to \eqref{eq:mfe:mu}. Under $\pi_{t}^{*}$, the action equals $1$ almost surely. Under the worst-case kernel $p_{t}^{*}$, we then have
    \begin{align*}
        &\int_{S\times A}p^*_{t}(s_{t+1} = 1 \mid s_{t}, a_{t}, {\mu}_{t}^*) \pi_{t}^*(da_{t} | s_{t}) \mu_{t}^*(ds_{t}) \\
        &=\int_{S\times A} \big(c + \delta a + \beta \mu_{t}^{*}(1)\big) \pi_{t}^*(da_{t} | s_{t}) \mu_{t}^*(ds_{t}) = c + \delta + \beta \mu_{t}^{*}(1) = \mu_{t+1}^{*}(s_{t+1} = 1),
    \end{align*}
    showing \eqref{eq:mfe:mu}. Hence, $(\mu_{0:T}^{*}, \pi_{0:T}^{*}, p_{0:T}^{*})$ is a mean-field equilibrium.
\end{proof}

\begin{proof}[Proof of Lemma~\ref{lem:len:mne}]
    For $\overline{s}^{N} = (s^{1}, \dots, s^{N}) \in S^{N}$, define
    \begin{equation*}
        m_{N}(\overline{s}^{N}) = \big(e^{N}(\overline{s}^{N})\big)(1) = \frac{1}{N} \sum_{j = 1}^{N} s^{j}.
    \end{equation*}
    For any $(\overline s_t^N,\overline a_t^N) \in S^{N} \times A^{N}$, the set $\mathfrak{P}_t^N(\overline s_t^N,\overline a_t^N)$ consists of product measures
    \begin{equation*}
        \P = \prod_{j = 1}^{N} p_{t}^{(\overline s_t^N,\overline a_t^N)}\big(\,\cdot \mid s^{j}, a^{j}, e^{N}(\overline{s}^{N})\big) = \prod_{j = 1}^{N} \mathrm{Bern}(\theta_{j})
    \end{equation*}
    for some kernel $p_{t}^{(\overline s_t^N,\overline a_t^N)} \in \mathcal{K}_{t}$ and
    \begin{equation} \label{eq:len:theta:j}
        \theta_{j} \in [c + \delta a^{j} + \beta m_{N}(\overline{s}^{N}), c + \delta a^{j} + \beta m_{N}(\overline{s}^{N}) + \varrho].
    \end{equation}
    In particular, $\theta_{j} = \theta_{k}$ if $(s^{j}, a^{j}) = (s^{k}, a^{k})$.

    Fix a deviating agent $i$ and an arbitrary Markov policy $\pi_{0:T}^{(N)} \in \Pi$. We then consider
    \begin{align} \label{eq:len:N:pi}
        \begin{aligned}
            &\overline{\pi}_{t,i}^{N|(N)}:S^N\ni \overline s_t^N\mapsto \overline \pi^{N|(N)}_{t,i}(d\overline a_t^N|\overline s_t^N):= \pi_t^{(N)}(da_t^i|s_t^i) \prod_{j=1,j\neq i}^N \pi_{t}^*(da_{t}^j|s_{t}^j) \\
            &\hspace{165pt}= \pi_t^{(N)}(da_t^i|s_t^i) \prod_{j=1,j\neq i}^N \delta_{1}(da_{t}^j).
        \end{aligned}
    \end{align}
    Recall the dynamic programming principle defined via \eqref{eq:DPP_min_Nash0} and \eqref{eq:DPP_min_Nash}, i.e., for $t = T-1$,
    \begin{align*}
        \widehat{J}_{T-1,i}^{N}(\overline s^N_{T-1},\overline a^N_{T-1};\overline\pi^N_{0:T}):=  \inf_{\P \in \mathfrak{P}_{T-1}^{N}(\overline s_{T-1}^{N}, \overline a_{T-1}^{N})} \int_{S^{N}} r\big(s_{T-1}^{i}, a_{T-1}^{i}, s^{i}_{T}, e^{N}(\overline s_{T-1}^{N})\big)  \P(d\overline s^{N}_{T}),
    \end{align*}
    and, for $t \leq T-2$,
    \begin{align*}
        \widehat{J}_{t,i}^{N} (\overline s^N_{t},\overline a^N_{t};\overline\pi^N_{0:T})
        &:=  \inf_{\P \in \mathfrak{P}_{t}^{N}(\overline s^N_{t},\overline a^N_{t})} \int_{S^{N}}\bigg( r\big(s_{t}^{i}, a_{t}^{i}, s^{i}_{t+1}, e^{N}(\overline s_t^{N})\big) \bigg.\\
        &\bigg. \quad\quad + \int_{A^N}\widehat{J}_{t+1,i}^{N} (\overline s^N_{t+1},\overline a^N_{t+1};\overline\pi^N_{0:T})\overline \pi_{t+1}^N(d\overline a_{t+1}^N|\overline s_{t+1}^N)\bigg) \P(d\overline s^{N}_{t+1}).
    \end{align*}
    As the deviating agent and policy are fixed, we simplify the notation and write
    \begin{align}\label{eq:len:N:Jt}
        \begin{aligned}
            J_{t}(\overline s^N,\overline a^N) &:= \widehat{J}_{t,i}^{N} (\overline s^N,\overline a^N;\overline\pi^{N \vert (N)}_{0:T, i}) \\
            &\hspace{2.78pt}=  \inf_{\P \in \mathfrak{P}_{t}^{N}(\overline s^N,\overline a^N)} \int_{S^{N}} \bigg(r\big(s^{i}, a^{i}, y^{i}, e^{N}(\overline s^{N})\big) + V_{t+1}(\overline{y}^{N})\bigg) \P(d\overline y^{N}),
        \end{aligned}
    \end{align}
    where $V_{T}(\overline s^N) := 0$ and, for $t \leq T-1$,
    \begin{equation}
        V_{t}(\overline s^N) := \int_{A^{N}} J_{t}(\overline{s}^{N}, \overline{a}^{N}) \overline{\pi}_{t, i}^{N \vert (N)}(d\overline a^{N} \mid \overline{s}^{N}).
    \end{equation}
    We now prove the following statements simultaneously by backward induction:
    \begin{enumerate}
        \item[(i)] $V_{t}$ is non-decreasing in every state coordinate.

        \item[(ii)] The kernel defined in~\eqref{eq:len:N:ker} attains the infimum in~\eqref{eq:len:N:Jt}.
    \end{enumerate}
    By \eqref{eq:len:r}, it holds that  for every $(s, a, y, \mu) \in S \times A \times S \times \mathcal{P}(S)$, \(
        -\kappa \leq r(s, a, y, \mu) \leq M + 1 + \lambda.  \)
    The length of the interval is
    \begin{equation} \label{eq:len:R:osc}
        R_{\mathrm{osc}} := M + 1 + \lambda + \kappa.
    \end{equation}
    Starting from date $t+1$, there are $T - t - 1$ rewards remaining. Every expected payoff, every infimum of expected payoffs, and every policy average therefore belongs to an interval of length at most $(T - t - 1) R_{\mathrm{osc}}$. In particular,
    \begin{equation*}
        \max_{\overline{y}^{N} \in S^{N}} V_{t+1}(\overline{y}^{N}) - \min_{\overline{y}^{N} \in S^{N}} V_{t+1}(\overline{y}^{N}) \leq (T - t - 1) R_{\mathrm{osc}}.
    \end{equation*}
    Define
    \begin{equation*}
        H_{t+1}(\overline{y}^{N}) := y^{i} + V_{t+1}(\overline{y}^{N})
    \end{equation*}
    and note that for every $\P \in \mathfrak{P}_{t}^{N}(\overline s^N,\overline a^N)$,
    \begin{align} \label{eq:len:rew:V:H}
        \begin{aligned}
            &\int_{S^{N}}\Big(r\big(s^{i}, a^{i}, y^{i}, e^{N}(\overline s_t^{N})\big) + V_{t+1}(\overline{y}^{N})\Big) \P(d\overline y^{N}) \\
            &= M s^{i} + \lambda m_{N}(\overline{s}^{N}) - \kappa a^{i} + \int_{S^{N}} H_{t+1}(\overline{y}^{N}) \P(d\overline y^{N}).
        \end{aligned}
    \end{align}
    The oscillation of $H_{t+1}$ when only the deviating agent changes is bounded by
    \begin{align} \label{eq:len:osc:H}
        \begin{aligned}
            &\mathrm{osc}_{i}(H_{t+1}) := \max_{\overline{y}^{N, -i} \in S^{N-1}} \bigg\lvert H_{t+1}(1, \overline{y}^{N, -i}) - H_{t+1}(0, \overline{y}^{N, -i}) \bigg\rvert \\
            &\hspace{49pt}\leq 1 + (T - t - 1) R_{\mathrm{osc}}.
        \end{aligned}
    \end{align}
    Assume now the induction hypothesis that $V_{t+1}$ is non-decreasing. Then $H_{t+1}$ is also coordinatewise non-decreasing. For fixed $(\overline{s}^{N}, \overline{a}^{N}) \in S^{N} \times A^{N}$, define
    \begin{eqnarray*}
        &&\theta_{j}^{*} := c + \delta a^{j} + \beta m_{N}(\overline{s}^{N}), \quad j = 1, \dots, N\\
        &&\mathbb{Q}^{*} := \prod_{j = 1}^{N} \mathrm{Bern}(\theta_{j}^{*}) = \prod_{j = 1}^{N} p_{t}^{*}\big(\,\cdot \mid s^{j}, a^{j}, e^{N}(\overline{s}^{N})\big).
    \end{eqnarray*}
    As $p_{t}^{*}(\,\cdot \mid s^{j}, a^{j}, e^{N}(\overline{s}^{N})) \in \mathfrak{P}_{t}(s^{j}, a^{j}, e^{N}(\overline{s}^{N}))$, it follows that $\mathbb{Q}^{*} \in \mathfrak{P}_{t}^{N}(\overline{s}^{N}, \overline{a}^{N})$. Now take an arbitrary measure
    \begin{equation*}
        \mathbb{Q} = \prod_{j = 1}^{N} \mathrm{Bern}(\theta_{j}) \in \mathfrak{P}_{t}^{N}(\overline{s}^{N}, \overline{a}^{N}).
    \end{equation*}
    By \eqref{eq:len:theta:j}, $\theta_{j} \geq \theta_{j}^{*}$ for all $j \in \{1, \dots, N\}$. Let $Z^{1}, \dots, Z^{N}$ be independent uniform random variables on $[0, 1]$, and define
    \begin{equation} \label{eq:len:Y:j}
        Y^{*, j} := \mathbf{1}_{\{Z^{j} \leq \theta_{j}^{*}\}} \quad \text{and} \quad Y^{j} := \mathbf{1}_{\{Z^{j} \leq \theta_{j}\}}.
    \end{equation}
    Then $\overline{Y}^{*, N} \sim \mathbb{Q}^{*}$ and $\overline{Y}^{N} \sim \mathbb{Q}$, and, for every $j \in \{1, \dots, N\}$, $Y_{j}^{*} \leq Y_{j}$ almost surely. Coordinatewise monotonicity of $H_{t+1}$ yields $H_{t+1}(\overline{Y}^{*, N}) \leq H_{t+1}(\overline{Y}^{N})$ almost surely. Hence, by \eqref{eq:len:rew:V:H},
    \begin{align*}
        &\int_{S^{N}} \bigg(r\big(s^{i}, a^{i}, y^{i}, e^{N}(\overline s_t^{N})\big) + V_{t+1}(\overline{y}^{N})\bigg) \mathbb{Q}^{*}(d\overline y^{N}) \\
        &\quad = M s^{i} + \lambda m_{N}(\overline{s}^{N}) - \kappa a^{i}+ \int_{S^{N}} H_{t+1}(\overline{y}^{N}) \mathbb{Q}^{*}(d\overline y^{N}) \\
        &\quad \leq M s^{i} + \lambda m_{N}(\overline{s}^{N}) - \kappa a^{i} + \int_{S^{N}} H_{t+1}(\overline{y}^{N}) \mathbb{Q}(d\overline y^{N}) \\
        &\quad = \int_{S^{N}} \bigg(r\big(s^{i}, a^{i}, y^{i}, e^{N}(\overline s_t^{N})\big) + V_{t+1}(\overline{y}^{N})\bigg) \mathbb{Q}(d\overline y^{N}).
    \end{align*}
    As $\mathbb{Q} \in \mathfrak{P}_{t}^{N}(\overline{s}^{N}, \overline{a}^{N})$ was arbitrary, $\mathbb{Q}^{*}$ attains the infimum in \eqref{eq:len:N:Jt}. This proves induction claim~(ii) at time $t$, conditional on induction claim~(i) at time $t+1$.

    Next, we thus prove that $V_{t}$ is non-decreasing in every state coordinate. First, consider $j \neq i$ and compare two current configurations that differ only by replacing $s^{j} = 0$ with $s^{j} = 1$. Under \eqref{eq:len:N:pi}, agent $j$ chooses action $1$ in both configurations. The deviating agent's state and action distribution are unchanged. Thus,
    \begin{equation} \label{eq:pi:chg:j}
        \overline{\pi}^{N \vert (N)}_{t, i}\big(\,\cdot \mid (0, \overline{s}^{N, -j})\big) = \overline{\pi}^{N \vert (N)}_{t, i}\big(\,\cdot \mid (1, \overline{s}^{N, -j})\big).
    \end{equation}
    The emperical mean rises by $1/N$, namely,
    \begin{equation} \label{eq:emp:chg:j}
        m_{N}(1, \overline{s}^{N, -j}) - m_{N}(0, \overline{s}^{N, -j}) = \frac{1}{N}.
    \end{equation}
    Therefore, the current reward increases by $\lambda / N$ and every lower-endpoint transition probability rises by $\beta / N$. The action costs and all other current reward terms remain unchanged. For both state-configurations $(0, \overline{s}^{N, -j})$, $(1, \overline{s}^{N, -j})$, and any action configuration $\overline{a}^{N}$, we obtain corresponding worst-case measures
    \begin{eqnarray*}
        &&\overline{p}_{t, i}^{N \vert (N)}\big(\,\cdot \mid (0, \overline{s}^{N, -j}), \overline{a}^{N}\big) \in \mathfrak{P}_{t}^{N}((0, \overline{s}^{N, -j}), \overline{a}^{N}),\\
        &&\overline{p}_{t, i}^{N \vert (N)}\big(\,\cdot \mid (1, \overline{s}^{N, -j}), \overline{a}^{N}\big) \in \mathfrak{P}_{t}^{N}((1, \overline{s}^{N, -j}), \overline{a}^{N}).
    \end{eqnarray*}
    Since we already know that the lower-endpoint transition probabilities of all agents increase by $\beta / N$ from $\overline{p}_{t, i}^{N \vert (N)}(\,\cdot \mid (0, \overline{s}^{N, -j}), \overline{a}^{N})$ to $\overline{p}_{t, i}^{N \vert (N)}(\,\cdot \mid (1, \overline{s}^{N, -j}), \overline{a}^{N})$, we can use the same idea as in \eqref{eq:len:Y:j}, to obtain that
    \begin{equation} \label{eq:Ht+1:chg:j}
        \int_{S^{N}} H_{t+1}(\overline{y}^{N}) \overline{p}_{t, i}^{N \vert (N)}\big(d\overline y^{N} \mid (0, \overline{s}^{N, -j}), \overline{a}^{N}\big) \leq \int_{S^{N}} H_{t+1}(\overline{y}^{N}) \overline{p}_{t, i}^{N \vert (N)}\big(d\overline y^{N} \mid (1, \overline{s}^{N, -j}), \overline{a}^{N}\big).
    \end{equation}
    Hence, by \eqref{eq:len:rew:V:H}, \eqref{eq:pi:chg:j}, \eqref{eq:emp:chg:j}, and~\eqref{eq:Ht+1:chg:j},
    \begin{align*}
        &V_{t}(0, \overline{s}^{N, -j}) = \int_{A^{N}} J_{t}\big((0, \overline{s}^{N, -j}), \overline{a}^{N}\big) \overline{\pi}_{t, i}^{N \vert (N)}\big(d\overline a^{N} \mid (0, \overline{s}^{N, -j})\big) \\
        &= M s^{i} + \lambda m_{N}\big((0, \overline{s}^{N, -j})\big) \\
        &\hspace{10pt}+ \int_{A^{N}} \bigg((-\kappa a^{i}) + \int_{S^{N}} H_{t+1}(\overline{y}^{N}) \overline{p}_{t, i}^{N \vert (N)}\big(d\overline y^{N} \mid (0, \overline{s}^{N, -j}), \overline{a}^{N}\big)\bigg) \overline{\pi}_{t, i}^{N \vert (N)}\big(d\overline a^{N} \mid (0, \overline{s}^{N, -j})\big) \\
        &\leq M s^{i} + \lambda m_{N}\big((1, \overline{s}^{N, -j})\big) \\
        &\hspace{10pt}+ \int_{A^{N}} \bigg((-\kappa a^{i}) + \int_{S^{N}} H_{t+1}(\overline{y}^{N}) \overline{p}_{t, i}^{N \vert (N)}\big(d\overline y^{N} \mid (1, \overline{s}^{N, -j}), \overline{a}^{N}\big)\bigg) \overline{\pi}_{t, i}^{N \vert (N)}\big(d\overline a^{N} \mid (1, \overline{s}^{N, -j})\big) \\
        &= \int_{A^{N}} J_{t}\big((1, \overline{s}^{N, -j}), \overline{a}^{N}\big) \overline{\pi}_{t, i}^{N \vert (N)}\big(d\overline a^{N} \mid (1, \overline{s}^{N, -j})\big) = V_{t}(1, \overline{s}^{N, -j}).
    \end{align*}
    Thus, $V_{t}$ is non-decreasing in every state coordinate $j \neq i$.

    It remains to prove that $V_{t}$ is non-decreasing in state coordinate $i$. Again, we compare two current configurations that differ only by replacing $s^{i} = 0$ with $s^{i} = 1$. In contrast to before, the action laws of agent $i$ may be completely different depending on the state $s^{i}$. Choose any
    \begin{equation*}
        A_{0}^{i} \sim \pi_{t}^{(N)}(\,\cdot \mid 0) \quad \text{and} \quad A_{1}^{i} \sim \pi_{t}^{(N)}(\,\cdot \mid 1).
    \end{equation*}
    The estimates below hold for any coupling of the random variables.

    First, consider the current reward. Raising state $s^{i}$ from $0$ to $1$ increases the reward by $M$. The empirical mean again increases by $1 / N$ and therefore increases the reward by $\lambda / N$. The worst-possible change in the action cost is $-\kappa$, because $A_{1}^{i} - A_{0}^{i} \leq 1$. Hence,
    \begin{equation} \label{eq:len:rew:ineq}
        M s^{i} \big\vert_{s^{i} = 1} - M s^{i} \big\vert_{s^{i} = 0} + \lambda \Big(m_{N}\big((1, \overline{s}^{N, -i})\big) - m_{N}\big((0, \overline{s}^{N, -i})\big)\Big) - \kappa (A_{1}^{i} - A_{0}^{i}) \geq M + \frac{\lambda}{N} - \kappa.
    \end{equation}

    Next, consider the lower-endpoint transition probabilities. For any agent $j \neq i$, raising state $s^{i}$ from $0$ to $1$ again increases the lower-endpoint transition probability by $\beta / N$. For agent $i$, the change is
    \begin{equation*}
        \delta (A_{1}^{i} - A_{0}^{i}) + \frac{\beta}{N} \geq - \delta.
    \end{equation*}
    Thus, for any agent $j \neq i$, the transition probability improves and only for agent $i$ it can potentially decrease. The decrease in agent $i$'s Bernoulli success probability is at most $\delta$.

    We now define and couple random variables $Y_{0}^{j}$ and $Y_{1}^{j}$ for all agents $j \neq i$ as in \eqref{eq:len:Y:j}, i.e., let $Z^{j}$ be an independent uniform random variable on $[0, 1]$ and define
        \begin{equation*}
            Y_{0}^{j} := \mathbf{1}_{\{Z^{j} \leq \theta^{0}_{j}\}} \quad \text{and} \quad Y_{1}^{j} := \mathbf{1}_{\{Z^{j} \leq \theta^{1}_{j}\}},
        \end{equation*}
        where
        \begin{eqnarray*}
            &&\theta^{0}_{j} := c + \delta a^{j} + \beta m_{N}\big((0, \overline{s}^{N, -i})\big) = c + \delta + \frac{\beta}{N} \sum_{j \neq i} s^{j},\\
            &&\theta^{1}_{j} := c + \delta a^{j} + \beta m_{N}\big((1, \overline{s}^{N, -i})\big) = c + \delta + \frac{\beta}{N} + \frac{\beta}{N} \sum_{j \neq i} s^{j}.
        \end{eqnarray*}
        Conditional on the actions $A_{0}^{i}, A_{1}^{i}$, we couple the random variables $Y_{0}^{i}$ and $Y_{1}^{i}$ analogously via another common uniform random variable $Z^{i}$, i.e.,
        \begin{equation*}
            Y_{0}^{i} := \mathbf{1}_{\{Z^{i} \leq \theta^{0}_{i}\}} \quad \text{and} \quad Y_{1}^{i} := \mathbf{1}_{\{Z^{i} \leq \theta^{1}_{i}\}},
        \end{equation*}
        where
        \begin{eqnarray*}
            &&\theta^{0}_{i} := c + \delta A_{0}^{i} + \beta m_{N}\big((0, \overline{s}^{N, -i})\big) = c + \delta A_{0}^{i} + \frac{\beta}{N} \sum_{j \neq i} s^{j},\\
            &&\theta^{1}_{i} := c + \delta A_{1}^{i} + \beta m_{N}\big((1, \overline{s}^{N, -i})\big) = c + \delta A_{1}^{i} + \frac{\beta}{N} + \frac{\beta}{N} \sum_{j \neq i} s^{j}.
        \end{eqnarray*}

        Let actions $A_{0}^{i}, A_{1}^{i}$ and states $\overline{Y}_{0}^{N, -i}, \overline{Y}_{1}^{N, -i}$ be given. By definition of their coupling, it holds that for all $j \neq i$, $Y_{0}^{j} \leq Y_{1}^{j}$ almost surely and thus, if also $Y_{0}^{i} \leq Y_{1}^{i}$, then
        \begin{equation*}
            H_{t+1}(\overline{Y}_{1}^{N}) - H_{t+1}(\overline{Y}_{0}^{N}) \geq 0
        \end{equation*}
        almost surely, where we use that $V_{t+1}$ and thus $H_{t+1}$ is coordinatewise non-decreasing by the induction hypothesis.

        If $Y_{0}^{i} > Y_{1}^{i}$, i.e., $Y_{0}^{i} = 1$ and $Y_{1}^{i} = 0$, set \(
            \overline{\underline{Y}}_{0}^{i} := (0, \overline{Y}_{0}^{N, -i}).
        \)
        Then, for all $j \in \{1, \dots, N\}$, $Y_{1}^{j} \geq \underline{Y}_{0}^{j}$ almost surely and therefore, $H_{t+1}(\overline{Y}_{1}^{N}) \geq H_{t+1}(\overline{\underline{Y}}_{0}^{N})$ almost surely. Hence,
        \begin{align*}
            H_{t+1}(\overline{Y}_{0}^{N}) - H_{t+1}(\overline{Y}_{1}^{N}) &\leq H_{t+1}(\overline{Y}_{0}^{N}) - H_{t+1}(\overline{\underline{Y}}_{0}^{N}) \\
            &= H_{t+1}\big((1, \overline{Y}_{0}^{N, -i})\big) - H_{t+1}\big((0, \overline{Y}_{0}^{N, -i})\big) \leq \mathrm{osc}_{i}(H_{t+1}).
        \end{align*}
        Analogously, we have \(
            H_{t+1}(\overline{Y}_{1}^{N}) - H_{t+1}(\overline{Y}_{0}^{N}) \geq -\mathrm{osc}_{i}(H_{t+1}).\)

        By definition of $Y_{0}^{i}$ and $Y_{1}^{i}$, we have
        \begin{align*}
            \mathbb{P}\Big(Y_{0}^{i} > Y_{1}^{i} \Bigm\vert A_{0}^{i}, A_{1}^{i}, \overline{Y}_{0}^{N, -i}, \overline{Y}_{1}^{N, -i}\Big) &= \mathbb{P}\Big(Y_{0}^{i} = 1, Y_{1}^{i} = 0 \Bigm\vert A_{0}^{i}, A_{1}^{i}, \overline{Y}_{0}^{N, -i}, \overline{Y}_{1}^{N, -i}\Big) \\
            &= \mathbb{P}\Big(\theta^{1}_{i} < Z^{i} \leq \theta^{0}_{i} \Bigm\vert A_{0}^{i}, A_{1}^{i}, \overline{Y}_{0}^{N, -i}, \overline{Y}_{1}^{N, -i}\Big) \\
            &= \max\{\theta^{0}_{i} - \theta^{1}_{i}, 0\} = \max\{\delta (A_{0}^{i} - A_{1}^{i}) - \beta/N, 0\}\leq \delta.
        \end{align*}
    Therefore,
    \begin{align*}
        &\mathbb{E}\bigg[H_{t+1}(\overline{Y}_{1}^{N}) - H_{t+1}(\overline{Y}_{0}^{N})\bigg] = \mathbb{E}\bigg[\mathbb{E}\Big[H_{t+1}(\overline{Y}_{1}^{N}) - H_{t+1}(\overline{Y}_{0}^{N}) \biggm\vert A_{0}^{i}, A_{1}^{i}, \overline{Y}_{0}^{N, -i}, \overline{Y}_{1}^{N, -i}\Big]\bigg] \\
        &\geq \mathbb{E}\bigg[\mathbb{E}\Big[\big(H_{t+1}(\overline{Y}_{1}^{N}) - H_{t+1}(\overline{Y}_{0}^{N})\big) \mathbf{1}_{\{Y_{0}^{i} > Y_{1}^{i}\}} \Bigm\vert A_{0}^{i}, A_{1}^{i}, \overline{Y}_{0}^{N, -i}, \overline{Y}_{1}^{N, -i}\Big]\bigg] \\
        &\geq -\mathrm{osc}_{i}(H_{t+1}) \mathbb{E}\bigg[\mathbb{P}\Big(Y_{0}^{i} > Y_{1}^{i} \Bigm\vert A_{0}^{i}, A_{1}^{i}, \overline{Y}_{0}^{N, -i}, \overline{Y}_{1}^{N, -i}\Big)\bigg] \\
        &\geq -\delta \mathrm{osc}_{i}(H_{t+1}).
    \end{align*}
    Combining this with \eqref{eq:len:R:osc}--\eqref{eq:len:osc:H} and \eqref{eq:len:rew:ineq}, we obtain that
    \begin{align*}
        &V_{t}\big((1, \overline{s}^{N, -i})\big) - V_{t}\big((0, \overline{s}^{N, -i})\big) \\
        &= \mathbb{E}\bigg[r\Big(1, A_{1}^{i}, Y_{1}^{i}, e^{N}\big((1, \overline{s}^{N, -i})\big)\Big) - r\Big(0, A_{0}^{i}, Y_{0}^{i}, e^{N}\big((0, \overline{s}^{N, -i})\big)\Big) + V_{t+1}(\overline{Y}_{1}^{N}) - V_{t+1}(\overline{Y}_{0}^{N})\bigg] \\
        &\geq M + \frac{\lambda}{N} - \kappa + \mathbb{E}\bigg[H_{t+1}(\overline{Y}_{1}^{N}) - H_{t+1}(\overline{Y}_{0}^{N})\bigg] \\
        &\geq M + \frac{\lambda}{N} - \kappa - \delta \mathrm{osc}_{i}(H_{t+1})\\
        & \geq M + \frac{\lambda}{N} - \kappa - \delta (1 + (T - t - 1) R_{\mathrm{osc}}) \\
        &\geq M - \kappa - \delta \big(1 + (T - 1) (M + 1 + \lambda + \kappa)\big) > 0,
    \end{align*}
    where the last inequality follows from \eqref{eq:M:ineq}. Thus, $V_{t}$ is strictly increasing in the $i$-th state coordinate, completing the induction step.

    At time $T$, $V_{T} \equiv 0$ and is thus coordinatewise non-decreasing. Hence, we are able to apply the above backward induction for $t = T-1, \dots, 0$. Consequently, the product kernel defined in \eqref{eq:len:N:ker} is a valid minimizer of \eqref{eq:len:N:Jt} for every $N$, deviating agent $i$, Markov policy $\pi_{0:T}^{(N)}$, date $t$, and state-action configuration $(\overline{s}^{N}, \overline{a}^{N})$. This concludes the proof of part~(i).

    Next, we will show that Assumption~\ref{as:weak_conv_ptb} holds. By \eqref{eq:cand:p} and \eqref{eq:Bern:W1}, it follows that for any $(s, a) \in S \times A$ and $\mu, \nu \in \mathcal{P}(S)$,
    \begin{align*}
        d_{W_{1}}\big(p_{t}^{*}(\,\cdot \mid s, a, \mu), p_{t}^{*}(\,\cdot \mid s, a, \nu)\big) &= d_{W_{1}}\Big(\mathrm{Bern}\big(c + \delta a + \beta \mu(1)\big), \mathrm{Bern}\big(c + \delta a + \beta \nu(1)\big)\Big) \\
        &= \beta \lvert \mu(1) - \nu(1) \rvert = \beta d_{W_{1}}(\mu, \nu).
    \end{align*}
    Thus, $p_{t}^{*}$ is continuous w.r.t.\ $W_{1}$, verifying Assumption~\ref{as:weak_conv_ptb}~(i). In part~(i), we have already shown that the worst-case joint kernel can be chosen such that
    \begin{equation*}
        \overline{p}_{t, i}^{N \vert (N)}(d \overline{s}^{N}_{t+1}|\overline s_t^N,\overline a_t^N) = \prod_{j = 1}^{N} \overline{q}_{t, i, j}^{N \vert (N)}(d s_{t+1}^{j} \mid \overline s_t^N,\overline a_t^N) = \prod_{j = 1}^{N} p_{t}^{*}\big(d s_{t+1}^{j} \mid s_t^j, a_t^j, e^{N}(\overline{s}_{t}^{N})\big).
    \end{equation*}
    Consequently,
    \begin{equation*}
        \max_{0 \leq t \leq T-1} \max_{(\overline{s}^{N}, \overline{a}^{N}) \in S^{N} \times A^{N}} \max_{1 \leq j \leq N} d_{W_{1}}\Big(\overline{q}^{N \vert (N)}_{t, i, j}(\,\cdot \mid \overline{s}^{N}, \overline{a}^{N}), p_{t}^{*}\big(\,\cdot \mid s^{j}, a^{j}, e^{N}(\overline{s}^{N})\big)\Big) = 0.
    \end{equation*}
    Thus, Assumption~\ref{as:weak_conv_ptb}~(ii) holds with $\Delta_{N} = 0$, $N \in \N$. Hence, by Theorem~\ref{thm:MNE0}, for any $\varepsilon > 0$ there exists $N(\varepsilon) \in \N$ such that for all $N \geq N(\varepsilon)$, the policy $\overline{\pi}_{0:T}^{N \vert *}$ defined for $t = 0, \dots, T-1$ by
    \begin{align*}
        \overline{\pi}_{t}^{N|*}:S^N\ni \overline s_t^N\mapsto \overline \pi^{N|*}_{t}(d\overline a_t^N|\overline s_t^N):=\prod_{j=1}^N \pi_{t}^*(da_{t}^j|s_{t}^j)
    \end{align*}
    forms an $\varepsilon$-Markov-Nash equilibrium.
\end{proof}

\subsection{Proof of results in Section \ref{sec:numeric}}\label{sec:apdx:numeric}
\color{black}
\begin{proof}[Proof of Lemma~\ref{lem:exm:cro}]
    Fix arbitrary $\lambda\geq 0$ and $c>0$. We first claim that $\mathfrak{P}_{0:T}^\lambda$ satisfies Assumption\;\ref{as:msr}\;(ii). Let $t \in \{0, \dots, T-1\}$, and let $s_t \in S$, $a_t \in A$, and $\mu_t, \tilde{\mu}_t \in \mathcal{P}(S)$ be arbitrarily chosen. Since the reference kernel $p^{o}$ does not depend on the argument $\mu$ and hence, $\mathfrak{P}^{\lambda}_t(s_t, a_t, \mu_t) = \mathfrak{P}^{\lambda}_t(s_t, a_t, \tilde \mu_t)$. Furthermore, as $\mathfrak{P}^{\lambda}_t(s_t, a_t, \mu_t)$ is a 1-Wasserstein ball around $p^{o}(\cdot \mid s_t, a_t, \mu_t)$, it is clearly non-empty, convex-valued, compact-valued.

    Moreover, since $\mathfrak{P}_t^{\lambda}(s_t, a_t, {\mu}_t)= \mathfrak{P}_t^{\lambda}(s_t, a_t, \widetilde{\mu}_t)$, for any $\mathbb{P}\in \mathfrak{P}_t^{\lambda}(s_t, a_t, {\mu}_t)$, we choose the same one $\widetilde{\P} := \P\in \mathfrak{P}_t^{\lambda}(s_t, a_t, \widetilde{\mu}_t)$ and let $L_{\mathfrak{P}_{t}^{\lambda}} := 1$ to get
    \begin{align*}
        0=d_{W_{1}}(\P, \widetilde{\P})  \leq L_{\mathfrak{P}_{t}^{\lambda}} d_{W_{1}}(\mu, \widetilde{\mu}).
    \end{align*}

    It remains to show that $\mathfrak{P}^{\lambda}_t$ is continuous (i.e., upper- and lower-hemicontinuous). To that end, consider an arbitrary sequence\footnote{We denote by $\operatorname{Gr}(\mathfrak{P}_t^{\lambda})$ the graph of $\mathfrak{P}_t^{\lambda}$.}
    \(
    ((s_{n}, a_{n}, \mu_{n}), \P_{n})_{n \in \N} \subseteq \operatorname{Gr}(\mathfrak{P}_t^{\lambda})
    \)
    such that $(s_{n}, a_{n}) \to (s, a)$ and $  \mu_{n} \rightharpoonup \mu$ as $n \to \infty$.

    Since $S$ and $A$ are finite, there exists $N \in \N$ such that for every $n \geq N$ it holds that $(s_{n}, a_{n}, \mu_{n}) = (s, a, \mu_{n})$. Hence, $\P_{n} \in \mathfrak{P}_t^{\lambda}(s, a, \mu_{n}) = \mathfrak{P}_t^{\lambda}(s, a, \mu)$ for every $n \geq N$. Moreover, since $\mathfrak{P}_t^{\lambda}(s, a, \mu)$ is compact, there exists a subsequence $\left(\P_{n_{k}}\right)_{k \in \N}\subseteq \left(\P_{n}\right)_{n \in \N}$ with $\P_{n_{k}} \rightharpoonup \P \in \mathfrak{P}_t^{\lambda}(s, a, \mu)$ as $k \to \infty$. Thus, by \cite[Theorem 17.20]{CharalambosKim2006infinite}, $\mathfrak{P}_t^{\lambda}$ is upper-hemicontinuous.

    Again, consider an arbitrary sequence $((s_{n}, a_{n}, \mu_{n}))_{n \in \N} \subset S \times A \times \mathcal{P}(S)$ such that $(s_{n}, a_{n}) \to (s, a)$ and $\mu_n\rightharpoonup \mu$ as $n \to \infty$ and let $\P \in \mathfrak{P}_t^{\lambda}(s, a, \mu)$. As before, there exists $N \in \N$ such that for every $n \geq N$ it holds that $(s_{n}, a_{n}, \mu_{n}) = (s, a, \mu_{n})$. Define a sequence $(\P_{n})_{n \in \N} \subseteq \mathcal{P}(S)$ by setting
    \begin{align*}
        \P_{n} :=
        \begin{cases}
            p^{o}(\cdot \mid s_{n}, a_{n}, \mu_{n}) \quad &\text{if } n < N, \\
            \P &\text{else.}
        \end{cases}
    \end{align*}
    Then, $\P_{n} \in \mathfrak{P}_t^{\lambda}(s_{n}, a_{n}, \mu_{n})$ for all $n \in \N$ and $\P_{n} \rightharpoonup \P$ as $n \to \infty$. Hence, by \cite[Theorem 17.21]{CharalambosKim2006infinite}, $\mathfrak{P}_t^{\lambda}$ is lower-hemicontinuous. Hence $\mathfrak{P}_{0:T}^\lambda$ satisfies Assumption \ref{as:msr}\;(ii), as claimed.

    We now claim that $r$ given in Definition \ref{dfn:exm:cro}\;(ii) satisfies Assumption\;\ref{as:msr}\;(iii).

    Since $|\hat s|\leq4$ and $|a|\leq1$ for every $(\hat{s},a)\in S\times A$ (noting that $S = \{0, 1, \dots, 4\}$ and $A= \{-1, 0, 1\}$; Definition \ref{dfn:exm:cro}), we define a constant $C_{r}:=\frac{17}{4} + \max\big\{-\log(c), \log(1 + c)\big\}>0$ so that for every $s, \hat s \in S$, $a \in A$, and $\mu \in \mathcal{P}(S)$,
    \begin{align*}
        \lvert r(s, a, \hat s, \mu) \rvert &\leq \left\lvert 1 - \frac{1}{2} \lvert \hat s - 2 \rvert\right\rvert + \frac{\lvert a \rvert}{4} + \lvert \log(\mu(\hat s) + c) \rvert \\
        &\leq 1 + \frac{1}{2} (\lvert \hat s \rvert + 2) + \frac{1}{4} + \max\big\{-\log(c), \log(1 + c)\big\}\leq C_r.
    \end{align*}
    Moreover, by the mean value theorem, for every $s, \hat{s} \in S$, $a \in A$, and $\mu, \hat{\mu} \in \mathcal{P}(S)$,
    \begin{align*}
        &\lvert r(s, a, \hat s, \mu) - r(s, a,  \hat s, \hat {\mu}) \rvert
        = \left\lvert \log(\mu(\hat s) + c) - \log(\hat{\mu}(\hat s) + c) \right\rvert \leq \frac{1}{c} \lvert \mu(\hat{s}) - \hat{\mu}(\hat{s}) \rvert \leq \frac{1}{c} d_{W_{1}}(\mu, \hat{\mu}).
    \end{align*}
    Hence, we can choose $L_{r} := \frac{1}{c}$ such that $r$ satisfies Assumption\;\ref{as:msr}\;(iii), as claimed.
\end{proof}

\begin{proof}[Proof of Lemma~\ref{lem:exm:tra}]
    Fix $\lambda\geq 0$, $\beta, \varepsilon_{0}, \varepsilon_{1} \in [0, 1]$, and $R, C_{1}, C_{2} > 0$. We first claim that $\mathfrak{P}_{0:T}^\lambda$ satisfies Assumption\;\ref{as:msr}\;(ii). Let $t \in \{0, \dots, T-1\}$, and let $s_t \in S$, $a_t \in A$, and $\mu_t, \tilde{\mu}_t \in \mathcal{P}(S)$ be arbitrarily chosen. As $\mathfrak{P}^{\lambda}_t(s_t, a_t, \mu_t)$ is a 1-Wasserstein ball around $p^{o}(\cdot \,|\, s_t, a_t, \mu_t)$, it is clearly non-empty, convex-valued, and compact-valued.

    Since $S =  \{0, 1\}$, we have $d_{W_{1}}(\mu_{t}, \tilde{\mu}_{t}) = \lvert \mu_{t}(1) - \tilde{\mu}_{t}(1) \rvert$. Hence, we can identify $\mathfrak{P}_{t}^{\lambda}(s_{t}, a_{t}, \mu_{t})$ with the interval $[l(\mu_{t}), u(\mu_{t})] \subseteq [0, 1]$, where for any $\mu \in \mathcal{P}(S)$, $l(\mu) := \max(0, p(\mu) - \lambda)$, $u(\mu) := \min(1, p(\mu) + \lambda)$, and $p(\mu) := p^{o}(1 \,|\, s_{t}, a_{t}, \mu)$.

    Let $\P \in \mathfrak{P}_{t}^{\lambda}(s_{t}, a_{t}, \mu_{t})$. It follows that $\P(1) \in [l(\mu_{t}), u(\mu_{t})]$ and thus $\P(1) = p(\mu_{t}) + c$ for some $c \in [-\lambda, \lambda]$. Now take $\tilde{\P} \in \mathcal{P}(S)$ such that $\tilde{\P}(1) = \min(\max(l(\tilde{\mu}_{t}), p(\tilde{\mu}_{t}) + c), u(\tilde{\mu}_{t}))$. Then $\tilde{\P} \in \mathfrak{P}_{t}^{\lambda}(s_{t}, a_{t}, \tilde{\mu}_{t})$ and
    \begin{align*}
        d_{W_{1}}(\P, \tilde{\P}) &= \lvert \P(1) - \tilde{\P}(1) \rvert = \big\lvert p(\mu_{t}) + c - \min\big(\max(l(\tilde{\mu}_{t}), p(\tilde{\mu}_{t}) + c), u(\tilde{\mu}_{t})\big) \big\rvert \\
        &\leq \big\lvert p(\mu_{t}) + c - \min\big(p(\tilde{\mu}_{t}) + c, u(\tilde{\mu}_{t})\big) \big\rvert \leq \lvert p(\tilde{\mu}_{t}) + c - p(\mu_{t}) - c \rvert \\
        &= \lvert p^{o}(1 \,|\, s_{t}, a_{t}, \mu_{t}) - p^{o}(1 \,|\, s_{t}, a_{t}, \tilde{\mu}_{t}) \rvert.
    \end{align*}
    Clearly, if $(s_{t}, a_{t}) \in \{(0, -1), (0, 0), (1, 0), (1, 1)\}$, then
    \begin{equation*}
        d_{W_{1}}(\P, \tilde{\P}) \leq \lvert p^{o}(1 \,|\, s_{t}, a_{t}, \mu_{t}) - p^{o}(1 \,|\, s_{t}, a_{t}, \tilde{\mu}_{t}) \rvert = 0 \leq \beta d_{W_{1}}(\mu_{t}, \tilde{\mu}_{t}).
    \end{equation*}
    If $(s_{t}, a_{t}) = (0, 1)$, then
    \begin{equation*}
        d_{W_{1}}(\P, \tilde{\P}) \leq \lvert p^{o}(1 \,|\, s_{t}, a_{t}, \mu_{t}) - p^{o}(1 \,|\, s_{t}, a_{t}, \tilde{\mu}_{t}) \rvert = \lvert 1 - \beta \mu_{t}(1) - 1 + \beta \tilde{\mu}_{t}(1)\rvert = \beta d_{W_{1}}(\mu_{t}, \tilde{\mu}_{t}).
    \end{equation*}
    Analogously, we get that $d_{W_{1}}(\P, \tilde{\P}) \leq \beta d_{W_{1}}(\mu_{t}, \tilde{\mu}_{t})$ for $(s_{t}, a_{t}) = (1, -1)$. Thus, choose $L_{\mathfrak{P}_{t}^{\lambda}} > \beta$.

    It remains to show that $\mathfrak{P}^{\lambda}_t$ is continuous (i.e., upper- and lower-hemicontinuous). To that end, consider an arbitrary sequence
    \(
    ((s_{n}, a_{n}, \mu_{n}), \P_{n})_{n \in \N} \subseteq \operatorname{Gr}(\mathfrak{P}_t^{\lambda})
    \)
    such that $(s_{n}, a_{n}) \to (s, a)$ and $  \mu_{n} \rightharpoonup \mu$ as $n \to \infty$.

    Since $S$ and $A$ are finite, there exists $N \in \N$ such that for every $n \geq N$ it holds that $(s_{n}, a_{n}, \mu_{n}) = (s, a, \mu_{n})$. Hence, $\P_{n} \in \mathfrak{P}_t^{\lambda}(s, a, \mu_{n})$ for every $n \geq N$. Furthermore, for every $n \geq N$, $\P_{n}(1) \in [l(\mu_{n}), u(\mu_{n})] \subseteq [0, 1]$.
    Since $[0, 1]$ is compact, there exists a subsequence $\left(\P_{n_{k}}(1)\right)_{k \in \N} \subseteq \left(\P_{n}(1)\right)_{n \in \N}$ with $\P_{n_{k}}(1) \rightarrow q \in [0, 1]$ as $k \to \infty$.

    Suppose that $(s, a) = (0, 1)$ and again consider $l(\nu) := \max(0, p(\nu) - \lambda)$, $u(\nu) := \min(1, p(\nu) + \lambda)$, and $p(\nu) := p^{o}(1 \,|\, 0, 1, \nu) = 1 - \beta \nu(1)$. Note that $p$ and thus $l$ and $u$ are continuous. Hence, for all $k \in \N$ with $n_{k} \geq N$, $l(\mu_{n_{k}}) \leq \P_{n_{k}}(1) \leq u(\mu_{n_{k}})$ and thus $l(\mu) \leq q \leq u(\mu)$. Choose $\P \in \mathcal{P}(S)$ such that $\P(1) = q$. Then $\P_{n_{k}} \rightharpoonup \P \in \mathfrak{P}_{t}^{\lambda}(s, a, \mu)$. The other cases follow analogously. Thus, by \cite[Theorem 17.20]{CharalambosKim2006infinite}, $\mathfrak{P}_t^{\lambda}$ is upper-hemicontinuous.

    Again, consider an arbitrary sequence $((s_{n}, a_{n}, \mu_{n}))_{n \in \N} \subset S \times A \times \mathcal{P}(S)$ such that $(s_{n}, a_{n}) \to (s, a)$ and $\mu_n\rightharpoonup \mu$ as $n \to \infty$ and let $\P \in \mathfrak{P}_t^{\lambda}(s, a, \mu)$. Since $S = \{0, 1\}$, we can write $\P = \mathrm{Bern}(\theta)$ for some $\theta \in [l(\mu), u(\mu)]$. As before, there exists $N \in \N$ such that for every $n \geq N$ it holds that $(s_{n}, a_{n}, \mu_{n}) = (s, a, \mu_{n})$. Define
    \begin{equation*}
        \mathbb{P}_n :=
        \begin{cases}
            p^o(\cdot \mid s_n,a_n,\mu_n), & n < N, \\
            \mathrm{Bern}\bigl(\min\{\max\{l(\mu_n),\theta\}, u(\mu_n)\}\bigr), & n\geq N.
        \end{cases}
    \end{equation*}
    Then $\P_{n} \in \mathfrak{P}_t^{\lambda}(s_{n}, a_{n}, \mu_{n})$ for all $n \in \N$. Moreover, since $l$ and $u$ are continuous and $l(\mu) \leq \theta \leq u(\mu)$, we obtain $\P_{n} \rightharpoonup \P$ as $n \to \infty$. Thus, by \cite[Theorem 17.21]{CharalambosKim2006infinite}, $\mathfrak{P}_t^{\lambda}$ is lower-hemicontinuous. Hence, $\mathfrak{P}_{0:T}^\lambda$ satisfies Assumption \ref{as:msr}\;(ii), as claimed.

    We now claim that $r$ given in Definition \ref{dfn:exm:tra}\;(ii) satisfies Assumption\;\ref{as:msr}\;(iii). Since $|\hat s| \leq 1$ and $|a| \leq 1$ for every $(\hat{s},a)\in S\times A$ (noting that $S = \{0, 1\}$ and $A= \{-1, 0, 1\}$; Definition \ref{dfn:exm:tra}), there exists a constant $C_{r} := R + C_{1} + C_{2} > 0$ satisfying that for every $s, \hat s \in S$, $a \in A$, and $\mu \in \mathcal{P}(S)$,
    \begin{align*}
        \lvert r(s, a, \hat s, \mu) \rvert = \big\lvert \big(R - C_{1} \mu(1)\big) \hat{s} - C_{2} \lvert a \rvert \big\rvert \leq R + C_{1} + C_{2} = C_r.
    \end{align*}
    Moreover, there exists $L_{r} := C_{1} > 0$ satisfying that for every $s, \hat s \in S$, $a \in A$, and $\mu, \hat \mu \in \mathcal{P}(S)$,
    \begin{align*}
        \lvert r(s, a, \hat s, \mu) - r(s, a,  \hat s, \hat {\mu}) \rvert &= \big\lvert \big(R - C_{1} \mu(1)\big) \hat{s} - C_{2} \lvert a \rvert - \big(R - C_{1} \hat{\mu}(1)\big) \hat{s} + C_{2} \lvert a \rvert \big\rvert \\
        &= C_{1} \hat{s} \lvert \hat{\mu}(1) - \mu(1) \rvert \leq L_r d_{W_{1}}(\mu, \hat{\mu}).
    \end{align*}
    Hence, $r$ satisfies Assumption\;\ref{as:msr}\;(iii), as claimed.
\end{proof}

\bibliographystyle{abbrv}
\bibliography{references}

\end{document}